\documentclass[12pt]{amsart}

\usepackage[T1]{fontenc}
\usepackage{imakeidx}
\usepackage{yfonts}
\usepackage{amssymb}
\usepackage{fancyhdr}
\usepackage{amsmath}
\usepackage{amsfonts}
\usepackage{slashed} 
\usepackage[utf8]{inputenc}
\usepackage{mathrsfs}
\usepackage[all]{xy} 
\usepackage{graphicx} 
\usepackage{xcolor}
\usepackage{enumitem}
\usepackage{glossaries}

\newcommand{\disk}{\ensuremath{\mathbb{D}} } 
\newcommand{\sphere}{\bar{\Bbb{C}}} 
\newcommand{\riem}{\Sigma}  
\renewcommand{\Bbb}[1]{\ensuremath{\mathbb{#1}}}

\newcommand{\R}{\mathbb{R}}

\newglossaryentry{bergman}{%
name=\ensuremath{\mathcal{A}},
    description={Bergman space}
}

\newglossaryentry{Abw}{%
name=\ensuremath{\mathcal{A}_\mathrm{bw}},
    description={bridgeworthy harmonic oneforms}
}

\newglossaryentry{Aharm}{%
name=\ensuremath{\mathcal{A}_{\mathrm{harm}}},
    description={Harmonic Bergman space}
}

\newglossaryentry{Ahm}{%
name=\ensuremath{\mathcal{A}_{\mathrm{hm}}},
    description={Complex linear span of harmonic measures}
}

\newglossaryentry{exactA}{%
name=\ensuremath{\mathcal{A}^\mathrm{e}},
    description={Space of exact forms}
}

\newglossaryentry{seform}{%
name=\ensuremath{\mathcal{A}^{{\mathrm{se}}}},
    description={Semi-exact forms}
}

\newglossaryentry{peform}{%
name=\ensuremath{\mathcal{A}_{\mathrm{harm}}^{\mathrm{pe}}},
    description={Piecewise exact harmonic forms}
}

\newglossaryentry{annulus}{%
name=\ensuremath{\mathbb{A}_{a,b}},
    description={Annulus with inner radius $a$ and outer radius $b$}
}

\newglossaryentry{gota}{%
name=\ensuremath{ \mathrm{\textgoth{A}} },
    description={Forms with prescribed periods}
}

\newglossaryentry{Bphi}{%
name=\ensuremath{ \mathbf{B}(\phi)},
    description={Boundary map}
}

\newglossaryentry{cf}{%
name=\ensuremath{\mathbf{C}_{f}},
    description={right-composition with $f$}
}

\newglossaryentry{cl}{%
name=\ensuremath{\text{cl}},
    description={Closure of a set}
}

\newglossaryentry{D}{%
name=\ensuremath{\mathcal{D}},
    description={Dirichlet space}
}

\newglossaryentry{Dbw}{%
name=\ensuremath{\mathcal{D}_\mathrm{bw}},
    description={bridgeworthy harmonic functions}
}

\newglossaryentry{Dir}{%
name=\ensuremath{\mathbf{Dir}},
    description={Solution map to the Dirichlet problem}
}

\newglossaryentry{Dharm}{%
name=\ensuremath{\mathcal{D}_{\mathrm{harm}}},
    description={Harmonic Dirichlet space}
}

\newglossaryentry{Dhom}{%
name=\ensuremath{\dot{\mathcal{D}}},
    description={Dirichlet space modulo constants}
}

\newglossaryentry{harmeasure}{%
name=\ensuremath{d\omega_{k}},
    description={Harmonic measure}
}

\newglossaryentry{E}{%
name=\ensuremath{\mathbf{E}},
    description={Data to solution map}
}

\newglossaryentry{greenb}{%
name=\ensuremath{g},
    description={Green's function of a bordered Riemann surface}
}
\newglossaryentry{greenc}{%
name=\ensuremath{\mathscr{G}},
    description={Green's function of a compact Riemann surface}
}

\newglossaryentry{bounce}{%
name=\ensuremath{\mathbf{G}_{U,\riem}},
    description={Bounce operator}
}

\newglossaryentry{grunsk}{%
name=\ensuremath{\mathbf{Gr}_{f}},
    description={Grunsky operator}
}

\newglossaryentry{sobolev}{%
name=\ensuremath{H^s},
    description={Sobolev space}
}
\newglossaryentry{homsobolev}{%
name=\ensuremath{\dot{H}^s},
    description={Homogeneous Sobolev space}
}

\newglossaryentry{sobolevconf}{%
name=\ensuremath{{H}^{1}_{\mathrm{conf}}},
    description={Conformal Sobolev space}
}

\newglossaryentry{Dbvaluescomp}{%
name=\ensuremath{\mathcal{H}'(\partial_k \riem)},
    description={Dirichlet boundary values for one forms}
}

\newglossaryentry{Dbvalues}{%
name=\ensuremath{\mathcal{H}'(\partial\riem)},
    description={Dirichlet boundary values for one forms}
}

\newglossaryentry{BVexactcomp} {%
name=\ensuremath{\dot{H}'(\partial_k \riem)},
    description={Boundary values with exact representative}
}

\newglossaryentry{BVexact} {%
name=\ensuremath{\dot{H}'(\partial \riem)},
    description={Boundary values with exact representative}
}

\newglossaryentry{crop}{%
name=\ensuremath{\mathbf{J}_{1}^{q}},
    description={Cauchy-Royden operator}
}
    
\newglossaryentry{rcrop}{%
name=\ensuremath{\mathbf{J}_{1,k}^{q}},
    description={Restricted Cauchy-Royden operator}
}
\newglossaryentry{Jdot}{%
name=\ensuremath{\dot{\mathbf{J}}_1},
    description={Cauchy-Royden operator on $\dot{\mathcal{D}}$}
}

\newglossaryentry{bergmank}{%
name=\ensuremath{K},
    description={Bergman kernel}
}
\newglossaryentry{schifferk}{%
name=\ensuremath{L},
    description={Schiffer kernel}
}
\newglossaryentry{rest}{%
name=\ensuremath{\mathbf{R}},
    description={Restriction operator}
}

\newglossaryentry{harmrest}{%
name=\ensuremath{\mathbf{R}^{\mathrm{h}}},
    description={Harmonic restriction operator}
}

\newglossaryentry{S}{%
name=\ensuremath{\mathbf{S}},
    description={The Schiffer comparison operator}
}
  
\newglossaryentry{Sharm}{%
name=\ensuremath{\mathbf{S}_{k}^{\mathrm{h}}},
    description={Harmonic Schiffer operator}
}
  
\newglossaryentry{theta}{%
name=\ensuremath{\Theta},
    description={Map}
}

\newglossaryentry{Tmix}{%
name=\ensuremath{\mathbf{T}_{\riem_{j},\riem_{k}}},
    description={Schiffer operator}
}
\newglossaryentry{T}{%
name=\ensuremath{\mathbf{T}},
    description={Schiffer comparison operator}
}

\newglossaryentry{overfare}{%
name=\ensuremath{\mathbf{O}},
    description={Overfare operator}
}
\newglossaryentry{doto}{%
name=\ensuremath{\dot{\mathbf{O}}},
    description={Overfare operator on $\dot{\mathcal{D}}$}
}

\newglossaryentry{exacto}{%
name=\ensuremath{\mathbf{O}^{\mathrm{e}}_{2,1}},
    description={Exact overfare operator}
}
    
\newglossaryentry{ohat}{%
name=\ensuremath{\hat{\mathbf{O}}},
    description={Operator}
}

\newglossaryentry{augo}{%
name=\ensuremath{\mathbf{O}^{\mathrm{aug}}},
    description={Augmented overfare operator}
}
    
\newglossaryentry{oprime}{%
name=\ensuremath{\mathbf{O}'},
    description={Operator}
}
    
\newglossaryentry{oprimedot}{%
name=\ensuremath{\dot{\mathbf{O}}'},
    description={Operator}
}
    
\newglossaryentry{pcap}{%
name=\ensuremath{\mathbf{P}_{\mathrm{cap}}}, description={Projection operator}
}

\newglossaryentry{period}{%
name=\ensuremath{\mathbf{\Upsilon}},
    description={Period map}
}

\theoremstyle{plain}
        \newtheorem{theorem}{Theorem}[section]
        \newtheorem{lemma}[theorem]{Lemma}
        \newtheorem{proposition}[theorem]{Proposition}
        \newtheorem{corollary}[theorem]{Corollary}

\theoremstyle{definition}
        \newtheorem{definition}[theorem]{Definition}
        \newtheorem{example}{Example}[section]

\theoremstyle{remark}
    \newtheorem{remark}[theorem]{Remark}

\numberwithin{equation}{section} 
\numberwithin{figure}{section} 

\usepackage{fullpage}
\author[E. Schippers]{Eric Schippers}
\author[W. Staubach]{Wolfgang Staubach}

\address{\newline
       Eric Schippers \newline
       Machray Hall, Dept. of Mathematics,
   University of Manitoba, \newline Winnipeg, MB
   Canada R3T 2N2}
       \email{eric.schippers@umanitoba.ca}

\address{\newline
       Wolfgang Staubach \newline
       Department of  Mathematics, Uppsala University, \newline
       S-751 06 Uppsala, Sweden}
       \email{wulf@math.uu.se}
       

 
 \title{A scattering theory of harmonic one-forms on Riemann surfaces}

\makenoidxglossaries{}
\begin{document}
\begin{abstract}
We construct a scattering theory for harmonic one-forms on Riemann surfaces, obtained from boundary value problems through systems of curves and the jump problem. We obtain an explicit expression for the scattering matrix in terms of integral operators which we call Schiffer operators, and show that the matrix is unitary. As a consequence of this scattering theory, we prove index theorems relating these conformally invariant integral operators to topological invariants. We also obtain a general association of positive polarizing Lagrangian spaces to bordered Riemann surfaces, which unifies the classical polarizations for compact surfaces of algebraic geometry with the infinite-dimensional period map of the universal Teichm\"uller space. 
\end{abstract}

\maketitle

\tableofcontents

\begin{section}{Introduction}
\begin{subsection}{Statement of results and literature}

In this paper, we develop a theory of scattering of $L^2$ harmonic one-forms on Riemann surfaces.  The scattering takes place in a network of curves which separate the Riemann surface in at least two connected components.  The process is as follows. Let $\mathscr{R}$ be a compact surface divided by a complex $\Gamma$ of simple closed curves into surfaces $\riem_1$ and $\riem_2$.  The number of curves is arbitrary, and we allow $\riem_1$ or $\riem_2$ to be disconnected. The reader may find it helpful to first imagine the case that $\riem_1$ and $\riem_2$ are connected and separated by $n$ closed curves. 
Given a harmonic function $h_1$ on $\riem_1$, it has boundary values on $\Gamma$, which in turn uniquely determine a harmonic function $h_2$ on $\riem_2$ with the same boundary values. We call $h_2$ the ``overfare'' of $h_1$ and write $h_2=\mathbf{O}_{1,2}h_1$. 

For harmonic one-forms, there is a similar overfare procedure. Briefly, one finds an anti-derivative of a form $\alpha_1$ on $\riem_1$, applies the overfare $\mathbf{O}_{1,2}$ to the anti-derivative, and differentiates the result to obtain a form $\alpha_2$ on $\riem_2$. Of course, $\alpha_1$ need not be exact, and one must also specify the cohomological properties of the form $\alpha_2$. We deal with this by specifying a harmonic one-form $\zeta$ on $\mathscr{R}$ such that $\alpha_1 - \zeta$ is exact on $\riem_1$, and let $\alpha_2$ be such that $\alpha_2 - \zeta$ is exact on $\riem_2$. Thus the extra cohomological data required to specify the overfare of harmonic one-forms is identified with the finite-dimensional vector space of harmonic one-forms on $\mathscr{R}$. In general, the overfared harmonic one-form is not harmonic on the union.\\

In analogy with potential-well scattering on the real line, we can regard the aforementioned $\alpha_2$ as the form obtained from $\alpha_1$ through scattering. In this scattering process, the curves themselves play the role of the potential well. We assume only that the curves are quasicircles, which generically are non-rectifiable curves arising in Teichm\"uller theory. The holomorphic and anti-holomorphic parts play the role of the left- and right- moving solutions, and the  asymptotic negative and positive directions are played by the two surfaces. The majority of the results of this work are directly related to the problem of developing various aspects of this scattering theory, including the unitarity of the scattering matrix.

We also apply this scattering theory to derive new results in the geometry of Riemann surfaces, for example index theorems for conformally invariant operators, and a generalization of polarizations to Riemann surfaces with boundary which incorporate boundary values.\\  

We state our main results, emphasizing their geometric or analytic nature. Expanded statements, together with background and literature, will be given in separate sections ahead. \\

 \noindent    {\bf Geometric Results:} We obtain
    \begin{enumerate} 
    \item an explicit expression for the scattering matrix for harmonic one-forms in terms of the Schiffer operators, and that it is unitary;\\
    \item an association of positive polarizing Lagrangian subspaces to  bordered Riemann surfaces, which unifies the classical polarizations for compact surfaces with the infinite-dimensional Kirillov-Yuri'ev-Nag-Sullivan embedding of the universal Teichm\"uller space into a Lagrangian Grassmannian; \\
    \item index theorems for conformally invariant integral operators related to the Riemann jump problem on $\mathscr{R}$ (which we call Schiffer operators), relating conformal invariants to topological invariants.\\ 
\end{enumerate} 

\noindent The results above require the following. \\

\noindent {\bf Analytic Results:} We prove that
\begin{enumerate}\setcounter{enumi}{3}
    \item the boundary values of $L^2$ harmonic one-forms on a genus $g$ surface with $n$ borders, in a certain non-tangential sense, is the Sobolev $H^{-1/2}$ space;\\
    
    \item conversely, the Dirichlet problem for one forms with $H^{-1/2}$ boundary values is well-posed, and the solutions are $L^2$ harmonic one-forms;\\
    \item the overfare of harmonic functions is bounded in the following two cases:
     \begin{enumerate}
        \item for quasicircles, with respect to the Dirichlet semi-norm when the originating surface is connected, and 
         \item for more regular curves, with respect to a conformally invariant norm extending the Dirichlet semi-norm. \\
     \end{enumerate}
    \end{enumerate}
    

{We prove these theorems in a very general analytic setting, which in the case at a hand, amounts to the assumption that the curve complex dividing the Riemann surface consists of a collection of quasicircles. Also, we use $L^2$ harmonic one-forms and Dirichlet-bounded harmonic functions throughout.} 

{At first glance, one might think that the point of this manuscript could be made by developing the scattering theory with stronger regularity assumptions (say smooth curves and forms). However there are good reasons for the choices that have been made here in this paper.} Two of these are geometric: all constructions are conformally invariant, and our analytic choices are necessary for applications to the Teichm\"uller theory. For example, an obvious next step is to show that the generalized period mapping yields coordinates on Teichm\"uller space; to do so will require both the choice of quasicircles and of $L^2$ harmonic one-forms. In the long run, the investigation of geometric structures on Teichm\"uller space (and its refinement, the Weil-Petersson class Teichm\"uller space) will require the theory on quasicircles. This will also be the case for the study of the symplectic actions by groups of boundary re-parametrizations. Another related motivation for considering quasicircles is a theorem of K. Vodopy'anov \cite{Vodopyanov}  and S. Nag-D. Sullivan \cite{NS}, that shows that the reparametrizations act by bounded symplectomorphisms precisely for quasisymmetric reparametrizations. 

Applicability to geometry aside, the conditions are analytically natural. This can be seen even in the plane, where for example it can be shown that overfare exists and is bounded if and only if the curve is a quasicircle.  See \cite{Schippers_Staubach_Grunsky_expository} which gives a strong case for the analytic naturality of these conditions.  It is remarkable that the conditions which are natural from the point of view of analysis, geometry, and algebra all coincide.\\  

The main results are described in the sections below.

\begin{subsubsection}{Overfare of harmonic functions} 

As described above, the process of overfare is as follows. Let $\mathscr{R}$ be a Riemann surface split into two pieces $\riem_1$ and $\riem_2$ by a Jordan curve or complex of curves. Given a harmonic function with $L^2$ derivatives on one of the pieces $\riem_1$ (a Dirichlet harmonic function), we find its boundary values. The ``overfare'' is the harmonic function on the other piece $\riem_2$ with the same boundary values as the original function. This is well-defined and bounded provided that the curves in the complex are quasicircles. 

Here, there are two analytic problems to be resolved. The first is to define the boundary values in preparation for overfare, and the second is to show the existence and continuous dependence of the overfare. The first problem is in a certain sense independent of the boundary regularity, while the second problem is more delicate and sensitive to the regularity of the curve. \\

In defining the boundary values, the nature of the approach to the boundary can be defined either extrinsically in terms of the geometry of the ambient space containing the curve, or in terms of the intrinsic geometry of the region on which the function is defined.  For example, since harmonic functions with $L^2$ derivatives are in the Sobolev space $H^1$ for a wide class of curves, one could consider the Sobolev trace to the boundary; in this case, one would need to take into account the regularity of the boundary for this to be defined. {The possibility of dealing with boundaries that may not be rectifiable would add additional difficulties that brings one into the realm of geometric measure theory see \cite{J}, \cite{JW}. Instead, our approach to boundary values proceeds intrinsically, in such a way that the boundary} can be viewed as the ideal boundary of $\riem_1$, which does not depend on the geometry of the boundary in $\mathscr{R}$. For example, it can be regarded as an analytic Jordan curve in the double of $\riem_1$. 

Our intrinsic approach to boundary values in some sense originates with H. Osborn \cite{Osborn}, who considered the boundary values of harmonic Dirichlet functions in planar domains $\riem_1$ along orthogonal trajectories of Green's function of that domain. This is conformally invariant and hence intrinsic, and can be formulated in terms of the ideal boundary. We improve this ``radial'' approach by defining a kind of conformally non-tangential boundary value (referred to as CNT boundary values), in which non-tangential cones are defined in terms of ``collar charts'' taking collar neighbourhoods of the boundary to annuli. Then, a classical theorem of A. Beurling applies to show that the boundary values exist except on a Borel set of logarithmic capacity zero in the circle under the chart (we call this a null set). We show that this notion of boundary value is independent of the choice of collar chart; this is essentially because the angle of approach to the ideal boundary is a well-defined conformal invariant. Thus we show that the boundary values are defined not just along orthogonal trajectories of Green's function but along any non-tangentially approaching curve. The independence of the boundary values on the choice of collar chart is a key tool in the application of the cutting and sewing approach to boundary value problems which we have developed in this and other papers \cite{Schippers_Staubach_transmission}, \cite{Schippers_Staubach_Plemelj}.\\ 

On the other hand, the overfare process is extrinsic, because the regularity of the boundary curve is crucial. We work with quasicircles; there are several reasons for this choice. The first is geometric: at a foundational level, Teichm\"uller theory of bordered surfaces involves viewing these surfaces as subsets of compact surfaces bounded by quasicircles. Classically, this is seen in the quasi-Fuschsian model of Teichm\"uller space \cite{Nagbook}; for example, the universal Teichm\"uller space can be viewed as the set of (normalized) planar domains bounded by quasicircles. The first author's work with D. Radnell \cite{RadnellSchippers_monster}, \cite{RS_fiber} also shows that the Teichm\"uller space can be modelled as the set of surfaces capped by domains bounded by quasicircles, and that this leads to a natural fibre structure on Teichm\"uller space. Thus, in this work, we choose quasicircles in order to have sufficient generality in order to provide the groundwork for applying our results to Teichm\"uller theory. 

The second reason for choosing quasicircles is analytic. The authors showed in \cite{Schippers_Staubach_transmission_sphere} that in the Riemann sphere, the overfare exists and is bounded precisely for quasicircles. This follows from a theorem of Nag-Sullivan/Vodopy'anov that shows that quasisymmetries are precisely the bounded composition operators on the {homogeneous Sobolev space $\dot{H}^{1/2}$} on the circle. As we will see ahead, this also relates to several characterizations of quasicircles in terms of the Cauchy-type and Schiffer integral operators which play the main role in this paper. A survey of such results in the Riemann sphere can be found in \cite{Schippers_Staubach_Grunsky_expository}. \\

{It should also be noted that the Sobolev theory techniques by themselves are not sufficient in dealing with all aspects of the boundary value problems that are involved in this paper, since Sobolev spaces involve functions defined up to sets of Lebesgue measure zero.  In fact, one needs to establish that boundary values exist up to a set which maps under a collar chart to a Borel set of logarithmic capacity zero in the unit circle. We call such sets null sets. By our earlier results, for quasicircles, a set which is null with respect to a collar chart on one side of the curve must be null with respect to a collar chart on the other side. This fact is central to establishing a well-defined overfare of harmonic functions. However, the claim fails if in the discussion above one replaces capacity zero with Lebesgue measure zero on the circles. Thus Sobolev theory on its own is not sufficient.  } \\  

In this paper, we extend our previous overfare results to Riemann surfaces divided by many curves, rather than just a single curve. There is an obstacle to doing so. If the region $\riem_2$ is bounded by several curves, but $\riem_1$ is not connected, then the Dirichlet semi-norm is not controlled by the Dirichlet norm of the input. This is because one may add different constants to different connected components of $\riem_1$, driving up the semi-norm of the overfare, while the Dirichlet norm on the originating surface is unchanged. If the originating surface is connected, this issue does not arise, and we are able to prove boundedness of overfare with respect to the Dirichlet seminorm.  

One can also obtain boundedness with respect to a genuine norm if more regularity is assumed. We introduce a conformally invariant norm: rather than adding the $L^2$ norm of the function as in Sobolev theory, we add an integral of the function around a boundary curve. With no connectivity assumptions, we obtain boundedness of overfare with respect to this conformally invariant norm, for curves with greater regularity. It suffices that the quasicircles are so-called Weil-Petersson quasicircles.  
For both of these results, in this paper we use a more flexible method of proof than in \cite{Schippers_Staubach_transmission}, and make systematic use of boundedness of the so-called bounce operator (see Definition \ref{defn:bounce op}).
\end{subsubsection}

\begin{subsubsection}{Dirichlet boundary value problem for $L^2$ one-forms  boundary values} {{A classical formulation of the Dirichlet problem on Riemannian manifolds with smooth boundary is as follows:\\

Let $M$ be a smooth, connected, compact, Riemannian manifold of real dimension $m$ and consider some arbitrary smooth domain  $\Omega \subseteq M$ with non-empty boundary. Assume that $f \in L^{2}\left(\partial \Omega, \wedge^{k} T M\right)$, $0\leq k\leq m,$ where $L^{2}\left(\partial \Omega, \wedge^{k} T M\right)$ denotes the space of $k$-forms which are $L^2$ on the boundary of $\Omega$. Denoting the Hodge Laplacian by $\Delta= d\delta+ \delta d$ (where $d$ is the exterior differentiation and $\delta$ its adjoint with respect to the Riemannian metric of $M$), 
the Dirichlet boundary value problem with boundary data $f$ is

\begin{equation}\label{dp for hodge}
\left\{\begin{array}{l}u \in C\left(\Omega, \wedge^{k} T M\right) \\ \Delta u=0 \text { in } \Omega  
\\ \left.u\right|_{\partial \Omega}=f \text { on } \partial \Omega \end{array}\right.
\end{equation}
For $0\leq k\leq m$, this problem was studied by G. Duff and D. Spencer \cite{DS1}, \cite{DS2}, \cite{Du3}, \cite{Sp},
C. Morrey and J. Eells \cite{ME1}, \cite{ME2}, and G. Schwarz \cite{Sch}. Through these investigations, it is known that for any $f \in L^{2}\left(\partial \Omega, \wedge^{k} T M\right)$ the Dirichlet problem has a unique solution $u \in H^{1/2}\left(\Omega, \wedge^{k} T \mathcal{M}\right)$ (Sobolev $\frac{1}{2}$-space), and moreover there exists $C>0$ independent of $f$ such that
\begin{equation}\label{mmt estim}
\|u\|_{H^{1/2}\left(\Omega, \wedge^{k} T M\right)} \leq C\|f\|_{L^{2}\left(\partial \Omega, \wedge^{k} T M\right)}.
\end{equation}
Another well-known fact is that if $k=0$, $u\in L^2(\Omega)$ and $\Delta u \in L^2(\Omega)$ then $u|_{\partial \Omega} \in H^{-1/2}(\partial \Omega).$\\

In this paper we investigate the well-posedness of \eqref{dp for hodge} when $k=1$ and $f$ in the Sobolev space of forms  $H^{-1/2}\left(\partial \Sigma, \wedge^{k} T \Sigma\right)$, where $\Sigma$ is a bordered Riemann surface. This amounts to the demonstration of the fact that for an element of $f\in H^{-1/2}$ together with sufficient cohomological data, there always exists a unique $u\in L^2$ harmonic one-form on $\riem$ with boundary value $f$. We also show that $u$ depends continuously on $f$, i.e. the analogue of \eqref{mmt estim} is valid in this setting.  
 
 The problem for $H^{-1/2}$ boundary values is solved by reformulating the $H^{-1/2}$-space conformally invariantly, and using the theory of CNT boundary values, mentioned above. That is, we show that elements of $H^{-1/2}$ can be represented by equivalence classes of $L^2$ harmonic one-forms defined in collar  neighbourhoods. Using the fact that $H^{-1/2}$ is the dual space to $H^{1/2}$, we will show that there is a one-to-one correspondence between elements of $H^{-1/2}$ and such equivalence classes, and this allows us to use the theory of conformally nontangential boundary values to solve the problem. It turns out that anti-derivatives of such forms have well-defined boundary values in the conformally nontangential sense, which after removing a period, can be identified with elements of $H^{1/2}$. In this context, the so-called {\it anchor lemmas} (Lemmas  \ref{le:anchor_lemma_one} and \ref{le:anchor_lemma_two}) are of fundamental importance since they imply that the limiting integral of $f \in H^{1/2}$ against any $\alpha \in \mathcal{A}(A)$ ($\mathcal{A}(A)$ is the Bergman space of holomorphic one forms on $A$, and $A$ is a collar neighbourhood of the boundary) exists and depends only on the CNT boundary values of $f$.}}

\end{subsubsection}

\begin{subsubsection}{Calculus of Schiffer operators, cohomology, and index theorems}

The cornerstone of this paper is the theory of certain integral operators of Schiffer. These integral operators are integral operators on holomorphic and anti-holomorphic one-forms, whose integral kernels are the two possible second derivatives of Green's function, often called the Bergman and Schiffer kernels. These are defined as follows. Let $\mathscr{R}$ be a compact Riemann surface  split into two surfaces $\riem_1$ and $\riem_2$ by a collection of Jordan curves. Let $\mathscr{G}(w;z,q)$ be Green's function of $\mathscr{R}$ (the fundamental harmonic function with logarithmic singularities at $z$ and $q$ of opposite weight, defined up to an additive constant).  We have, denoting the Bergman space of holomorphic one-forms on $\riem_k$ by $\mathcal{A}(\riem_k)$ for $k=1,2$, 
\begin{align*}
  \mathbf{T}_{1,k}:\overline{\mathcal{A}(\riem_1)} & \rightarrow \mathcal{A}(\riem_k) \\
  \overline{\alpha} & \mapsto \iint_{\riem_1} \partial_w \partial_z \mathscr{G}(w;z,q) \wedge_w \overline{\alpha(w)}.  
\end{align*} 
The two choices of $k$ are obtained by restricting $z$ to $\riem_k$. 
If $k=1$, this has a singularity and can be treated as a Calder\'on-Zygmund singular integral operator. We also have the operator
\begin{align*}
    \mathbf{S}_{1}: {\mathcal{A}(\riem_1)} & \rightarrow \mathcal{A}(\mathscr{R}) \\
  \overline{\alpha} & \mapsto \iint_{\riem_1} \overline{\partial}_w \partial_z \mathscr{G}(w;z,q) \wedge_w \overline{\alpha(w)}. 
\end{align*}
We may of course switch the roles of $1$ and $2$ above.
These were investigated extensively by M. Schiffer with various co-authors \cite{BergmanSchiffer} \cite{Schiffer_first}, in relation to potential theory and conformal mapping,  eventually culminating in a comparison theory of domains \cite{Courant_Schiffer}. The Schiffer kernel is closely related to the so-called fundamental bidifferential and figures in geometry of function spaces on Riemann surfaces  \cite{Eynard_notes}, \cite{Schiffer_Spencer}.  

By a striking result of V. Napalkov and R. Yulmukhametov \cite{Nap_Yulm}, if $\mathscr{R}$ is the Riemann sphere, and $\riem_1$ and $\riem_2$ are the two complementary components of a Jordan curve $\Gamma$ on the sphere, then the Schiffer operator $\mathbf{T}_{1,2}$ is an isomorphism if and only if $\Gamma$ is a quasicircle. {This is closely related to the fact that functions can be approximated in the Dirichlet semi-norm by Faber series precisely for domains bounded by quasicircles; see \cite{Schippers_Staubach_Grunsky_expository} for an overview.} The authors showed in \cite{Schippers_Staubach_Plemelj} that, for a compact Riemann surface divided in two by a quasicircle,   $\mathbf{T}_{1,2}$ is an isomorphism on the orthogonal complement of anti-holomorphic one-forms on $\mathscr{R}$. This was further generalized by M. Shirazi to the case of many curves where all but one of the components is simply connected in \cite{Shirazi_thesis}, \cite{Schippers_Shirazi_Staubach}.  The boundedness of overfare plays a central role in the formulation and proof of this fact.
{This extension of the isomorphism theorem was used by the authors and Shirazi to show that one-forms on a domain in a Riemann surface bounded by quasicircles can be approximated in $L^2$ on a larger domain \cite{Schippers_Shirazi_Staubach}. Approximability theorems  for general $k$-differentials with respect to the conformally invariant $L^2$ norm and less regular boundaries were obtained by N. Askaripour and T. Barron \cite{AskBar,AskBar2} using very different methods. So far as we know, these were the first results for nested domains on Riemann surfaces in the $L^2$ setting}.
\\

In this paper, we characterize the kernel and image of $\mathbf{T}_{1,2}$ in the case of a Riemann surface split by a complex of quasicircles. The main tool is an extended Plemelj-Sokthoski jump formula, which is in turn based on a relation between the Schiffer operators and a generalization of the Cauchy operator originating with H. Royden \cite{Royden} which we call the Cauchy-Royden operator. As quasicircles are not rectifiable, we are required to define the Cauchy-Royden integral using curves which approach the boundary. In the sphere with one curve, the authors showed that the resulting Plemelj-Sokhotski jump decomposition is an isomorphism if and only if the curve is a quasicircle.  The analytic issues in those papers, as in this one, are resolved by the fact that the limiting integral is the same from both sides up to constants. This in turn is a consequence of the anchor lemmas and boundedness of the bounce operator. The equality of the limiting integral from both sides is also a key geometric tool; in combination with the bounded overfare it allows one to find preimages of elements of the image of $\mathbf{T}_{1,2}$. \\

We further use this  to investigate the cohomology of the images of $\mathbf{T}_{1,1}$, $\mathbf{T}_{1,2}$ and $\mathbf{S}$. In particular we show that for any anti-holomorphic one form $\overline{\alpha}$ in $\riem_1$, $\mathbf{T}_{1,2} \overline{\alpha}$ and $\overline{\mathbf{S}}_1 \overline{\alpha}$ are in the same cohomology class. This simple fact is surprisingly versatile. Along with the characterization of the kernels and images of $\mathbf{T}_{1,2}$ mentioned above, we also show that in the case that $\riem_1$ and $\riem_2$ are connected, the Fredholm index of $\mathbf{T}_{1,2}$ is $g_1-g_2$ where $g_1$ and $g_2$ are the genuses of $\riem_1$ and $\riem_2$. This index theorem relates a conformal invariant (the index of $\mathbf{T}_{1,2}$) to the topological invariant $g_1-g_2$. \\

Finally, we derive a number of new identities for Schiffer operators and their adjoints, as well as extend identities obtained earlier in \cite{Schippers_Staubach_Plemelj} to the case of a compact surface split by a complex of curves.
These identities play a central role in the scattering theory. It should be mentioned that one of these identities is a reformulation and significant generalization of an norm identity of Bergman and Schiffer for planar domains \cite{BergmanSchiffer}. This identity can be used to derive the Grunsky inequalities (see ahead). 
\end{subsubsection}
\begin{subsubsection}{Scattering matrix and unitarity}
 
 We define a scattering process for one-forms in the following way. The overfare process defined above for functions uniquely defines the overfare of exact one-forms from connected surfaces to arbitrary ones, by 
 \begin{align*}
  \mathbf{O}^{\mathrm{e}}_{\riem_1,\riem_2}: \mathcal{A}^{\mathrm{e}}_{\mathrm{harm}}(\riem_1) & \rightarrow \mathcal{A}^{\mathrm{e}}_{\mathrm{harm}}(\riem_2) \\
  \alpha & \mapsto d \mathbf{O}_{\riem_1,\riem_2} d^{-1}
 \end{align*}
 where $\mathbf{O}_{\riem_1,\riem_2}$ is overfare of harmonic functions and $\mathcal{A}_{\mathrm{harm}}(\riem)$ denotes $L^2$ harmonic one-forms on $\riem$.  For arbitrary one-forms on a connected surface, we specify the cohomological data as follows: let $\zeta \in \mathcal{A}_{\mathrm{harm}}(\mathscr{R})$ be a one-form such that $\alpha - \zeta$ is exact on $\riem_1$.  We seek a one-form with the same boundary values as $\alpha$ and in the cohomology class of $\zeta$ on $\riem_2$. This form is
 \[  \mathbf{O}^{\mathrm{e}}_{\riem_1,\riem_2} \left( \alpha - \left. \zeta \right|_{\riem_1} \right) + \left. \zeta \right|_{\riem_2}.   \]
 We call $\zeta$ a ``catalyzing form'', and forms which are related by overfare via $\zeta$ compatible.
 
 From this overfare process we define a scattering operator which takes the holomorphic parts of the compatible forms, together with the anti-holomorphic part of the catalyzing forms, and produces the anti-holomorphic parts of the compatible forms and the holomorphic part of the catalyzing form. The anti-holomorphic parts can be thought of as left moving waves, while the holomorphic parts can be thought of as right moving waves.  
 
 We give an explicit form for the scattering matrix in terms of the Schiffer operators, using the identities and cohomological results of Section \ref{se:Schiffer_Cauchy}. We furthermore show that this scattering matrix is unitary, using the adjoint identities of Section \ref{se:Schiffer_Cauchy}. 
 
 These adjoint identities can be thought of as generalizations of norm inequalities relating the Schiffer operators \cite{BergmanSchiffer}, which are themselves closely related to identities relating the Faber and Grunsky operators. However neither the unitarity of the scattering process nor the adjoint identities were recognized even in the case of the plane.  
\end{subsubsection}
\begin{subsubsection}{Polarizations and Grunsky operators}

 For context, we sketch the well-known classical polarization for compact surfaces. Given a compact Riemann surface $\mathscr{R}$, by the Hodge decomposition theorem, every $L^2$ one-form has a harmonic representative. The spaces of harmonic one-forms in turn decompose into the spaces of holomorphic and anti-holomorphic one-forms. Thus the cohomology classes of a Riemann surface are represented by the direct sum of the vector spaces of holomorphic and  anti-holomorphic one-forms. This decomposition depends on the complex structure. 
 
In complex algebraic geometry, this picture is often represented in terms of the so-called period-matrix. Given a basis of the homology, divided into $a$ and $b$ curves satisfying the usual intersection conditions, one normalizes half of the periods of the holomorphic one-forms, and encoding the remaining periods in a $g \times g$ matrix where $g$ is the genus. Most often one normalizes matrix of $a$ periods to be the identity matrix; in that case, by the Riemann bilinear relations, the matrix of $b$ periods lies in the Siegel upper half-space of symmetric matrices with positive definite imaginary part. 
 It is also possible to represent the periods with a matrix of norm less than one (that is, a matrix in the Siegel disk). It was shown by {{L. Ahlfors \cite{Ahlfors}}} that the period matrix can be used to give coordinates on Teichm\"uller space; the idea of using periods as coordinates on the moduli space goes back to B. Riemann \cite{Riemann}.\\
 
 An analogue of the period map exists for the case of the Teichm\"uller space of the disk. Nag and Sullivan \cite{NS}, following earlier work of A. Kirillov and D. Yuri'ev in the smooth case \cite{KY2}, showed that the set of quasisymmetries of the circle acts symplectically on the space of polarizations of the set of Dirichlet-bounded harmonic functions on the disk, and that the space of polarizations can be identified with an infinite-dimensional Siegel disk. They further outlined various analogies with the classical period matrix. L. Takhtajan and L-P. Teo \cite{Takhtajan_Teo_Memoirs} showed that this ``period matrix'' is in fact the Grunsky matrix, and proved that the period map is a holomorphic map of the Teichm\"uller space of the unit disk (which is also the universal Teichm\"uller space). Later, with Radnell, the authors generalized this holomorphicity to genus zero surfaces with $n$ boundary curves. All of these results demonstrate the existence of a powerful analogy with the classical period matrix. Nevertheless they do not indicate the mathematical source of the analogy, nor how to unify the classical case for compact surfaces and the case of surfaces with border. \\

 For genus zero surfaces with $n$ boundary curves, we showed with Radnell that the graph of the Grunsky matrix gives the boundary values 
 of the set of Dirichlet-bounded harmonic functions  curves \cite{RSS_Dirichletspace}, using overfare. This was extended by M. Shirazi  \cite{Shirazi_thesis,Shirazi_Grunsky} to the genus $g$ case.
 In this paper, we show that by treating polarizations as decompositions of boundary values of semi-exact one-forms, all the versions of the polarizations can be viewed as special cases of a single general theorem. The unifying principle is provided by boundary values of harmonic one-forms.
 In particular, 
 we show that the polarizing subspace of holomorphic one-forms on a bordered surface can be viewed as the graph of an operator in an infinite Siegel disk, from which the polarizations in both the compact case and the case of genus zero surfaces with borders can be recovered.  The overfare process is a crucial part of establishing this unified picture. 
 
 The bound on the polarizing operator can be viewed as a far-reaching generalization of the Grunsky inequalities. We also show how special cases of the Grunsky inequalities can be recovered from this one.
\end{subsubsection}

\end{subsection}
\begin{subsection}{Outline of the paper}
  Here we give a sparing outline of the paper. 
  
 In Section \ref{se:preliminaries} we gather the preliminary material about Riemann surfaces, their boundaries, and spaces of harmonic and holomorphic functions and forms.  Section \ref{se:CNT_all} defines the conformally non-tangential boundary values of Dirichlet bounded harmonic functions, and proves the existence and boundedness of the overfare map. Section \ref{se:Schiffer_Cauchy} we define and prove the basic properties of the Schiffer and Cauchy-Royden operators. Furthermore we gather a collection of identities which form the computational backbone of the paper. 
 
 Section \ref{se:Dirichlet_problem} contains a full treatment of the Dirichlet problem for $L^2$ harmonic one-forms with $H^{-1/2}$ boundary values. This is followed by the definition and properties of the overfare process for forms in Section \ref{se:Overfare}. 

 In Section \ref{se:index_cohomology} we derive the cohomological results about the Schiffer operator, including characterizations of the kernel and image, the generalized jump theorem, and index theorems. Section \ref{se:scattering} derives the form of the scattering matrix for harmonic one-forms and proves that it is unitary. Finally, in Section \ref{se:period_mapping} we give the generalized polarizations, and apply it to solve the boundary value problem for semi-exact $L^2$ harmonic one-forms on bordered surfaces. We also explain its relation to the classical Grunsky inequalities and their generalizations.   
\end{subsection}
\end{section}
\begin{section}{Preliminaries} \label{se:preliminaries}
\begin{subsection}{About this section}
 This section gathers the definitions and basic results used throughout the paper.  This includes Dirichlet spaces of functions and Bergman spaces of forms; Riemann surfaces, their boundaries and specialized charts called collar charts; sewing; Green's functions on compact surfaces and surfaces with boundary; Sobolev spaces; and harmonic measures and boundary period matrices. 
\end{subsection}
\begin{subsection}{Bordered surfaces}
We briefly recall the definition of a bordered surface in order to remove any ambiguity.  See for example \cite{Ahlfors_Sario} for a complete treatment.\\

In what follows we denote by $\gls{annulus}$ the annulus $ \{ z;\, a<|z|<b \}$.
\begin{definition}\label{defn:bordered surface}
Let $\mathbb{C}$ denote the complex plane, let 
 $\mathbb{H} = \{ z \in \mathbb{C} : \text{Im} z >0 \}$ denote the upper half plane, and let $\gls{cl}\,( \mathbb{H})$ denote its closure (we will let $\text{cl}$ denote closure throughout).  We say that a connected Hausdorff topological space $\hat{\riem}$ is a \emph{bordered Riemann surface} if there is an atlas of charts $\phi:U \rightarrow \text{cl} \, ( \mathbb{H})$ with the following properties. 
 \begin{enumerate}
     \item  Each chart is a local homeomorphism with respect to the relative topology;
     \item  Every point in $\hat{\riem}$ is contained in the domain of some chart;
     \item  Given any pair of charts $\phi_k:U_k \rightarrow \text{cl}\, ( \mathbb{H})$, $k=1,2$, if $U_1 \cap U_2$ is non-empty, then $\phi_1 \circ \phi_2^{-1}$ is a biholomorphism on $U_1 \cap U_2 \cap \mathbb{H}$.   
 \end{enumerate}    
\end{definition}

 This defines a distinction between interior and border points (see e.g. \cite[p23-24]{Ahlfors_Sario}).  That is, we say $p$ is on the border if there is a chart in the atlas such that $\phi(p)$ is on the real axis, and $p$ is in the interior if there a chart mapping $p$ to a point in $\mathbb{H}$.  In either case, if the claim holds for one chart, it holds for all of them.  We will denote the set of interior points by $\riem$ and the set of border points by $\partial \riem$.  We call $\partial \riem$ the border, and note that the border is also the topological boundary of $\riem$ in $\hat{\riem}$.  Observe that $\riem$ is a Riemann surface in the standard sense.
 
 
 We will call a chart $\phi$ which contains a boundary point in its domain a ``boundary chart'' or ``border chart''.   Now regarding the notion of the double of a bordered Riemann surface, assume that $\phi_k:U_k \rightarrow \text{cl}\, (\mathbb{H})$, $k=1,2$, are charts such that $U_1 \cap U_2 \cap \partial \riem$ is non-empty.  Then by the Schwarz reflection principle, $\phi_1 \circ \phi_2^{-1}$ extends to a biholomorphism of an open set containing $\phi_2(U_1 \cap U_2)$.  This open set can be taken to be the union of $\phi_2(U_1 \cap U_2)$ with its reflection in the real axis.  In the usual construction of the double, any chart $\phi$ which contains border points can be extended to a chart in the double by reflection.  By the above argument, the overlap map $\phi_1 \circ \phi_2^{-1}$ for any pair $\phi_1,\phi_2$ of such extensions is a biholomorphism.  This defines the atlas on the double of $\riem$ which is denoted here by $\riem^{d}$.  
 
 \begin{remark}  Once the border structure is established as above, for convenience we will allow interior charts to have image in $\mathbb{C}$ and not necessarily in $\mathbb{H}$.  Moreover, we will also consider border charts which map into the closure of the disk 
 $\disk^+ = \{ z \in \sphere: |z| <1 \}$, with border points mapping to $|z|=1$.  Every such chart is a border chart in the original sense after composition by a M\"obius transformation. 
 \end{remark}

One of our main objects of study is a particular type of bordered Riemann surface which is defined as follows:
\begin{definition}\label{defn:gn-type surface}
We say that $\riem$ is a \emph{bordered Riemann surface of type} $(g,n),$ if it is bordered (in the sense Definition \ref{defn:bordered surface}), the border has $n$ connected components, each of which is homeomorphic to $\mathbb{S}^1$, and its double $\riem^d$ is a compact surface of genus $2g + n-1$.  
 \end{definition}
Visually, a bordered surface of type $(g,n)$ is a $g$-handled surface bounded by $n$ simple closed curves.  
 We order the borders and label them accordingly, so that $\partial \riem = \partial_1 \riem \cup \cdots \cup \partial_n \riem$.  The borders can be identified with analytic curves in the double $\riem^d$, and we denote the union $\riem\cup \partial \riem$ by $\text{cl}(\riem)$.  
 
 Finally, we observe that borders are conformally invariant.  That is, if $\riem_1$ and $\riem_2$ are bordered surfaces, then any biholomorphism $f:\riem_1 \rightarrow \riem_2$ extends to a homeomorphism of the borders.  In fact, $f$ extends to a biholomorphism between the doubles $\riem_1^d$ and $\riem_2^d$ which takes $\partial \riem_1$ to $\partial \riem_2$.  Finally, if only one of the two surfaces has a border, say $\riem_1$, then one can endow $\riem_2$ with a border using $f$. In particular, there is a unique maximal border structure.  
 \begin{remark}
  Note that if $\riem$ has type $(g,n)$, the border structure is maximal, since $\riem^d$ is a compact surface.  
 \end{remark}
 
 \begin{definition}
 We say that a homeomorphic image $\Gamma$ of $\mathbb{S}^1$ is a \emph{strip-cutting Jordan curve} if it is contained in an open set $U$ and there is a biholomorphism $\phi:U \rightarrow \mathbb{A}_{r,R} $ for some annulus 
 \[  \mathbb{A}_{r,R} \subset \mathbb{C}, \ \ \  r<1<R,  \] in such a way that $\phi(\Gamma)$ is isotopic to the circle $|z|=1.$  We call $U$ a doubly-connected neighbouhood of $\Gamma$ and $\phi$ a doubly-connected chart. 
 \end{definition}
 \begin{remark}
  If $\Gamma$ is a strip-cutting curve, by shrinking $\mathbb{A}_{r,R} $, we can assume that (1) $\phi$ extends biholomorphically to an open neighourhood of $\text{cl} \,(U)$, (2), that the boundary curves of $U$ are themselves strip cutting (in fact analytic), and  (3) that $\Gamma$ is isotopic to each of the boundary curves (using $\phi^{-1}$ to provide the isotopy).  
 \end{remark}
 
 \begin{remark} \label{re:analytic_curve_stripcutting}
  An analytic Jordan curve is by definition strip-cutting.
 \end{remark}

 Throughout the paper we consider \emph{nested Riemann surfaces}.  That is, we are 
 given a type $(g,n)$ bordered surface $\riem$, another Riemann surface $\mathscr{R}$ which is compact, and a holomorphic inclusion map $\iota( \riem )\subset \mathscr{R}$.  Assume that the closure of $\riem$ is compact in $R$, and furthermore the boundary consists of $n$ closed strip-cutting Jordan curves, which do not intersect.  In that case, the inclusion map $\iota$ extends homeomorphically to a map from the border to the strip-cutting Jordan curves.  Thus $\partial \riem$ is in one-to-one correspondence with its image under the homeomorphic extension of $\iota$, and in fact the image is the boundary of $\iota (\riem)$
 in the ordinary topological sense.  For this reason, we will not notationally distinguish $\riem$ from $\iota (\riem)$.  We will also use the notation $\partial \riem$ for both the boundary of $\iota (\riem)$ in $\mathscr{R}$ and the abstract border of $\riem$, and denote both closures by $\text{cl}\, (\riem)$. 
 
 In fact, the assumption that the surface $\riem$ is bordered can be removed in the following way.  
 \begin{theorem}  \label{th:embedded_is_bordered}
  Let $\riem$ be an open connected subset of a Riemann surface $\mathscr{R}$. 
  Assume that the topological boundary of $\riem$ in $\mathscr{R}$ is a finite collection $\Gamma = \Gamma_1 \cup \cdots \cup \Gamma_n$ of strip-cutting Jordan curves.  Furthermore suppose that there are doubly-connected charts $\phi_k:U_k \rightarrow \mathbb{A}_k$ of $\Gamma_k$ for $k=1,\ldots,n$ $($where $\mathbb{A}_k$'s are annuli$)$ such that the closures of $U_k$ are mutually disjoint, and $U_k \backslash \Gamma$ consists of two connected components, one of which is entirely contained in $\riem$ and one which is in $\mathscr{R} \backslash \riem$.   Then $\riem$ is a bordered surface and the inclusion map is a homeomorphism.
 \end{theorem}
 \begin{proof}
  First, observe that $\riem$ has a unique complex structure compatible with $\mathscr{R}$, so we let $\mathscr{A}$ be an atlas compatible with this structure.  
 
  Let $U_k^+$ denote the component of $U_k \backslash \Gamma$ in $\riem$.  Then $\phi_k(U_k^+)$ is an open subset of $\mathbb{C}$ bounded by two Jordan curves, one of which is a boundary $\gamma$ of $\mathbb{A}_k$ and one of which is the Jordan curve $\phi_k(\Gamma)$.  By \cite[Theorems 3.3, 3.4 Sect 15.3]{ConwayII}, there is a biholomorphism
  $\psi_k: \phi_k(U_k^+) \rightarrow \mathbb{A}_{r,1}$ which extends to a homeomorphism of the boundaries, taking $\gamma$ to $|z|=r$ and $\phi_k(\Gamma)$ to $\mathbb{S}^1$.  Adjoining the points in $\Gamma_k$ to $\riem$, 
  Then 
  \[  \mathscr{A} \cup \left\{ \left. \psi_1 \circ \phi_1 \right|_{U_1^+}, \ldots, \left. \psi_n \circ \phi_n \right|_{U_n^+} \right\}   \]
  is an atlas making $\riem \cup \partial \riem$ into a bordered surface.
 \end{proof}
 
 \begin{remark}
   The embedding of the border $\partial \riem$ in $\mathscr{R}$ need not be regular.  That is, the inclusion map does not extend to a smooth or analytic map from $\partial \riem$ onto its image under inclusion $\iota$, unless the image consists of smooth or analytic curves. 
 \end{remark}
 
 By another application of \cite[Theorems 3.3, 3.4 Sect 15.3]{ConwayII}, it is easily shown that if $\riem_1\subset \mathscr{R}_1$ and $\riem_2 \subset \mathscr{R}_2$
 satisfy the conditions above, and $f:\riem_1 \rightarrow \riem_2$ is a biholomorphism, then $f$ extends continuously to a map taking each Jordan curve in $\partial \riem_1$ homeomorphically to one of the Jordan curves of $\partial \riem_2$. \\
 
 It is helpful to have the following distinction in mind throughout the paper: certain statements are ``intrinsic'' while others are ``extrinsic''.   Intrinsic statements about a Riemann surface $\riem$ are those which depend only on the surface itself and are unchanged under a biholomorphism.  For example, the border is intrinsic, and the harmonic function which is one on $\partial_k \riem$ and $0$ on other curves is intrinsic.  Extrinsic statements about a Riemann surfaces $\riem$ nested in another surface $\mathscr{R}$, are those which make reference to $\mathscr{R}$.  For example, ``strip-cutting'' is an extrinsic property, as is the regularity of $\iota(\partial \riem)$. An example of an extrinsic object is the restriction of Green's function of $\mathscr{R}$ to $\riem$ (see the next subsection for the definition of Green's functions).\\ 
 
 When dealing with intrinsically phrased boundary value problems, regularity of the boundary is not an issue, since we can treat the boundary as a border and thus we have its analytic structure at our disposal. Examples of this are the Dirichlet problem (as we phrase it) in Section \ref{se:Dirichlet_problem} and CNT boundary values of $L^2$ harmonic forms on $\riem$ in Section \ref{se:CNT_limits_BVs}.  On the other hand, when dealing with extrinsically phrased boundary value problems, regularity of the boundary is a major concern.  Overfare/Transmission phenomena in Section \ref{se:function_Overfare}, in which the boundary values of a harmonic function on $\riem$ become data for the Dirichlet problem on $\mathscr{R} \backslash \riem$, are of this nature, as are the Schiffer operators and results regarding them in Section \ref{se:Schiffer_Cauchy} and onward.   
 \end{subsection}

\begin{subsection}{Collar charts}
 We also define a kind of chart on bordered surfaces near the boundary, which we call a collar chart.  
 \begin{definition}  Let $\riem$ be a bordered Riemann surface of type $(g,n)$.  
  A biholomorphism $\phi: U \rightarrow \mathbb{A}_{r,1}$ is called a \emph{collar chart} of $\partial_k \riem$ (for some fixed $k$) if $U$ is an open set in $\riem$ bounded by two Jordan curves $\partial \riem$ and $\Gamma$, such that $\Gamma$ is isotopic to $\partial_k \riem$ within the closure of $U$, and such that $\phi$ extends continuously to the closure.   A domain $U$ is a collar neighbourhood of $\partial_k \riem$ if it is the domain of some collar chart.
 \end{definition}
  
 \begin{proposition}
  Let $\riem$ be a type $(g,n)$ surface.  Then every boundary curve $\partial_k \riem$ has a collar chart.
 \end{proposition} 
 \begin{proof}
  Let $\riem^d$ be the double of $\riem$, so that each boundary $\partial_k \riem$ is an analytic Jordan curve and hence strip-cutting.  Let $U_k$, $U_k^+$, $\phi_k$ and $\psi_k$ be as in the proof of Theorem \ref{th:embedded_is_bordered}.  Then $\left. \psi_k \circ \phi_k \right|_{U_k^+}$ is a collar chart.
 \end{proof}

 Furthermore, we have the following consequence of Carath\'eodory's theorem.
 \begin{theorem}\label{theorem on collar charts}  Let $\riem$ be a bordered surface and $\Gamma$ be a component of the border which is homeomorphic to $\mathbb{S}^1$.  
  If $\phi:U \rightarrow \mathbb{A}$ is a collar chart, then $\phi$ extends continuously to $\partial_k \riem$.  The extension is a homeomorphism of $\partial_k \riem$ onto $\mathbb{S}^1$.   
 \end{theorem}
 \begin{proof}
  $\Gamma$ is an analytic Jordan curve in the double, and hence strip-cutting.  Let $\psi:V \rightarrow \mathbb{A}$ be a doubly-connected chart for $\Gamma$. By shrinking $V$ we may assume that the boundaries of $\psi(V)$ are Jordan curves. Then $\psi \circ \phi^{-1}$ maps $\mathbb{A}$ onto a doubly-connected region bounded by Jordan curves, so the claim follows from \cite[Theorems 3.4 Sect 15.3]{ConwayII}.
 \end{proof}
 To keep the notation simple, we will also denote the continuous extension by $\phi$.

 \begin{remark}[Isotopy and extension]   \label{re:collar_chart_provides_isotopy}  By shrinking $r$, for any collar chart $\phi:U \rightarrow \mathbb{A}_{1,r}$ we can always assume that the inner boundaries are analytic curves and $\phi$ has an analytic extensions to these curves.  Furthermore, $H(t,\theta) = \phi^{-1}(e^{t} e^{i\theta})$ defines an isotopy between the level curve $|\phi|=r$ and $\partial_k \riem$, running through the level curves of $|\phi|$. 
 \end{remark}
 
 In fact the homeomorphic extension is analytic on the border.  This can be phrased in various ways, one of which is as follows.
 Treat $\riem$ as a subset of its double $\riem^d$ with involution $\tilde{(\cdot) }$.  For a collar neighbourhood $U$ of $\partial_k \riem$, let $U^d = U \cup \tilde{U} \cup \partial_k \riem$.   We then have
 \begin{proposition} \label{prop:collar_charts_extend_analytically} Let $\phi:U \rightarrow \mathbb{A}_{r,1}$ be a collar chart.  Let $U^d = U \cup \tilde{U} \cup \partial_k \riem$ be the double of $U$.
  If $\riem$ is included in its double $\riem^d$, then $\phi$ extends to a doubly-connected chart $\phi^d$ of $\partial_k \riem$ mapping $U^d$ onto the annulus $\mathbb{A}_{r,1/r}$ satisfying $\phi^d(\tilde{z}) = 1/\overline{\phi^d(z)}$. 
 \end{proposition}  
 \begin{remark} \label{re:border_form_regularity_meaning}
 In particular, the border charts give a well-defined meaning to continuous, $C^k$, analytic functions, vector fields, one-forms and so forth, on $\partial_k \riem$ for $k=1,\ldots, n$.  For example, a one-form $\alpha$ on $\partial_k \riem$ is continuous, $C^k$, or analytic if its expression in a boundary chart $\psi:U \rightarrow \mathbb{H}$ near $p$ is $h(x)\, dx$ where $h$ is continuous, $C^k$ or analytic respectively, and this holds for all $p \in \partial_k \riem$.  If the property holds for any collection of boundary charts covering $\partial_k \riem$ then it holds for all boundary charts.  Thus, it is enough that the property in question holds for one collar chart; that is, $\alpha$ is continuous, $C^k$, or analytic if and only if in the local  coordinates defined using  $\left. \phi \right|_{\partial_k \riem}$ for a collar chart $\phi$, $\alpha$ is given by $h(e^{i\theta}) \,d\theta$ where $h$ is respectively continuous, $C^k$ or analytic on $\mathbb{S}
^1$.  
 \end{remark} 

 Finally, we have the following useful fact. 
 \begin{proposition} \label{prop:collar_charts_intersection}  Let $\riem$ be a Riemann surface with border $\Gamma$ homeomorphic to $\mathbb{S}^1$, and let $U$ and $V$ be collar neighbourhoods of a boundary curve $\partial_k \riem$.  There is a collar chart $\phi :W \rightarrow \mathbb{A}_{r,1}$ such that $W \subseteq U \cap V$. Moreover  $r$ can be chosen so that the inner boundary of $W$ is contained in $U \cap V$.
 \end{proposition}
 \begin{proof} By Remark \ref{re:collar_chart_provides_isotopy} we can choose collar  neighbourhoods $U'$ and $V'$ whose inner boundaries are analytic curves $\gamma_1$ and $\gamma_2$ contained in $U$ and $V$, with corresponding collar charts $\psi_{U'}$ and $\psi_{V'}$ extending analytically to $\gamma_1$ and $\gamma_2$.   By composing with $\psi_{U'}$, we can assume that $\Gamma = \mathbb{S}^1$, $\psi_{U'}(z)=z$,  $U' = \mathbb{A}_{r,1}$ for some $r$, and $\gamma_1 = \{z : |z|=r \}$.   
 
 Now let $M$ be the maximum value of $|\psi_{V'}(z)|$ on $\gamma_2$, which exists because $\gamma_2$ is compact. In that case $\text{cl} \, (\mathbb{A}_{{s},1}) \subseteq V' \cap U'$ for $s=(1+M)/2$.   
 We may now choose $W = \mathbb{A}_{{s},1}$ and $\phi(z)=z$ to prove the claim.  
 \end{proof} 
 \begin{proposition} \label{prop:collar_charts_in_doubly_connected_charts}
   Let $\Gamma$ be a strip-cutting Jordan curve in $\mathscr{R}$,  and let $\phi:U \rightarrow \mathbb{A}_{r,R}$ be a doubly-connected chart.
    There are canonical collar charts $\psi_k:U_k \rightarrow \mathbb{A}$ with $U_k \subseteq U \cap \riem_k$ for $k=1,2$.  $U_k$ may be chosen so that their inner boundaries are analytic curves contained in $U$.   
 \end{proposition}
 \begin{proof} Applying the proof of Theorem \ref{th:embedded_is_bordered} to each side of $\Gamma$ we obtain the desired $\psi_k$.  
 \end{proof}
 
 \end{subsection}
 
 \begin{subsection}{Function spaces and holomorphic and harmonic forms}
  In this paper, we will denote positive constants in the inequalities by $C$ whose 
value is not crucial to the problem at hand. The value of $C$ may differ
from line to line, but in each instance could be estimated if necessary.  Moreover, when the values of constants in our estimates are of no significance for our main purpose, then we use the notation $a\lesssim b$ as a shorthand for $a\leq Cb$. If $a\lesssim b$ and $b\lesssim a$ then we write $a\approx b.$  \\
 
 On any Riemann surface,  define the dual of the almost 
 complex structure,  $\ast$ in local coordinates $z=x+iy$,  by 
 \[  \ast (a\, dx + b \, dy) = a \,dy - b \,dx. \]
This is independent of the choice of coordinates.
It can also be computed in coordinates that for any complex function $h$ 
\begin{equation} \label{eq:Wirtinger_to_hodge}
    2 \partial_z h = dh + i \ast dh.
\end{equation}

\begin{definition}
 We say a complex-valued function $f$ on an open set $U$ is \emph{harmonic} if it is $C^2$ on $U$ and $d \ast d f =0$. We say that a complex one-form $\alpha$ is harmonic if it is $C^1$ and satisfies
 both $d\alpha =0$ and $d \ast \alpha =0$. 
\end{definition}
  Equivalently, harmonic one-forms are those which can be expressed locally as $df$ for some harmonic function $f$. Harmonic one-forms and functions must of course be $C^\infty$. \\
  
   Denote complex conjugation of functions and forms with a bar, e.g. $\overline{\alpha}$.
 A holomorphic one-form is one which can be written in coordinates as $h(z)\,dz$ for a holomorphic function $h$, while an anti-holomorphic one-form is one which can be locally written $\overline{h(z)}\, d\bar{z}$ for a holomorphic function $h$.  
 
Denote by $L^2(U)$ the set of one-forms $\omega$ on an open set $U$ which satisfy
\[   \iint_U \omega \wedge \ast \overline{\omega} < \infty  \]
(observe that the integrand is positive at every point, as can be seen by writing the expression in local coordinates).  
This is a Hilbert space with respect to the inner product
\begin{equation} \label{eq:form_inner_product}
 (\omega_1,\omega_2) =  \iint_U \omega_1 \wedge \ast \overline{\omega_2}.
\end{equation}

\begin{definition}\label{defn:bergman spaces}
The \emph{Bergman space of holomorphic one forms} is 
\begin{equation}
    \gls{bergman}(U) = \{ \alpha \in L^2(U) \,:\, \alpha \ \text{holomorphic} \}.
\end{equation} 
 The anti-holomorphic Bergman space is denoted $\overline{\mathcal{A}(U)}$.   We will also denote 
\begin{equation}
    \gls{Aharm}(U) =\{ \alpha \in L^2(U) \,:\, \alpha \ \text{harmonic} \}.
\end{equation}
\end{definition}

Observe that $\mathcal{A}(U)$ and $\overline{\mathcal{A}(U)}$ are orthogonal with respect to the inner product \eqref{eq:form_inner_product}.  In fact we have the direct sum decomposition
\begin{equation}\label{direct sum decomposition}
 \mathcal{A}_{\mathrm{harm}}(U) = \mathcal{A}(U) \oplus \overline{\mathcal{A}(U)}.    
\end{equation}      
If we restrict the inner product to 
 $\alpha, \beta \in \mathcal{A}(U)$ then since $\ast \overline{\beta} = i \overline{\beta}$, we have   
\[  (\alpha,\beta) = i \iint_U \alpha \wedge \overline{\beta}.      \]

Denote the projections induced by this decomposition by 
\begin{align}  \label{eq:hol_antihol_projections_Bergman}
  \mathbf{P}_U: \mathcal{A}_{\mathrm{harm}}(U) & \rightarrow \mathcal{A}(U) \nonumber \\
  \overline{\mathbf{P}}_U : \mathcal{A}_{\mathrm{harm}}(U) & \rightarrow  \overline{\mathcal{A}(U)}.
\end{align}

Let $f: U \rightarrow V$ be a biholomorphism. We denote the pull-back of $\alpha \in \mathcal{A}_{\mathrm{harm}}(V)$
under $f$ by $f^*\alpha.$ 
Explicitly, if $\alpha$ is given in local coordinates $w$ by $a(w)\, dw + \overline{b(w)} \, d\bar{w}$ and $w=f(z),$ then the pull-back is given by 
\[   f^* \left( a(w)\, dw + \overline{b(w)} \,d\bar{w} \right)= a(f(z)) f'(z)\, dz + \overline{b(f(z))} \overline{f'(z)}\, d\bar{z}.   \]
The Bergman spaces are all conformally invariant, in the sense that if $f:U \rightarrow V$ is a biholomorphism, then $f^*\mathcal{A}(V) = \mathcal{A}(U)$ and this preserves the inner product.  Similar statements hold for the anti-holomorphic and harmonic spaces.\\ 

\begin{definition}\label{def: exact holo and harm forms}
 We define the space $\gls{exactA}_{\mathrm{harm}}(U)$ as the subspace of exact elements of $\mathcal{A}_{\mathrm{harm}}(U)$, and similarly for $\mathcal{A}^\mathrm{e}(\riem)$ and $\overline{\mathcal{A}^\mathrm{e}(\riem)}$.  
\end{definition}

The following spaces also play significant roles in this paper.
\begin{definition}\label{defn:dirichlet spaces}
The \emph{Dirichlet spaces of functions} are defined by 
\begin{align*}
   \gls{Dharm}(U)& = \{ f:U \rightarrow \mathbb{C}, f \in C^2(U), \,:\, 
   df\in L^2 (U)\,\,\,\mathrm{and}\,\, \, d\ast df =0 \},\\
   \gls{D}(U)& = \{ f:U \rightarrow \mathbb{C} \,:\, 
    df \in \mathcal{A}(U) \}, \ \text{and} \\
    \overline{\mathcal{D}(U)} & = \{ f:U \rightarrow \mathbb{C} \,:\, 
    df \in \overline{\mathcal{A}(U)} \}. \\
\end{align*}
\end{definition}
We can define a degenerate inner product on $\mathcal{D}_{\mathrm{harm}}(U)$ by 
\[   (f,g)_{\mathcal{D}_{\mathrm{harm}}(U)} = (df,dg)_{\mathcal{A}_{\mathrm{harm}}(U)},   \] 
where the right hand side is the inner product (\ref{eq:form_inner_product}) restricted to elements of $\mathcal{A}_{\mathrm{harm}}(U)$.  The inner product can be used to define a seminorm on $\mathcal{D}_{\mathrm{harm}}(U)$, by letting $$\Vert f\Vert^2_{\mathcal{D}_{\mathrm{harm}}(U)}:=(df,df)_{\mathcal{A}_{\mathrm{harm}}(U)}.$$ 



We note that if one defines the \emph{Wirtinger operators} via their local coordinate expressions
\[   \partial f = \frac{\partial f}{\partial z}\, dz,  \ \ \ 
   \overline{\partial} f =  \frac{\partial f}{\partial \bar{z}}\, d \bar{z}, \]
then the aforementioned inner product can be written as
\begin{equation} \label{eq:inner_product_with_dbar_and_d}
 (f,g)_{\mathcal{D}_{\mathrm{harm}}(U)} =i \iint_{U} \left[ \partial f \wedge \overline{\partial  g} -  \overline{\partial} f \wedge 
 \partial \overline{g} \right].    
\end{equation}
Although this implies that $\mathcal{D}(U)$ and $\overline{\mathcal{D}(U)}$ are orthogonal, there is no direct sum decomposition into $\mathcal{D}(U)$ and $\overline{\mathcal{D}(U)}$.  This is because in general there exist exact harmonic one-forms whose holomorphic and anti-holomorphic parts are not exact.  

Observe that the Dirichlet spaces are conformally invariant in the same sense as the Bergman spaces.  That is, if $f: U \rightarrow V$ is a biholomorphism then 
\begin{equation*}
 \gls{cf} h = h \circ f
\end{equation*} 
satisfies
\[  \mathbf{C}_f :\mathcal{D}(V) \rightarrow \mathcal{D}(U) \]
and this is a semi-norm preserving bijection. If $f(p)= q$ then $\mathbf{C}_f$ is an isometry from $\mathcal{D}(V)_q$ to $\mathcal{D}(U)_p$.  Similar statements hold for the anti-holomorphic and harmonic spaces. 

We also note that if $h \in \mathcal{D}(U)$ and $\tilde{h}(z) = h \circ \phi^{-1}(z)$ is the expression for $h$ in local coordinates $z= \phi(w)$ in an open set $\phi(U) \subseteq \mathbb{C}$,  then we have the local expression
\[  (h,h)_{\mathcal{D}(U)} = \iint_{\phi(U)}  |\tilde{h}'(z)|^2 dA_z   \] 
where $dA$ denotes Lebesgue measure in the plane. 
 Similar expressions hold for the other Dirichlet spaces.\\

Next we gather some results from the theory of Sobolev spaces which we shall use in this paper.\\

\begin{definition}\label{defn:sobolev space on Rn}
For $s\in \R$, one defines the \emph{Sobolev space} $\gls{sobolev}(\mathbb{R}^n),$  which consists of tempered distributions $u$ such that $$\Vert u\Vert^{2}_{H^{s}(\mathbb{R}^n)}:=\Vert (1-\Delta)^{s/2} u\Vert^{2}_{L^2(\mathbb{R}^n)}=\int_{\mathbb{R}^n} (1+|\xi|^2)^{s}|\widehat{u}(\xi)|^2 \, d\xi<\infty,$$ where $\widehat{u}(\xi)$ is the Fourier transform of $u$ defined by $\widehat{u}(\xi)=\int_{\mathbb{R}^n} u(x) \, e^{-ix\cdot\xi} \, dx$ and $$(1-\Delta)^{s/2} u(x)=\frac{1}{(2\pi)^n} \int_{\mathbb{R}^n} (1+|\xi|^2)^{s/2}\,\widehat{u}(\xi)\, e^{ix\cdot\xi}  \, d\xi .$$ The homogeneous Sobolev space $\gls{homsobolev}(\mathbb{R}^n),$ is the space of tempered distributions such that $\int_{\mathbb{R}^n} |\xi|
^{2s}\,|\widehat{u}(\xi)|^2 \, d\xi<\infty.$   
\end{definition}
 The scales of Sobolev spaces that are of particular interest for us are $s=1,\pm \frac{1}{2}$ (defined on various manifolds). For instance
$H^1(\R^n)$ consists of the space of tempered distributions $u$ for which 

\begin{equation}
\begin{split}
\|u\|_{H^{1}(\mathbb{R}^n)}:=\left(\int_{\mathbb{R}^n} |\nabla u(x)|^2 \, dx+\int_{\mathbb{R}^n}|u|^{2} d x\right)^{\frac{1}{2}}\\=:\left(\Vert u\Vert^2_{\dot{H}^{1}(\mathbb{R}^n)}+ \Vert u\Vert^2_{L^2(\mathbb{R}^n)}\right)^{\frac{1}{2}}<\infty,
\end{split}
\end{equation}
and $H^{1/2}(\R^n)$ consists of the space of tempered distributions $u$ for which
\begin{equation}\label{defn:H onehalf norm}
\begin{split}
\|u\|_{H^{1/2}(\mathbb{R}^n)}:=\left(\int_{\mathbb{R}^n} \int_{\mathbb{R}^n} \frac{|u(x)-u(y)|^{2}}{|x-y|^{n+1}} d x d y +\int_{\mathbb{R}^n}|u|^{2} d x\right)^{\frac{1}{2}}\\=:\left(\Vert u\Vert^2_{\dot{H}^{1/2}(\mathbb{R}^n)}+ \Vert u\Vert^2_{L^2(\mathbb{R}^n)}\right)^{\frac{1}{2}}<\infty.
\end{split}
\end{equation}

The Sobolev space $H^{s}(\mathbb{S}^1),$ $s\geq 0,$ will also play an important role in our investigations, whose definition we also recall. Given $f\in L^2(\mathbb{S}^1)$ one defines the Fourier coefficients and the Fourier series associated to $f$ by 
\begin{equation}\label{Fourier expansion}
\hat{f}(n)=\frac{1}{2 \pi} \int_{0}^{2 \pi} f(t) e^{-i n \theta} \, d\theta,\,\,\,
f=\sum_{n=-\infty}^{\infty} \hat{f}(n) e^{i n \theta}, \quad 
\end{equation}
where the convergence of the series is both in the $L^2$-norm and also pointwise almost everywhere. The Sobolev space $H^{s}(\mathbb{S}^1)$ is defined by
\begin{equation}
H^{s}(\mathbb{S}^1)=\left\{\varphi \in L^{2}(\mathbb{S}^1): \sum_{n=-\infty}^{\infty}\left(1+|n|^{2}\right)^{s}|\hat{f}(n)|^{2}<\infty\right\}.
\end{equation}
Like all other $L^2$-based Sobolev spaces, $H^{s}(\mathbb{S}^1)$ is a Hilbert space and given $f, \, g\in H^{s}(\mathbb{S}^1)$ their scalar product is given by

\begin{equation}
\langle f, g\rangle_{H^{s}(\mathbb{S}^1)}=\sum_{n=-\infty}^{\infty}\left(1+|n|^{2}\right)^{s} \hat{f}(n) \overline{\hat{g}(n)},
\end{equation}
and so 
\begin{equation}
 \Vert f\Vert_{H^{s}(\mathbb{S}^1)}= \left(\sum_{n=-\infty}^{\infty}\left(1+|n|^{2}\right)^{s}|\hat{f}(n)|^{2}\right)^{1/2}.   
\end{equation}
Of particular interest in this paper, are the functions in the Sobolev space $H^{1/2}(\mathbb{S}^1)$ for which one also has the analogue of \eqref{defn:H onehalf norm}, i.e.
  
    \begin{equation}\label{defn:sobolev half norm}
      \Vert f\Vert_{H^{1/2}(\mathbb{S}^1)}:=\Big( \int_{\mathbb{S}^1} \int_{\mathbb{S}^1} \frac{|f(z)-f(\zeta)|^2}{|z-\zeta|^2}\,|dz|\, |d \zeta|+\Vert f\Vert^2_{L^2(\mathbb{S}^1)}\Big)^{1/2}.
  \end{equation}
As was shown by J. Douglas \cite{Douglas}, for a function $F\in \mathcal{D}_{\mathrm{harm}}(\disk)$ ($\disk$ denotes the unit disk), then the restriction of $F$ to $\mathbb{S}^1$ is in $H^{1/2}(\mathbb{S}^1)$ and if the boundary value of $F$ is denoted by $f$ then one has that

\begin{equation}\label{Douglas formula}
  \Vert F\Vert^2_{\mathcal{D}_{\mathrm{harm}}(\disk)}=\pi \int_{0}^{2\pi} \int_{0}^{2\pi} \frac{|{f}(z)-f(\zeta)|^2}{|z-\zeta|^2}\,|dz|\, |d \zeta|.
\end{equation}
The dual of $H^{1/2}(\mathbb{S}^1)$, identified with $H^{-1/2}(\mathbb{S}^1)$, consists of linear functionals $L$ on $H^{1/2}(\mathbb{S}^1)$ with the property that if $\alpha_n:= L(e^{in \theta})$ (this is the action of the funcional $L$ on the function $e^{in \theta}$), then
 \begin{equation}
     \sum_{n=-\infty}^{\infty }\frac{|\alpha_n |^{2}}{(1+|n|^2)^{1/2}} <\infty.
 \end{equation}
 Moreover one has
 \begin{equation}\label{negative onehalf norm}
\|L\|_{H^{-1/2}(\mathbb{S}^1)}=\sup _{\|g\|_{H^{1/2}(\mathbb{S}^1)}=1}|\sum_{n=-\infty}^{\infty} \alpha(n) \overline{\hat{g}(n)}|.
\end{equation}

We shall also recall the following useful embedding result, whose proof can be found in \cite{Triebel}.

\begin{theorem}
 Let $1\leq p<\infty$ and $s\geq 0$ with $s+\frac{1}{p}\geq \frac{1}{2}$ then one has the continuous inclusion $($embedding$)$
 \begin{equation}\label{Sobolev embed}
     H^{s}(\mathbb{S}^1)\subset L^p(\mathbb{S}^1).
 \end{equation}
\end{theorem}

Now regarding Sobolev spaces on manifolds, we first recall the definition of Sobolev $H^{s}(M)$, $s\in \R$ for compact manifolds $M$, see e.g. \cite{Booss}. 
\begin{definition} 
Let $M$ be an $n-$dimensional smooth compact manifold without boundary, with the smooth atlas $(\phi_j, U_j)$ and the corresponding smooth partition of unity $\psi_j$ with $\psi_j\geq 0$, $\mathrm{supp}\, \psi_j \subset U_j$ and $\sum_j \psi_j =1$. Given $s\geq 0,$ the Sobolev spaces $H^{s}(M)$ are the space of complex-valued $L^2$ functions on $M$ for which
\begin{equation}\label{defn: Sobolev norm}
\Vert f\Vert_{H^{s} (M)}:=\sum_j \Vert (\psi_j f)\circ \phi^{-1}_j \Vert_{H^{s}(\mathbb{R}^n)}<\infty.
\end{equation}
The homogeneous Sobolev space $\dot{H}^{s} (M)$ is defined using \eqref{defn: Sobolev norm} by substituting $H^{s}(\mathbb{R}^n)$ with $\dot{H}^{s}(\mathbb{R}^n)$. 
\end{definition}

It is also well-known that different choices of the atlas and its corresponding partition of unity, produces norms that are equivalent with \eqref{defn: Sobolev norm}.

 Next let $X$ be a smooth compact $n$-dimensional manifold with smooth boundary $\mathrm{bd}(X)$ and fix a Riemannian structure on $X$. Use the Riemannian structure to construct a collar neighbourhood $N=\mathrm{bd}(X) \times I$ of the boundary
$\mathrm{bd}(X)$ and denote the (inward) normal coordinate by $t \in I=[0,1]$. We may assume that $X$ is a submanifold of a closed  compact, smooth manifold $M$, which is the compact double of $X$.\\
\begin{definition}
\label{defn:sobolev space on manifold with bound} 
Let $X$ be a smooth compact $n$-dimensional manifold with boundary. We can regard $X$ as a submanifold of a closed smooth $n$-dimensional manifold $M$ (i.e. $M$ is compact without boundary as above). Then the space $H^{s}(X)$ consists of the restrictions $\left\{\mathbf{R} u;\, u \in H^{s}(M)\right\}$
where $\mathbf{R}: L^{2}(M) \rightarrow L^{2}(X)$ denotes the \emph{restriction operator} $\left.u \mapsto u\right|_{X}.$
\end{definition}
 In this connection one also has the fundamental fact about Sobolev spaces on manifolds with boundary that asserts that the \emph{trace map}, i.e. the map 
 \begin{equation*}
    \mathrm{Tr}: u\mapsto u|_{\mathrm{bd}(X)}
 \end{equation*}
 from $H^{s}(X) \rightarrow H^{s-\frac{1}{2}} (\mathrm{bd}(X))$ is continuous for $s>\frac{1}{2},$ see e.g. \cite[Theorem 11.4, p 68]{Booss}.\\

 Ahead, we will show that the border structure on a Riemann surface induces a smooth boundary in the Riemannian sense above, so that Sobolev trace can be applied. In this section, we will keep the notation $\mathrm{bd}(X)$ to denote the boundary in the sense above. Once it is established that the theory applies to the case of the border of a Riemann surface, we will return to the notation $\partial \riem$. 

Occasionally, we will also use the invariance of the Sobolev space $H^s$  under diffeomorphisms. We state this below as a lemma whose proof could be found in Lemma 1.3.3 in \cite{Gilkey}, or even more explicitly as Theorem 9.2.3 in \cite{van den ban}, or by using interpolation between the well-known results for Sobolev spaces of integer scales.
 \begin{lemma}\label{lem:invariance of sobolev under diffeo}
Let $s\in \mathbb{R}$ and $\psi$ be a diffeomorphism of an open set $U_{1} \subset
\R ^{n}$ onto another open set $U_{2} \subset \mathbb{R}^{n}$ such that $\psi\in \mathcal{C}^{\infty}(\mathrm{cl}(U_1))$ and $\psi^{-1}\in \mathcal{C}^{\infty}(\mathrm{cl}(U_2))$. Then one has 
$$
\|f \circ \psi\|_{{H}^s (U_1)} \approx \|f\|_{{H}^s (U_{2})}.
$$
 \end{lemma}

 { The following result is quite useful in connection to the boundedness of certain operators which will be introduced later. In fact this theorem enables us to turn our estimates into conformally invariant ones through suitable choices of the norms involved in the estimates.
 \begin{theorem}\label{thm:equiv sobolev norm}
  Let $X$ be a compact Riemannian manifold with smooth boundary, for which the homogeneous and inhomogeneous Sobolev spaces are well-defined. Assume that $\mathscr{F}$ is a non-negative functional on $H^s(X)$, $s>0$, with the following properties:\\
  \begin{enumerate}
      \item [$(1)$] $\mathscr{F}$ is real-valued and for all $c\in \mathbb{C}$ and $f\in{H}^s(X)$, $\mathscr{F}(cf)= |c|\mathscr{F}(f)$;\\
      \item [$(2)$] For $f\in{H}^s(X)$, there exists a constant $C$ $($independent of $f$$)$ such that \begin{equation*}0\leq \mathscr{F}(f)\leq C\Vert f\Vert_{H^s(X)};\end{equation*}\vspace{0.08cm} \item  [$(3)$] For $f\equiv1$ on $\mathrm{cl}(X)$ one has that $\mathscr{F}(f)\neq 0$.
  \end{enumerate}
  \vspace{0.1cm}
 Then there are constants $C_1$ and $C_2$ such that for $f\in H^s(X)$ one has 
  
 \begin{equation}\label{control Sobolev by Dirichlet}
  C_1\Big( \Vert f\Vert^2{_{\dot{H}^s(X)}}+ (\mathscr{F}(f))^2\Big)^{1/2}\leq \Vert f\Vert_{H^s(X)}  \leq C_2 \Big( \Vert f\Vert^2{_{\dot{H}^s(X)}}+ (\mathscr{F}(f))^2\Big)^{1/2}.
 \end{equation}
 \end{theorem}
 \begin{proof}
 Set $\Phi(f):= \Big( \Vert f\Vert^2{_{\dot{H}^s(X)}}+ (\mathscr{F}(f))^2\Big)^{1/2}.$ Then trivially one has that $\Phi( f_1+ f_2)\leq \Phi(f_1)+ \Phi(f_2),$ and for any $c\in \mathbb{C}$ one has $\Phi(c f)= |c|\Phi(f)$. Moreover $\Phi$ is injective, since if $\Phi(f)=0$ then $\Vert f\Vert_{_{\dot{H}^s(X)}}=0$ and $\mathscr{F}(f)=0$. The first equality yields that $f= \mathrm{constant}$, and from the second inequality and the assumption on $\mathscr{F}$ it follows that $f=0.$ This shows that $\Phi(\cdot)$ defines a norm on $H
^s (X).$ Furthermore since the continuity of $\mathscr{F}$ implies that $\Phi(f)\leq A \Vert f\Vert_{H^s(X)},$ a result based on Banach's open mapping theorem, see e.g. \cite{Rudin} Corollary 2.12b, yields that $\Vert f\Vert \leq B \Phi(f).$ Taking $C_1={1}/{A}$ and $C_2= B$ we obtain \eqref{control Sobolev by Dirichlet}.
 \end{proof}} A useful corollary of this result is the following
\begin{corollary}\label{good sobolev cor}
 Let $F\in \mathcal{D}_{\mathrm{harm}}(\disk)$and let $f$ denote the boundary value of $F$. Then one has
 
  \begin{equation}\label{cor:sobolev half norm}
      \Vert f\Vert_{H^{1/2}(\mathbb{S}^1)}\approx |F(0)|+\Vert F\Vert_{\mathcal{D}_{\mathrm{harm}}(\disk)}.
  \end{equation}
 
\end{corollary}
 \begin{proof}
 Since $f\in H^{1/2}(\mathbb{S}^1)$, we know that $f\in L^2(\mathbb{S}^1)$ and so $f=\sum_{n=-\infty}^{\infty} \hat{f}(n) e^{i n \theta},$ with convergence almost everywhere, where $\hat{f}(n)$ is given by \eqref{Fourier expansion}.
  Therefore, for the harmonic extension $F$ of $f$, one has that $F(0)= \hat{f}(0)$ and using Parseval's identity we obtain
  \begin{equation}\label{L2boundary domibation}
  |F(0)|=|\hat{f}(0)|\leq \Big(\sum_{n=-\infty}^{\infty} |\hat{f}(n)|^2\Big)^{1/2}=\frac{1}{\sqrt{2\pi}} \Vert f\Vert_{L^2(\mathbb{S}^1)}.  
 \end{equation}
 Hence using \eqref{L2boundary domibation} and \eqref{defn:sobolev half norm} among others, one can easily check that the functional $$\mathscr{F}(f):= |F(0)|$$ satisfies all the conditions of Theorem \ref{thm:equiv sobolev norm}. Hence Theorem \ref{thm:equiv sobolev norm} and equation \eqref{defn:sobolev half norm} yield that
  $$\Vert f\Vert_{H^{1/2}(\mathbb{S}^1)}\approx \Big(|F(0)|^2+\int_{0}^{2\pi} \int_{0}^{2\pi} \frac{|{f}(z)-f(\zeta)|^2}{|z-\zeta|^2}\,|dz|\, |d \zeta|\Big)^{1/2}.$$ Finally, \eqref{Douglas formula} and the elementary inequality $\frac{1}{\sqrt{2}}(|a|+|b|)\leq (|a|^2 + |b|^2)^{1/2}\leq |a|+|b|$ shows that \eqref{cor:sobolev half norm} is valid. 
 \end{proof}


We also record a rather general fact that is often useful in connection to various boundedness results involving Sobolev spaces, see e.g. Theorem 2.6 in \cite{Chipot} for a proof.
\begin{theorem}\label{thm:H1-L2} 
Let $\Omega$ be a domain whose boundary is locally the graph of a Lipschitz function $($i.e. a Lipschitz domain$)$. Then there exits a unique continuous linear mapping $\gamma: H^1(\Omega) \to L^2(\mathrm{bd}(\Omega))$ such that $\gamma(u)=u|_{\mathrm{bd}(\Omega)}$. In particular, one as the estimate

\begin{equation}\label{L2-sobolev estim}
    \int_{\mathrm{bd}(\Omega)} |u|^2 \lesssim \Vert u \Vert^{2}_{H^1(\Omega)}.
\end{equation}

\end{theorem}

Now let us turn to Sobolev spaces on bordered Riemann surfaces. Let $(\mathscr{R}, h)$ be a compact Riemann surface endowed with a hyperbolic metric $h$ and $f$ a function defined on $\mathscr{R}$. 
Set $d\sigma(h):= \sqrt{|\det h_{ij}|}\,|dz|^2$ which is the area-element of $\mathscr{R}$, where $h_{ij}$
are the components of the metric with respect to coordinates $z = x_1 + i x_2$.  We define the inhomogeneous and homogeneous Sobolev norms and semi-norms respectively of $f$ as

\begin{equation}\label{defn: Sobolev norm on riem}
\Vert f\Vert_{H^1 (\mathscr{R})}:= \Big(\iint_\mathscr{R} df \wedge \ast \overline{df}+\iint_\mathscr{R} |f|^2 d\sigma(h)\Big)^{\frac{1}{2}}=:\Big(\Vert f\Vert^2_{\dot{H}^1 (\mathscr{R})}+\Vert f\Vert^2_{L^2(\mathscr{R})}\Big)^{\frac{1}{2}}.
\end{equation}
Observe that the Dirichlet semi-norm and the homogeneous Sobolev semi-norm $\Vert \cdot\Vert_{\dot{H}^1(\mathscr{R})}$ are given by the same expression up to a constant.\\ 

We also note that since any two smooth metrics on $\mathscr{R}$ have comparable determinants, choosing different metrics in the definitions above yield equivalent norms. Now if $\mathscr{R}$ is a compact Riemann surface and $\riem$ is an open subset of $\mathscr{R}$ with analytic boundary $\partial \riem$, then the pull back of the metric $h_{ij}$ under the inclusion map yields a metric on $\riem$. Using that metric, we can define the inhomogeneous and homogeneous Sobolev spaces $H^1(\Sigma)$ and $\dot{H}^1(\Sigma)$. However these definitions will a-priori depend on the choice of the metric induced by $\mathscr{R}$, due to the non-compactness of $\riem$, unless further conditions on $\riem$
are specified.  \\

\begin{remark} \label{re:Sobolev_space_uses_analytic_boundary}
Whenever we consider the Sobolev space $H^{1/2}(\partial \Sigma)$ in this paper, we assume that $\riem \subset \riem^d$ where $\riem^d$ is the compact double, so that $\partial \riem$ is  an analytic curve (and in particular smooth) and thus an embedded submanifold of $R$.  Thus the charts on  $\partial \riem$ can be taken to be restrictions of charts from $\mathscr{R}$. Equivalently, the boundary $\partial \riem$ is endowed with the manifold structure obtained by treating it as the border of $\riem$.  For roughly bounded $\riem \subseteq \mathscr{R}$, we will not apply the Sobolev theory directly to the boundary $\partial \riem$ as a subset of $\mathscr{R}$.  Indeed in those cases the boundary is of course not a submanifold of $\mathscr{R}$.  However, we may still make use of the Sobolev space on the abstract border by making use of the double.
\end{remark}

Regarding the homogeneous and inhomogeneous Sobolev spaces, it was proved in \cite{Schippers_Staubach_transmission} that
\begin{theorem}\label{equiv and sobolevs in riem}
 Let $\mathscr{R}$ be a compact surface and let $\riem \subset \mathscr{R}$ be bounded by a closed analytic curve $\Gamma$.  Fix a Riemannian metric $\Lambda_\mathscr{R}$ on $\mathscr{R}$ as follows.   If $\mathscr{R}$ has genus $g>1$ then let $\Lambda_\mathscr{R}$ be the hyperbolic metric; if $\mathscr{R}$ has genus $1$ then let $\Lambda_\mathscr{R}$ be the Euclidean metric, and if $\mathscr{R}$ has genus $0$ then let $\Lambda_\mathscr{R}$ be a spherical metric.   
 Let 
 $H^1(\riem)$ and $\dot{H}^1(\riem)$ denote the Sobolev spaces with respect to $\Lambda_\mathscr{R}$.  Then $\dot{H}^1(\riem) = H^1(\riem)$ as sets. 
\end{theorem}
\end{subsection}
\begin{subsection}{Harmonic measures} \label{ harmonic measures}
We start with the definition of harmonic measure in the context of bordered Riemann surfaces.
\begin{definition}\label{def:harmonic measure}
Let $\omega_k$, $k=1,\ldots,n$ be the unique harmonic function which is continuous 
 on the closure of $\riem$ and which satisfies
 \begin{equation*}
  \omega_k = \left\{ \begin{array}{ll} 
    1 & \mathrm{on}\,\,\, \partial_k \riem \\ 0 & \mathrm{on}\,\,\, \partial_j \riem, \ \  j \neq k.  \end{array}  \right.   
 \end{equation*}
 The one-forms $\gls{harmeasure}$ are the {\it harmonic measures}.  We denote the complex linear span of the harmonic measures by $\gls{Ahm}(\riem).$ Moreover we define $\ast \mathcal{A}_{\mathrm{hm}}(\riem) = \{ \ast \alpha : \alpha \in \mathcal{A}_{\mathrm{hm}}(\riem) \}.$
\end{definition}

By definition any element of $\mathcal{A}_{\mathrm{hm}}(\riem)$ is exact, and its anti-derivative $\omega$ is constant on each boundary curve. On the other hand, the elements of $\ast \mathcal{A}_{\mathrm{hm}}(\riem)$ are all closed. Elements of $\mathcal{A}_{\mathrm{hm}}(\riem)$ and $\ast \mathcal{A}_{\mathrm{hm}}(\riem)$ extend real analytically to the border, in the sense that they are restrictions to $\riem$ of harmonic one-forms on the double.  In particular they are square-integrable, which explains our choice of notation above. Thus to summarize:
 
\begin{proposition} \label{pr:harmonic_measures_L2}
 Let $\riem$ be a bordered surface of type $(g,n)$.  Then $\mathcal{A}_{\mathrm{hm}}(\riem) \subseteq \mathcal{A}^{\mathrm{e}}(\riem)$ and 
 $\ast \mathcal{A}_{\mathrm{hm}}(\riem) \subseteq \mathcal{A}(\riem)$.  
\end{proposition} 

 \begin{definition}\label{periodmatrix}
 The \emph{boundary period matrix} $\Pi_{jk}$ of a non-compact surface $\riem$ of type $(g,n)$ is defined by
 \[  \Pi_{jk} := \int_{\partial \riem} \omega_j  \ast d\omega_k = \int_{\partial_j \riem} \ast d \omega_k.   \]
\end{definition} 
 \begin{theorem} \label{th:period_matrix_invertible} If we let $j,k$ run from $1$ to $n$, omitting one fixed value $m$ say, then the resulting matrix
  $\Pi_{jk}$ is symmetric and positive definite.  
 \end{theorem}
 \begin{proof}
 
  The matrix is symmetric, because 
  \[ \Pi_{jk} - \Pi_{kj} =  \int_{\partial \riem} \left( \omega_j \ast d\omega_k - \omega_k \ast d \omega_j  \right) = \iint_{\riem} \left( \omega_j \, d \ast d \omega_k - \omega_k \, d \ast d \omega_j \right) =0.      \]
 
  Now let $\lambda_1,\ldots,\lambda_{n}$ denote fixed real numbers, where $\lambda_m$ is omitted from the list.  Define 
  \[  \omega = \sum_{\substack{k=1 \\ k \neq m}}^n\lambda_k \omega_k  \]
  then using the fact that $\omega$ is harmonic we obtain (implicitly using Proposition \ref{pr:harmonic_measures_L2}) 
  \begin{align*}
   \|d\omega \|^2 & = \iint_{\riem} d \omega \wedge \ast d \omega = \int_{\partial \riem}
    \omega \wedge \ast d \omega \\
    & = \int_{\partial \riem} \left( \sum_{j \neq m} \lambda_j \omega_j \right) \ast d \left( \sum_{k \neq m} \lambda_k \omega_k \right)  \\
    & =   \sum_{j \neq m}  \sum_{k \neq m} \Pi_{jk} \lambda_j \lambda_k. 
  \end{align*}    
  
  Since $d\omega_1,\ldots,d\omega_n$ (omitting $d\omega_m$) are linearly independent, this completes the proof. 
 \end{proof}
 
 Thus $\Pi_{jk}$, $j,k = 1,\ldots\hat{m},\ldots,n$ is an invertible matrix, and we can specify $n-1$ of the boundary periods of elements of $\ast \mathcal{A}_{\mathrm{hm}}(\riem)$.
 \begin{corollary}  \label{co:boundary_periods_specified_starmeasure}
  Let $\riem$ be of type $(g,n)$ and $\lambda_1,\ldots,\lambda_n \in \mathbb{C}$ be such that $\lambda_1 + \cdots + \lambda_n =0$.  Then there is an $\alpha \in \ast \mathcal{A}_{\mathrm{hm}}(\riem)$ such that 
  \begin{equation}\label{periodjaveln}
    \int_{\partial_k \riem} \alpha = \lambda_k   
  \end{equation}   
  for all $k=1,\ldots,n$.  
 \end{corollary}
 \begin{proof}
Since for any $\alpha$ in $\ast \mathcal{A}_{\mathrm{hm}}(\riem),$ the exactness of the elements of $\mathcal{A}_{\mathrm{hm}}(\riem)$ yields that $\alpha= \ast d \left( \sum_l a_l \omega_l \right),$ it is enough to determine the $a_l$'s in such a way that \eqref{periodjaveln} holds. Removing one value, say $\lambda_n$, we conclude that solving \eqref{periodjaveln} amount to solving the system of equations
  \begin{align*}
    \lambda_k & = \int_{\partial_k \riem} \ast d \left( \sum_l a_l^k \omega_l \right) = \int_{\partial \riem}  \omega_k \ast d \left( \sum_l a_l^k \omega_l \right) \\
    & = \sum_{l=1}^{n-1} \Pi_{kl}\, a_l^k.
  \end{align*}
  By Theorem \ref{th:period_matrix_invertible}, the matrix $\Pi_{kl}$ is invertible so this has a unique solution $a_l^k $.  Once this solution is found, the remaining period equals $\lambda_n$ by noting that $\sum_{k=1}^{n} \int_{\partial_k \riem} \alpha =0$.
 
 \end{proof}
\end{subsection}

\begin{subsection}{Green's functions} 
Another basic notion which is of fundamental importance in our investigations is that of Green's functions.
 
 \begin{definition}\label{defn:greens function} Let $\riem$ be a type $(g,n)$ surface.
For fixed $z \in \riem$, we define \emph{Green's function of} $\riem$ to be a function $\gls{greenb}(w;z)$ such that 
 \begin{enumerate}
     \item for a local coordinate $\phi$ vanishing at $z$ the function $w\mapsto g(w;z) + \log|\phi(w)|$ is harmonic in an open neighbourhood
     of $z$;
     \item $\lim_{w \rightarrow \zeta} g(w;z) =0$ for any $\zeta \in \partial \riem$.   
 \end{enumerate}
 \end{definition}
  That such a function exists, follows from \cite[II.3 11H, III.1 4D]{Ahlfors_Sario}, considering $\riem$ to be a subset of its double $\riem^d$.\\  
\begin{definition}
 For compact surfaces $\mathscr{R}$, one defines the \emph{Green's function} $\gls{greenc}$ (see e.g. \cite{Royden}) as the unique function
 $\mathscr{G}(w,w_0;z,q)$ satisfying 
 \begin{enumerate}
  \item $\mathscr{G}$ is harmonic in $w$ on $\mathscr{R} \backslash \{z,q\}$;
  \item for a local coordinate $\phi$ on an open set $U$ containing $z$, $\mathscr{G}(w,w_0;z,q) + \log| \phi(w) -\phi(z) |$ is harmonic 
   for $w \in U$;
  \item for a local coordinate $\phi$ on an open set $U$ containing $q$, $\mathscr{G}(w,w_0;z,q) - \log| \phi(w) -\phi(z) |$ is harmonic 
   for $w \in U$;
  \item $\mathscr{G}(w_0,w_0;z,q)=0$ for all $z,q,w_0$.  
 \end{enumerate}
\end{definition}

 The existence of such a function is a standard fact about Riemann surfaces, see for example \cite{Royden}.  It satisfies the following identities:
 \begin{align}
  \mathscr{G}(w,w_1;z,q) & = \mathscr{G}(w,w_0;z,q) - \mathscr{G}(w_1,w_0;z,q) \label{eq:g_w_0_dependence} \\
  \mathscr{G}(w_0,w;z,q) & = - \mathscr{G}(w,w_0;z,q) \label{eq:g_interchange_w} \\
  \mathscr{G}(z,q;w,w_0) & = \mathscr{G}(w,w_0;z,q).  \label{eq:g_interchange_both}
 \end{align}
 In particular, $\mathscr{G}$ is also harmonic in $z$ where it is non-singular.
 
 \begin{remark}
  The condition (4) involving the point $w_0$ simply determines an arbitrary additive constant, and is not of any interest in the paper. 
 This is because by the property (\ref{eq:g_w_0_dependence}), $\partial_w \mathscr{G}$ is independent
 of $w_0$, and only such derivatives enter in the paper.  For this reason, we usually leave $w_0$ out of the expression for $\mathscr{G}$.
 \end{remark}

 

 Green's function is conformally invariant.  That is, if $\riem$ is of type $(g,n)$, and $f:\riem \rightarrow \riem'$ is conformal, then 
 {\begin{equation} \label{eq:Greens_conf_inv_gntype}
  g_{\riem'}(f(w);f(z)) = g_{\riem}(w;z). 
 \end{equation}}
 
 Similarly if $\mathscr{R}$ is compact and $f:\mathscr{R} \rightarrow \mathscr{R}'$ is a biholomorphism, then 
 \begin{equation} \label{eq:Greens_conf_inv_compact}
   \mathscr{G}_{\mathscr{R}'}(f(w),f(w_0);f(z);f(q)) = \mathscr{G}_{\mathscr{R}}(w,w_0;z,q). 
 \end{equation}

 These follow from uniqueness of Green's function; in the case of type $(g,n)$ surfaces, one also needs the fact that a biholomorphism extends to a homeomorphism of the boundary curves. 

 \end{subsection}
\begin{subsection}{Sewing}  \label{se:sewing}
We start by defining the quasisymmetric homeomorphisms of the circle.
 \begin{definition}\label{de:quasisymmetriccircle}
 An orientation-preserving homeomorphism $h$ of $\mathbb{S}^1$ is called an \emph{orientation-preserving quasisymmetric mapping}, iff there is a constant $k>0$, such that for every $\theta$, and every $\psi$ not equal to a multiple of $2\pi$, the inequality
 \[  \frac{1}{k} \leq \left| \frac{h(e^{i(\theta+\psi)})-h(e^{i\theta})}{h(e^{i\theta})-h(e^{i(\theta-\psi)})} \right|
    \leq k \]
 holds. 
 We say that $h$ is an orientation-reversing quasisymmetry if $h \circ s$ is an orientation-preserving quasisymmetry where $s(e^{i\theta}) = e^{-i\theta}$. 
\end{definition}
 A quasisymmetry is either an orientation-preserving or orientation-reversing quasisymmetry.
 
 We generalize this to general Riemann surfaces of type $(g,n)$.  
 \begin{definition} Fix $k \in \{1,\ldots, n\}$.
  Let $\tau:\mathbb{S}^1 \rightarrow \partial_k \riem$ be a homeomorphism.  We say that 
  $\tau$ is a \emph{quasisymmetry} if there is a collar chart $\phi: U \rightarrow \mathbb{A}_{r,1}$ of $\partial_k \riem$ such that $\phi \circ \tau$ is a quasisymmetry in the sense of Definition \ref{de:quasisymmetriccircle}.  We say that $\tau$ is orientation-preserving (resp. orientation-reversing) when $\phi \circ \tau$ is orientation-preserving (resp. orientation-reversing).   
 \end{definition}
 \begin{theorem} Let $\tau: \mathbb{S}^1 \rightarrow \partial_k \riem$ be a homeomorphism 
 for some fixed $k \in \{1,\ldots,n\}$.  If $\phi \circ \tau$ is a quasisymmetry of $\mathbb{S}^1$ for some collar chart $\phi$ of $\partial_k \riem$, then $\phi \circ \tau$ is a quasisymmetry of $\mathbb{S}^1$ for any collar chart $\psi$ of $\partial_k \riem$.  
 \end{theorem}
 \begin{proof} 
  If $\psi$ is another collar chart, then $\psi \circ \phi^{-1}$ is a conformal map from some collar neighbourhood of $\mathbb{S}^1$ to another collar neighbourhood of $\mathbb{S}^1$. It extends homeomorphically to the boundary by Theorem \ref{theorem on collar charts}. Thus by Schwarz reflection $\psi \circ \phi^{-1}$ extends to a conformal map of a neighbourhood of $\mathbb{S}^1$. Thus $\psi \circ \tau = \psi \circ \phi^{-1} \circ \phi \circ \tau$ is also a quasisymmetry.
 \end{proof}
 In a similar way, we can define the notion of analytic parametrization.
 \begin{definition}
  We say that $\tau$ is an \emph{analytic parametrization} if $\phi \circ \tau$ is analytic for any collar chart $\phi$. 
 \end{definition}
   Using the quasisymmetric homeomorphisms above, one can define a sewing operation between two bordered Riemann surfaces as follows
   \begin{definition} \label{de:sewing_surfaces}
    Let $\riem_1$ and $\riem_2$ be bordered surfaces of type $(g_1,n_1)$ and $(g_2,n_2)$ respectively.  Let $\tau_1: \mathbb{S}^1 \rightarrow \partial_{k_1} \riem_1$ and $\tau_2: \mathbb{S}^1 \rightarrow \partial_{k_2} \riem_2$ be orientation-reversing quasisymmetries. 
 We can \emph{sew} these surfaces to get a new topological space $\riem$ defined by the equivalence relation
 \[  q_1 \sim q_2 \Leftrightarrow q_2 = \tau_2 \circ \tau_1^{-1}(q_1).    \]
 We call the set of points in $\riem$ corresponding to the boundaries \emph{the seam}. 
   \end{definition}
  
 In this connection we have the following:  
 \begin{theorem}[\cite{RadnellSchippers_monster}] \label{th:sewing_complex_structure} The surface  $\riem$ in \emph{Definition \ref{de:sewing_surfaces}} has a complex structure which is compatible with that of $\riem_1$ and $\riem_2$.  This complex structure is unique.  The seam is a quasicircle.  If $\tau_1$ and $\tau_2$ are analytic then the seam is an analytic Jordan curve.
 \end{theorem}
 Recall that analytic Jordan curves are strip-cutting by definition.\\
 
 In what follows we shall denote the unit disk $ \{z:|z| <1 \}$ by $\disk.$
 \begin{corollary} \label{co:caps_can_be_sewn_on}  Let $\riem$ be a bordered surface of type $(g,n)$.  There is a compact surface $\mathscr{R}$ and an inclusion $\iota : \riem \rightarrow \mathscr{R}$ which is a biholomorphism onto its image, which extends continuously to a homeomorphism of the boundary curves of $\riem$ into $n$ disjoint quasicircles in $\mathscr{R}$, such that $\mathscr{R} \backslash \mathrm{cl} \, (\riem)$ consists of $n$ open regions biholomorphic to $\disk$.  
  If desired, the quasicircles can be chosen to be analytic curves.
 \end{corollary}
 \begin{proof} Let $\tau_k:\mathbb{S}^1 \rightarrow \partial_k \riem$ be orientation-reversing quasisymmetries for $k=1,\ldots,n$.  Using the parametrization 
 $z \mapsto {1/z}$, sew on $n$ copies of $\disk$ to $\riem$.  The claim follows from 
 Theorem \ref{th:sewing_complex_structure}.
 \end{proof}
 
\begin{definition} \label{de:sewing_on_caps}
 We refer to this procedure as \emph{sewing caps on} $\riem$, where a \emph{cap} is a connected component of $\mathscr{R} \backslash \riem$.   
\end{definition}  
\end{subsection}

\end{section}
\begin{section}{Conformally non-tangential limits and overfare of harmonic functions} \label{se:CNT_all}
\begin{subsection}{About this section}
 This section accomplishes two goals. The first is to develop a theory of boundary values of Dirichlet bounded harmonic functions. The second is to overfare these functions in quasicircles. By overfare, we mean the following process. We are given a compact Riemann surface $\mathscr{R}$ divided in two pieces $\riem_1$ and $\riem_2$ by a collection of quasicircles $\Gamma$. A function $h_1 \in \mathcal{D}_{\mathrm{harm}}(\riem_1)$ has boundary values on $\Gamma$. There is then a unique function $h_2 \in \mathcal{D}_{\mathrm{harm}}(\riem_2)$ with these same boundary values. We say $h_2$ is the ``overfare'' of $h_1$ and denote it by $h_2=\mathbf{O}_{1,2}h_1$. 
 
 This simple idea some technical work to make rigorous. The sewing technique is a key tool throughout. First, we need a notion of boundary values; these are what we call conformally non-tangential boundary values. They are defined in Section \ref{se:CNT_limits_BVs}; briefly, we use a collar chart to map the function near the boundary to the disk, and apply Beurling's theorem on non-tangential boundary values of Dirichlet bounded functions. We then show that this is independent of the choice of collar chart. To prove that the overfare process makes sense, it must be shown that the set of possible boundary values is the same from either side. This includes showing that a set which is negligible from the point of view of $\riem_1$ is also negligible from the point of view of $\riem_2$.  Here, by negligible, we mean that the boundary values can be changed on this set without changing the solution to the boundary value problem. Again, this is accomplished by cutting and pasting neighbourhoods of the boundary, applying a chart, and using the corresponding result in the plane. A negligible set (which we call ``null'') is a Borel set whose image under the chart is a set of logarithmic capacity zero. This is done in Section \ref{se:function_Overfare}. 
 
 We will also prove that the overfare operator is bounded, using sewing techniques. The proof proceeds in steps. First, we show that a certain ``bounce operator'' is bounded. This bounce operator acts entirely within one surface, say $\riem_1$. It takes Dirichlet bounded functions defined on a collar neighbourhood of the collection of quasicircles, and produces the unique Dirichlet bounded function on the Riemann surface $\riem_1$ with the same boundary values.  We show in Section \ref{se:bounce} that this is bounded; this follows essentially from the existence and continuous dependence of solutions to the Dirichlet problem together with the fact that the Sobolev trace is bounded. Then, we
 define a ``local'' overfare as follows. Given a function defined in a collar neighbourhood of a boundary curve in $\riem_1$,  we cut out a tubular neighbourhood of a quasicircle, and map it into the plane with a doubly connected chart.  Using the fact that bounce and overfare are bounded in the plane, we obtain a bounded map taking Dirichlet bounded functions on a collar neighbourhood in $\riem_1$ to Dirichlet bounded functions in a collar neighbourhood in $\riem_2$.  
 
 The overfare operator is then shown to be bounded by first overfaring locally and then applying the bounce operator on $\riem_2$. Since every step is bounded, this will complete the proof.

 In previous works of the authors, only one curve was involved. This meant that constant functions overfare to constant functions. For this reason, it was sufficient to work with the Dirichlet semi-norm. However, if there are many curves, it is possible that many constants are involved, and indeed it is even possible that the overfare of a locally constant function is a non-constant function. It is then possible to drive up the Dirichlet semi-norm on one side while it is unchanged on the other. 
 
 If the originating surface is connected, this problem does not arise. In this case, we show that overfare is bounded with respect to the Dirichlet semi-norm for general quasicircles. To control the constants, we need to work with a  true norm. We introduce such a conformally invariant norm, which can be given in several equivalent forms. We show that for quasicircles with greater regularity the overfare is bounded with respect to this true norm. This conformally invariant norm also plays an important role in the theory of boundary values of $L^2$ harmonic one-forms established in Section \ref{se:Dirichlet_problem}.
\end{subsection}
\begin{subsection}{CNT limits and boundary values of functions and forms}  \label{se:CNT_limits_BVs}
 
 In this section, we define a notion of non-tangential limit which is conformally invariant.  Existence of this limit is independent of coordinate.   In a sense, this is the natural notion of non-tangential limit on the border of a Riemann surface.  The main idea is that any border chart determines a notion of non-tangential approach to a point on the boundary, and the compatibility of border charts implies that this notion is independent of chart. \\
 
 We now give the precise definition.
 First, we recall the definition of non-tangential limit on the upper half plane and the disk $\mathbb{D}$.  For $\theta \in (0,\pi/2)$ and $p \in \partial \mathbb{H}$ define the wedge
 \[  V_{p,\theta} = \{ z \in \mathbb{H} : \pi/2 - \theta < \text{arg}(z - p) < \pi/2 + \theta \}.  \]
 Let $h:U \rightarrow \mathbb{C}$ be a function defined on an open set $U$ in $\mathbb{H}$ which contains a half disk $D_r = \{z:  |z-p| <r, z \in \mathbb{H} \}$.
 \begin{definition}
  We say that $h$ has a \emph{non-tangential limit at} $p$ if 
 \[   \lim_{z \rightarrow p} \left. h\right|_{V_{p,\theta} \cap U}  \]
 exists for every $\theta \in (0,\pi/2)$.  
 \end{definition}
 
Similarly, we can define non-tangential limit for functions $h$ on open subsets $U$ of $\disk$ containing a set $D_r = \{z:  |z-p| <r, z \in \disk \}$.  
 A non-tangential wedge in $\disk$ with vertex at $p \in \mathbb{S}^1$ is a set of the form
 \[ W(p,M)  = \{ z \in \disk : |p-z| < M(1-|z|)   \} \]
 for some $M \in (1,\infty)$.
 We say that a function $h:\disk \rightarrow \mathbb{C}$ has a non-tangential limit at $p$
 if the limit of $\left. h \right|_{W(p,M) \cap U}$ as $z \rightarrow p$ exists for all $M \in (1,\infty)$.
 One may of course equivalently use Stolz angles, that is sets of the form
 \[ S(p, \alpha)= \{z: \text{arg}(1-\bar{p} z) < \alpha, |z| < \rho \cos{\alpha} \}   \]
 where $\alpha \in (0,\pi/2)$ \cite[p6]{Pommerenke_boundary_behaviour}.  
 
 It is easily seen that if $T:\disk \rightarrow \disk$ is a disk automorphism, then $h$ has a non-tangential limit at $p$ if and only if $h \circ T$ has a non-tangential limit at $T(p)$.  A similar statement holds for non-tangential limits in the upper half plane.  Finally, observe that if $T$ is a M\"obius transformation from $\disk$ to $\mathbb{H}$ then a function $h$ on a subset of the upper half plane has a non-tangential limit at $p$ if and only if $h \circ T$ has a non-tangential limit at $T^{-1}(p)$.  
 
 We now define conformally non-tangential limits.  Let $U$ be an open subset of $\riem$ and let $h:U \rightarrow \mathbb{C}$.  Let $p \in \partial \riem$.  We say that $h$ is ``defined near $p$'' if there is a boundary chart $\phi:V \rightarrow \text{cl}( \mathbb{H})$ such that $\phi(U)$ contains a half-disk $D_r = \{ z: |z-p|<r, \ \ z \in \mathbb{H} \}$.  
 \begin{definition}\label{defn:CNTlimit}  Let $\riem$ be a Riemann surface with border $\partial \riem$.  Fix $p \in \partial \riem$ and let $h:U \rightarrow \mathbb{C}$ be defined near $p$.  
  We say that $h$ has a conformally non-tangential limit at $p$ if there is a boundary chart $\phi:V \rightarrow \text{cl}(\mathbb{H})$ such that $p \in V$ and $h \circ \phi^{-1}$ has a non-tangential limit at $\phi(p)$.  
 \end{definition}
 
 We will use the acronym CNT in place of ``conformally non-tangential''.  
 The following theorem shows that the existence of the CNT limit does not depend on the chart, in the sense that the condition of the definition holds either for all boundary charts or none.  
 \begin{proposition}
For fixed $p \in \partial \riem$,  let $h: U \rightarrow \mathbb{C}$ be defined near $p$ and let $h$ have a \emph{CNT} limit equal to $\zeta$ at $p$. Then the \emph{CNT} limit is independent of the boundary chart used in \emph{Definition \ref{defn:CNTlimit}}. That is, for any boundary chart $\psi:W \rightarrow \mathbb{H}$, $h \circ \psi^{-1}$ has a non-tangential limit equal to $\zeta$ at $\psi(p)$.  The same claims holds for boundary charts $\psi:W \rightarrow \mathbb{D}^+$.  
 \end{proposition}
 \begin{proof}  
  Assume that $h \circ \phi^{-1}$ has a non-tangential limit equal to $\zeta$ at $\phi(p)$ for some boundary chart $\phi:V \rightarrow \mathbb{H}$.  Let $\psi:W \rightarrow \mathbb{H}$ be any other boundary chart near $p$.  By the Schwarz reflection principle, $\phi \circ \psi^{-1}$ extends to a biholomorphism from an open neighbourhood of $\psi(p)$ to an open neighbourhood of $\phi(p)$.  In particular, for any non-tangential wedge $V_{\psi(p),\theta}$ there is a disk $D$ at $\phi(p)$ such that $\phi \circ \psi(p)(D \cap V_{\psi(p),\theta})$ is contained in a non-tangential wedge at $\phi(p)$. Thus the limit as $z$ approaches $\psi(p)$ of $h \circ \psi^{-1} = h \circ \phi \circ \phi \circ \psi^{-1}$ within $D \cap V_{\psi(p),\theta}$ equals $\zeta$.  
 \end{proof}  
 It follows immediately from the definition of CNT limits that they are conformally invariant.  Although this is a simple consequence it deserves to be highlighted.  
 \begin{theorem}[Conformal invariance of CNT limits] 
  \label{th:conformal_inv_CNT} Let $\riem$ be a bordered Riemann surface and  $h: U \rightarrow \mathbb{C}$  be a function defined near $p \in \partial \riem$.  If $F: \riem_1 \rightarrow \riem$ is a   conformal map, then $h$ has a \emph{CNT} limit of $\zeta$ at $p$  if and only if $h \circ F$ has a \emph{CNT} limit of $\zeta$ at  $F^{-1}(p)$.   
 \end{theorem} 
 
 Next, we define a potentially-theoretically negligible set on the border which we call a null set.  We first need a lemma.  
 \begin{lemma} \label{le:null_chart_independent}
  Let $\riem$ be a type $(g,n)$ bordered surface and let $\phi_k:U_k \rightarrow \mathbb{A}_{r_k,1}$ be collar charts of a boundary curve $\partial_j \riem$ for $k=1,2$ and some fixed $j \in \{1,\ldots n\}$.  Let $I\subset \partial_j \riem$ be a Borel set. Then $\phi_1(I)$ has logarithmic capacity zero 
  if and only if $\phi_2(I)$ has logarithmic capacity zero. 
 \end{lemma}
 \begin{proof}
If $K \subset \mathbb{S}^1 = \{z\,:\, |z|=1 \}$ is 
 a Borel set of logarithmic capacity zero, and $\phi$ is a quasisymmetry, then $\phi(K)$ has logarithmic capacity zero \cite[Theorem 2.9]{Schippers_Staubach_JMAA}. Since the inverse of a quasisymmetric map is also a quasisymmetry (and in particular 
 a homeomorphism), we see that a Borel set $K$ has logarithmic capacity zero if and only if $\phi(K)$  is a Borel set of logarithmic capacity zero. 
 
 Now let $\phi_1:U_1 \rightarrow \mathbb{A}_{r_1, 1}$ and $\phi_2:U_2 \rightarrow \mathbb{A}_{r_2, 1}$ be collar charts
  such that $U_1$ and $U_2$ are in $\riem$.   By 
 composing with a scaling and translation we can obtain maps $\tilde{\phi_1}$ and $\tilde{\phi_2}$ such that the image of $\Gamma$ under both 
 $\tilde{\phi}_1$ and $\tilde{\phi}_2$ is $\mathbb{S}^1$; we can also arrange that $\mathbb{S}^1$ is the outer boundary of both $\mathbb{A}$ and $\mathbb{B}$ by composing with $1/z$ if necessary. By Lemma \ref{theorem on collar charts}, $\tilde{\phi}_1\circ \tilde{\phi}_{2}^{-1}$  has a homeomorphic extension to $\mathbb{S}^1$.  By the Schwarz reflection principle, it has an extension to a conformal map of an open neighbourhood of $\mathbb{S}^1$, so it is an analytic diffeomorphism of $\mathbb{S}^1$ and in particular a 
 quasisymmetry. Thus $\tilde{\phi}_{2}(I)$ has logarithmic capacity zero if and only if $\tilde{\phi}_{1}(I)$ has capacity zero.  Since linear maps $z \mapsto az  + b$ take  
 Borel sets of capacity zero to Borel sets of capacity zero, as does $z \mapsto 1/z$, we have that $\phi_{1}(I)$ has logarithmic 
 capacity zero if and only if $\phi_{2}(I)$ does.  This completes the proof.
 \end{proof} 
 
The previous lemma motivates and justifies the following definition.
 \begin{definition}
  Let $\riem$ be a bordered Riemann surface of type $(g,n)$. We say that a Borel set $I \subset \partial_k \riem$ is a \emph{null set} if  $\phi(I)$ is a set of logarithmic capacity zero in $\mathbb{S}^1$ for some collar chart $\phi$ of $\partial_k \riem$.  
  We say that a Borel set $I$ in $\partial \riem$ is null if it is a union of null sets $I_k \subset \partial_k \riem$, $k=1,\ldots,n$.  
 \end{definition}

 We also have the following two results:
 
 \begin{proposition} If $I_1$ and $I_2$ are null in $\partial_k \riem$ then $I_1 \cup I_2$ is null.  
 \end{proposition}
 \begin{proof}
  It is enough to show that the union of Borel sets $I_1$ and $I_2$ of logarithmic capacity zero in $\mathbb{S}^1$ are of logarithmic capacity zero.  By Choquet's theorem, the outer capacity of $I_1$ and $I_2$ equal their capacity.  Since outer capacity is subadditive, 
  the outer capacity of $I_1 \cup I_2$ is zero.  The claim follows from another application of Choquet's theorem.
 \end{proof}

 Harmonic functions which are Dirichlet bounded near a border have CNT boundary values except possibly on a null set.
 \begin{theorem} \label{th:CNT_BVs_collar_existence} Let $\riem$ be a bordered Riemann surface of type $(g,n)$.  Let $U_k$ be a collar neighbourhood of $\partial_k \riem$ for some $k \in \{1,\ldots,n \}$.  If $h \in \mathcal{D}_{\mathrm{harm}}(U_k)$ then $h$ has \emph{CNT} boundary values on $\partial_k \riem \backslash I$ for some null set $I \subset \partial_k \riem$.  
 \end{theorem}  
 \begin{proof}
  By conformal invariance of the Dirichlet space and CNT boundary values (Theorem \ref{th:conformal_inv_CNT}), it is enough to prove this for an annulus in the plane, which is a special case of  \cite[Theorem 3.12]{Schippers_Staubach_transmission}. 
 \end{proof}
 
  \begin{remark} \label{re:Sobolev_CNT_agree} 
  The non-tangential boundary values agree with the Sobolev trace up to a set of measure zero, if the boundary is sufficiently regular.  This holds for example if we treat the border as an analytic curve in the double.\\  In fact if one has an $(\varepsilon,\delta)$ domain $\Omega$ (in the plane these are quasidisks) with Ahlfors-regular boundary in the sense of Definition 1.1 of \cite{brewster_mitrea}, then using Theorem 8.7 (iii) in \cite{brewster_mitrea} and taking $s=1,$ $p=2$ and $n=2$, we have that  their condition $s-\frac{n-d}{p}= 1-\frac{2-1}{2}=\frac{1}{2}\in (0,\infty)$ is satisfied. Thus, Theorem 8.7 (iii) in \cite{brewster_mitrea} yields that the Sobolev trace belonging to $H^{1/2}(\partial \Omega)$ agrees almost everywhere (since the $1$-dimensional Hausdorff measure on $\partial \Omega$ is the 1-dimensional Lebesgue measure) with the non-tangential limit of the function $h\in H^1 (\Omega).$ Note that chord-arc domains, are examples of $(\varepsilon,\delta)$ domains with Ahlfors-regular boundary.
 \end{remark}
 \begin{theorem}  \label{th:CNT_BVs_existence_uniqueness} Let $\riem$ be a bordered surface of type $(g,n)$.  If $h \in \mathcal{D}_{\mathrm{harm}}(\riem)$, then there is a null set $I \subset \partial \riem$ such that $h$ has \emph{CNT} boundary values on $\riem \backslash I$.  If $H$ is any element of $\mathcal{D}_{\mathrm{harm}}(\riem)$ with \emph{CNT} boundary values which agree with those of $h$ except possibly on a null set $J$, then $h=H$.  
 \end{theorem}
 \begin{proof} 
  The first claim follows directly from Theorem \ref{th:CNT_BVs_collar_existence}. 
   For the uniqueness part, it is well-known that if $X$ is a smooth compact Riemannian manifold with boundary, then the Dirichlet problem 
  
  \begin{equation}\label{dirichlets problem}
    \begin{cases}
    \Delta u=0 \\
    u|_{\partial X}= f\in H^{1/2}(\partial X)
    \end{cases}
  \end{equation}
  
 has a unique solution that satisfies
 $$\Vert u\Vert_{H^1(X)}\leq C  \Vert f\Vert_{H^{1/2}(\partial X)},$$
 see e.g. \cite[Proposition 1.7, p 360]{Taylor}.  Using this together with Remark \ref{re:Sobolev_CNT_agree} it follows that if $H=h$ up to a null set on $\partial\riem$ then $h=H$.
 \end{proof}
 
A suitable adaptation of the proof of \cite[Theorem 3.17]{Schippers_Staubach_transmission} also yields
 \begin{theorem} \label{th:bounce_existence} Let $\riem$ be a bordered surface of type $(g,n)$ and 
  let $U_k \subseteq \riem$ be collar neighbourhoods of $\partial_k \riem$ for $k=1,\ldots,n$. Let $h_k \in \mathcal{D}_{\mathrm{harm}}(U_k)$ for $k=1,\ldots,n$.  There is a function $H \in \mathcal{D}_{\mathrm{harm}}(\riem)$ whose \emph{CNT} boundary values agree with those of $h_k$ on $\partial_k \riem$ up to a null set for each $k=1,\ldots,n$.  
 \end{theorem}
 
 We thus make the following definition.  
 \begin{definition}\label{defn: Osborn space}  Let $\riem$ be a Riemann surface and let $\Gamma$ be a finite collection of borders of $\riem$ each of which is homeomorphic to $\mathbb{S}^1$.  
 Given functions $h_k :\Gamma \backslash I_k \rightarrow \mathbb{C}$ where $I_1$ and $I_2$ are null sets, we say that $h_1 \sim h_2$ if $h_1$ and $h_2$ are both defined on $\Gamma \backslash I$ for some null set $I$ and $h_1 = h_2$ on $\Gamma \backslash I$.  
  The \emph{Osborn space} $\mathcal{H}(\Gamma)$ is the set of equivalence classes of such functions.  
 \end{definition}
 \begin{remark} It follows from the results of this section that 
  every element of $H^{1/2}(\Gamma)$, which is defined almost everywhere, has a unique extension to an element of $\mathcal{H}(\Gamma)$ which is defined except possibly on a null set.  
 \end{remark}
 \end{subsection}
 \begin{subsection}{Anchor lemma and boundary integrals}
 
Having defined the notion of CNT boundary values in the previous section, we establish two lemmas which allow us to consistently define integrals of the form 
 \[  \int_\Gamma \alpha \, h  \]
 where $\Gamma$ is a boundary curve of a Riemann surface, $\alpha$ is an $L^2$ harmonic one-form in a collar neighbourhood of $\Gamma$, and $h$ is a harmonic function with finite Dirichlet norm in a collar neighbourhood of $\Gamma$. Moreover the integral, as far as $h$ is concerned, depends only on the CNT boundary values of $h$ on $\Gamma$.
 
 We do this by evaluating the integral along curves which approach $\Gamma$ in the limit. We first describe these limits.  Let $\riem$ be a Riemann surface of type $(g,n)$, $\Gamma_k$ be one of its boundary curves, and $\phi:A \rightarrow \mathbb{A}$ be a collar chart for $\Gamma_k$.  We assume that $A \subseteq \riem_1$ for the sake of the definition; the identical construction will hold for $\riem_2$.  By Remark \ref{re:collar_chart_provides_isotopy}, setting $C_r = \{z : |z|=r \}$ for $r \in (0,1)$
 \begin{equation} \label{eq:isotopy} 
  \Gamma_k^r = \phi(C_r)
 \end{equation} 
 is an isotopy of analytic Jordan curves on $[R,1]$  for some $R \in (0,1)$, such that $\Gamma_k^1 = \Gamma_k$.  
 
  The following two lemmas show that the limiting integrals are well-defined in the sense that they are independent of the choice of limiting curves (the first anchor lemma, and depend only on the boundary values (the second anchor lemma).    
  \begin{lemma}[First anchor lemma] \label{le:anchor_lemma_one}  Let $\phi:A \rightarrow \mathbb{A}$ be a collar chart of $\Gamma_k$ in $\riem_1$.
  Let $\alpha \in \mathcal{A}(A)$.  For any $h \in \mathcal{D}_{\mathrm{harm}}(A)$ 
  \[  \lim_{r \nearrow 1} \int_{\Gamma^r_k} \alpha(w) h(w)  \]
   exists. Furthermore, this quantity is independent of the collar chart.  
  \end{lemma}
  \begin{proof} 
   Existence follows from Stokes' theorem, since 
   \begin{equation} \label{eq:anchor_lemma_temp} 
      \lim_{r \nearrow 1} \int_{\Gamma_k^r} \alpha(w) h(w) = 
         \int_{\Gamma_k} \alpha(w) h(w) + \iint_{A_r} \alpha \wedge \overline{\partial} h (w).  
   \end{equation}
   where $A_r$ is the region bounded by $\Gamma^r_k$ and $\Gamma_k$.
   This existence argument of course applies to any choice of collar chart. 
   
   We need to show that it gives the same result regardless of the choice.
   By change of variables, it is enough to prove this in the situation that $\Gamma_k = \mathbb{S}^1$ and $A =\mathbb{A}_{r,1}$, and $\phi=\text{Id}$. The curves $\Gamma_k^r$ are then just $|z|=r$.  Let $\phi':A' \rightarrow \mathbb{A}'$ be some collar chart of $\mathbb{S}^1$.  Let $\gamma_k^r$ denote the isotopy induced by $\phi'$.

   Fix any $\varepsilon >0$ and choose $R$ such that 
   \[  \left| \lim_{r \nearrow 1} \int_{\gamma_k^r} \alpha(w) h(w)  - 
   \int_{\gamma_k^{R}} \alpha(w) h(w) \right| < \varepsilon/2  \]
   and 
   \begin{equation} \label{eq:anchor_lemma_one_temptwo}
    \| \alpha \|_{\mathcal{A}(A'_R)} \| \overline{\partial} h\|_{\overline{\mathcal{A}(A'_R)}} < \varepsilon/2  
   \end{equation}
   where $A'_R$ is the region bounded by $\mathbb{S}^1$ and $\gamma_k^R$.  
   Since $\gamma_k^{R}$ is compact,  $|z|$ has a maximum $M<1$ on $\gamma_k^R$.  For any $r>M$, $\Gamma_k^r$ is contained in $A'_R$ and does not intersect $\gamma_k^R$. If we let $B$ denote the region bounded by these two curves, then $B \subseteq A'_R$. Therefore using Cauchy-Schwarz's inequality we deduce that
   \begin{align*}
    \left|  \lim_{r \nearrow 1} \int_{\gamma_{k}^r} \alpha(w) h(w) 
     - \int_{\Gamma_k^r} \alpha(w) h(w)  \right| & \leq 
     \left|  \lim_{r \nearrow 1} \int_{\gamma_{k}^r} \alpha(w) h(w) -
   \int_{\gamma_k^{R}} \alpha(w) h(w)  \right| \\ 
    & \ \ \ + 
   \left| \int_{\gamma_k^{R}} \alpha(w) h(w) -  \int_{\Gamma_k^r} \alpha(w) h(w)  \right| \\
   & < \frac{\varepsilon}{2} +\left| \iint_{B} \alpha(w) \wedge \overline{\partial} h(w) \right| \\
   & \leq \frac{\varepsilon}{2} +  \| \alpha \|_{\mathcal{A}(A'_R)} \| \overline{\partial} h\|_{\overline{\mathcal{A}(A'_R)}} 
   \end{align*}
   which by \eqref{eq:anchor_lemma_one_temptwo} proves the claim.  
  \end{proof}
  
  Henceforth we will denote this limiting integral by 
  \[  \int_{\Gamma_k} \alpha(w) h(w)   \ \ \ \ \mathrm{or} \ \ \ \ 
    \int_{\partial_k \riem} \alpha(w) h(w)   \]
  if $\Gamma_k = \partial_k \riem$, 
  where the notation is justified by Lemma \ref{le:anchor_lemma_one}.

 Another useful Anchor Lemma goes as follows.
  \begin{lemma}[Second anchor lemma]  \label{le:anchor_lemma_two} Let $A$ be a collar neighbourhood  of $\Gamma_k$ in $\riem_1$ for some $k \in \{1,\ldots,n\}$.   
   If $h_1$ and $h_2$ are any two elements of $\mathcal{D}_{\mathrm{harm}}(A)$ with the same \emph{CNT} boundary values on $\Gamma_k$ up to a null set, then for any $\alpha \in \mathcal{A}(A)$ 
   \[  \int_{\partial_k \riem_1} \alpha(w) h_1(w) =  \int_{\partial_k \riem_1} \alpha(w) h_2(w).  \]
  \end{lemma}
  \begin{proof}  By Lemma \ref{le:anchor_lemma_one} we may use any collar chart to determining a limiting sequence of curves.  By Proposition \ref{prop:collar_charts_intersection} we can find a collar chart whose domain is in $A$.  Since the integral along a curve is invariant under composition with a conformal map, it is enough to prove this for $\Gamma_k = \mathbb{S}^1$ and $A= \mathbb{A}_{r,1}$ for some $r$, with limiting curves $\Gamma_k^r$ given by $|z|=r$.  
  We can apply \cite[Theorem 4.7]{Schippers_Staubach_Plemelj} or \cite[Lemma 3.21]{Schippers_Staubach_Grunsky_expository}
  to ($h_1 - h_2$) in this case.   
  \end{proof}
 Thus, as was mentioned earlier, the limiting integral of $h$ against any $\alpha \in \mathcal{A}(A)$ exists and depends only on the CNT boundary values of $h$. 
 
   \begin{remark} 
   We will often consider the situation where the Riemann surface $\riem$ is a subset of a compact surface $\mathscr{R}$, where the boundary is irregular (such as a quasicircle). However the anchor lemmas involve only the assumption that the boundary is a border (and hence, a collar chart exists).
   In particular, no reference is made to any outside surface, and thus they apply in the situation above.  
  \end{remark}
  
Next we define certain boundary integrals of Dirichlet-bounded harmonic functions.
 Let $d\omega_k$ be the harmonic measures given in Definition \ref{def:harmonic measure}. For a collar neighbourhood $U_k$ of $\partial_k \riem$ and $h_k \in \mathcal{D}_{\mathrm{harm}}(U_k)$, assume that the inner boundary of $U_k$ is an analytic curve $\gamma_k$.  By Stokes' theorem (where recall that the left hand side is defined by a limit of curves approaching $\partial_k \riem$, and well-defined by Lemma \ref{le:anchor_lemma_one}) we have
 \begin{equation}\label{defn:integration over rough}
  \int_{\partial_k \riem}  h_k \ast d \omega_k: =  \iint_{U_k} dh_k \wedge \ast d\omega_k - \int_{\gamma_k} h_k \ast d \omega_k   
 \end{equation}   
 where we give $\gamma_k$ the same orientation as $\partial_k \riem$. 
 The left hand side is in independent of the choice of curve $\gamma_k$, and thus so is the right hand side. 
 
  Given $h \in \mathcal{D}_{\mathrm{harm}}(\riem)$ we set
   \[  \mathscr{H}_k  := \int_{\partial_k \riem} h_k \ast d\omega_k.  \]
In the case that $n=1$,  fix a point $p\in \riem \setminus U_1$ and define instead 
  \begin{equation} \label{eq:constant_hk_definition_onecurve}
    \mathscr{H}_1 := \int_{\partial_1 \riem} h_1 \ast d \mathscr{G}(w,p),
  \end{equation}
  where $ \mathscr{G}(w,p)$ is Green's function of $\riem.$
  
  We can also use Green's function to define the norm in the case that $n>1$, as the following lemma shows. The different characterizations will be of use to us later.
  \begin{lemma} \label{le:point_also_H1conf} Let $\riem$ be a connected Riemann surface of type   $(g,n)$. For any fixed point $p \in \riem$, the norms given by
  \begin{align*}
   \| h \|^2_{\mathcal{D}_{\mathrm{harm}}(\riem),p} & = \| h \|^2_{\mathcal{D}_{\mathrm{harm}}(\riem)} + |h(p)|^2  \\
   & =  \| h \|^2_{\mathcal{D}_{\mathrm{harm}}(\riem)} + \left| \lim_{\varepsilon \searrow 0} \frac{1}{\pi i}  \int_{\Gamma_\varepsilon} \partial_w {g}(w;p)  h(w)  \right|^{2},
  \end{align*}
  where $g$ is Green's function of $\riem$ and $\Gamma_\varepsilon$ are the level curves of Green's function based at $p$, and the $H_{\mathrm{conf}}^1(\riem)$ norm are equivalent.  
  \end{lemma}
 
  \begin{proof} If $n=1$ there is nothing to prove.
  First we note that if $U\subset \riem$ is a small neighbourhood of $p\in \riem$  then by the mean-value theorem for harmonic functions and Jensens inequality we have that $|h(p)|^2 \lesssim \Vert h\Vert^2_{L^2(\riem)}$, which confirms condition (2) of  Theorem \ref{thm:equiv sobolev norm}. Therefore, since conditions (1) and (3) of that theorem are also trivially satisfied, the Lemma follows.
  \end{proof}
  
{{This can be used to construct a conformally invariant version of Sobolev spaces on Riemann surfaces. 
  \begin{definition}\label{def:norm on H1confU}
  Set $U = U_1 \cup \cdots \cup U_n$ as above. By ${H}^{1}_{\mathrm{conf}}(U)$ we denote the harmonic Dirichlet space $\mathcal{D}_{\mathrm{harm}}(U)$ endowed with the norm
  \begin{equation}\label{defn:Dconf norm}
 \| h \|_{\gls{sobolevconf}(U)} := \Big(\| h \|^2_{\mathcal{D}_{\mathrm{harm}}(U)} + 
  \sum_{k=1}^{n}   |\mathscr{H}_k |^2\Big)^{\frac{1}{2}}.  
 \end{equation} 
   For the Riemann surface $\riem$, we can choose any fixed boundary curve $\partial_n \riem$ say, and define the norm 
 \begin{equation}\label{equivalent sobolev norm}  
     \| h \|_{{H}^1_{\mathrm{conf}}(\riem)} := \left( \| h \|^2_{\mathcal{D}_{\mathrm{harm}}(\riem)} + |\mathscr{H}_n |^2 \right)^{1/2},
 \end{equation}  
(where any of the $\mathscr{H}_k$ could in fact be used in place of $\mathscr{H}_n$).  
 \end{definition}}}
 
 \begin{theorem}\label{thm:equivnorms on riemsurfaces}
  Let $\riem$ be a Riemann surface of type $(g,n)$.  Then, the ${H}^{1}_{\mathrm{conf}}(\riem)$ norm is equivalent to the $H^1(\riem)$ norm.  
  In particular, any choice of boundary curve in the definition of ${H}^{1}_{\mathrm{conf}}(\riem)$ leads to an equivalent norm.\\

 \end{theorem}  
 \begin{proof}
   We note that for any integer $0\leq k\leq n$, $|\mathscr{H}_k |\geq 0$, $\int_{\partial_k \riem}  \ast d\omega_k= -\int_{\gamma_k} \ast d\omega_k\neq 0,$ and $|\mathscr{H}_k |\leq C \Vert h\Vert_{H^1(\riem)},$ by the Cauchy-Schwarz inequality and \eqref{L2-sobolev estim}. Therefore Theorem \ref{thm:equiv sobolev norm} yields the desired result.
   \end{proof}
 
 The elements of ${H}^1_{\mathrm{conf}}$  have well-defined boundary integrals, as will be demonstrated below. This in turn hinges on existence of a collar chart stemming from the harmonic measure. A similar collar chart is also available in connection to Green's functions. These two canonical collar charts are very useful, especially in association with the evaluation of certain boundary integrals and have the property that the resulting limiting curves are level curves of harmonic measures or Green's function respectively. The first lemma tackles the case of the collar char from harmonic measure.
 
  \begin{lemma}[Collar chart from harmonic measure]  \label{le:harmonic_measure_collar_chart}  Let $\riem$ be a type $(g,n)$ Riemann surface.  
  Let $\omega_k$ be the harmonic function which is one on $\partial_k \riem$ and $0$ on the other boundary curves. Let $\psi$ be the multi-valued holomorphic function with real part $\omega_k-1$ and set
  \[  i a_k = i \int_{\partial_k \riem} \ast d\omega_k. \]
  Then
  \[   \phi(z) = \exp{\left( 2 \pi \psi/a_k \right)}  \]
  is a collar chart on some collar neighbourhood $U$ of $\partial_k \riem$. 
  Furthermore 
  \[   (\ast d\omega_k) = \frac{a_k}{2 \pi} \phi^* d\theta, \]
  and thus for any $h \in H^1_{\mathrm{conf}}(U)$ in a collar neighbourhood of $\partial_k \riem$ we have
  \[ \int_{\partial_k \riem} h \ast d\omega_k = \frac{a_k}{2\pi} \int_{\mathbb{S}^1} h \circ \phi(e^{i\theta}) d\theta.  \]
 \end{lemma}
 \begin{proof} It is clear that $\phi$ takes level curves of $\omega_k = 1-\epsilon$ to curves $|z|=e^{-\epsilon}$ for $\epsilon$ sufficiently small. Observe that $d\psi = d\omega_k $ so that the harmonic conjugate of $\omega_k-1$ is a primitive of $\ast d\omega_k$. 
 In particular $\phi$ is single-valued. 
  An elementary application of the argument principle shows that the map is a bijection for some collar neighbourhood defined by $0<\epsilon <s$ for some fixed $s$. This proves the first claim. 
  
  The second claim follows from
  \[  d \theta = d \, \mathrm{Im} \log{\psi} = \frac{2 \pi}{a_k} \ast d\omega_k.  \]  
  The final claim follows from a change of variables and the second claim.
 \end{proof}
 
 We also have
 \begin{lemma}[Collar chart from Green's function] \label{le:Greens_collar_chart}
  Let $\riem$ be a type $(g,n)$ Riemann surface and let $g$ be Green's function of $\riem$. For fixed $p$, let $\psi(w)$ be the multi-valued holomorphic function with real part $g(w;p)$. Setting 
  \[  i a_k = i \int_{\partial_k \riem} \ast dg(\cdot;p)  \]
  it holds that
  \[ \phi(w) = \exp (2\pi \psi(w)/a_k)  \]
  is a collar chart on some open neighbourhood $U$ of $\partial_k \riem$.   
 \end{lemma}
 \begin{proof}
   The proof is similar to that of the above, and can be found in \cite{Schippers_Staubach_transmission}.
 \end{proof}
 The important property of these two collar charts is that the limiting curves are level curves of the harmonic measure and Green's function respectively.\\
  
 A very useful application occurs in the following well-known reproducing property of Green's functions, which also uses the fact that Green's function based at the point $p$ induces a collar chart.
 \begin{proposition} \label{pr:Greens_reproducing} Let $\riem$ be a type $(g,n)$ Riemann surface and let $g$ be its Green's function.  
 Let $\Gamma^p_\epsilon$ denote the level curves of Green's function for any fixed $p \in \riem$. 
  For any $h \in \mathcal{D}_{\mathrm{harm}}(\riem)$
  \[  h(z)= - \frac{1}{2 \pi} \int_{\partial \riem}  \ast d_wg(w;z) h(w) = - \lim_{\epsilon \searrow 0} \frac{1}{2 \pi} \int_{\Gamma_\epsilon^p} \ast d_wg(w;z) h(w).      \]
  We also have 
  \[  h(z) = - \frac{1}{\pi i} \int_{\partial \riem} \partial_w g(w;z) h(w).   \]
 \end{proposition}
 \begin{proof} We prove the first displayed equation. 
  By Lemma \ref{le:anchor_lemma_one}, the integral on the left is well-defined, that is, the right hand side is the same no matter what the choice of $p$ is. Thus we may assume that $p=z$. In that case, Stokes' theorem and the harmonicity of $h$ yield that
  \[  \int_{\Gamma_\epsilon^p} g \ast dh = \epsilon \int_{\Gamma_\epsilon^p} \ast dh = 0.  \]
  From here, the proof proceeds in the usual way using Green's identity
  \[  \int_{\Gamma_\epsilon^p} ( g(w;p) \ast dh(w)  -  h(w) \ast dg(w;p) ) =
      \int_{\gamma_r}  ( g(w;p) \ast dh(w) -  h(w) \ast dg(w;p) )  \]
  where $\gamma_r$ is a curve $|w-p|=r$ in some local coordinate, and letting $r \searrow 0$.   
  
  To prove the second displayed equation, choose the limiting curves to be level curves of $g(\cdot\,;z)$; again, this can be done by Lemma \ref{le:anchor_lemma_one}.  Along such curves $dg =0$, so that 
  \[    \partial_w g(w;z) = \frac{i}{2} \ast d_w g    \]
  by equation \eqref{eq:Wirtinger_to_hodge}, which proves the claim. 
 \end{proof}
Note that, this is usually written in terms of an integral around the boundary, under the assumption that $h$ is more regular. 
 Note that the boundary $\partial \riem$ is treated as an analytic curve.
 \end{subsection}
 
  \begin{subsection}{The bounce operator}  \label{se:bounce}
 Let $\riem$ be a bordered surface of type $(g,n)$ and 
  let $U_k \subseteq \riem$ be collar neighbourhoods of $\partial_k \riem$ for $k=1,\ldots,n$. Let $h_k \in \mathcal{D}_{\mathrm{harm}}(U_k)$ for $k=1,\ldots,n$.  Recall that by Theorems \ref{th:CNT_BVs_existence_uniqueness} and \ref{th:bounce_existence}, there is a unique $H \in \mathcal{D}_{\mathrm{harm}}(\riem)$ whose CNT boundary values agree with those of $h_k$ on $\partial_k \riem$ up to a null set for each $k=1,\ldots,n$. This fact allows us to define the following operator, which plays a major role in what follows.
 \begin{definition}\label{defn:bounce op}
 Set $U = U_1 \cup \cdots \cup U_n$ and let $h:U \rightarrow \mathbb{C}$ be the function whose restriction to $U_k$ is $h_k$ for each $k=1,\ldots,n$.  We define 
  \begin{align*}
   \mathbf{G}_{U,\riem}: \mathcal{D}_{\mathrm{harm}}(U) & \rightarrow \mathcal{D}_{\mathrm{harm}}(\riem) \\ h & \mapsto H
  \end{align*}
   We call this operator the {\it{bounce operator}}.  
 \end{definition}
  By conformal invariance of CNT limits, the bounce operator is conformally invariant, that is, if $f:\riem \rightarrow \riem'$ is a biholomorphism and $f(U)=U'$, then 
  
  \begin{equation} \label{eq:bounce_conformally_invariant}
   \gls{bounce} \mathbf{C}_f = \mathbf{C}_f \mathbf{G}_{U',\riem'}.
  \end{equation}

\begin{theorem}[Boundedness of bounce operator] \label{th:bounce_bounded} 
  Let $\riem$, $U_k$ and $h_k$ be as above for $k=1,\ldots,n$.
  Then $\mathbf{G}_{U,\riem}$ is bounded from ${H}^{1}_{\mathrm{conf}}(U)$ to ${H}^{1}_{\mathrm{conf}}(\riem)$.
 \end{theorem}
\begin{remark}
Note that a proof of the special case of Theorem \ref{th:bounce_bounded} can be found in \cite{Schippers_Staubach_Grunsky_expository}, Theorem 4.6.
\end{remark}
 \begin{proof}
  The goal is to show that if $U = U_1 \cup \cdots \cup U_n$ and if $h:U \rightarrow \mathbb{C}$ is a function in $ \mathcal{D}_{\mathrm{harm}}(U)$ whose restriction to $U_k$ is $h_k$ for each $k=1,\ldots,n$, then $$\Vert \mathbf{G}_{U,\riem}\,h \Vert_{{H}^{1}_{\mathrm{conf}}(\riem)}\lesssim  \| h \|_{{H}^{1}_{\mathrm{conf}}(U)},$$ 
  for $h\in \mathcal{D}_{\mathrm{harm}}(U).$ First, observe that we can assume that the inner boundary of $U_k$ is analytic.  To see this, let $U_k' \subseteq U_k$ be a collar neighbourhood whose inner boundary is analytic.  Since $\| \left. h_k \right|_{U_k'} \|_{H ^{1}_{\mathrm{conf}}(U_k')} \leq \| h_k\|_{H ^{1}_{\mathrm{conf}}(U_k)}$, it is enough to show that $\mathbf{G}_{U,\riem}$ is bounded 
   with respect to the $H^{1}_{\mathrm{conf}}(U')$ norm, where $U'=U'_1 \cup \cdots \cup U'_n \subset U$. In what follows, we relabel the new sets by removing the primes. 
   
  Next, observe that because CNT boundary values and the Dirichlet norms are conformally invariant, it is enough to prove this for analytic strip-cutting curves $\partial_k \riem$, and
  this can be arranged for example by embedding $\riem$ in its double.  Thus, we can
  assume that $\partial U_k$ is analytic.\\ 
 Furthermore by the result on the unique Sobolev extension, see e.g. \cite[Proposition 4.5, p 334]{Taylor} and the fact that $\partial_k \riem \subsetneq \partial U_k$, 
yields that
  \begin{equation}\label{estim: interior to boundary}
  \| \left. h \right|_{\partial_k\riem} \|_{H^{1/2}(\partial_k\riem)} \leq  \| \left. h \right|_{\partial_k\riem} \|_{H^{1/2}(\partial U_k)}  \lesssim \| h_k \|_{H^1(U_k)}.    
  \end{equation} 
  
  Also, since $\partial \riem= \cup_{k=1}^{n}\partial_k\riem,$ given the Dirichlet data $h|_{\partial_k\riem}$, $k=1, \dots, n$, on each of the boundary components, Theorem \ref{th:CNT_BVs_existence_uniqueness} yields that the unique harmonic extension $H$ of the boundary values $ h|_{\partial_k\riem}$ satisfies 
  \begin{equation}\label{estim: boundary to interior}
   \| H \|_{{H}^1(\riem)} \lesssim \sum_{k=1}^{n}\| \left. h \right|_{\partial_k\riem} \|_{H^{1/2}(\partial_k\riem)}.   
  \end{equation}   
  
 Now since $H=\mathbf{G}_{U,\riem}\, h$,  using \eqref {estim: interior to boundary} and \eqref{estim: boundary to interior} one has
  \begin{equation}\label{Sobolev estimate for bounce}
  \begin{split}
  \Vert  \mathbf{G}_{U,\riem}\, h \Vert _{H^1(\riem)} \lesssim \sum_{k=1}^{n}  \Vert h_k\Vert _{H^1(U_k)}\lesssim \Vert h\Vert _{H^1(U)}.
  \end{split}
  \end{equation}
  
Now let $\mathscr{F}(h):= \Big( \sum_{k=1}^{n}  | \mathscr{H}_k |^2 \Big)^{1/2}$ then $\mathscr{F}$  is clearly non-negative, and \eqref{defn:integration over rough} yields that 

{$$\mathscr{F}(1)=  \Big(\sum_{k=1}^{n} { \left| \int_{\partial_k \riem} \ast d\omega_k \right|^2}  \Big)^{1/2}= \Big(\sum_{k=1}^{n} {\left| \int_{\gamma_k } \ast d\omega_k \right|^2}  \Big)^{1/2}\neq 0. $$}

Furthermore the definition \eqref{defn:integration over rough}, the Cauchy-Schwarz inequality, \eqref{estim: interior to boundary} and Theorem \ref{thm:H1-L2} yield that
\begin{equation}
\begin{split}
    \mathscr{F}(h)\leq \sum_{k=1}^{n}  |\iint_{U_k} dh_k \wedge \ast d\omega_k |+\sum_{k=1}^{n} | \int_{\gamma_k} h_k \ast d \omega_k| \lesssim\Vert h\Vert_{H^1(U)}
    \end{split}
\end{equation}
  This shows that $\mathscr{F}$ is a bounded linear functional on $H^1(U)$ and thereby the conditions of Theorem \ref{thm:equiv sobolev norm} are all satisfied. Hence using \eqref{Sobolev estimate for bounce} and Theorem \ref{thm:equiv sobolev norm} we obtain 
    \begin{equation}
    \begin{split}
    \Vert  \mathbf{G}_{U,\riem}\, h \Vert _{H^1(\riem)}\lesssim\Vert h\Vert_{H^1(U)}\lesssim \Big(\Vert h\Vert^2_{\dot{H}^1(U)}+  (\mathscr{F}(h))^2\Big)^{1/2}\lesssim\\ \Big(\| h \|^2_{\mathcal{D}_{\mathrm{harm}}(U)} + 
  \sum_{k=1}^{n}   |\mathscr{H}_k |^2\Big)^{1/2}\lesssim \| h \|_{{H}^{1}_{\mathrm{conf}}(U)}.
 \end{split}
  \end{equation}
Now Theorem \ref{thm:equivnorms on riemsurfaces} on equivalence of the norms ends the proof of the boundedness of the bounce operator.
\end{proof}

  Now as an illuminating example, choose $\riem= \disk$ and $U = \mathbb{A}$ where $\mathbb{A}= \mathbb{A}_{r,1}$. Choosing $p=0$ in \eqref{eq:constant_hk_definition_onecurve}, we observe that $\ast d\mathscr{G} = d\theta$ where $\theta$ is angle in polar coordinates $z=r e^{i\theta}$ on $\disk$.  Thus 
  \[  \mathscr{H}_1 = \int_{\mathbb{S}^1} h(e^{i\theta}) \,d\theta,    \]
  that is, it is just the constant term in the Fourier expansion of the trace of $h$ to the boundary.  Using this fact it is elementary to show that
  \begin{proposition} \label{pr:bounce_hol_dense_disk} The subset
   $\mathbf{G}_{\mathbb{A},\disk} \mathcal{D}(\mathbb{A})$ is dense 
   in $H^1_{\mathrm{conf}}(\disk)$.  
  \end{proposition}
 \begin{proof}
 Given $f\in H^1_{\mathrm{conf}}(\disk)$ and $\varepsilon>0$, take a polynomial $p(z)\in \mathcal{D}(\mathbb{A})$ such that $\Vert f-p\Vert_{H^1_{\mathrm{conf}}(\disk)}< \varepsilon. $ 
 Now since $\mathbf{G}_{\mathbb{A},\disk} f=f$, Theorem \ref{th:bounce_bounded} yields that
 \begin{equation}
     \Vert f- \mathbf{G}_{\mathbb{A},\disk}p\Vert_{H^1_{\mathrm{conf}}(\disk)}= \Vert \mathbf{G}_{\mathbb{A},\disk}(f- p)\Vert_{H^1_{\mathrm{conf}}(\disk)}\lesssim \Vert f- p\Vert_{H^1_{\mathrm{conf}}(\disk)}<\varepsilon,
 \end{equation}
 which proves the desired density.
 \end{proof}
  
In order to prove a density result in the case of many boundary curves, we need the following lemma.

 \begin{lemma} \label{le:for_density_factorization}
  Let $\riem$ be a Riemann surface of type $(g,n)$.  For any collar neighbourhood $\phi:U \rightarrow \mathbb{A}$ of a border $\partial_k \riem$, the map 
  \[ \mathbf{C}_\phi: H^1_{\mathrm{conf}}(U) \rightarrow  H^1_{\mathrm{conf}}(\mathbb{A}) \]
  is a bounded isomorphism.
 \end{lemma}
 \begin{proof}
 Note that we can treat $\partial_k \riem$ as an analytic curve in the double and in fact there is a biholomorphism of a doubly-connected neighbourhood of $\partial_k \riem$ in the double to a doubly-connected neighbourhood of $\mathbb{S}^1$.  So after localizations and partition of unity  on the boundary structure of $\riem$, and using Theorem \ref{thm:equiv sobolev norm}, matters  reduce to Lemma \ref{lem:invariance of sobolev under diffeo}.  
 \end{proof}
 
 In particular, this shows the following.
 \begin{proposition}
  Let $\riem$ be of type $(g,1)$.  The $H^1_{\mathrm{conf}}$-norms induced by different choices of $p$ in \eqref{eq:constant_hk_definition_onecurve}
  are equivalent.
 \end{proposition}
 \begin{proof}
  By Lemma \ref{le:for_density_factorization} each choice is equivalent to the $H^1_{\mathrm{conf}}(\mathbb{A})$-norm with singularity at $0$, under a fixed collar chart.
 \end{proof}
 
 We will also need the following when we study overfare in the next section. 
 \begin{proposition}  \label{pr:restriction_Hconf_bounded}
  Let $\riem$ be a type $(g,n)$ Riemann surface.  Let $\phi_k:U_k \rightarrow \mathbb{A}_k$ be a collection of collar charts of the boundaries $\partial_k \riem$ for $k=1,\ldots,n$ and let $U = U_1 \cup \cdots \cup U_n$.  Then the restriction map
  \[ \gls{rest}_{\riem,U} : H^1_{\mathrm{conf}}(\riem) \rightarrow H^1_{\mathrm{conf}}(U)    \]
  is bounded.
 \end{proposition}
\begin{proof}
  This follows from the definitions of the norms \eqref{defn:Dconf norm}, \eqref{equivalent sobolev norm} and Theorem \ref{thm:equivnorms on riemsurfaces}.  
 \end{proof}

 We may now prove the following:  
 \begin{theorem} \label{th:density_of_holo_collar}
  Let $\riem$ be a type $(g,n)$ Riemann surface.  Let $U = U_1 \cup \cdots \cup U_n$ be a union of collar neighbourhoods $U_k$ of $\partial_k \riem$.  Then $\mathbf{G}_{U,\riem} \mathcal{D}(U)$ is dense in ${H}^1_{\mathrm{conf}}(\riem)$. 
 \end{theorem}
 \begin{proof}
   A factorization trick makes the proof somewhat simple.  
   Let $F_k:U_k \rightarrow \mathbb{A}_k$ be collar charts, and denote 
   $\mathbb{A}^n = \mathbb{A}_1 \times \cdots \times \mathbb{A}_n$ 
   and $\disk^n = \mathbb{D} \times \cdots \times \mathbb{D}$,
   and define $F:U \rightarrow \mathbb{A}$ by $F(z)=(F_1(z),\ldots,F_n(z))$.  Define the restriction maps
   \begin{align*}
    \mathbf{R}_{\riem,U}: \mathcal{D}_{\mathrm{harm}}(\riem) & \rightarrow \mathcal{D}_{\mathrm{harm}}(U) \\
    h & \mapsto \left. h \right|_U 
   \end{align*}
   and similarly 
   \begin{align*}
    \mathbf{R}_{\disk^n,\mathbb{A}^n}: \bigoplus^n \mathcal{D}_{\mathrm{harm}}(\disk) & \rightarrow \bigoplus_{k=1}^n  \mathcal{D}_{\mathrm{harm}}(\mathbb{A}_k) \\
    (h_1,\ldots,h_n)  & \mapsto \left(\left. h_1\right|_{\mathbb{A}_1},\ldots, \left. h_n\right|_{\mathbb{A}_n} \right).
   \end{align*}
   Now
   \[ \mathbf{C}_F \mathbf{R}_{\mathbb{D}^n,\mathbb{A}^n}: H^1_{\mathrm{conf}}(\mathbb{D}^n) \rightarrow H^1_{\mathrm{conf}}(U)    \]
   is bounded by Lemma \ref{le:for_density_factorization}, where we put the direct sum norm on $H^1_{\mathrm{conf}}(\disk^n)$.  Similarly 
   \[  \mathbf{C}_{F^{-1}} \mathbf{R}_{\riem,U} : H^1_{\mathrm{conf}}(\riem) \rightarrow {H}^1_{\mathrm{conf}}(\mathbb{A}^n)    \]
   is bounded.  
   Thus 
   \[ \rho = \mathbf{G}_{U,\riem} \, \mathbf{C}_{F} \, \mathbf{R}_{\disk^n,\mathbb{A}^n}: H^1_{\mathrm{conf}}(\disk^n)
   \rightarrow H^1_{\mathrm{conf}}(\disk^n) \]
   is bounded by Theorem \ref{th:bounce_bounded}, as is 
   \[ \rho^{-1} = \mathbf{G}_{\mathbb{A}^n,\disk^n} \mathbf{C}_{F^{-1}} \mathbf{R}_{\riem,U}. \]
   The fact that this is the inverse of $\rho$ follows from conformal invariance of CNT boundary values.
   
   Again by conformal invariance of CNT boundary values and the definition of the bounce operator, we have the factorization
   \[   \mathbf{G}_{U,\riem} = \rho \, \mathbf{G}_{\mathbb{A}^n,\mathbb{D}^n} \, \mathbf{C}_{F^{-1}}. \]
   Since $\mathbf{C}_{F^{-1}}$ is a bounded invertible map by Lemma \ref{le:for_density_factorization}, and we have shown that $\rho$ is a bounded invertible map, then density follows from Proposition \ref{pr:bounce_hol_dense_disk}.  
 \end{proof}

 \end{subsection}
\begin{subsection}{ {Overfare} of harmonic functions} \label{se:function_Overfare}
 Let $\Gamma$ be a collection of curves separating a Riemann surface $\mathscr{R}$  into two components $\riem_1$ and $\riem_2$,  consider the following problem.  Given $h_1 \in \mathcal{D}_{\mathrm{harm}}(\riem_1)$, is there an $h_2 \in \mathcal{D}_{\mathrm{harm}}(\riem_2)$ with the same boundary values up to a negligible set? We call this the {overfare} of $h_1$ to $h_2$.\\
 
 We saw that for the Dirichlet space, the negligible sets are null sets.  However, a null set with respect to $\riem_1$ need not be null with respect to $\riem_2$.  Thus we must restrict to curves for which this is true: quasicircles. Furthermore, for quasicircles, the {overfare} exists and is a bounded map with respect to the Dirichlet semi-norms, when the originating surface is connected. Note also that the overfare map is bounded with respect to $H^1_{\mathrm{conf}}$ in the general case, if we assume that the quasicircle is more regular.  As we shall see ahead, the so-called Weil-Petersson class quasicircles are sufficient for this purpose, which will be outlined below.
 
 \begin{definition} 
  We say that a simple closed curve in the Riemann sphere $\sphere$ is a \emph{quasicircle} if it is the image of $\mathbb{S}^1$ under a quasiconformal map of the plane.  
 
  A simple closed curve $\Gamma$ in a Riemann surface $R$ is a quasicircle if there is an open set $U$ containing $\Gamma$ and a biholomorphism $\phi:U \rightarrow \mathbb{A}$ where $\mathbb{A}$ is an annulus in $\mathbb{C}$, such that $\phi(\Gamma)$ is a quasicircle.  
 \end{definition}  
 By definition, a quasicircle is a strip-cutting Jordan curve.\\
 
 {{There is a class of quasicircles, called \emph{Weil-Petersson quasicircles}, that arise naturally and frequently in geometric function theory, Teichm\"uller theory, the theory of Schramm-Loewner evolution, and conformal field theory.}}
 \begin{definition}
  We say that a quasicircle in the Riemann sphere $\sphere$ is a Weil-Petersson class quasicircle (or WP quasicircle) if there is a conformal map $f:\disk \rightarrow \Omega$ where $\Omega$ is one of the connected components of the complement, such that the Schwarzian derivative $S(f)= f'''/f' - 3/2 (f''/f')^2$ satisfies
  \[  \iint_{\disk}  (1-|z|^2)^2 |S(f)|^2 \frac{d\bar{z} \wedge d z}{2i}< \infty.  \]
  
  We say that a quasicircle $\Gamma$ in a Riemann surface $\mathscr{R}$ is a WP class quasicircle if there is an open set $U$ containing $\Gamma$ and a biholomorphism $\phi:U \rightarrow \mathbb{A}$ where $\mathbb{A}$ is an annulus, such that $\phi(\Gamma)$ is a WP quasicircle. 
 \end{definition}
 {{One characterization of WP quasicircles is that $\Gamma$ is a WP quasicircle if and only if the $f$ in the definition above has a quasiconformal extension whose Beltrami differential is $L_{\mathrm{hyp}}^2(\disk^-)$ where $\disk^- = \{ z: |z|>1 \} \cup \{ \infty \},$ and $L_{\mathrm{hyp}}^2(\disk^-)$ is the set of $L^2$ functions with respect to the hyperbolic metric of the disk. As with the case of general quasicircles, there are in fact a large number of characterizations of WP quasicircles. C. Bishop  \cite{Bishop} has listed over twenty, many of which are new. His paper also contains other far-reaching generalizations of the concept of WP quasicircles to higher dimensions.}}\\
 
 Having the definition of quasicircles at hand, we consider the following situation. 
 \begin{definition}  \label{de:separating_complex}
 Let $\mathscr{R}$ be a compact Riemann surface, and let $\Gamma_1,\ldots,\Gamma_m$ be a collection of  quasicircles in $\mathscr{R}$.  Denote $\Gamma = \Gamma_1 \cup \cdots \cup \Gamma_m$.  We say that $\Gamma$ \emph{separates} $\mathscr{R}$ into $\riem_1$ and $\riem_2$ if 
 \begin{enumerate}
     \item there are doubly-connected neighbourhoods $U_k$ of $\Gamma_k$ for $k=1,\ldots,n$   such that $U_k \cap U_j$ is empty for all $j \neq k$, 
     \item one of the two connected components of $U_k \backslash \Gamma_k$ is in $\riem_1$, while the other is in $\riem_2$; 
     \item $\mathscr{R} \backslash \Gamma = \riem_1 \cup \riem_2$;
     \item $\mathscr{R} \backslash \Gamma$ consists of finitely many connected components;
     \item $\riem_1$ and $\riem_2$ are disjoint.
 \end{enumerate}
 \end{definition}
 
 Briefly, $\riem_1$ and $\riem_2$ are the two ``sides'' of $\Gamma$, and
 each side is a finite union of Riemann surfaces. 
 \begin{proposition}  Let $\mathscr{R}$ be a compact Riemann surface and $\Gamma = \Gamma_1 \cup \cdots \cup \Gamma_m$ be a collection of quasicircles separating $\mathscr{R}$ into $\riem_1$ and $\riem_2$.  Then $\riem_1$ and $\riem_2$ are each a finite union of bordered surfaces.  For $k=1,2$, the inclusion map of $\riem_k$ into $\mathscr{R}$ extends continuously to the border $\partial_k \riem$, and this extension is a homeomorphism onto $\Gamma$.  
 \end{proposition}
 \begin{proof}
  This follows immediately from Theorem \ref{th:embedded_is_bordered}.
 \end{proof}
 
 Thus, we will identify $\partial \riem_1$ and $\partial \riem_2$ pointwise with $\Gamma$. 
 It is important to note that the border structure is entirely independent of the inclusion map, and furthermore the border structures induced by $\riem_1$ and $\riem_2$ do not agree in general (unless the curves are analytic). In particular, a border chart in $\riem_1$ does not in general extend to a chart in $\mathscr{R}$ which is also a border chart of $\riem_2$, unless the curves $\Gamma_k$ are analytic.
 
 It is not obvious that a null set in $\partial \riem_1$ is null in $\partial \riem_2$, even though they are the same set.  This holds for quasicircles.
 
{ \begin{theorem} \label{th:null_same_both_sides} Let $\mathscr{R}$ be a Riemann surface $($not necessarily compact$)$ and $\Gamma = \Gamma_1 \cup \cdots \cup \Gamma_m$ be a collection of quasicircles separating $\mathscr{R}$ into $\riem_1$ and $\riem_2$.  Then $I \subseteq \Gamma$ is null in $\partial \riem_1$ if and only if $I$ is null in $\partial \riem_2$.   
 \end{theorem}
 \begin{proof}
  It is enough that this is true for a single boundary curve $\Gamma_k$.   Let $g:U \rightarrow V$ be a doubly-connected chart in a neighbourhood of $\Gamma$. 
 By shrinking $U$ if necessary, we can assume that $U$ is bounded by analytic curves $\gamma_1$ 
 and $\gamma_2$ in $\riem_1$ and $\riem_2$ respectively, and 
 that $g$ has a conformal extension to an open set containing the closure of $U$ so that $g(\gamma_1)$
 and $g(\gamma_2)$ are analytic curves in $\mathbb{C}$.  
 Let $\phi:A \rightarrow \mathbb{A}$ be a collar chart in a neighbourhood of $\Gamma$ in $\riem_1$
 and $\psi:B \rightarrow \mathbb{B}$ be a collar chart in a neighbourhood of $\Gamma$ in $\riem_2$.
 For definiteness, we arrange that the outer boundary of both annuli $\mathbb{A}$ and $\mathbb{B}$
 is $\mathbb{S}^1$, and that $\phi$ and $\psi$ both take $\Gamma$ to $\mathbb{S}^1$.  This can 
 be done by composing with an affine transformation and $z \mapsto 1/z$ if necessary.   Let $\Omega^+$
 denote the bounded component of the complement of $g(\Gamma)$ in $\sphere$ and $\Omega^-$ denote
 the unbounded component.  We assume that $g$ takes $U \cap \riem_1$ into $\Omega^+$, again by 
 composing with $z \mapsto 1/z$ if necessary.  Finally, by possibly shrinking the domain of $g$ again,
 we can assume that the analytic curve $\gamma_1$ is contained in the domain of $\phi$.  
 
 Thus, $\phi \circ g^{-1}$ is a conformal map of a collar neighbourhood $W$ of $g(\Gamma)$ in $\Omega^+$
 onto a collar neighbourhood of 
 $\mathbb{S}^1$ in $\disk$, whose inner boundary $\phi(\gamma_1)$ is an analytic curve.
 By the previous paragraph it has a conformal extension to an open neighbourhood of $g(\gamma_1)$, and thus 
 the restriction of $\phi \circ g^{-1}$ is an analytic diffeomorphism from $g(\gamma_1)$ to $\phi(\gamma_1)$.
 Thus if we let $W'$ be the simply connected set in $\Omega^+$ bounded by $g(\gamma_1)$, then there is a quasiconformal 
 map $F$ of $W'$ with a homeomorphic extension to $g(\gamma_1)$ equalling $\psi \circ g^{-1}$.   The 
 map 
 \begin{equation}
 \Phi(z) = \left\{ \begin{array}{ll} F(z) & z \in W' \\ \phi \circ g^{-1}(z) & z \in W \cup 
     g(\gamma_1)   \end{array} \right.  
\end{equation}
 is therefore a quasiconformal map from $\Omega^+$ to $\disk$.   A similar argument 
 shows that
 $\psi \circ g^{-1}$ has a quasiconformal extension to a  map from $\Omega^-$ to $\disk$.  
 
 Since $g(\Gamma)$ is a quasicircle, there is a quasiconformal reflection $r$ of the plane which fixes each point in $g(\Gamma)$.  Thus $\psi \circ g^{-1} \circ r \circ (\phi \circ g^{-1})^{-1}$ has an extension to an (orientation reversing) quasiconformal self-map of the disk.  Thus it extends continuously to a quasisymmetry of $\mathbb{S}^1$, which takes Borel sets of capacity zero to Borel sets of capacity zero. Furthermore, on $\mathbb{S}^1$, this map equals $\psi \circ \phi^{-1}$.  Since the same argument applies to $\phi \circ \psi^{-1}$, 
 we have shown that $\phi(I)$ has capacity zero in $\mathbb{S}^1$ if and only if $\psi(I)$ has capacity zero
 in $\mathbb{S}^1$. This completes the proof.  
 \end{proof}}
 
\begin{definition}\label{Dirchlet norm on multiconnected} 
 In the case that $\riem$ is a finite union of connected Riemann surfaces $\riem_1,\ldots,\riem_s$, we define the Dirichlet semi-norm on these components by 
 \[  \| h \|_{\mathcal{D}_{\mathrm{harm}}(\riem)} := 
     \sum_{k=1}^s \left\| \left. h \right|_{\riem_k} \right\|_{\mathcal{D}_{\mathrm{harm}}(\riem_k)} \]
 and similarly for the holomorphic and anti-holomorphic Dirichlet spaces, Bergman spaces, etc.
\end{definition} 

Before defining the overfare process on Riemann surfaces, we will define it in the plane. 

  \begin{theorem} \label{th:overfare_sphere_control_Dirichlet}
  Let $\Gamma$ be a quasicircle in $\sphere$, and let $\Omega_1$ and $\Omega_2$ be the connected components of the complement of $\Gamma$. 
  For all $h_1 \in \mathcal{D}_{\mathrm{harm}}(\Omega_1)$ there is an $h_2 \in \mathcal{D}_{\mathrm{harm}}(\Omega_2)$ whose \emph{CNT} boundary values agree with those of $h_1$ up to a null set, and one has the estimate
  \[    \| h_2 \|_{\mathcal{D}_{\mathrm{harm}}(\Omega_2)} \lesssim\| h_1 \|_{\mathcal{D}_{\mathrm{harm}}(\Omega_1)}. \] 
 \end{theorem}
 \begin{proof} { See \cite{Schippers_Staubach_transmission}  Theorem 3.25.}  
 \end{proof} 
 In particular, we have well-defined operators
 \[ \mathbf{O}_{\Omega_1,\Omega_2}:\mathcal{D}_{\mathrm{harm}}(\riem_1) \rightarrow \mathcal{D}_{\mathrm{harm}}(\riem_2)  \]
 and
 \[ \mathbf{O}_{\Omega_2,\Omega_1}:\mathcal{D}_{\mathrm{harm}}(\riem_2) \rightarrow \mathcal{D}_{\mathrm{harm}}(\riem_1).  \]
 
 If the quasicircle is more regular, we can also control the $H^1_{\mathrm{conf}}$ norm. 
 \begin{definition}[BZM quasicircle]  Let $\Gamma$ be a quasicircle in $\sphere$, and let $\Omega_1$ and $\Omega_2$ denote the connected components of the complement.  We say that $\Gamma$ is a \emph{bounded zero mode  quasicircle} (BZM for short), if 
 $\mathbf{O}_{\Omega_1,\Omega_2}$ and $\mathbf{O}_{\Omega_2,\Omega_1}$ are bounded with respect to $H^1_{\mathrm{conf}}(\Omega_k)$.  
 
  A quasicircle $\Gamma$ in a compact Riemann surface $\mathscr{R}$ is called an BZM quasicircle if there is an open set $U$ containing $\Gamma$ and a conformal map $\phi:U \rightarrow \mathbb{A}$ onto an annulus $\mathbb{A} \subseteq \mathbb{C}$ such that $\phi(\Gamma)$ is a BZM quasicircle.  
 \end{definition}
  
 In this connection we have the following theorem which is built upon deep results regarding flows of Sobolev-vector fields on the unit circle, and also a basic result regarding the action of the group of quasisymmetries of the unit circle, by bounded automorphisms on the homogeneous Sobolev space $\dot{H}^{1/2}(\mathbb{S}^1)$ (the action is essentially a composition).
 \begin{theorem} \label{th:WP_are_BZM}
  $\mathrm{WP}$ quasicircles are $\mathrm{BZM}$ quasicircles. 
 \end{theorem}
 \begin{proof}
 It is enough to show that for a WP-class quasisymmetric homeomorphism of the circle $\phi$, the composition operator $\mathbf{C}_\phi$ is bounded on the Sobolev space $H^{1/2}(\mathbb{S}^1),$ which amounts to show that for $f\in H^{1/2}(\mathbb{S}^1)$
 \begin{equation}\label{WP Sobolev-boundedness}
     \Vert \mathbf{C}_\phi f \Vert_{{H}^{1/2}(\mathbb{S}^1)}\leq c   \Vert f \Vert_{{H}^{1/2}(\mathbb{S}^1)},
 \end{equation}
where the constant $c$ in the estimate only depends on $\phi$.

Using change of variables and Cauchy-Schwarz's inequality one has 

\begin{equation}\label{firststep}
   \Vert \mathbf{C}_\phi f \Vert^2_{{L}^{2}(\mathbb{S}^1)}=  \int _{\mathbb{S}^1} |f \circ\phi|^2 = \int_{\mathbb{S}^1} |f|^2 (\phi^{-1}) '\leq \Big( \int_{\mathbb{S}^1} |f|^4\Big) ^{1/2}
\Big( \int_{\mathbb{S}^1} |(\phi^{-1}) ' |^2 \Big)^{1/2}.
\end{equation}

Here we note that by the results of A. Figalli \cite{Figalli} and F. Gay-Balmaz and T. Ratiu \cite{Gay-Balmaz-Ratiu}, one has that for WP-class quasisymmetries $\phi$ on $\mathbb{S}^1$ both $\phi$ and its inverse $\phi^{-1}$ are in $H^{3/2-\varepsilon}(\mathbb{S}^1)$ for all $\varepsilon >0$. Therefore by taking $\varepsilon=1/2$ we have that  $\phi^{-1} \in H^{1} (\mathbb{S}^1)$ which means that the first derivative of $\phi^{-1}$ belong to $L^2(\mathbb{S}^1),$ hence

\begin{equation}\label{secondstep}
   \Big( \int_{\mathbb{S}^1} |(\phi^{-1}) ' |^2 \Big)^{1/2}\leq \Vert \phi^{-1}\Vert_{H^1(\mathbb{S}^1)}< \infty.
\end{equation} 

Now if $f\in H^{1/2}(\mathbb{S}^1)$, then the Sobolev embedding \eqref{Sobolev embed} with $p=4$ and $s=\frac{1}{2}$ yields that

 \begin{equation}\label{thirdstep}
     \Vert f\Vert_{L^{4}(\mathbb{S}^1)}\lesssim \Vert f\Vert_{H^{1/2}(\mathbb{S}^1)}.
 \end{equation}
 
 \item Thus taking the square root of both sides of \eqref{firststep}, and using \eqref{secondstep} and \eqref{thirdstep}, we obtain for $f\in H^{1/2}(\mathbb{S}^1)$ and 
\begin{equation}\label{fourthstep}
    \Vert \mathbf{C}_\phi f\Vert_{L^2(\mathbb{S}^1)} \lesssim \Vert f\Vert_{H^{1/2}(\mathbb{S}^1)}
\Vert \phi^{-1}\Vert^{1/2}_{H^1(\mathbb{S}^1)}\lesssim \Vert f\Vert_{H^{1/2}(\mathbb{S}^1)} .
\end{equation}

Moreover, by a result of Vodopyanov-Nag-Sullivan \cite{Vodopyanov} and \cite{NS}, we also know that
\begin{equation}\label{VNS}
    \Vert \mathbf{C}_\phi f\Vert_{\dot{H}^{1/2}(\mathbb{S}^1)} \lesssim \Vert f\Vert_{\dot{H}^{1/2}(\mathbb{S}^1)}.
\end{equation}
Consequently \eqref{fourthstep} and \eqref{VNS} yield that
    \begin{equation}
    \Vert\mathbf{C}_\phi f\Vert_{{H}^{1/2}(\mathbb{S}^1)} \approx  \Vert \mathbf{C}_\phi f\Vert_{\dot{H}^{1/2}(\mathbb{S}^1)}+  \Vert \mathbf{C}_\phi f\Vert_{{L}^{2}(\mathbb{S}^1)} \lesssim \Vert f\Vert_{\dot{H}^{1/2}(\mathbb{S}^1)}+ \Vert f\Vert_{H^{1/2}(\mathbb{S}^1)} \lesssim \Vert f\Vert_{H^{1/2}(\mathbb{S}^1)}.
\end{equation}
 \end{proof}

\begin{remark}
 Regarding the hidden exponentials in the calculations above, let us assume that the angle in the image of the quasisymmtric homeomorphism $\chi:\mathbb{S}^1 \to \mathbb{S}^1$ is $\psi(\theta)$, so that 
\[  e^{i\psi} = \chi(e^{i\theta}),  \]
i.e.
\[ \psi = - i \log{\chi(e^{i\theta})}. \]
Then if $\chi'(z)$ denotes the derivative of $\chi$ with respect to  $z$, and $\dot{\psi}$ denotes the derivative with respect to $\theta$, we would have 
\[  \dot{\psi} =  \chi'(e^{i\theta}) \frac{e^{i \theta}}{\chi(e^{i \theta})}.    \]
From this it immediately follows that
\[  |\dot{\psi}| = |\chi'(e^{i\theta})|.   \]
In particular if one makes an estimate using one or the other, it doesn't affect the outcome of the estimate.
\end{remark}

  The next three theorems concern existence and boundedness of the overfare operator for general curve complexes. Their proofs are somewhat involved and will be approached together in stages. 
  \begin{theorem}[Existence of overfare]\label{thm:bounded_overfare_existence}  Let $\mathscr{R}$ be a compact Riemann surface and let $\Gamma = \Gamma_1 \cup \cdots \cup \Gamma_m$ be a collection of  quasicircles separating $\mathscr{R}$ into $\riem_1$ and $\riem_2$. 
  Let $h_1 \in \mathcal{D}_{\mathrm{harm}}(\riem_1)$.  There is a $h_2 \in \mathcal{D}_{\mathrm{harm}}(\riem_2)$ whose \emph{CNT} boundary values agree with those of $h_1$ up to a null set, and this $h_2$ is unique.   
  \end{theorem}
  
    This theorem, which we will prove shortly, motivates the definition of the following operator which plays an important role in the scattering theory that is developed here.
  \begin{definition}\label{defn:overfare operator} 
With the assumption of Theorem \ref{thm:bounded_overfare_existence}, we define the \emph{overfare operator} $\gls{overfare}_{\riem_1,\riem_2}$ by   
  \begin{align*}
   \mathbf{O}_{\riem_1,\riem_2} : \mathcal{D}_{\mathrm{harm}}(\riem_1) & \rightarrow 
   \mathcal{D}_{\mathrm{harm}}(\riem_2) \\
   h_1 & \mapsto h_2    
 \end{align*}
  \end{definition}
 One obviously has that
 \[  \mathbf{O}_{\riem_2,\riem_1} \mathbf{O}_{\riem_1,\riem_2} = \text{Id} \]
 and of course one can switch the roles of $\riem_1$ and $\riem_2$.  
 
 The overfare operator is conformally invariant.  That is, if $f:\mathscr{R} \rightarrow \mathscr{R}'$ is a biholomorphism and we set $f(\riem_k) = \riem_k'$ for $k=1,2$ then it follows immediately from conformal invariance of CNT limits that
 \begin{equation} \label{eq:overfare_conformally_invariant}
  \mathbf{O}_{\riem_1,\riem_2} \mathbf{C}_f = \mathbf{C}_f \mathbf{O}_{\riem_1',\riem_2'}.   
 \end{equation}
{\bf{Notation.}}  If $\riem_1$ and $\riem_2$ are clear from context, we will denote the overfare operator by $\mathbf{O}_{1,2}$. 
  
  We will also obtain boundedness of this operator with respect to $H^1_{\mathrm{conf}}$ and the Dirichlet semi-norm.  In both cases, certain further conditions on the curve complex $\Gamma$ are required. 

 \begin{theorem}[Bounded overfare theorem for BZM quasicircles]\label{thm:bounded_overfare_conf}  Let $\mathscr{R}$ be a compact Riemann surface and let $\Gamma = \Gamma_1 \cup \cdots \cup \Gamma_m$ be a collection of \emph{BZM} quasicircles separating $\mathscr{R}$ into $\riem_1$ and $\riem_2$. 
  There is a constant $C$ such that 
     \[  \| \mathbf{O}_{1,2} h \|_{H^1_{\mathrm{conf}}(\riem_2)}  \leq C \| h \|_{H^1_{\mathrm{conf}}(\riem_1)}  \]
     for all $h \in \mathcal{D}_{\mathrm{harm}}(\riem_1)$.  
  \end{theorem}
  One can also obtain Dirichlet boundedness for general quasicircles, but one must assume that the originating surface is connected. 
 \begin{theorem}[Bounded overfare theorem for general quasicircles]\label{thm:bounded_overfare_Dirichlet} Let $\mathscr{R}$ be a compact Riemann surface and let $\Gamma = \Gamma_1 \cup \cdots \cup \Gamma_m$ be a collection of \emph{BZM} quasicircles separating $\mathscr{R}$ into $\riem_1$ and $\riem_2$.  Assume that $\riem_1$ is connected. 
  There is a constant $C$ such that 
     \[  \| \mathbf{O}_{1,2} h \|_{\mathcal{D}_{\mathrm{harm}}(\riem_2)}  \leq C \| h \|_{\mathcal{D}_{\mathrm{harm}}(\riem_1)}  \]
     for all $h \in \mathcal{D}_{\mathrm{harm}}(\riem_1)$.  
  \end{theorem}  
  Needless to say, the roles of $1$ and $2$ can be interchanged.
  
  The remainder of the section is dedicated to proving these three theorems.

 \begin{lemma} \label{le:local_Overfare}  Let $\mathscr{R}$ be a Riemann surface and let $\Gamma$ be a quasicircle in $\mathscr{R}$.  Let $\phi:U \rightarrow \mathbb{A}$ be a doubly-connected chart, and let $U_1,U_2$ be the connected components of $U \backslash \Gamma$. 
 There is an operator
 \[  \mathbf{O}(\phi)_{1,2}:\mathcal{D}_{\mathrm{harm}}(U_1) \rightarrow \mathcal{D}_{\mathrm{harm}}(U_2)   \]
 such that the \emph{CNT} boundary values of $\mathbf{O}(\phi)_{1,2}h$ agree with those of $h$ up to a null set, and a $C$ such that 
 \[  \|\mathbf{O}(\phi)_{1,2}h\|_{\mathcal{D}_{\mathrm{harm}}(U_2)} \leq C    \| h \|_{\mathcal{D}_{\mathrm{harm}}(U_1)}.  \]
 
 If $\Gamma$ is a \emph{BZM} quasicircle, then there is a $C'$ such that for all $h \in \mathcal{D}_{\mathrm{harm}}(U_1)$ \[  \|\mathbf{O}(\phi)_{1,2}h\|_{H^1_{\mathrm{conf}}(U_2)} \leq C'    \| h \|_{H^1_{\mathrm{conf}}(U_1)}.  \]
 \end{lemma}
 \begin{proof}
  Let $\Omega_1$ and $\Omega_2$ be the connected components of $\sphere \backslash \phi(\Gamma)$ containing $\phi(U_1)$ and $\phi(U_2)$ respectively.  We then have a bounded overfare $\mathbf{O}_{\Omega_1,\Omega_2}:\mathcal{D}_{\mathrm{harm}}(\Omega_1) \rightarrow \mathcal{D}_{\mathrm{harm}}(\Omega_2)$ by Theorem \ref{th:overfare_sphere_control_Dirichlet}.  Furthermore, the bounce operator $\mathbf{G}_{\phi(U_k),\Omega_k}$ is bounded with respect to $\mathcal{D}_{\mathrm{harm}}$ by \cite[Theorem 4.6]{Schippers_Staubach_Grunsky_expository}.  Defining 
  \begin{equation} \label{eq:local_Overfare}
    h_2 = \mathbf{C}_{\phi} \mathbf{R}_{\Omega_2,U_2}   \mathbf{O}_{\Omega_1,\Omega_2}
   \mathbf{G}_{\phi(U_1),\Omega_1} \mathbf{C}_{\phi^{-1}} h_1, 
  \end{equation}
  by conformal invariance of the Dirichlet semi-norm, we have proven the first claim.  The second claim follows by definition of BZM quasicircles, using Theorem \ref{th:bounce_bounded}, and Proposition \ref{pr:restriction_Hconf_bounded}. 
 \end{proof}
 We call \eqref{eq:local_Overfare} the local overfare of induced by $\phi$.  It is non-canonical in the sense that it depends on $\phi$. Since the values on the other boundaries of $U$ are not specified, the local overfare is not unique. 
 
 On the other hand, the overfare to $\riem_2$ is unique. By combining local overfare with the bounce operator, we can show that the overfare exists.
 \begin{proof}(of Theorem \ref{thm:bounded_overfare_existence}).
  Let $\phi_k:U^k \rightarrow \mathbb{A}^k$ be the doubly-connected charts corresponding to the curves $\Gamma_1,\ldots,\Gamma_m$.  Denote $U^k_j = U^k \cap \riem_j$. Given $h \in \mathcal{D}_{\mathrm{harm}}(\riem_1)$, Lemma \ref{le:local_Overfare} produces a collection of functions $H_2^k \in \mathcal{D}_{\mathrm{harm}}(U_2^k)$ whose boundary values agree with $h$.  
  
  For each connected component $\riem^j_2$ of $\riem_2$, let $\hat{U}_j$ denote the union of those $U_2^k$ which lie in this component. We now apply the bounce operator $\mathbf{G}_{\hat{U}_j,\riem_2^j}:\mathcal{D}_{\mathrm{harm}}(\hat{U}_j) \rightarrow \mathcal{D}_{\mathrm{harm}}(\riem_2^j)$ on each component separately to obtain a harmonic function in $\mathcal{D}_{\mathrm{harm}}(\riem_2)$ whose CNT boundary values agree with $h$. 
 \end{proof}
 
 We now prove the boundedness for BZM quasicircles.
 \begin{proof}(of Theorem \ref{thm:bounded_overfare_conf}).
  The idea is the same as in the previous proof, except that we must keep track of the bounds.
  Let $\phi_k:U^k \rightarrow \mathbb{A}^k$ be doubly-connected charts corresponding to the curves $\Gamma_1,\ldots,\Gamma_m$, and let $U^k_l$ be the components of $U_k \backslash \Gamma$ in $\riem_l$ for $l=1,2$.  
  Let $C = \sup \{C_1,\ldots,C_m\}$ where $C_1,\ldots,C_m$ are the constants in the second estimate of Lemma \ref{le:local_Overfare} for the local overfares from $\mathcal{D}_{\mathrm{harm}}(U^k_1)$ to $\mathcal{D}_{\mathrm{harm}}(U^k_2)$ determined by $\phi_k$ for $k=1,\ldots,m$.  For any $h_1 \in \mathcal{D}_{\mathrm{harm}}(\riem_1)$ 
  we have therefore a collection of functions $H_2^k \in \mathcal{D}_{\mathrm{harm}}(U^k_2)$ such that 
  \begin{equation} \label{eq:bt_proof_temp1}
    \left\| H_2^k \right\|_{H^1_{\mathrm{conf}}(U^k_2)} 
     \leq C \left\| \left. h_1 \right|_{U^k_1} \right\|_{H^1_{\mathrm{conf}}(U^k_1)} \leq 
     C \| h_1 \|_{H^1_{\mathrm{conf}}(\riem_1)} 
  \end{equation}     
  where we have also used Proposition \ref{pr:restriction_Hconf_bounded}.  
  
  Now let $\riem_2^1,\ldots,\riem_2^s$ be the connected components of $\riem_2$. For each fixed $j \in 1,\ldots,s$, let $\hat{U}_j$ be the union of those $U^k_2$ which are in $\riem_2^j$, and let $h^j_2$ be the function whose restriction to $\hat{U}_j$ agrees with the corresponding functions $H_2^k$.  
  By Theorem \ref{th:bounce_bounded} there is a constant $C'_j$ such that 
  \begin{equation} \label{eq:bt_proof_temp2}
  \| \mathbf{G}_{\hat{U}_j,\riem_2} h^j_2  \|_{H^1_{\mathrm{conf}}(\riem_2)} \leq C'_j \| h^j_2 \|_{H^1_{\mathrm{conf}}(\hat{U}_j)}.         
  \end{equation}
  Combining \eqref{eq:bt_proof_temp1} and \eqref{eq:bt_proof_temp2} we obtain
  \begin{equation*}
        \| \mathbf{G}_{\hat{U}_j,\riem_2^j} h^j_2  \|_{H^1_{\mathrm{conf}}(\riem_2^j)} \leq m C C'_j 
         \| h_1 \|_{H^1_{\mathrm{conf}}(\riem_1)}
  \end{equation*}
  (where the $m$ appears because there are at most $m$ curves bounding the component $\riem_2^j$).  
  
  Set $C' = \sup \{ mCC'_1,\ldots,mCC'_s\}$.
  If we now let $h_2$ be the function on $\riem_2$ whose restriction to $\riem_2^j$ is $\mathbf{G}_{\hat{U}_j,\riem_2^j} h^j_2$ for $j=1,\ldots,s$, we have that the CNT boundary values of $h_2$ agree with those of $h_1$ and 
  \[  \| h_2 \|_{H^1_{\mathrm{conf}}(\riem_2)} 
   = \sum_{j=1}^s  \| \mathbf{G}_{\hat{U}_j,\riem_2^j} h^j_2  \|_{H^1_{\mathrm{conf}}(\riem_2^j)} \leq s C' \| h_1 \|_{H^1_{\mathrm{conf}}(\riem_1)}.      \]
 \end{proof}
 
 To prove boundedness with respect to the Dirichlet semi-norm, we require three lemmas. 
 \begin{lemma} \label{le:constants_trick} Let $\mathscr{R}$ be a compact Riemann surface and $\Gamma$ be a collection of quasicircles separating $\mathscr{R}$ into components $\riem_1$ and $\riem_2$. Assume that $\riem_1$ is connected.  If $\Gamma$ has the property that 
\[  \| \mathbf{O}_{1,2} h \|_{H^1_{\mathrm{conf}}(\riem_2)} \leq K  \| h \|_{H^1_{\mathrm{conf}}(\riem_1)} \]
then $\Gamma$ also has the property that 
\[  \| \mathbf{O}_{1,2} h \|_{\mathcal{D}_{\mathrm{harm}}(\riem_2)} \leq K  \| h \|_{\mathcal{D}_{\mathrm{harm}}(\riem_1)}.   \]
\end{lemma}
\begin{proof}
 For all $c$ constant on $\riem_1$ we have 
 \begin{align*}
  \| \mathbf{O}_{1,2} h \|^2_{\mathcal{D}_{\mathrm{harm}}(\riem_2)} & = \| \mathbf{O}_{1,2} ( h +c) \|^2_{\mathcal{D}_{\mathrm{harm}}(\riem_2)} \leq 
  \| \mathbf{O}_{1,2} ( h +c) \|^2_{H^1_{\mathrm{conf}}(\riem_2)} \\
  & \leq K^2 \| h+c \|^2_{H^1_{\mathrm{conf}}(\riem_1)} \\
  & =  K^2 \left( \| h \|^2_{{\mathcal{D}}_{\mathrm{harm}}(\riem_1)} + |\widehat{h+c}|^2 \right).
 \end{align*}
 The claim follows by choosing $c$ such that $\hat{c} =-\hat{h}$. 
\end{proof}
\begin{lemma} \label{le:qcmaps_preserve_CNT}
 For $k=1,2$ let $\Gamma_k$ be a quasicircle in a Riemann surface $\mathscr{R}_k$, and let $U_k$ be collar neighbourhoods of $\Gamma_k$.  Let $f:U_1 \rightarrow U_2$ be a quasiconformal map of an open neighbourhood of $U_1 \cup \Gamma_1$ which takes $\Gamma_1$ to $\Gamma_2$. Let $h:U_2 \rightarrow \mathbb{C}$. Then $h$ has a \emph{CNT} limit of $\xi$ at $p \in \Gamma_2$ if and only if $h \circ f$ has a \emph{CNT} limit of $\xi$ at $f^{-1}(p)$.  
\end{lemma}
\begin{proof}
 By conformal invariance of CNT boundary values, it's enough for this to hold for $\Gamma_k = \mathbb{S}^1$ for $k=1,2$, and a quasiconformal map $f:\mathbb{A}_{r} \rightarrow \mathbb{A}_s$ where $\mathbb{A}_r = \{ z : r< |z| <1 \}$ and $\mathbb{A}_2 = \{ z : s< |z| <1 \}$. For a proof of this fact see \cite{Schippers_Staubach_transmission}.  
\end{proof}
\begin{lemma} \label{le:quasiconformal_normalize_to_analytic}
 Let $\mathscr{R}$ be a compact Riemann surface, and $\Gamma = \Gamma_1 \cup \ldots \cup \Gamma_m$ be a collection of quasicircles separating $\mathscr{R}$ into components $\riem_1$ and $\riem_2$. Let $U_1,\ldots,U_m$ be collar neighbourhoods of $\Gamma_1,\ldots,\Gamma_n$ in $\Gamma_2$.  There is a quasiconformal map $f:\mathscr{R} \rightarrow \mathscr{R}'$ which is conformal on the complement of the closure of $U_1 \cup \cdots \cup U_m$, such that $f(\Gamma_k)$ is analytic for $k=1,\ldots,m$. 
\end{lemma}
\begin{proof}
 This was proven in \cite{Schippers_Staubach_transmission} for a single quasicircle using a sewing argument. The proof extends to a complex of curves without issue.
\end{proof}

With these three lemmas in hand, we may now prove boundedness with respect to the Dirichlet semi-norm.
\begin{proof}(of Theorem \ref{thm:bounded_overfare_Dirichlet}).
 By Lemma \ref{le:quasiconformal_normalize_to_analytic} there is a quasiconformal map $f:\mathscr{R} \rightarrow \mathscr{R}'$, which is conformal on $\riem_1$ and takes each quasicircle $\Gamma_j$ to an analytic curve $\Gamma'_j$.  Denote $\riem_1'=f(\riem_1)$ and $\riem_2' = f(\riem_2)$.  
 
 By quasi-invariance of $\dot{H}^1$, there is a fixed $K$ such that for any $h \in \mathcal{D}_{\mathrm{harm}}(\riem_1)$ we have
 \begin{equation} \label{eq:Dirichlet_overfare_temp1}
  \| \mathbf{O}_{\riem'_1,\riem'_2}( h \circ f^{-1}) \circ f \|_{\dot{H}^1(\riem_2)} \leq K \| \mathbf{O}_{\riem'_1,\riem'_2} (h \circ f^{-1}) \|_{\mathcal{D}_{\mathrm{harm}}(\riem_2')}.  
 \end{equation}
 
 Now analytic curves are WP quasicircles, so by Theorems  \ref{th:WP_are_BZM} and \ref{thm:bounded_overfare_conf}, $\mathbf{O}_{\riem'_1,\riem'_2}$ is bounded with respect to $H^1_{\mathrm{conf}}$.  Since $\riem_1'$ is connected, by Lemma \ref{le:constants_trick} there is a $K'$ such that 
 \begin{align} \label{eq:Dirichlet_overfare_temp2}
  \| \mathbf{O}_{\riem_1',\riem_2'} (h \circ f^{-1}) \|_{\mathcal{D}_{\mathrm{harm}}(\riem'_2)} & \leq K'  \nonumber
  \| h \circ f^{-1} \|_{\mathcal{D}_{\mathrm{harm}}(\riem'_1)} \\ & = K' \| h  \|_{\mathcal{D}_{\mathrm{harm}}(\riem_1)} 
 \end{align}
 where the second equality is just invariance of Dirichlet energy under conformal maps. 
 
 Finally, by Lemma \ref{le:qcmaps_preserve_CNT}, $\mathbf{O}_{\riem_1,\riem_2} h$ has the same CNT boundary values as $\mathbf{O}_{\riem_1',\riem_2'}(h \circ f^{-1}) \circ f$. 
 Let $F:=  \mathbf{O}_{\riem_1',\riem_2'}(h \circ f^{-1}) \circ f -\mathbf{O}_{\riem_1,\riem_2} h\in H^1(\riem_2)$. Then using  $F|_{\partial \riem_2} =0,$ the harmonicity of $\mathbf{O}_{\riem_1,\riem_2} h$ and the Sobolev space Stokes' theorem (see e.g. Theorem 4.3.1 page 133 in \cite{EvansGariepy}; note that we treat $\partial \riem_2$ as analytic in the double), which also works for manifolds with several oriented boundary curves, one can show that
  $$\int_{\riem_2} \partial (\mathbf{O}_{\riem_1,\riem_2} h) \, \overline{\partial F}\, dA=0.$$ This yields that
  \begin{align}
  \|\mathbf{O}_{\riem_1,\riem_2} h \|^{2}_{\mathcal{D}_{\mathrm{harm}}(\riem_2)} & \leq    \Vert  \mathbf{O}_{\riem_1,\riem_2} h\Vert^2_{\dot{H}^1(\riem_2)} +   \Vert F\Vert^2_{\dot{H}^1(\riem_2)} \nonumber \\ & =  \Vert \mathbf{O}_{\riem_1,\riem_2} h\Vert^2_{\dot{H}^1(\riem_2)} +2 \mathrm{Re}  \int_{\riem_2} \partial( \mathbf{O}_{\riem_1,\riem_2} h) \, \overline{\partial F}\, dA +\Vert F\Vert^2_{\dot{H}^1(\riem_2)} \\ & =   \|\mathbf{O}_{\riem_1',\riem_2'}(h \circ f^{-1}) \circ f \|^{2}_{\dot{H}^1(\riem_2)}, \nonumber
  \end{align}
 
 which is just the manifestation of the Dirichlet principle. Therefore  we have
 \begin{equation} \label{eq:Dirichlet_overfare_temp3}
  \|\mathbf{O}_{\riem_1,\riem_2} h \|_{\mathcal{D}_{\mathrm{harm}}(\riem_2)} \leq  
  \|\mathbf{O}_{\riem_1',\riem_2'}(h \circ f^{-1}) \circ f \|_{\dot{H}^1(\riem_2)}.
 \end{equation} 
 The claim follows from \eqref{eq:Dirichlet_overfare_temp1}, \eqref{eq:Dirichlet_overfare_temp2}, \eqref{eq:Dirichlet_overfare_temp3}.
\end{proof}

 \begin{definition} \label{de:homogeneous_Dirichlet_space} For a Riemann surface $\riem$, with finitely many connected components, let 
  $\gls{Dhom}_{\mathrm{harm}}(\riem)$ be the equivalence classes of $\mathcal{D}_{\mathrm{harm}}(\riem)$ modulo functions which are constant on each connected component of $\riem$. 
 \end{definition}
 It is clear that on $\dot{\mathcal{D}}_{\mathrm{harm}}(\riem)$ the Dirichlet semi-norm becomes a norm. 
 
 Let $\mathscr{R}$ be a compact Riemann surface, separated by quasicircles into $\riem_1$ and $\riem_2$.  
 If $\riem_1$ is connected and $c$ is a constant, then $\mathbf{O}_{\riem_1,\riem_2}$ is also constant on $\riem_2$ so the operator
 \begin{equation} \label{eq:O_dot_definition}
 \gls{doto}_{\riem_1,\riem_2} : \dot{\mathcal{D}}_{\mathrm{harm}}(\riem_1) \rightarrow  \dot{\mathcal{D}}_{\mathrm{harm}}(\riem_2) 
 \end{equation}
 is well-defined. We have
 \begin{corollary}
  Let $\mathscr{R}$ be a compact Riemann surface, separated by quasicircles into $\riem_1$ and $\riem_2$.  Assume that $\riem_1$ is connected.  Then $\dot{\mathbf{O}}_{\riem_1,\riem_2}$ is bounded with respect to the Dirichlet norm.  
 \end{corollary}
 
 One further observation must be made. 
 As a set, $\partial \riem_1 = \Gamma = \partial \riem_2$.  
 By Theorem \ref{thm:bounded_overfare_existence} and Theorem \ref{th:null_same_both_sides}, we now have the following striking result.
 \begin{corollary}  Let $\mathscr{R}$ be a compact Riemann surface and $\Gamma = \Gamma_1 \cup \cdots \cup \Gamma_m$ be a family of quasicircles separating $\mathscr{R}$ into $\riem_1$ and $\riem_2$.  Then 
 \[ \mathcal{H}(\partial \riem_1) = \mathcal{H}(\partial \riem_2). \]
 \end{corollary}
 We can now define 
 \[  \mathcal{H}(\Gamma) = \mathcal{H}(\partial \riem_1) =  \mathcal{H}(\partial \riem_2).  \]
 
 This result requires the fact that $\Gamma$ consists of quasicircles and does not appear to hold in general.  In the case of the Riemann sphere, the authors have shown that it holds with a Dirichlet-bounded identification of the spaces, precisely for quasicircles \cite{Schippers_Staubach_JMAA}.  
 
\end{subsection}
\end{section}
\begin{section}{Schiffer and Cauchy-Royden operators} \label{se:Schiffer_Cauchy}
\begin{subsection}{Assumptions throughout this section} \label{se:assumptions_scattering_section} The following assumptions
 will be in force throughout Section \ref{se:Schiffer_Cauchy}. Additional hypotheses are added to the statement of each theorem where necessary.
 \begin{enumerate}
     \item $\mathscr{R}$ is a compact Riemann surface;
     \item $\Gamma = \Gamma_1 \cup \cdots \cup \Gamma_n$ is a collection of quasicircles;
     \item $\Gamma$ separates $\mathscr{R}$ into $\riem_1$ and $\riem_2$ in the sense of Definition \ref{de:separating_complex}.
 \end{enumerate} 
  
  We will furthermore assume that the ordering of the boundaries of $\partial \riem_1$ and $\partial \riem_2$ is such that $\partial_k \riem_1 = \partial_k \riem_2 = \Gamma_k$ as sets for $k=1,\ldots,n$.  
\end{subsection}
\begin{subsection}{About this section}
 In this section, we define certain operators called Schiffer operators, and an associated integral operator which we call the Cauchy-Royden operator. We show that these are bounded and derive a simple set of relations between the Schiffer and Cauchy-Royden operators. We also derive a number of identities for the adjoints of these operators, which play a central role in the proof of the unitarity of the scattering operator.  We also establish their action on harmonic measures, which also plays a role in the scattering theory. Finally, we derive a kind of jump formula which mixes overfare and the Cauchy-Royden operator. This is the main tool in investigating the effect of the operators on cohomology classes, as well as the investigation of their kernels and images.  
\end{subsection}
\begin{subsection}{Definitions of Schiffer and Cauchy operators}\label{subsec:defSchiffer} 
Denote by $\mathscr{G}$ Green's function of $\mathscr{R}$, and let $g_k$ be Green's functions of $\riem_k$, $k=1,2$.   Here, if $\riem_k$ has more than one connected component, then $g_k$ stands for the function whose restriction to each connected component is the Green's function of that component.

  First, we define the Schiffer operators.  
 To that end, we need to define certain bi-differentials, which will be the integral kernels of the Schiffer operators. 
 
 \begin{definition}
  For a compact Riemann surface $\mathscr{R}$ with Green's function $\mathscr{G}(w,w_0;z,q),$ the \emph{Schiffer kernel} is defined by
 \[  \gls{schifferk}_{\mathscr{R}}(z,w) =  \frac{1}{\pi i} \partial_z \partial_w \mathscr{G}(w,w_0;z,q),    \]
 and the \emph{Bergman kernel} is given by
 \[  \gls{bergmank}_{\mathscr{R}} (z,w) = - \frac{1}{\pi i} \partial_z \overline{\partial}_{{w}} \mathscr{G}(w,w_0;z,q).    \]

 For a non-compact surface $\riem$ of type $(g,n)$ with Green's function $g_\riem$, we define
 \[  L_\riem(z,w) =   \frac{1}{\pi i} \partial_z \partial_w g_\riem(w,z),  \]
 and 
 \[  K_\riem (z,w) = - \frac{1}{\pi i} \partial_z \overline{\partial}_{{w}}g_\riem(w,z).  \]
 \end{definition}

  The kernel functions satisfy the following:
 \begin{enumerate} 
  \item[$(1)$] $L_{\mathscr{R}}$ and $K_{\mathscr{R}}$ are independent of $q$ and $w_0$.  
  \item[$(2)$] $K_{\mathscr{R}}$ is holomorphic in $z$ for fixed $w$, and anti-holomorphic in $w$ for fixed $z$.
  \item[$(3)$] $L_{\mathscr{R}}$ is holomorphic in $w$ and $z$, except for a pole of order two when $w=z$.
\item[$(4)$] $L_{\mathscr{R}}(z,w)=L_{\mathscr{R}}(w,z)$.  
  \item[$(5)$] $K_{\mathscr{R}}(w,z)= - \overline{K_{\mathscr{R}}(z,w)}$.  
 \end{enumerate}

For non-compact Riemann surfaces $\riem$ with Green's function, $(2)-(5)$ hold with $L_{\mathscr{R}}$ and $K_{\mathscr{R}}$ replaced by $L_\riem$ and $K_\riem$. Moreover for any vector $v$ tangent to $\Gamma^{w}_{\varepsilon}$ at a point $z$, we have
  \begin{equation} \label{eq:level_curve_identity}
   \overline{K_\riem(z,w)}(\cdot,v) = -L_\riem(z,w)(\cdot,v).
  \end{equation} 
   Note that here we can treat the boundary as an analytic curve in the double, so that it makes sense to consider vectors tangent to the boundary.
 Also, the well-known reproducing property of the Bergman kernel holds, i.e.
 \begin{equation}  \label{eq:Bergman_reproducing}
  \iint_\riem K_\riem(z,w) \wedge h(w) = h(z),
 \end{equation}
 for $h \in A(\riem)$ \cite{Royden}.

Another basic fact about the kernels above is that they are conformally invariant.  That is, for a compact surface $\mathscr{R}$ and a biholomorphism $f:\mathscr{R} \rightarrow \mathscr{R}'$ we have 
 \begin{align} \label{eq:kernels_conformally_invariant_compact}
  (f^* \times f^*)  \; L_\mathscr{R'} & = L_{\mathscr{R}}  \nonumber \\
  (f^* \times f^*) \;  K_\mathscr{R'} & = K_{\mathscr{R}} 
 \end{align}
 and similarly, for surfaces $\riem$, $\riem'$ of type $(g,n)$ and a biholomorphism
 $f:\riem \rightarrow \riem'$, 
 \begin{align} \label{eq:kernels_conformally_invariant_gntype}
  (f^* \times f^*)  \; L_{\riem'} & = L_{{\riem}}  \nonumber \\
  (f^* \times f^*)\;  K_{\riem'} & = K_{{\riem}}.  
 \end{align}
 These follow immediately from conformal invariance of Green's function 
 (\ref{eq:Greens_conf_inv_compact},\ref{eq:Greens_conf_inv_gntype}).\\ 
 
 \begin{definition}\label{def: restriction ops}
For $k=1,2$ define the \emph{restriction operators}
 \begin{align*}
  \mathbf{R}_{\riem_k}:\mathcal{A}(\mathscr{R}) & \rightarrow \mathcal{A}(\riem_k) \\
  \alpha & \mapsto \left. \alpha \right|_{\riem_k}
 \end{align*}
 and 
 \begin{align*}
  \mathbf{R}^0_{\riem_k}: \mathcal{A}(\riem_1 \cup \riem_2) & \rightarrow \mathcal{A}(\riem_k) \\
    \alpha & \mapsto \left. \alpha \right|_{\riem_k}.
 \end{align*}
 \end{definition}
 
 It is obvious that these are bounded operators. In a similar way, we also define the restriction operator
 \[  \gls{harmrest}_{\riem_k} : \mathcal{A}_{\mathrm{harm}}(\mathscr{R}) \rightarrow \mathcal{A}_{\mathrm{harm}}(\riem_k).   \] 

Having the Bergman and Schiffer kernels and the restriction operators at hand, we can now define the Schiffer operators as follows.
\begin{definition}
For $k=1,2$, we define the {\it Schiffer comparison operators} by
 \begin{align*}
  \gls{T}_{\riem_k}: \overline{\mathcal{A}(\riem_k)} & \rightarrow \mathcal{A}(\riem_1 \cup \riem_2)  \\
  \overline{\alpha} & \mapsto \iint_{\riem_k} L_{\mathscr{R}}(\cdot,w) \wedge \overline{\alpha(w)}.
 \end{align*}
\\
and 
\begin{align*}
 \gls{S}_{\riem_k}: \mathcal{A}(\riem_k) & \rightarrow \mathcal{A}(\mathscr{R}) \\
  \alpha & \mapsto \iint_{\riem_k}  K_{\mathscr{R}}(\cdot,w) \wedge \alpha(w).
\end{align*}
The integral defining $\mathbf{T}_{\riem_k}$ is interpreted as a principal value integral whenever $z \in \riem_k$.  
Also, we define for $j,k \in \{1,2\}$
 \begin{equation}\label{defn: T sigmajsigmak}
     \gls{Tmix} = \mathbf{R}^0_{\riem_k} \mathbf{T}_{\riem_j}: \overline{\mathcal{A}(\riem_j)} \rightarrow \mathcal{A}(\riem_k). 
 \end{equation}   
\end{definition} 
 \begin{theorem} \label{th:Schiffer_operators_bounded}  $\mathbf{T}_{\riem_k}$, $\mathbf{T}_{\riem_j,\riem_k}$, and $\mathbf{S}_{\riem_k}$ are bounded for all $j,k =1,2$.  
 \end{theorem}
 \begin{proof}
The operator $\mathbf{T}_{\Sigma_k}$ is defined by integration against the $L_{\mathscr{R}}$-Kernel which in local coordinates $\zeta = f(z)$, $\eta = f(w)$ is given by 
\[  (f \times f)^*L_{\mathscr{R}}(\zeta,\eta) = \frac{d\zeta \, d \eta}{\pi(\zeta-\eta)^2} + \alpha(\zeta,\eta) \] where $\alpha$ is a holomorphic bi-differential. Since the singular part of the kernel is a Calder\'on-Zygmund kernel we can use the theory of singular integral operators to conclude that the operators with kernel $L_\mathscr{R}(z,w)$ are bounded on $L^2$.  The same proof applies to $L_{\riem}$. The boundedness of $\mathbf{T}_{\riem_j, \riem_k}$ follows from this and the fact that $\mathbf{R}
^0_{\riem_k}$ is also bounded.\\
That the operator $\mathbf{S}_{\Sigma_k}$ is bounded and its image is $\mathcal{A}(\mathscr{R}),$  can be seen from the fact that the kernel $K_{\mathscr{R}}(., w)$ is holomorphic in $w$ and $\mathscr{R}$ is compact.
 \end{proof}
 
{\bf{Notation.}} As in the case of the overfare operator $\mathbf{O}$, we will use the notations 
 \[  \mathbf{S}_k, \ \ \mathbf{T}_{j,k}, \ \ \mathbf{T}_k, \ \ \mathbf{R}_{k}, \ \  \mathbf{P}_k=\mathbf{P}_{\riem_k}, \]
 wherever the choice of surfaces $\riem_1$ and $\riem_2$ is clear from context.\\
 
For any operator $\mathbf{M}$, we define the complex conjugate operator by 
 \[  \overline{\mathbf{M}} \overline{\alpha} = \overline{\mathbf{M} \alpha}  \]
 So for example 
 \[   \overline{\mathbf{T}}_{1,2}:\mathcal{A}(\riem_1) \rightarrow \overline{\mathbf{A}(\riem_2)} \]
 and similarly for $\overline{\mathbf{R}}_{\riem_k}$, etc.  
 
 The restriction operator is conformally invariant by conformal invariance of Bergman space of one-forms. 
 By (\ref{eq:kernels_conformally_invariant_compact}), the operators $\mathbf{T}$ and $\mathbf{S}$ are also conformally invariant.  Explicitly, if $f:\mathscr{R} \rightarrow \mathscr{R}'$ is a biholomorphism between compact surfaces, and we denote
 $\riem_k'= f(\riem_k)$ for $k=1,2$, then 
 \begin{align}  \label{eq:Schiffer_operators_conformally_invariant}
   f^* \; \mathbf{R}_{\riem_k'} & = \mathbf{R}_{\riem_k'} \; f^* \nonumber \\
   f^*  \; \mathbf{R}^0_{\riem_k'} & = \mathbf{R}^0_{\riem_k'} \; f^* \nonumber \\
   f^* \; \mathbf{T}_{\riem_k'} & = \mathbf{T}_{\riem_k}\; f^* \\
   f^* \; \mathbf{T}_{\riem_j',\riem_k'} & = \mathbf{T}_{\riem_j,\riem_k} \; f^* \nonumber \\ 
   f^* \; \mathbf{S}_{\riem_k'} & = \mathbf{S}_{\riem_k} \; f^*. \nonumber
 \end{align}
 
 The following basic lemma which we will used frequently in this paper, is crucial in establishing some of the forthcoming identities concerning Schiffer and Bergman kernels.
\begin{lemma} \label{le:limiting_circle_Schiffer_identity}  Fix 
 a  point $z$ and local coordinates $\phi$ near $z$.  Let $\gamma_r$ be a curve such that $|\phi(w)-\phi(z)|=r$.  Then 
 for any holomorphic one-form $\alpha$ defined near $z$, and fixed $q$, we have 
 \[   \lim_{r \searrow 0} \int_{\gamma_r,w} \frac{1}{\pi i} \partial_z \mathscr{G}(w;z,q) \overline{\alpha(w)}=0    \]
 and
 \[   \lim_{r \searrow 0} \int_{\gamma_r,w} \frac{1}{\pi i} \partial_z \mathscr{G}(w;z,q)  {\alpha(w)}= \alpha(z).    \]
 Similarly for $z \in \riem_k$ we have 
 \[   \lim_{r \searrow 0} \int_{\gamma_r,w} \frac{1}{\pi i} \partial_z {g}_{k}(w;z) \overline{\alpha(w)}=0    \]
 and
 \[   \lim_{r \searrow 0} \int_{\gamma_r,w} \frac{1}{\pi i} \partial_z {g}_{k}(w;z)  {\alpha(w)}= \alpha(z).    \]
\end{lemma}
\begin{proof}
 In coordinates denote $\phi(w)=\zeta$ and $\phi(z)=\eta$.  We have, writing $\alpha(w) = h(\zeta)d\zeta$ (with $h$ holomorphic)
 and observing that  
 \[ \partial_z \mathscr{G}(w;z,q) = G_\eta(\zeta) d\eta  + \frac{1}{2}\frac{1}{\zeta-\eta}d\eta    \]
 where $G_\eta(\zeta)$ is non-singular at $\eta$, 
 \begin{align*}
  \lim_{r \searrow 0} \int_{\gamma_r,w} \frac{1}{\pi i} \partial_z \mathscr{G}(w;z,q)  {\alpha(w)} & = 
     \lim_{r \searrow 0} \int_{|\zeta-\eta|=r,\zeta} \frac{1}{2 \pi i} \frac{d\zeta}{\zeta-\eta} h(\zeta) d\eta = h(\eta) d\eta 
     \\ & = \alpha(z).
 \end{align*}
 Similarly 
 \begin{align*}
  \lim_{r \searrow 0} \int_{\gamma_r,w} \frac{1}{\pi i} \partial_z \mathscr{G}(w;z,q) \overline{\alpha(w)} & = 
     \lim_{r \searrow 0} \int_{|\zeta-\eta|=r,\zeta} \frac{1}{2 \pi i} \frac{d\overline{\zeta}}{\zeta-\eta} \overline{h(\zeta)} d {\eta} 
     \\ & = 0
 \end{align*}
 by writing a power series expansion of $h$ and integrating in polar coordinates.  The proof for $g_k$ is identical. 
\end{proof}

 We will frequently use the following identity, which we refer to as \emph{Schiffer's identity}.
 \begin{theorem}  \label{th:Schiffer_vanishing_identity}  Let $\riem$ be a bordered surface of type $(g,n)$.  
  For all $\overline{\alpha} \in \overline{\mathcal{A}(\riem)}$ 
  \[   \iint_{\riem,w}  L_\riem(z,w) \wedge \overline{\alpha(w)}=0. \]
 \end{theorem}
 \begin{proof} Let $\riem$ be embedded in its double $\riem^d$, so that the boundary is an analytic curve.  Fixing $z \in \riem$ and Applying Stokes' theorem we then have
 \[ \iint_{\riem} L_\riem(z,w) \overline{\alpha(w)} = \int_{\partial \riem} \frac{1}{\pi i} \partial_z g_{\riem}(z;w) \overline{\alpha(w)} - \lim_{r \searrow 0}\int_{\gamma_{r,w}}
 \frac{1}{\pi i} \partial_z g_{\riem}(z;w) \overline{\alpha(w)}    \]
 with $\gamma_{r,w}$ as in Lemma \ref{le:limiting_circle_Schiffer_identity}. 
 The claim now follows from Lemma \ref{le:limiting_circle_Schiffer_identity} and the fact that for any fixed $z$, $\partial_z g_k(z;w)$ vanishes for all $w \in \partial \riem$.
 \end{proof}
 This implies that  
 \begin{equation}  \label{eq:nonsingular_Schiffer}
  \mathbf{T}_{1,1} \alpha(z)= \iint_{\riem_1,w} \left(L_\mathscr{R}(z,w) - L_{\riem_1}(z,w) \right) 
  \wedge \overline{\alpha(w)}.  \end{equation}
 This desingularizes the kernel function, and will be useful below in some of the proofs. Incidentally, it also gives a direct way to see that the principal value integral defining $\mathbf{T}_{1,1}$ is independent of the choice of local coordinate, even though the omitted disks in the integral depend on this choice.

\begin{example} \label{ex:sphere_kernels}
 If $\mathscr{R}$ is the Riemann sphere $\sphere$, we have
 \[  \mathscr{G}(w,\infty;z,q) = - \log{ \frac{|w-z|}{|w-q|}}.     \]
 Thus
 \[  K_{\sphere}(z,w) =0  \]
 and
 \[  L_{\sphere}(z,w) = - \frac{1}{2 \pi i} \frac{dw \,dz}{(w-z)^2}.  \]
 So the Schiffer operators are given by, for $\overline{\alpha(z)} = \overline{h(z)}d \bar{z}$,
 \[  \ \mathbf{T}_1  \, \overline{\alpha} (z) =
    \frac{1}{\pi} \iint_{\riem_1} \frac{\overline{h(w)}}{(w-z)^2} \frac{d\bar{w} 
    \wedge d w}{2i}  \cdot dz  \]
 and $\mathbf{S}_k=0$, $k=1,2$.    
 
 By the uniformization theorem, if $\mathscr{R}$ is a compact surface of genus zero, it is biholomorphic to $\sphere$.  Thus by conformal invariance of the Schiffer kernels \eqref{eq:kernels_conformally_invariant_compact} we see that $K_{\mathscr{R}}=0$ and $\mathbf{S}_k=0$.   
 \end{example}    
 \begin{example} \label{ex:disk_kernels}    
  For $\riem = \disk$, we have
 \[  \mathscr{G}(z,w) = - \log{\frac{|z-w|}{|1-\bar{w}z|}}.  \]
 So
 \[  L_{\disk}(z,w) = \frac{-1}{2\pi i} \frac{dw \, dz}{(w-z)^2}    \]
 and 
 \[  K_{\disk}(z,w) = \frac{1}{2\pi i}  \frac{d\overline{w} \, dz}{(1-\bar{w}z)^2}. \]
 
 For a M\"obius transformation $M$, we can verify the identities
 \[   \frac{M'(w) M'(z)}{(M(w)-M(z))^2} = \frac{1}{(z-w)^2}     \]
 and 
 \[  \frac{\overline{M'(w)} M'(z)}{(1-\overline{M(w)}M(z))^2} = \frac{1}{(1-\bar{w}z)^2}. \]
 By conformal invariance of the Schiffer kernels (\ref{eq:kernels_conformally_invariant_gntype}) we see that for any disk or half plane $U$
 \[  L_{U}(z,w) = \frac{-1}{2\pi i} \frac{dw \, dz}{(w-z)^2}    \]
 and 
 \[  K_{U}(z,w) = \frac{1}{2\pi i}  \frac{d\overline{w} \, dz}{(1-\bar{w}z)^2}. \]
\end{example}    
 
 Next we consider a kind of Cauchy operator defined using Green's function.  This operator involves integrals over the separating quasicircles, which are not in general rectifiable.  So we define the integral using limits along analytic curves which approach the quasicircle.  This is well-defined by the Anchor Lemmas \ref{le:anchor_lemma_one} and \ref{le:anchor_lemma_two}. Furthermore for quasicircles, up to constants, this limit does not depend on the side from which the curve is approached.  This significant fact, which depends on the bounded overfare theorem, is one of the motivations for the use of quasicircles throughout the paper.    
 We now define the Cauchy operators.
   
   \begin{definition}
   Let $A = A_1 \cup \cdots \cup A_n$ be a union of non-intersecting collar neighbourhoods of $\Gamma$ in $\riem_1$. For $q \in \mathscr{R} \backslash \Gamma$ and $h \in \mathcal{D}_{\mathrm{harm}}(A)$  define, for $z \in \mathscr{R} \backslash \Gamma$, the \emph{Cauchy-{Royden} operator} by
  \begin{align}  \label{eq:J_definition}
      \gls{crop}(\Gamma) h(z) = {-} {\frac{1}{\pi i}} \int_{\partial \riem_1} \partial_w \mathscr{G}(w;z,q) h(w) = {-} {\frac{1}{\pi i}} \sum_{k=1}^n 
      \int_{\partial_k \riem_1} \partial_w \mathscr{G}(w;z,q) h(w),
  \end{align} 

and the \emph{restricted Cauchy-Royden operators} by
\begin{equation}
    \gls{rcrop}(\Gamma)=\mathbf{J}_1^q(\Gamma) h |_{\riem_k}
\end{equation}
where, as will be shown later,  $\mathbf{J}_1^q(\Gamma):\mathcal{D}_{\mathrm{harm}}(\riem_1) \rightarrow \mathcal{D}_{\mathrm{harm}}(\riem_1 \cup \riem_2)$ and $\mathbf{J}_{1,k}^q(\Gamma): \mathcal{D}_{\mathrm{harm}}(\riem_1) \rightarrow \mathcal{D}_{\mathrm{harm}}(\riem_k).$

\end{definition} 
    Note that by Definition \ref{de:separating_complex} and Proposition \ref{prop:collar_charts_in_doubly_connected_charts} non-intersecting collections of collar charts exist, and the integral exists by Lemma \ref{le:anchor_lemma_one}. 
  
 The Cauchy operator is closely related to the Schiffer operators, as the following theorem shows. 
 { \begin{theorem}  \label{th:jump_derivatives}  
  For all $h \in \mathcal{D}_{\mathrm{harm}}(\riem_1)$ and any $q \in \mathscr{R} \backslash \Gamma$, 
   \begin{align*}
    \partial \mathbf{J}_{1} ^q(\Gamma)h(z)   & = \mathbf{T}_{1,2} \overline{\partial} h(z),\, \, \, z\in \riem_2  \\
    \partial \mathbf{J}_{1}^q(\Gamma)  h(z) & = \partial h + \mathbf{T}_{1,1} \overline{\partial} h,\, \, \, z\in \riem_1   \\
    \overline{\partial} \mathbf{J}_{1}^q(\Gamma) h(z) & = \overline{\mathbf{S}}_1 \overline{\partial} h(z),\, \, \, z\in \riem_1 \cup \riem_2 
   \end{align*}
  \end{theorem}
  }
  \begin{remark}
   There is a sign error in \cite{Schippers_Staubach_Plemelj}, {which is corrected here.}
  \end{remark}
  \begin{proof}
    { Assume first that $q \in \riem_2$. The first claim follows from the application of the Stokes theorem to \eqref{eq:J_definition} and the fact that the integrand is non-singular.  Similarly for $q,z \in \riem_2$, the third claim follows 
    from the same reasoning.  
    
    The second claim also follows from Stokes theorem, namely if $\Gamma_\varepsilon$ are curves given by $|w-z|=\varepsilon$ in local coordinates, positively oriented with respect to $z$,
    \begin{align} \label{eq:add_q_temp}
     \partial \mathbf{J}_{1} ^q(\Gamma)h(z) & = \partial_z \left( - \frac{1}{\pi i} \lim_{\varepsilon \searrow 0}
     \int_{\Gamma_\varepsilon} (\partial_w \mathscr{G}(w;z,q) - \partial_w g_{1}(w,z) )\,h(w) \right) \nonumber \\
     & \ \ \ \   -   \partial_z \lim_{\varepsilon \searrow 0} \frac{1}{\pi i} \int_{\Gamma_\varepsilon}
     \partial_w g_{1} (w,z) \,h(w) \nonumber\\
     & = \partial_z \left(  \frac{1}{\pi i}  
     \iint_{\riem_1} (\partial_w \mathscr{G}(w;z,q) - \partial_w g_{1}(w,z) )\wedge_w \overline{\partial} h(w) \right) \nonumber \\
     & \ \ \ \  - \partial_z \lim_{\varepsilon \searrow 0} \frac{1}{\pi i} \int_{\Gamma_\varepsilon}
     \partial_w g_{1} (w,z) \,h(w) \nonumber\\
     & = \frac{1}{\pi i}  
     \iint_{\riem_1} (\partial_z \partial_w \mathscr{G}(w;z,q) - \partial_z \partial_w g_{1}(w,z) )
      \wedge_w \overline{\partial} h(w)  + \partial h(z) 
    \end{align}
    where we have used the harmonicity of $h$. Derivation under the integral sign in the first term is justified by the fact that the integrand
    of the first term is non-singular and holomorphic in $z$ for each $w\in \riem_1$, and that $$\iint_{\riem_1, w}|( \partial_w \mathscr{G}(w;z,q) - \partial_w g_{1}(w,z) ) \wedge_w \overline{\partial}_{{w}} h(w)|$$ is locally bounded in $z$. 
    
    Similarly removing the singularity using $\partial_w g_{\riem}$, and then using the harmonicity of $h$ and Stokes' theorem yield that
    \begin{align*}
      \overline{\partial} \mathbf{J}_{1} ^q(\Gamma) h(z) & = - \overline{\partial}_{{z}} \frac{1}{\pi i} \lim_{\varepsilon \searrow 0}
     \int_{\Gamma_\varepsilon} ( \partial_w \mathscr{G}(w;z,q) -  \partial_w g_{1}(w,z) )\,h(w)  + \overline{\partial} h(z) \\
     & =  \frac{1}{\pi i} \iint_{\riem_1}(\overline{\partial}_{{z}} \partial_w \mathscr{G}(w;z,q) - \overline{\partial}_{{z}} \partial_w g_{1}(w,z) ) \wedge_w \overline{\partial}_{{w}} h(w)  + \overline{\partial} h(z).
    \end{align*}
    The third claim now follows by observing that the second term in the integral is just $- \overline{\partial} h$ because the integrand is just the complex conjugate of the Bergman kernel.
    
    Now assume that $q \in \riem_1$.  We show the second claim in the theorem.  We argue as in equation (\ref{eq:add_q_temp}), except that we must also add 
    a term $\partial_w g_{1}(w;q) h(w)$.  We obtain instead
    \[ \partial \mathbf{J}_{1} ^q(\Gamma)h(z)=  \frac{1}{\pi i} 
     \iint_{\riem_1} (\partial_z \partial_w \mathscr{G}(w;z,q) - \partial_z \partial_w g_{1}(w;z) )
      \wedge_w \overline{\partial}_{{w}} h(w)  + \partial_z \left( h(z) + h(q) \right)  \]
     and the claim follows from $\partial_z h(q) =0$. The remaining claims follow similarly.} 
  \end{proof}
  
  Combining this with Theorem \ref{th:Schiffer_operators_bounded}, we obtain
  \begin{theorem}   \label{th:jump_bounded_Dirichlet}
   $\mathbf{J}_1^q(\Gamma):\mathcal{D}_{\mathrm{harm}}(\riem_1) \rightarrow \mathcal{D}_{\mathrm{harm}}(\riem_1 \cup \riem_2)$ 
   is bounded with respect to the Dirichlet semi-norm. 
  \end{theorem}
  Of course, the roles of the surfaces $\riem_1$ and $\riem_2$ can be switched.\\  

  It follows from conformal invariance of Green's functions (\ref{eq:Greens_conf_inv_compact},\ref{eq:Greens_conf_inv_gntype}) and Dirichlet space that the Cauchy-Royden operator $\mathbf{J}$ is conformally invariant.  That is,
  if $f:\mathscr{R} \rightarrow \mathscr{R}'$ is a biholomorphism between compact surfaces, $\Gamma' = f(\Gamma)$, and $\riem_k' = f(\riem_k)$ for $k=1,2$, then 
  \begin{equation} 
    \mathbf{C}_f \mathbf{J}_k(\Gamma') = \mathbf{J}_k (\Gamma)  \mathbf{C}_f 
  \end{equation}
  which of course implies the same for $\mathbf{J}_{j,k}(\Gamma)$ and $\mathbf{J}_{j,k}(\Gamma')$ for $j,k=1,2$.

  The operator $\mathbf{J}_1^q$ is in fact bounded with respect to the $H^1_{\text{conf}}$-norm.  
  \begin{theorem} \label{th:J_bounded_Hconf}
   $\mathbf{J}^q_{1,k}(\Gamma):H^1_{\mathrm{conf}}(\riem_1) \rightarrow H^1_{\mathrm{conf}}(\riem_k)$ is bounded for $k=1,2$. 
  \end{theorem}
  Note that strictly speaking, this is not stronger Theorem \ref{th:jump_bounded_Dirichlet}, since that theorem shows that the $H^1_{\mathrm{conf}}$-norm is not necessary to control the Dirichlet norm of the output. 
  
  The proof requires a lemma.   
  \begin{lemma} \label{le:difference_Greens_in_Dirichlet} Let $g_1$ denote Green's function of $\riem_1$ 
   for  $k=1$ and $\mathscr{G}$ denote Green's function of $\mathscr{R}$.  Then for any fixed $p \in \riem_1$ and $q \in \riem_2$
   \[  \partial_w \mathscr{G}(w,w_0;p,q) - \partial_w g_1(w;p) \in \mathcal{A}_{\mathrm{harm}}(\riem_1).     \]
   If $q \in \riem_1$ then 
    \[  \partial_w \mathscr{G}(w,w_0;p,q) - \partial_w g_1(w;p) + \partial_w g_1(w;q) \in \mathcal{A}_{\mathrm{harm}}(\riem_1).     \]
   The same holds with $1$ and $2$ switched. 
  \end{lemma}
  \begin{proof}
   By definitions of $\mathscr{G}$ and $g_1$, this is a non-singular harmonic function on $\riem_1$.  So it suffices to show that the function is in $\mathcal{A}_{\mathrm{harm}}(A)$ for some collar neighbourhood of $A= A_1\cup \cdots \cup A_n$ of $\partial \riem_1$. 
   The first term $\partial_w \mathscr{G}(w,w_0;p,q)$ is obviously in $\mathcal{A}_{\mathrm{harm}}(A)$ since it is holomorphic on an open neighbourhood of the closure of $A$. By conformal invariance of Green's function and the Bergman norm, the second term can be evaluated on the double $\riem^d$, where the boundary $\partial \riem$ is then an analytic curve.  Assuming that the inner boundary of $A$ consists of $n$ analytic curves $\Gamma = \Gamma_1 \cup \cdots \cup \Gamma_n$ we get 
   \[ \iint_{A} \partial_{\bar{w}} g_1(w;p) \wedge_w \partial_w g_1(w;p) = - \int_{\Gamma} g_1(w;p)\,  \partial_w g_1(w;p) <\infty  \]
   where we have used Stokes' theorem and the fact that $g_1$ vanishes on $\partial \riem_1$.  The proof for $q \in \riem_1$ is similar.
  \end{proof}

  We can now prove Theorem \ref{th:J_bounded_Hconf}.
  \begin{proof}(of Theorem \ref{th:J_bounded_Hconf}.)  
   By Theorem \ref{th:jump_bounded_Dirichlet} and Lemma \ref{le:point_also_H1conf}, to prove that $\mathbf{J}_{1,k}^q$ is bounded, it's enough to show that for a $p$ in one of the connected components of $\riem_k$, $|(\mathbf{J}_{1,k}^q h)(p)| \lesssim \| h \|_{H^1_{\mathrm{conf}}}.$ 
   
   We first do the case of $\mathbf{J}_{1,1}^q$. First assume that $q \in \riem_2$, and $p \in \riem_1$. Then, we have using the reproducing property of Green's function (Proposition \ref{pr:Greens_reproducing}) and Stokes' theorem 
   \begin{align*}
     \mathbf{J}_{11}^q h(p) & = - \lim_{\epsilon \searrow 0} \frac{1}{\pi i}  \int_{\Gamma_\epsilon} \partial_w \mathscr{G}(w;p,q)h(w) \\
     & =  \lim_{\epsilon \searrow 0} \frac{1}{\pi i}  \int_{\Gamma_\epsilon} \left( - \partial_w \mathscr{G}(w;p,q) 
     + \partial_w g_1(w;p) \right) h(w) + h(p) \\
     & =   \frac{1}{\pi i}  \iint_{\riem_1} \left( \partial_w \mathscr{G}(w;p,q) 
     - \partial_w g_1(w;p) \right) \wedge_w \overline{\partial} h(w) + h(p). \\
   \end{align*}
   By Lemma \ref{le:point_also_H1conf} we have $|h(p)|\leq \| h \|_{H^1_{\mathrm{conf}}(\riem_1)}$, and by Cauchy-Schwarz and Lemma 
   \ref{le:difference_Greens_in_Dirichlet} we obtain
   \[  \left|  \frac{1}{\pi i}  \iint_{\riem_1} \left( \partial_w \mathscr{G}(w;p,q) 
     - \partial_w g_1(w;p) \right) \wedge_w \overline{\partial} h(w)\right| \leq C \| \overline{\partial} h \|_{\mathcal{A}_{\mathrm{harm}}(\riem_2)} \leq C \| h \|_{H^1_{\mathrm{conf}}(\riem_1)}.  \]
    If on the other hand $q \in \riem_1$, the claim follows similarly from the second part of Lemma \ref{le:difference_Greens_in_Dirichlet} and  
    \[  \mathbf{J}_{11}^q h(p) = \frac{1}{\pi i}  \iint_{\riem_1} \left( \partial_w \mathscr{G}(w;p,q) 
     - \partial_w g_1(w;p) + \partial_w g_1(w;q) \right) \wedge_w \overline{\partial} h(w) + h(p) - h(q).  \] 
     Because any point can be used in Lemma \ref{le:point_also_H1conf} to obtain a norm equivalent to the $H^1_{\mathrm{conf}}$ norm, it holds that $|h(q)| \lesssim \| h \|_{H^1_{\mathrm{conf}(\riem_1)}}$ for the norm determined by $p$. 
     
     Now we estimate $\mathbf{J}_{12}^q$. If $q \in \riem_2$, then for $p \in \riem_2$ we have similarly by Stokes' theorem
     \[  \left| (\mathbf{J}_{11}^q h)(p) \right| = \left| 
       \lim_{\epsilon \searrow 0} \frac{1}{\pi i}  \iint_{\riem_1} \partial_w \mathscr{G}(w;p,q)\wedge_w \overline{\partial} h(w) \right|     \]
     so the claim follows once again by Cauchy-Schwarz and the fact that $\partial_w \mathscr{G}(w;p,q) \in \mathcal{D}(\riem_1)$ for $p,q \in \riem_2$.  
     The case that $q \in \riem_1$ can be dealt with as above.
  \end{proof}

  Like the Cauchy integral, this operator reproduces holomorphic functions (up to constants).
  \begin{theorem} \label{th:jump_on_holomorphic}
   Assume that $h \in \mathcal{D}(\riem_1)$.  If $q \in \riem_1$, 
   let $c_q(z)$ be the function which is equal to $h(q)$ in the connected component of $\riem_k$ containing $q$ and $0$ otherwise.  
   Then 
   \[   \mathbf{J}^q_{1,1} h (z) = \left\{  \begin{array}{ll} h(z) - c_q(z)  & q \in \riem_1 \\
      h(z)   & q \in \riem_2 \end{array} \right. \]
   and 
   \[  \mathbf{J}^q_{1,2} h (z) = \left\{  \begin{array}{ll} - c_q(z)  & q \in \riem_1 \\
      0  & q \in \riem_2 \end{array} \right.  \]
   This holds with the roles of $1$ and $2$ interchanged.
  \end{theorem}
  \begin{proof}  
   Since $h \in \mathcal{D}(\riem_1)$, the integrand of $\mathbf{J}^q_1 h$ is holomorphic, except for possible singularities at $z$ and $q$ depending on their locations. 
   If $z$ is contained in $\riem_1$, and $C_r$ are curves given by $|w-z|=r$ in local coordinates, positively oriented with respect to $z$, then 
   \[  - \frac{1}{\pi i} \lim_{r \searrow 0}  \int_{C_r} \partial_w \mathscr{G}(w;z,q) h(w) = h(z)    \]
   and if $q$ is in $\riem_1$ and $C_r$ are the curves $|w-q|=r$ then 
   \[   - \frac{1}{\pi i} \lim_{r \searrow 0}  \int_{C_r} \partial_w \mathscr{G}(w;z,q) h(w) = -h(q).    \]
   The claim follows from Stokes' theorem applied to the connected components of $\riem_1$.  
  \end{proof}
  In particular, for any $q \notin \Gamma$, and any locally constant function $c$,  $\mathbf{J}^q_1 c$ is also locally constant.  Thus we obtain a well-defined operator 
  \[  \gls{Jdot}: \dot{\mathcal{D}}_{\mathrm{harm}}(\riem_1) \rightarrow   \dot{\mathcal{D}}(\riem_1 \cup \riem_2).  \]
  The Dirichlet norm becomes a semi-norm on the homogeneous space, and $\dot{\mathbf{J}}_1$ is bounded with respect to this norm. It is easily verified that $\dot{\mathbf{J}}_1$ is independent of $q$.

   Next we will prove some results about the interaction with $\mathbf{J}_1^q$ with the bounce and overfare operators. 
  \begin{proposition} \label{prop:J_agrees_with_bounce} Let $A= A_1 \cup \cdots \cup A_n$ be a union of collar neighbourhoods $A_k$ of $\Gamma_k$ in $\riem$.  For $h \in \mathcal{D}_{\mathrm{harm}}(A)$ 
  \[   \mathbf{J}_1^q(\Gamma) h  = \mathbf{J}_1^q(\Gamma) \mathbf{G}_{A,\riem_1} h.     \]
  \end{proposition}
  \begin{proof}
    The kernel $\partial_w \mathscr{G}(w;z)$ is holomorphic in $w$ in an open neighbourhood of the boundary $\Gamma$, so $\partial_w \mathscr{G}(w;z) \in \mathcal{A}(A)$. The claim now follows from Lemma \ref{le:anchor_lemma_two}.
  \end{proof}
  \begin{remark} In fact, this applies for any collection of strip-cutting Jordan curves, but we do not require this here.
  \end{remark}
  A deeper result is that for quasicircles, the limiting integral is the same from both sides up to constants; for BZM quasicircles, they are the same. 
  \begin{theorem}  \label{th:J_same_both_sides} The following statements hold:\\
   \begin{enumerate} 
   \item [$(1)$] If $\Gamma$ consists of \emph{BZM} quasicircles, then 
   for any $h \in \mathcal{D}_{\mathrm{harm}}(\riem_1)$
    \[ \mathbf{J}_1^q(\Gamma) h = - \mathbf{J}_2^q(\Gamma) \mathbf{O}_{1,2} h.  \] 
  \item [$(2)$] If $\riem_1$ is connected and $\Gamma$ is an arbitrary complex of quasicircles, then for any $\dot{h} \in \dot{\mathcal{D}}_{\mathrm{harm}}(\riem_1)$
   \[ \dot{\mathbf{J}}_1(\Gamma) \dot{h} = - \dot{\mathbf{J}}_2(\Gamma) \dot{\mathbf{O}}_{1,2} \dot{h}.  \] 
  \end{enumerate}
  \end{theorem}
  \begin{proof} We prove first claim.
   Choose doubly-connected neighbourhoods $U_1,\ldots,U_n$ of the boundary curves $\Gamma$ with charts $\phi_m:U_m \rightarrow \mathbb{A}_m$, where each $\mathbb{A}_m = \{z : r_m < |z|<R_m \}$ is an annular region in the plane. For $k=1,2$ let $A_m^k = U_m \cap \riem_k$ be collar neighbourhoods of $\Gamma$ in $\riem_k$, and set $B_m^k=\phi_m(A_m^k)$.  We claim that $\mathcal{D}_{\text{harm}}(U_m)$ is dense in
   $\mathcal{D}_{\text{harm}}(A^k_m)$ for each $k=1,\ldots,n$ with respect to the $H^1_{\mathrm{conf}}$ norms. 
   By conformal invariance of the $H^1_{\mathrm{conf}}$-norm, it is enough to prove that $\mathcal{D}_{\mathrm{harm}}(\mathbb{A}_m)$ is dense in $\mathcal{D}_{\mathrm{harm}}(B_m^k)$. 
   This follows immediately from the fact that polynomials 
   \[ p(z) = \sum_{l=s}^t z^{l}, \ \ s,t \in \mathbb{Z}, t\geq s   \]
   are dense in both $\mathcal{D}(\mathbb{A}_m)$ and $\mathcal{D}(B_m^k)$.
   
   Now let $h \in \mathcal{D}_{\text{harm}}$, and let $\Gamma_\epsilon^k$ be the level sets of Green's function $g_k$ for $k=1,2$, which are analytic curves for $\epsilon$ sufficiently close to zero. Letting $E_\epsilon$ be the region enclosed by these analytic curves, we have 
   \[ - \frac{1}{\pi i} \int_{\Gamma^2_\epsilon} \partial_w \mathscr{G}(w;z,q) \,h(w) -  \frac{1}{\pi i} \int_{\Gamma^1_\epsilon} \partial_w \mathscr{G}(w;z,q) \,h(w)  =  \iint_{E_\epsilon} \partial_w \mathscr{G}(w;z,q) \wedge_w \overline{\partial} h(w)  \]
   (note that the reversal of orientation of the contour integrals is taken into account). 
   Applying the Cauchy-Schwarz inequality to the right hand side we get 
   \[ \left| \frac{1}{\pi i} \int_{\Gamma^2_\epsilon} \partial_w \mathscr{G}(w;z,q) \,h(w) +  \frac{1}{\pi i} \int_{\Gamma^1_\epsilon} \partial_w \mathscr{G}(w;z,q) \,h(w) \right| \leq   \|\partial_w \mathscr{G}(w;z,q) \|_{\mathcal{A}_{\mathrm{harm}}(E_\epsilon)}   \| \overline{\partial} h(w) \|_{\mathcal{A}_{\mathrm{harm}}(E_\epsilon)}.   \]
   Since quasicircles have measure zero and $\cap_\epsilon E_\epsilon = \Gamma$, the right hand side goes to zero as $\epsilon \searrow 0$.  Thus
   \[ - \lim_{\epsilon \searrow 0} \frac{1}{\pi i} \int_{\Gamma^2_\epsilon} \partial_w \mathscr{G}(w;z,q) \,h(w) = \lim_{\epsilon \searrow 0}  \frac{1}{\pi i} \int_{\Gamma^1_\epsilon} \partial_w \mathscr{G}(w;z,q) \,h(w). \]  
   Now set $U=U_1 \cup \cdots U_n$ and $A^k=A^k_1 \cup \cdots \cup A^k_n$ and assume that $h \in \mathcal{D}_{\mathrm{harm}}(U)$.
   Using the above, together with the second anchor lemma \ref{le:anchor_lemma_two} and the fact that $\mathbf{G}_{A^2,\riem_2} h = \mathbf{O}_{1,2}\mathbf{G}_{A^1,\riem_1} h$, we have
   \begin{align} \label{eq:thingy_temp}
       \mathbf{J}_1^q \mathbf{G}_{A^1,\riem_1} h & = \mathbf{J}_1^q h = - \mathbf{J}_2^q h \nonumber \\
       & = -\mathbf{J}_2^q \mathbf{G}_{A^2,\riem_2} h \nonumber \\
       & = -\mathbf{J}_2^q \mathbf{O}_{1,2} \mathbf{G}_{A^1,\riem_1} h.
   \end{align}
   The proof is completed by the density of $H^1_{\text{conf}}(U)$ in $H^1_{\text{conf}}(A^1)$, the density of $\mathbf{G}_{A^1,\riem_1} H^1_{\text{conf}}(A^1)$ in $H^1_{\text{conf}}(\riem_1)$ (Theorem \ref{th:density_of_holo_collar}), and the boundedness of $\mathbf{J}^q_k$, $\mathbf{O}_{1,2}$, and $\mathbf{G}_{A^k,\riem_k}$ (Theorems \ref{th:J_bounded_Hconf}, \ref{thm:bounded_overfare_conf}, and \ref{th:bounce_bounded}).  
   
   {The proof of the second claim follows the same line, but requires a bit of care with the constants. First, observe that 
   \[  \mathbf{G}_{A^1,\riem_1} :\mathcal{D}_{\mathrm{harm}}(A^1) \rightarrow \dot{\mathcal{D}}_{\mathrm{harm}}(\riem_1)   \]
   is well-defined.  Furthermore, it is bounded with respect to the $H^1_{\mathrm{conf}}(A^1)$ and $\dot{\mathcal{D}}(\riem_1)$ norms, since the $H^1_{\mathrm{conf}}(\riem_1)$ norm dominates the Dirichlet semi-norm. The image is dense.  
   
   By \eqref{eq:thingy_temp} we have 
   \[  \dot{\mathbf{J}}_1 \dot{H} = - \dot{\mathbf{J}}_2 \dot{\mathbf{O}}_{1,2} \dot{H}  \]
   for all $\dot{H}$ arising from $H \in \mathbf{G}_{A^1,\riem_1} \mathcal{D}_{\mathrm{harm}}(A^1)$.  The second claim now follows from boundedness of $\dot{J}$ (Theorem \ref{th:jump_bounded_Dirichlet}) and boundedness of $\dot{\mathbf{O}}_{1,2}$ (Theorem \ref{thm:bounded_overfare_Dirichlet}).}

  \end{proof}   
  {
  \begin{remark}
   It is not true that  $\mathbf{G}_{A^1,\riem_1} :\mathcal{D}_{\mathrm{harm}}(A^1) \rightarrow \dot{\mathcal{D}}_{\mathrm{harm}}(\riem_1)$ is bounded with respect to the Dirichlet semi-norms. To see this, let $\riem_1$ be the annulus $\{z : 1<|z|<4 \}$, and let $A^1 = \{ z: 1<|z|<2 \} \cup \{ z: 3<|z|<4 \}$.  The claim is falsified by considering the function which is $1$ on $\{ z: 1<|z|<2 \}$ and $N$ on $\{ z: 3<|z|<4 \}$, and letting $N \rightarrow \infty$.  
  \end{remark}}

  The operator $\mathbf{J}^q_k$ satisfies a Plemelj-Sokhotski jump formula. Although we will not emphasize this role in this paper, the following theorem represents this fact. 
    The following improvement of Theorem 4.13 in \cite{Schippers_Staubach_Plemelj}, can be viewed as a CNT version of the Plemelj-Sokhotski jump formula. However, rather than referring to a function on the curve, we express the result in terms of the extensions into $\riem_1$, with the help of the overfare operator.  
   \begin{theorem} \label{th:Overfare_with_correction_functions}  
    The following statements hold:\\
    \begin{enumerate}
        \item  [$(1)$] Assume that every curve in the complex $\Gamma$ is a \emph{BZM} quasicircle. 
    For any $h \in \mathcal{D}_{\mathrm{harm}}(\riem_1)$, 
    \[  \mathbf{O}_{2,1} \mathbf{J}^q_{1,2} h = \mathbf{J}^q_{1,1} h - h   \]
    and for all $h \in \mathcal{D}_{\mathrm{harm}}(\riem_2)$
    \[  \mathbf{O}_{2,1} \mathbf{J}^q_{2,2} h - \mathbf{J}^q_{2,1} h =\mathbf{O}_{2,1} h.   \]
    \item  [$(2)$] Assume that $\riem_2$ is connected and $\Gamma$ is an arbitrary complex of quasicircles. Then for any $\dot{h} \in \dot{\mathcal{D}}_{\mathrm{harm}}(\riem_1)$, 
    \[  \dot{\mathbf{O}}_{2,1} \dot{\mathbf{J}}_{1,2} \dot{h} = \dot{\mathbf{J}}_{1,1} \dot{h} - \dot{h} \]
    and 
    \[  \dot{\mathbf{O}}_{2,1} \dot{\mathbf{J}}_{2,2} \dot{h} - \dot{\mathbf{J}}_{2,1} \dot{h} = \dot{\mathbf{O}}_{2,1} \dot{h}.   \]
    \end{enumerate}
   \end{theorem}
   \begin{proof}  We prove (1). 
    Let $A_1$ be a collar neighbourhood of $\Gamma$ in $\riem_1$. 
     Assume that the boundary $\Gamma'$ is an analytic curve which is isotopic in the closure of $A_1$ to $\Gamma$.  Orient both curves positively with respect to $\riem_1$.    By shrinking $A$ and moving $\Gamma'$ we may assume that  $q$ is not in $A_1$. We assume that $z$ is in $A_1$.  Let $\gamma_r$ denote the curve $|w-z|=r$ in local coordinates, oriented positively with respect to $z$.
     
     Applying Stokes' theorem and assuming $h \in \mathcal{D}(A_1)$, for  $z \in A_1$ we have
    {\begin{equation*}
         - \frac{1}{\pi i}  \int_{\Gamma} \partial_w 
         \mathscr{G}(w;z,q) h(w) + \frac{1}{\pi i} \int_{\Gamma'} \partial_w \mathscr{G}(w;z,q) h(w)
          = - \frac{1}{\pi i} \lim_{r \searrow 0} \int_{\gamma_r}  \partial_w \mathscr{G}(w;z,q) h(w) = h(z).  
     \end{equation*}}
     The integrand of the second integral on the left hand side is holomorphic in $w$. Therefore the integral equals the limiting integral $-\mathbf{J}^q_{1,2} h$ for any $z \in \riem_2$, and furthermore, the integral over $\Gamma'$ is a harmonic function $H_2$ in $z$ extending $-\mathbf{J}^q_{1,2} h$ into $A_1 \cup \text{cl} (\riem_2)$. For $z \in A_1$ this function thus satisfies 
     \[  \mathbf{J}^q_{1,1} h(z) - H_2(z) = h(z).        \]
     Since the CNT boundary values of $\mathbf{J}^q_{1,2} h$ equal those of the extension $H_2$, we have proved that 
     \[  \mathbf{O}_{2,1} \mathbf{J}^q_{1,2} h(z) = \mathbf{G}_{A,\riem_1} H_2  =
      \mathbf{J}^q_{1,1} h(z) - \mathbf{G}_{A,\riem_1} h(z)  \]
      by the equation above. 
      
      Applying the second anchor lemma \ref{le:anchor_lemma_two}, we obtain for all $h \in \mathcal{D}(A_1)$
      \begin{equation} \label{eq:jump_on_dense_temp1}
        \mathbf{O}_{2,1} \mathbf{J}^q_{1,2} \mathbf{G}_{A_1,\riem_1} h(z)    =
      \mathbf{J}^q_{1,1}  \mathbf{G}_{A_1,\riem_1} h(z) - \mathbf{G}_{A_1,\riem_1} h(z),   
      \end{equation}
      as claimed.   A similar argument shows that for a collar neighbourhood $A_2$ of $\Gamma$ in $\riem_2$, for all $h \in \mathcal{D}(A_2)$ we have
      \begin{equation} \label{eq:jump_on_dense_temp2}
       { \mathbf{O}_{2,1} \mathbf{J}^q_{2,2} \mathbf{G}_{A_2,\riem_2} h(z)    =
      \mathbf{J}^q_{2,1}  \mathbf{G}_{A_2,\riem_2} h(z) +   \mathbf{O}_{2,1}\mathbf{G}_{A_2,\riem_2} h(z).}   
      \end{equation}
      Observe that the derivations of \eqref{eq:jump_on_dense_temp1} and \eqref{eq:jump_on_dense_temp2} required neither the assumption that $\Gamma$ is a BZM quasicircle nor the assumption that $\riem_2$ is connected. 
      
      A density argument completes the proof of the first claim of (1). 
       Recall that $\mathbf{G}_{A_1,\riem_1} \mathcal{D}(A_1)$ is dense in $H^1_{\mathrm{conf}}(\riem_1)$ by Theorem \ref{th:density_of_holo_collar}.
    Thus it is enough to prove the claim for $\mathbf{G}_{A_1,\riem_1} h$ for $h \in \mathcal{D}(A_1)$, since $\mathbf{G}_{A_1,\riem_1}$, $\mathbf{O}_{1,2}$, and $\mathbf{J}_{1,k}^q$ are bounded with respect to $H^1_{\mathrm{conf}}$ by Theorems \ref{th:bounce_bounded}, \ref{thm:bounded_overfare_conf}, and \ref{th:J_bounded_Hconf} respectively.  A similar density argument using \eqref{eq:jump_on_dense_temp2} shows the second claim of (1).  
      
      {We now prove the first claim of (2).  For any $h \in \mathcal{D}_{\mathrm{harm}}(A_1)$ we have that \eqref{eq:jump_on_dense_temp1} holds.  Arguing as in the proof of part (2) of Theorem \ref{th:J_same_both_sides}, we have that the set of $\dot{H}$ in $\dot{\mathcal{D}}_{\mathrm{harm}}(\riem_1)$
      of the form $H = \mathbf{G}_{A_1,\riem_1} h$ for $h \in \mathcal{D}(A_1)$ are dense in $\dot{\mathcal{D}}_{\mathrm{harm}}(\riem_1)$.  By \eqref{eq:jump_on_dense_temp1} we have for such $\dot{H}$ that
      \[   \dot{\mathbf{O}}_{2,1} \dot{\mathbf{J}}_{1,2} \dot{H} = \dot{\mathbf{J}}_{1,1} \dot{H} - \dot{H}.   \]
      The claim now follows from boundedness of $\dot{\mathbf{J}}_1$ and $\dot{\mathbf{O}}_{2,1}$, which is Theorems \ref{th:jump_bounded_Dirichlet} and \ref{thm:bounded_overfare_Dirichlet} respectively.  The proof of the second claim is similar. 
      }
   \end{proof}
  
\end{subsection}

\begin{subsection}{Adjoint identities for the Schiffer operators}
 \label{se:Schiffer_adjoint_identities}
 In this section, we prove some identities for the Schiffer operators.

 \begin{theorem}[Adjoint identities]  \label{th:adjoint_identities} For $j,k =1,2$ 
  $($not necessarily distinct$)$, 
  \[    \mathbf{T}_{j,k}^* = \overline{\mathbf{T}}_{k,j}.      \]
  If the genus of $\mathscr{R}$ is non-zero, then for $k=1,2$ we have 
  \[  \mathbf{R}_k^* = \mathbf{S}_k.      \]
 \end{theorem}
 \begin{proof}
  In the case of a single quasicircle $\Gamma$, these are \cite[Theorems 3.11, 3.12]{Schippers_Staubach_Plemelj}. The proofs there hold for the case of 
  several quasicircles.    
 \end{proof}
 Also, observe that if we define 
 \begin{equation}\label{defn: Shk}
      \gls{Sharm} = \mathbf{S}_k \mathbf{P}_k + \overline{\mathbf{S}}_k \overline{\mathbf{P}}_k : \mathcal{A}_{\text{harm}}(\riem_k) \rightarrow \mathcal{A}_{\text{harm}}(\mathscr{R}) 
 \end{equation} 
 then we have by an elementary computation
 \begin{corollary} If the genus of $\mathscr{R}$ is non-zero then for $k=1,2$
   $(\mathbf{R}_k^{\mathrm{h}})^* = \mathbf{S}^{\mathrm{h}}_k$.
 \end{corollary}
 
 \begin{theorem}[Quadratic adjoint identities, Part I]  \label{th:quadratic_adjoint_one} 
    If $\mathscr{R}$ is of genus $g>0$ then
  \[  \mathbf{S}_1 \mathbf{S}_1^* + \mathbf{S}_2 \mathbf{S}_2^* = \mathbf{I}  \]
  and 
  \[  \overline{\mathbf{S}}_1 \overline{\mathbf{S}}_1^* + \overline{\mathbf{S}}_2 \overline{\mathbf{S}}_2^* = \mathbf{I}.      \]
 \end{theorem}
 \begin{proof}
  These identities follow from the 
  reproducing property of Bergman kernel.  For $\alpha \in \mathcal{A}(\mathscr{R})$ we have, using the fact that quasicircles have measure zero (see e.g. \cite{Lehto})
  \begin{align*}
    (\mathbf{S}_1 \mathbf{S}_1^* + \mathbf{S}_2 \mathbf{S}_2^* )\, \alpha (z) & = \iint_{\riem_1} K_{\mathscr{R}}(z,w) \alpha(w)  + \iint_{\riem_2} K_{\mathscr{R}}(z,w) \alpha(w) \\ & = \iint_{\mathscr{R}} K_{\mathscr{R}}(z,w) \alpha(w)
    = \alpha(w)
  \end{align*}
  which proves the first identity.  The second identity is the complex conjugate of the first.
 \end{proof}
 We will repeatedly use the fact that quasicircles have measure zero in this way, in order to express an integral over $\mathscr{R}$ as the sum of integrals over $\riem_1$ and $\riem_2$, without mentioning it each time. 
 
 To prove quadratic adjoint identities involving $\mathbf{T}$, we require a lemma.
 \begin{lemma}  \label{le:double_L_vanishes_general_genus}
 For any $w,z \in \mathscr{R}$, 
  \[  \iint_{\mathscr{R},\zeta} L_{\mathscr{R}}(z,\zeta) \wedge \overline{L_\mathscr{R}(\zeta,w)} = K_\mathscr{R}(z,w)      \]
  where the integral is interpreted as a principal value integral.  In particular, if the genus of $\mathscr{R}$ is zero then
   \[  \iint_{\mathscr{R},\zeta} L_{\mathscr{R}}(z,\zeta) \wedge \overline{L_\mathscr{R}(\zeta,w)} = 0.     \]
 \end{lemma}
 {\begin{proof}
  Fix $w=w_0$ and $z=z_0$ in the integrals above.  Let  $\gamma^\varepsilon_{w_0}$ be curves such that $\psi \circ \gamma^\varepsilon_{w_0}$ are given by $|\eta|=\varepsilon$ for a chart $\psi(\zeta)=\eta$ near $w_0$, which takes $w_0$ to $0$. Define $\gamma^\varepsilon_{z_0}$ similarly.  Let $\mathscr{R}^\varepsilon$ be the region in $\mathscr{R}$ bounded by the curves $\gamma^\varepsilon_{z_0}$ and $\gamma^\varepsilon_{w_0}$ but not containing $z_0$ and $w_0$.  We assume these curves have positive orientation with respect to $\mathscr{R}^\varepsilon$. 
  
   In these coordinates, we have (setting $u=\psi(z)$ and $v=\psi(w)$)
  \[ \frac{1}{\pi i} \partial_{\bar{w}} \mathscr{G}(\zeta,w) = - \left( \frac{1}{2\pi i} 
    \frac{1}{\bar{\eta}} + \overline{\phi(\eta)} \right)d\overline{v}  \]
  where $\phi$ is a smooth function of $\eta$ which is uniformly bounded near $z$. We suppress dependence on $u$ and $v$ because we are fixing $w=w_0$ and $z=z_0$; however, we retain $d\bar{v}$ to emphasize that the quantity is a form in the $w$ variable.
  
  We then have
  \begin{align*}
      \iint_\mathscr{R} L_\mathscr{R}(z;\zeta) \wedge_\zeta \overline{L_\mathscr{R}(\zeta;w)} & = \lim_{\varepsilon \searrow 0} \iint_{\mathscr{R}^\varepsilon}  L_\mathscr{R}(z;\zeta) \wedge_\zeta \overline{L_\mathscr{R}(\zeta;w)} \\
      & = \lim_{\varepsilon \searrow 0} \left[ \int_{\gamma^\varepsilon_{w_0}}  L_{\mathscr{R}}(z;\zeta)   \overline{ \frac{1}{\pi i} \partial_w \mathscr{G}(\zeta,w)}   
      + \int_{\gamma^\varepsilon_{z_0}} L_\mathscr{R}(z;\zeta)  \overline{ \frac{1}{\pi i} \partial_w \mathscr{G}(\zeta,w)}   \right]. \\
  \end{align*}
  
  Now in $\eta$-coordinates $L(z;\zeta) = \rho(\eta)\, d\eta$ for some holomorphic function $\rho(\eta)$, and so the first term is (where the integral is with respect to $\eta$)
  \[  \lim_{\varepsilon \rightarrow 0} \int_{\psi \circ \gamma^\varepsilon_{w_0}} \left( \rho(\eta)   \left( \frac{1}{2\pi i} 
    \frac{1}{\overline{\eta}} + \overline{\phi(\eta)}  \right)  d\eta \right)\,du \, d\bar{v} =0.   \]
   Here we have used the fact that if $\eta = \varepsilon e^{i\theta}$ then 
  \[  \frac{d\eta}{\overline{\eta}}=e^{2 i \theta} d\theta.      \]     
  
  On the other hand, in the second term it is $L_{\mathscr{R}}$ that is singular while $\partial_w \mathscr{G}$ is non-singular.  Fix $w$ and ignore the $dw$.  Now let $\eta=\phi(\zeta)$ be a holomorphic coordinate vanishing at $z$ and let the level curves $\gamma_z^\varepsilon$ be as above, and let $u=\phi(z)$ and $v=\phi(w)$. We may write 
  \[  \frac{1}{\pi i} \partial_{\bar{w}} \mathscr{G}(\zeta,w) = \left( h_1(\eta) + \overline{h_2(\eta)}\right) d\overline{u}       \]
  where $h_1$ and $h_2$ are holomorphic.  Now writing 
  $\overline{h_2(\eta)} = a_0 + a_1 \bar{\eta} + a_2 \bar{\eta}^2 + \cdots$
  and observing that (suppressing the fixed $z$, but keeping $dv$ to indicate the fact that it is a form) 
  \[  L_\mathscr{R}(z,\zeta) = \left(-  \frac{1}{2 \pi i} \frac{d\eta}{\eta^2} + k(\eta) d\eta  \right) dv    \]
  where $k$ is holomorphic.  Integrating this kernel against $\overline{h_2(\eta)} d\bar{u}$ is zero in the limit, so from this it is easily seen that 
  \begin{align*}
    \lim_{\varepsilon \rightarrow 0} \int_{\gamma^\varepsilon_z} L_\mathscr{R}(z;\zeta)   \overline{ \frac{1}{\pi i} \partial_w \mathscr{G}(\zeta,w)}  &  =   
    - \lim_{\varepsilon \rightarrow 0} \int_{\gamma^\varepsilon_z} L_\mathscr{R}(z;\zeta)    \frac{1}{\pi i} \partial_{\bar{w}} \mathscr{G}(\zeta,w) \, du d \bar{v} \\
    & = \lim_{\varepsilon \rightarrow 0} \int_{\psi \circ \gamma^\varepsilon_z}  
    h_1(\eta) \frac{d \eta}{2 \pi i \eta^2} \, du \, d \bar{v}  =  - h_1'(0)  \, du \, d \bar{v}
  \end{align*}
  where the final sign change results from the fact that the curve $\gamma^\varepsilon_z$ is negatively oriented with respect to $z$.   Now observing that 
  \[  - h_1'(0)  \, du \, d \bar{v}= - \frac{1}{\pi i} \partial_z \partial_{\bar{w}} \mathscr{G}(z,w)        = K_\mathscr{R}(z,w) \]
  the proof of the first claim is complete.   
  
  In the case that $\mathscr{R}$ has genus zero, by  Example \ref{ex:sphere_kernels} we have that $K_{\mathscr{R}}=0$, which proves the second claim. 
 \end{proof}}
 
Using this lemma, we can prove the following. 
 \begin{theorem}[Quadratic adjoint identities, part II]  \label{th:general_adjoint_double_identities}
  If $\mathscr{R}$ has genus $g>0$, then 
  \begin{align*}
    \mathbf{I} & = \mathbf{T}_{1,1}^* \mathbf{T}_{1,1} + \mathbf{T}_{1,2}^* \mathbf{T}_{1,2} + \overline{\mathbf{S}}_1^* \overline{\mathbf{S}}_1 \\
    \mathbf{I} & = \mathbf{T}_{2,1}^* \mathbf{T}_{2,1} + \mathbf{T}_{2,2}^* \mathbf{T}_{2,2} + \overline{\mathbf{S}}_2^* \overline{\mathbf{S}}_2 \\
    0 & = \mathbf{T}_{1,1}^* \mathbf{T}_{2,1} + \mathbf{T}_{1,2}^* \mathbf{T}_{2,2} + \overline{\mathbf{S}}_1^* \overline{\mathbf{S}}_2 \\
    0 & = \mathbf{T}_{2,2}^* \mathbf{T}_{1,2} + \mathbf{T}_{2,1}^* \mathbf{T}_{1,1} + \overline{\mathbf{S}}_2^* \overline{\mathbf{S}}_1. 
  \end{align*}
  If $\mathscr{R}$ has genus $g=0$, then 
   \begin{align*}
    \mathbf{I} & = \mathbf{T}_{1,1}^* \mathbf{T}_{1,1} + \mathbf{T}_{1,2}^* \mathbf{T}_{1,2}  \\
    \mathbf{I} & = \mathbf{T}_{2,1}^* \mathbf{T}_{2,1} + \mathbf{T}_{2,2}^* \mathbf{T}_{2,2}  \\
    0 & = \mathbf{T}_{1,1}^* \mathbf{T}_{2,1} + \mathbf{T}_{1,2}^* \mathbf{T}_{2,2} \\
    0 & = \mathbf{T}_{2,2}^* \mathbf{T}_{1,2} + \mathbf{T}_{2,1}^* \mathbf{T}_{1,1}.
  \end{align*}
 \end{theorem}
 \begin{proof}  Assume that $\mathscr{R}$ has genus $g >0$.    
  The first identity was proven in \cite{Schippers_Staubach_Plemelj}, in the case of one boundary curve.  The proof given there extends verbatim to the case of several boundary curves and disconnected components without issue. The second identity is just the first, with the roles of $\riem_1$ and $\riem_2$ switched. 
  The fourth identity is just the third with the roles of $\riem_1$ and $\riem_2$ interchanged.  So it is enough to prove the third identity. 
  
  Let $v \in \mathcal{A}(\riem_1)$ and $u \in \mathcal{A}(\riem_2)$, and denote the Schiffer kernels of $\riem_k$ by $L_k$ for $k=1,2$.  Then setting $M = \mathbf{T}_{1,1}^* \mathbf{T}_{2,1} + \mathbf{T}_{1,2}^*\mathbf{T}_{2,2}$ and applying Theorems \ref{th:Schiffer_vanishing_identity} and \ref{th:adjoint_identities} yields that 
  \begin{align*}
   2i \left< v, M u \right> & = 
    2i \left< \mathbf{T}_{1,1} v , \mathbf{T}_{2,1} u \right> + 2i \left< \mathbf{T}_{1,2} v,\mathbf{T}_{2,2} u \right> \\
    & =  \iint_{1,z} \iint_{1,w} \iint_{2,\zeta} L_{\mathscr{R}}(z,w) \wedge_w \overline{v(w)}
    \wedge_z \overline{L_{\mathscr{R}}(z,\zeta)} \wedge_\zeta u(\zeta)  \\
    & \ \ \ + \iint_{2,z} \iint_{1,w} \iint_{2,\zeta} L_{\mathscr{R}}(z,w) \wedge_w \overline{v(w)}
    \wedge_z \overline{L_{\mathscr{R}}(z,\zeta)} \wedge_\zeta u(\zeta) \\
    & = \iint_{1,z} \iint_{1,w} \iint_{2,\zeta} \left(L_{\mathscr{R}}(z,w) - L_1(z,w) \right) \wedge_w \overline{v(w)} \wedge_z
    \overline{L_{\mathscr{R}}(z,\zeta)}\wedge_\zeta u(\zeta) \\
    & \ \ \ + \iint_{2,z} \iint_{1,w} \iint_{2,\zeta} L_{\mathscr{R}}(z,w) \wedge_w \overline{v(w)}
    \wedge_z \overline{\left( L_{\mathscr{R}}(z,\zeta) -L_2(z,\zeta) \right)} \wedge_\zeta u(\zeta)  \\
    & =  \iint_{1,w} \iint_{2,\zeta}  \overline{v(w)} \wedge_w u(\zeta) \wedge_\zeta \iint_{1,z} \left(L_{\mathscr{R}}(z,w) - L_1(z,w) \right) \wedge_z 
    \overline{L_{\mathscr{R}}(z,\zeta)}      \\
    & \ \ \ + \iint_{1,w} \iint_{2,\zeta} \overline{v(w)}  \wedge_w u(\zeta) \wedge_\zeta \iint_{2,z} L_{\mathscr{R}}(z,w) \wedge_z 
    \overline{\left( L_{\mathscr{R}}(z,\zeta) -L_2(z,\zeta) \right)}. 
  \end{align*}
  Reorganizing the two terms above we obtain 
  \begin{align*}
   2i \left< v, M u \right> & = - \iint_{1,w} \iint_{2,\zeta}  \overline{v(w)} \wedge_w u(\zeta) \wedge_\zeta \iint_{1,z} L_1(z,w) \wedge_z
    \overline{L_{\mathscr{R}}(z,\zeta)}      \\
    & \ \ \ - \iint_{1,w} \iint_{2,\zeta} \overline{v(w)} \wedge_w u(\zeta) \wedge_\zeta \iint_{2,z} L_{\mathscr{R}}(z,w) \wedge_z 
    \overline{L_2(z,\zeta)}  \\
    & \ \ \ + 2 \iint_{1,w} \iint_{2,\zeta} \overline{v(w)} \wedge_w u(\zeta) \wedge_\zeta \iint_{\mathscr{R},z} L_{\mathscr{R}}(z,w) \wedge_z
    \overline{L_{\mathscr{R}}(z,\zeta)}.
  \end{align*}
  Observing that $\zeta$ is not in the closure of $\riem_1$, the first term vanishes by Schiffer's identity (i.e. Theorem \ref{th:Schiffer_vanishing_identity}) applied to $L_1$.  Similarly the second term vanishes because $z$ is not in the closure of $\Sigma_2$.  Thus applying Lemma \ref{le:double_L_vanishes_general_genus}  and Theorem \ref{th:adjoint_identities} yield that
  \begin{align*}
   2i \left< v,Mu \right> & =  2 \iint_{1,w} \iint_{2,\zeta} \overline{v(w)} \wedge_w u(\zeta) \wedge_\zeta K_{\mathscr{R}}(w,\zeta) \\
   & = - {2}  \iint_{1,w} \iint_{2,\zeta} \overline{v(w)} \wedge_w K_{\mathscr{R}}(w,\zeta)  \wedge_\zeta u(\zeta) \\
   & = - 2i \left< v, \overline{\mathbf{R}}_1  \overline{\mathbf{S}}_2 u\right>.
  \end{align*}
  This completes the proof in the case of non-zero genus.  If $\mathscr{R}$ has genus zero, then all the computations above are still valid.  We need only observe that in the last step $K_{\mathscr{R}}=0$ by Example \ref{ex:sphere_kernels}.    
 \end{proof}

 Taking complex conjugates and using the adjoint identities of Theorem \ref{th:adjoint_identities}, we also have for non-zero genus: 
 \begin{align} \label{eq:other_general_adjoint_double}
    \mathbf{I} & = \mathbf{T}_{1,1} \mathbf{T}_{1,1}^* + \mathbf{T}_{2,1} \mathbf{T}_{2,1}^* + {\mathbf{S}}_1^* {\mathbf{S}}_1 \nonumber \\
    \mathbf{I} & = \mathbf{T}_{1,2} \mathbf{T}_{1,2}^* + \mathbf{T}_{2,2} \mathbf{T}_{2,2}^* + {\mathbf{S}}_2^* {\mathbf{S}}_2 \nonumber \\
    0 & = \mathbf{T}_{1,1} \mathbf{T}_{1,2}^* + \mathbf{T}_{2,1} \mathbf{T}_{2,2}^* + {\mathbf{S}}_1^* {\mathbf{S}}_2 \\ \nonumber
    0 & = \mathbf{T}_{2,2} \mathbf{T}_{2,1}^* + \mathbf{T}_{1,2} \mathbf{T}_{1,1}^* + {\mathbf{S}}_2^* {\mathbf{S}}_1, 
  \end{align}
 and in the genus zero case we have 
  \begin{align*}  
    \mathbf{I} & = \mathbf{T}_{1,1} \mathbf{T}_{1,1}^* + \mathbf{T}_{2,1} \mathbf{T}_{2,1}^*  \\
    \mathbf{I} & = \mathbf{T}_{1,2} \mathbf{T}_{1,2}^* + \mathbf{T}_{2,2} \mathbf{T}_{2,2}^*  \\
    0 & = \mathbf{T}_{1,1} \mathbf{T}_{1,2}^* + \mathbf{T}_{2,1} \mathbf{T}_{2,2}^*  \\ 
    0 & = \mathbf{T}_{2,2} \mathbf{T}_{2,1}^* + \mathbf{T}_{1,2} \mathbf{T}_{1,1}^*.
  \end{align*}
  
 Finally we have the following identity:
  \begin{theorem}[Quadratic adjoint identities, part III]  \label{th:general_adjoint_identity_double_final}  If $\mathscr{R}$ has non-zero genus, then
   \begin{align*}
     0 & = \mathbf{T}_{1,1} \overline{\mathbf{S}}_1^* + \mathbf{T}_{2,1} \overline{\mathbf{S}}_2^* \\
     0 & = \mathbf{T}_{1,2} \overline{\mathbf{S}}_1^* + \mathbf{T}_{2,2} \overline{\mathbf{S}}_2^*.
   \end{align*}
  \end{theorem}
  \begin{proof}
   First, recall that $\mathbf{S}_k^* = \mathbf{R}_k$ by Theorem \ref{th:adjoint_identities}. Thus the identities are equivalent to showing that  
   \[  \iint_{\mathscr{R}} L_{\mathscr{R}}(z,\zeta) \wedge_\zeta \overline{\alpha(\zeta)} = 0   \]
   for $z \in \riem_k$, $k=1,2$ and all $\alpha\in \mathcal{A}(\mathscr{R})$.  
  
   Fix $z \in \riem_k$.  Let $\gamma_\varepsilon$ be a curve given by 
   $|z-\zeta|=\varepsilon$ in a local coordinate chart, with orientation chosen to be positive with respect to $z$.  
  Stokes' theorem yields that the principal value integral is given by 
   \[   -\frac{1}{\pi i} \lim_{\varepsilon \searrow 0}  \int_{\gamma_\varepsilon}  \partial_z \mathscr{G}(z;\zeta, q) \overline{\alpha(\zeta)}  \]
   where $\mathscr{G}$ is Green's function of $\mathscr{R}$.  
   By Lemma \ref{le:limiting_circle_Schiffer_identity} this is zero.  
  \end{proof}
  
Taking complex conjugates and using \ref{th:adjoint_identities} we also obtain 
  \begin{align} \label{eq:general_adjoint_identity_double_final}
   0 & = \mathbf{S}_1 \mathbf{T}_{1,1} + \mathbf{S}_{2} \mathbf{T}_{1,2} \nonumber \\
   0 & = \mathbf{S}_{1} \mathbf{T}_{2,1} + \mathbf{S}_{2} \mathbf{T}_{2,2}.
  \end{align}
\end{subsection}
\begin{subsection}{The Schiffer operators on harmonic measures}  \label{se:Schiffer_on_harmonic_measures}
A relationship between the Schiffer operators and the harmonic measure is established in the following result:
\begin{theorem} \label{th:T_and_S_on_harmonic_measures} 
 Let  $d\omega$ be a harmonic measure on $\riem_1$.  Then
 \[  \mathbf{T}_{1,1} \overline{\partial} \omega = - \partial \omega + \mathbf{R}_1 \mathbf{S}_1 {\partial} \omega      \]
 and
     \[  \mathbf{T}_{1,2} \overline{\partial} \omega = \mathbf{R}_2 \mathbf{S}_1 \partial \omega. \]
\end{theorem}
\begin{proof}  Assume that $\omega=1$ on one boundary of $\riem_1$ and $0$ on the others.  It is enough to prove the claim for such $\omega$.  {We will need to use a particular set of limiting curves in computing the boundary integrals, for which the computation simplifies.}
{Let $\Gamma_\varepsilon$ denote the union of the level sets $\omega = \varepsilon$, $\omega= 1-\varepsilon$}. Using Lemma \ref{le:harmonic_measure_collar_chart},  for $\varepsilon$ sufficiently small this consists of $n$ disjoint curves each homotopic to $\partial_k \riem_1$ for a particular $k$.  Also, let $\gamma_r$ be as in Lemma \ref{le:limiting_circle_Schiffer_identity}.  
 We then have that, fixing $q \in \riem_2$, 
 \begin{align*}
    \mathbf{T}_{1,1} \overline{\partial} \omega & = \frac{1}{ \pi i } \iint_{\riem_1} 
    \partial_z \partial_w \mathscr{G}(w;z,q) \wedge_w \overline{\partial} \omega (w) \\
    & = \lim_{\varepsilon \searrow 0} \frac{1}{ \pi i } \int_{\Gamma_\varepsilon,w} \partial_z \mathscr{G}(w;z,q)\, \overline{\partial} \omega(w) - 
    \lim_{r \searrow 0}\frac{1}{ \pi i } \int_{\gamma_r,w} \partial_z \mathscr{G}(w;z,q)\, \overline{\partial} \omega(w).
 \end{align*}
 Applying Lemma \ref{le:limiting_circle_Schiffer_identity} twice and using the fact that $\overline{\partial}\omega(w)=- \partial \omega(w)$ on the level curves $\Gamma_\varepsilon$, we see that 
 \begin{align*}
     \mathbf{T}_{1,1} \overline{\partial} \omega(z) & = - \lim_{\varepsilon \searrow 0} \frac{1}{ \pi i } \int_{\Gamma_\varepsilon,w} \partial_z \mathscr{G}(w;z,q)\, {\partial}\, \omega(w) \\
     & = - \lim_{\varepsilon \searrow 0} \frac{1}{ \pi i } \int_{\Gamma_\varepsilon,w} \partial_z \mathscr{G}(w;z,q)\, {\partial} \omega(w) + 
    \lim_{r \searrow 0}\frac{1}{ \pi i } \int_{\gamma_r,w} \partial_z \mathscr{G}(w;z,q)\, {\partial} \omega(w) \\
     &  \ \ \ \ \ \    - 
    \lim_{r \searrow 0}\frac{1}{ \pi i } \int_{\gamma_r,w} \partial_z \mathscr{G}(w;z,q)\, {\partial} \omega(w) \\
    & = - \frac{1}{ \pi i } \iint_{\riem_1} 
    \partial_z \overline{\partial}_w \mathscr{G}(w;z,q) \wedge_w  {\partial} \omega (w) - \partial \omega(z)  \\
    & = \mathbf{R}_1 \mathbf{S}_1 \partial \omega (z) - \partial \omega (z).
 \end{align*}
 The proof of the second claim is similar, except with the integrals over $\gamma_r$ removed, since for $z \in \riem_2$ there are no singularities in $\riem_1$.
\end{proof}

\begin{definition}
 We say that $u \in \mathcal{A}_{\text{harm}}(\mathscr{R})$ is \emph{piecewise exact} if 
\[  \mathbf{R}_k^{\mathrm{h}} u  \in \mathcal{A}^{\mathrm{e}}_{\mathrm{harm}}(\riem_k)   \]
for $k=1,2$, where $\mathbf{R}_k^{\mathrm{h}}$ is as in Definition \ref{def: restriction ops}.  We denote the space of piecewise exact harmonic forms on $\mathscr{R}$ by $\gls{peform}(\mathscr{R})$. 
\end{definition}

{\begin{definition}\label{def:exact overfare}
  If $\riem_2$ is connected, the \emph{exact overfare} can be defined as follows. Given 
 the spaces $\mathcal{A}^\mathrm{e}$ of Definition \ref{def: exact holo and harm forms} we define 
 \[  \gls{exacto}: \mathcal{A}^{\mathrm{e}}(\riem_2) \rightarrow \mathcal{A}^{\mathrm{e}}(\riem_1) \]
 to be the unique operator satisfying 
 \begin{equation}\label{charac of exact overfare}
   \mathbf{O}^{\mathrm{e}}_{2,1} d = d \mathbf{O}_{2,1}.    
 \end{equation}      
 If $\riem_1$ is connected we may define $\mathbf{O}^\mathrm{e}_{1,2}$ in the same way.
\end{definition}}
\begin{corollary} \label{co:overfare_piecewise_exact} 
 For any harmonic measure $d\omega \in \mathcal{A}_{\mathrm{harm}}(\riem_1)$, $\mathbf{S}_1 \partial \omega + \overline{\mathbf{S}}_1 \overline{\partial} \omega = \mathbf{S}_1^{\mathrm{h}} d\omega \in \mathcal{A}_{\mathrm{harm}}^{\mathrm{pe}}(\mathscr{R})$.  
 Furthermore, if $\riem_2$ is connected, then
 \[  \mathbf{O}^{\mathrm{e}}_{2,1} \left(  \mathbf{R}^{\mathrm{h}}_2 \mathbf{S}_1 d \omega \right) =   \mathbf{R}^\mathrm{h}_1 \mathbf{S}^\mathrm{h}_1  d \omega + d\omega.   \]
\end{corollary}
\begin{proof}
 Since all operators involved are complex linear, it is enough to prove this for real harmonic measures $d\omega \in \mathcal{A}_{\mathrm{hm}}(\riem_1)$.  For such harmonic measures, by Theorems \ref{th:jump_derivatives} and \ref{th:T_and_S_on_harmonic_measures},
 \[     d\mathbf{J}^q_{1,1}\, \omega  = \partial \omega + \mathbf{T}_{1,1} \overline{\partial} \omega + \overline{\mathbf{R}}_1 \overline{\mathbf{S}}_1 \overline{\partial} \omega = \mathbf{R}_1 \mathbf{S}_1 \partial \omega + \overline{\mathbf{R}}_1 \overline{\mathbf{S}}_1 \overline{\partial} \omega  \]
 and 
 \[  d \mathbf{J}^q_{1,2}\, \omega = \mathbf{T}_{1,2} \overline{\partial} \omega + \overline{\mathbf{R}}_2 \overline{\mathbf{S}}_1 \overline{\partial} \omega = \mathbf{R}_2 \mathbf{S}_1 \partial \omega + \overline{\mathbf{R}}_2 \overline{\mathbf{S}}_1 \overline{\partial} \omega  \]
 which proves the first claim. 
 
 Now if $\riem_2$ is connected, then $\mathbf{O}_{2,1}^{\mathrm{e}}$ is well-defined (if not, it's only defined up to addition of a harmonic measure on $\riem_1$.) 
 By the transmitted jump formula (Theorem \ref{th:Overfare_with_correction_functions}), 
 \begin{align*}
     \dot{\mathbf{O}}_{2,1} \dot{\mathbf{J}}_{1,2} \, \dot{\omega} = \dot{\mathbf{J}}_{1,1} \, \dot{\omega} - \dot{\omega}.
 \end{align*}
 Taking $d$ of both sides and applying the previous two equations and \eqref{charac of exact overfare}, completes the proof of the theorem. 
\end{proof}
\begin{remark}  \label{re:nonself-overfare}
 Note that this shows that an element of $\mathcal{A}^{\mathrm{pe}}_{\text{harm}}$ need not be its own overfare.
\end{remark}

Another identity for the harmonic measures and $\mathbf{S}$ operator is the following. 

\begin{theorem} \label{th:S_on_harmonic_measures}
 Let $\omega_1$ be a harmonic function on $\riem_1$ which is constant on each boundary curve.  Let $\omega_2 = \mathbf{O}_{1,2}\, \omega_1$.  If either $\riem_1$ or $\riem_2$ is connected, then 
 \[ \mathbf{S}_1^{\mathrm{h}} \, d\omega_1 = -\mathbf{S}_2^{\mathrm{h}} \, d \omega_2. \]
\end{theorem}
\begin{proof} Assume that $\riem_1$ is connected. 
 By Theorem \ref{th:J_same_both_sides}
 \[ \dot{\mathbf{J}}_{1,2}\, \dot{\omega}_1 = - \dot{\mathbf{J}}_{2,2} \dot{\mathbf{O}}_{1,2}\, \dot{\omega}_1 = - \dot{\mathbf{J}}_{2,2}\, \dot{\omega}_2.    \]
 Applying $d$ to both sides we get
 \begin{align*}
   \mathbf{R}_2 \mathbf{S}_1  \partial \omega_1 +  \overline{\mathbf{R}}_2 \overline{\mathbf{S}}_1 \overline{\partial}   \omega_1 & = \mathbf{T}_{1,2}\, \overline{\partial} \omega_1 +  {\mathbf{R}}_2  {\mathbf{S}}_1 \partial \omega_1 =
  d \mathbf{J}_{1,2}\, \omega_1 = - d \mathbf{J}_{2,2}\, \omega_2  \\
  & = - \partial \omega_2 - \mathbf{T}_{22}\, \overline{\partial} \omega_2 - \overline{\mathbf{R}}_{2} \overline{\mathbf{S}}_2\, \overline{\partial} \omega_2 \\ &  =  -{\mathbf{R}}_{2} {\mathbf{S}}_2 {\partial} \omega_2 - \overline{\mathbf{R}}_{2} \overline{\mathbf{S}}_2 \, \overline{\partial} \omega_2
 \end{align*}
 where we have used Theorem \ref{th:T_and_S_on_harmonic_measures}.  Therefore $\mathbf{R}_2^\mathrm{h} \mathbf{S}_1^\mathrm{h}\, d \omega_1 = - \mathbf{R}_2^\mathrm{h} \mathbf{S}_2^\mathrm{h} \, \omega_2$ which proves the claim by analytic continuation to $\mathscr{R}$. 
 
 If, on the other hand, $\riem_2$ is connected, we have $\omega_1=\mathbf{O}_{2,1} \omega_2$. One can now repeat the proof with the roles of $1$ and $2$ switched.
\end{proof}

This immediately leads to a characterization of which harmonic measures lie in the kernel of $\mathbf{S}$ and $\mathbf{T}_{12}$. 
{
\begin{corollary} \label{co:bridgeworthy_TFAE}
Let $\omega_1 \in \mathcal{D}_{\mathrm{harm}}(\riem_1)$ have constant boundary values. 
The following statements are equivalent. 
\begin{enumerate}[label=(\arabic*),font=\upshape]
    \item $\mathbf{S}_1\, \partial \omega_1 =0$; 
    \item $\mathbf{T}_{1,2}\, \overline{\partial} \omega_1 =0$;
    \item $\mathbf{T}_{1,1}\, \overline{\partial} \omega_1 = -\partial \omega_1$.
\end{enumerate}
If at least one of $\riem_1$ or $\riem_2$ is connected, then $(1)-(3)$ are also equivalent to each of the following: 
\begin{enumerate}[label=(\arabic*),font=\upshape] \setcounter{enumi}{3} 
    \item $\mathbf{S}_2 \, \partial\, \mathbf{O}_{1,2}\, \omega_1 =0$;
    \item $\mathbf{T}_{2,1}\, \overline{\partial}\, \mathbf{O}_{1,2}\, \omega_1 =0$;
    \item $\mathbf{T}_{2,2}\, \overline{\partial}\, \mathbf{O}_{1,2}\, \omega_1 = -\partial \, \mathbf{O}_{1,2}\, \omega_1$.
\end{enumerate}

The complex conjugates of the statements above also hold.
\end{corollary} 
\begin{proof}
  The equivalence of the first three claims follows from Theorem \ref{th:T_and_S_on_harmonic_measures}, as does the equivalence of claims (4) through (6). If one of $\riem_1$ and $\riem_2$ is connected, then by Theorem \ref{th:S_on_harmonic_measures}, the holomorphic part of $\mathbf{S}_1^{\mathrm{h}} d \omega_1$ is zero if and only if the holomorphic part of $\mathbf{S}_2^{\mathrm{h}} d \mathbf{O}_{1,2} \omega_1$ is zero. This proves the equivalence of (1) and (4).
  The remaining claim is obvious.
\end{proof} 
}

 We review some facts about harmonic one-forms on compact Riemann surfaces; see for example \cite[Chapter III]{Farkas_Kra}. We start by recalling  the standard way to define harmonic one-forms $H_C$ such that 
\begin{equation}\label{eq:H_forms_definition}
     (\alpha,\ast H_C) = \int_C \alpha 
\end{equation}   
for all $\alpha \in \mathcal{A}_{\text{harm}}(\mathscr{R})$.  
Given a curve $C$, let $\Omega$ be a strip to the left of $C$, and let $f$ a real-valued function which is $1$ on $C$, smooth on $\Omega$, and $0$ outside of $\Omega$.  Thus, there is a steady increase from $0$ to $1$ as one approaches $C$ from the left, and a jump back down to $0$ as one crosses the curve.  Then $df$ is smooth, and we have 
\begin{align*}
 (\alpha,\ast df)_{\mathscr{R}} & = \iint_{\mathscr{R}} \alpha \wedge \ast \ast \overline{df} = -\iint_{\mathscr{R}} \alpha \wedge   {df}  \\
 & = -\iint_{\Omega} \alpha \wedge   {df} = \int_C f \alpha \\
 &= \int_C \alpha.
\end{align*} 
By the Hodge theorem, there is a unique harmonic one-form $H_C$ in the same equivalence class as $df$.  Now co-exact forms are orthogonal to closed forms because 
\[  (\alpha,\ast dg )_{\mathscr{R}} ={-} \iint_{\mathscr{R}} \alpha \wedge dg =  \iint_{\mathscr{R}} d (g \alpha) =0.    \]
Thus 
\[  (\alpha,\ast H_C)_{\mathscr{R}}  = (\alpha,\ast df)_{\mathscr{R}} = \int_C \alpha. \]

Now let $\mathscr{R}$ be separated into two surfaces $\riem_1$ and $\riem_2$ by a collection of curves $\partial_k \riem_1$ as in Figure \ref{fig:boundary_curves}.  We assume that $\riem_1$ and $\riem_2$ are both connected.

\begin{figure}
     \includegraphics[width=10cm]{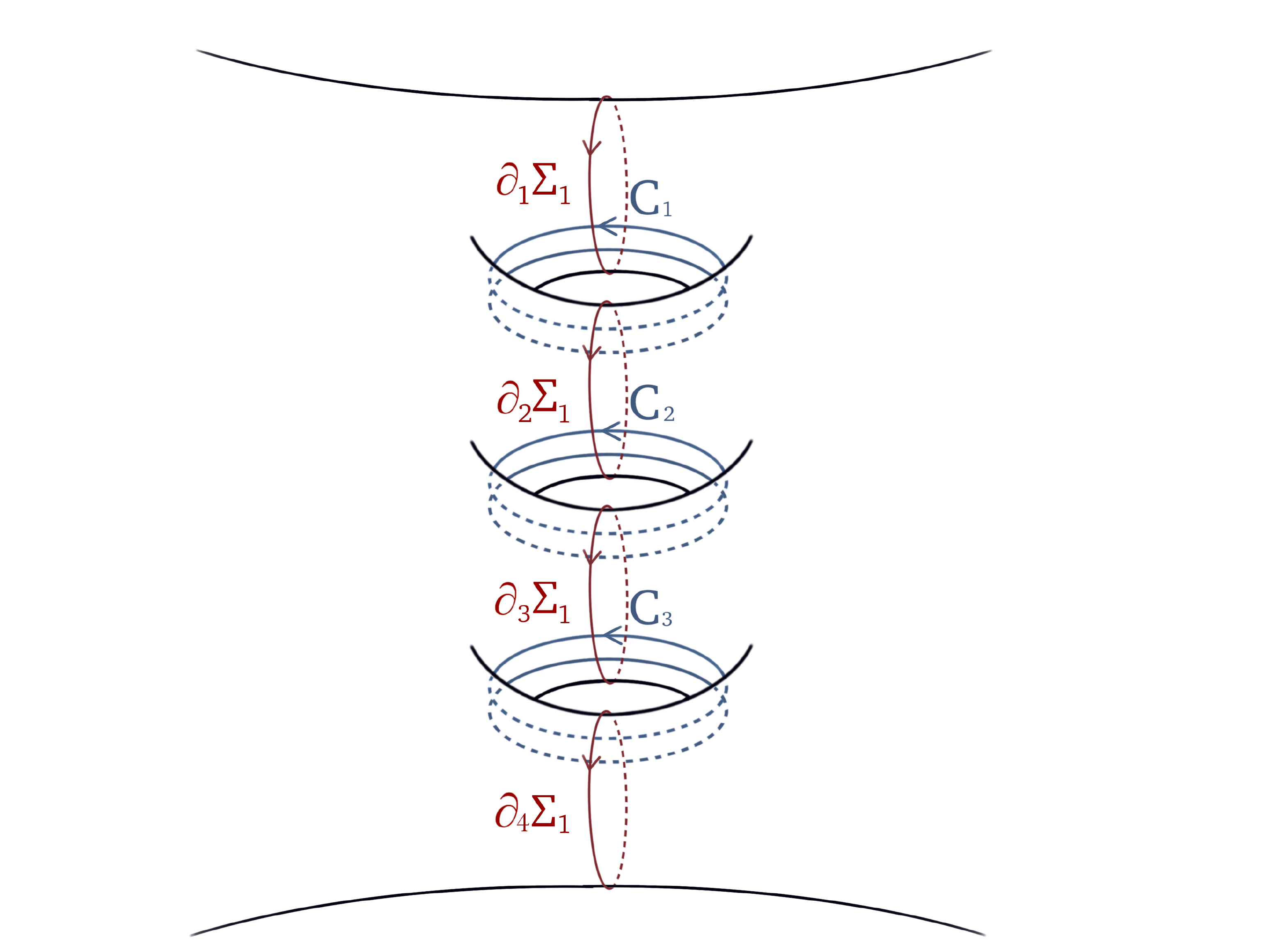}
     \caption{The one-forms $df_{k}$}
     \label{fig:boundary_curves}
 \end{figure} 
The strips defining $df_k$, $k=1,2,3$ are the horizontal strips and the boundary curves are the vertical curves. Since $H_{C_k}$ and $df_k$ are in the same cohomology class, we have (by taking analytic limiting curves approaching $\partial_m \riem_1$)
\[ \int_{\partial_m \riem_1}  H_{C_k} = \int_{\partial_m \riem_1} df_k. \]
Here we let $k=1,\ldots,n-1$ where $n$ is the number of boundary curves (in Figure \ref{fig:boundary_curves} we have $n=4$).
 
We see that
\begin{align}  \label{eq:integrals_of_Hcs} \nonumber
 \int_{\partial_1 \riem_1} H_{C_k} & = \left\{ \begin{array}{ll} -1 & k=1 \\ 0 & \text{otherwise} \end{array} \right. \\  \nonumber
\int_{\partial_2 \riem_1} H_{C_k} & = \left\{ \begin{array}{ll} 1 & k=1 \\ -1 & k=2 \\ 0 & \text{otherwise} \end{array} \right. \\ 
\vdots & \ \ \ \ \ \ \ \vdots \\ \nonumber
\int_{\partial_{n-1}  \riem_1} H_{C_k} & = \left\{ \begin{array}{ll} 1 & k={n-2} \\ -1 & k=n-1 \\ 0 & \text{otherwise} \end{array} \right. \\ \nonumber 
\int_{\partial_n \riem_1} H_{C_k} & = \left\{ \begin{array}{ll} 1 & k={n-1}  \\  0 & \text{otherwise} \end{array} \right. 
\end{align}
Furthermore, the integral around any $C_k$ or internal homology curve is zero.
\begin{remark}
 The reason for the differing first and last integrals is that we haven't included the redundant $H_{C_n}$ which corresponds to a curve traversing the outside of the entire surface.
\end{remark}
 
We have the following result:\\
\begin{theorem} \label{th:piecewise_exact_characterize}  
Assume that $\riem_1$ and $\riem_2$ are connected.   Then  
\[ \mathcal{A}_{\mathrm{harm}}^{\mathrm{pe}}(\mathscr{R})  = \{   \mathbf{S}_1 \partial \omega +   \overline{\mathbf{S}}_1 \overline{\partial} \omega : d\omega \in { \mathcal{A}_{\mathrm{harm}}(\riem_1) } \}.   \] 
\end{theorem}
\begin{proof}
 The fact that $\mathbf{S}_1 \partial \omega +   \overline{\mathbf{S}}_1 \overline{\partial} \omega$ is piecewise exact for $d\omega \in \mathcal{A}_{\mathrm{hm}}(\riem_1)$ follows from Theorem \ref{th:jump_derivatives}. We need to show that such forms span the space. 
 
 To do so we first need an identity. 
 Let $\omega_k$ be the harmonic function which is one on the $k$th boundary curve and zero on the remaining boundary curves. Set 
 \[  \beta_k = \mathbf{S}_1 \partial \omega_k +   \overline{\mathbf{S}}_1 \overline{\partial} \omega_k   \]
 We then have that by definition of $H_{C_j}$  
 \begin{align*}
   \int_{C_j} \beta_k & = (\beta_k,\ast H_{C_j})_{\mathscr{R}} \\
   & = (\mathbf{S}^{\mathrm{h}}_1 d\omega,\ast H_{C_j})_{\mathscr{R}} \\
   & = (d\omega, \mathbf{R}_1^{\mathrm{h}} \ast H_{C_j})_{\riem_1} \\
   & = \left(d\omega, \mathbf{R}_1^{\mathrm{h}}\left[ \frac{1}{2} \left(\ast H_{C_j} -i H_{C_j} \right)+  \frac{1}{2} \left(\ast H_{C_j} +i H_{C_j} \right) \right] \right)_{\riem_1}.
 \end{align*}
 So 
 \begin{align*}
   \int_{C_j} \beta_k & =  \left(  \partial \omega_k, \frac{1}{2} \mathbf{R}_1\left( \ast H_{C_j} -i H_{C_j} \right)  \right)_{\riem_1} +
   \left(  \, \overline{\partial} \omega_k, \frac{1}{2}\overline{\mathbf{R}_1}  \left( \ast H_{C_j} + i H_{C_j} \right)  \right)_{\riem_1}
   \\  & =  \left( d \omega_k, \frac{1}{2} \mathbf{R}_1\left( \ast H_{C_j} -i H_{C_j} \right)  \right)_{\riem_1} +
   \left(  \, d\omega_k, \frac{1}{2}\overline{\mathbf{R}_1}  \left( \ast H_{C_j} + i H_{C_j} \right)  \right)_{\riem_1}
   \\
   & = \left(  d \omega_k,  \mathbf{R}^{\mathrm{h}}_1 \ast H_{C_j}  \right)_{\riem_1}.
 \end{align*}
 
 So we have
 \begin{align} \label{eq:integral_beta_is_integral_H}
  \int_{C_j} \beta_k  
  & = - \iint_{\riem_1} d \omega_k \wedge H_{C_j} \nonumber \\
  & = - \int_{\partial \riem_1} \omega_k H_{C_j} \nonumber \\
  & = - \int_{\partial_k \riem_1} H_{C_j}.
 \end{align}
 It now follows immediately from \eqref{eq:integrals_of_Hcs} that the $\beta_k$ span $\mathcal{A}^{\mathrm{pe}}_{\mathrm{harm}}(\mathscr{R})$. This completes the proof.
\end{proof}

We also have the following elementary fact.
\begin{proposition} \label{pr:period_submatrix_also_positive}
 Fix a bordered Riemann surface $\riem$ of type $(g,n)$. Fix a subcollection $\gamma_1,\ldots,\gamma_m$ of the boundary curves $\{ \partial_1 \riem,\ldots, \partial_n \riem \}$. For any $c_1,\ldots,c_m \in \mathbb{C}$, there is an $\omega \in \mathcal{A}_{\mathrm{hm}}(\riem)$ whose boundary values are only non-zero on $\gamma_1,\ldots,\gamma_m$ such that 
 \[  \int_{\gamma_k} \omega = c_k.   \]
\end{proposition}
\begin{proof}
 Let $j_1,\ldots,j_m$ be the indices of the subcollection of curves; that is, $\gamma_{l} =\partial_{j_l} \riem$ for $l=1,\ldots,m$. Following a similar strategy as in the proof of Corollary \ref{co:boundary_periods_specified_starmeasure}, it is readily seen that we need to prove the existence of a solution to the system of equations 
  \[  c_k=  \int_{\gamma_k}  \sum_{l=1}^n a_l \ast d \omega_{j_l} =\sum_{l=1}^n \int_{\partial_{j_k} \riem}   a_l \ast d \omega_{j_l}= \sum_{l=1}^n\Pi_{j_k j_l} a_l .      \]
 This follows from the fact that any square submatrix of a positive-definite matrix is also positive-definite.
\end{proof}

Now we define a subclass of the harmonic measures which lie in the kernel of $\mathbf{T}_{1,2}$. 
  
 \begin{definition}  
  We say that $\omega \in \mathcal{D}_{\text{harm}}(\riem_1)$ is \emph{bridgeworthy} if
  \begin{enumerate}
   \item it is constant on each boundary curve;
   \item on any pair of boundary curves $\partial_k \riem_1$ and $\partial_m \riem_1$ that bound the same connected component of $\riem_2$, the boundary values of $\omega$ are equal.
  \end{enumerate}
  
  We say that $\alpha \in \mathcal{A}_{\mathrm{hm}}(\riem_1)$ is bridgeworthy if $\alpha = d\omega$ for some bridgeworthy harmonic function $\omega$.  Denote the collection of bridgeworthy harmonic functions by
 $\gls{Dbw}(\riem_1)$,  and the collection of bridgeworthy  harmonic one-forms by 
 $\gls{Abw}(\riem_1)$.  The same definitions apply to $\riem_2$.
 \end{definition}
 The name is meant to invoke the following geometric picture: $\omega$ is bridgeworthy if it has the same constant value on any pair of boundary curves which are connected by a ``bridge'' in $\riem_2$. 
 
 \begin{remark} If $\riem_1$ has more than one connected component, then  a bridgeworthy harmonic one-form has anti-derivatives which are not bridgeworthy. This is because one can add to a bridgeworthy harmonic function $\omega$ any function which is constant on connected components without changing $d\omega$. 
 \end{remark}

We will also need the following characterization of the kernel of $\mathbf{S}_k^{\mathrm{h}}$.
\begin{proposition}  \label{pr:kernel_S1h}   { Assume that either $\riem_1$ or $\riem_2$ is connected}.  Fix $k=1$ or $k=2$.  
 Let $d\omega_k \in \mathcal{A}_{\mathrm{hm}}(\riem_k)$.
 The following are equivalent.
 \begin{enumerate}[label=(\arabic*),font=\upshape]
     \item $\overline{\partial} \omega_k \in \overline{\partial} \mathcal{D}_{\mathrm{bw}}(\riem_k)$.
     \item $\partial \omega_k \in \partial \mathcal{D}_{\mathrm{bw}}(\riem_k)$.
     \item $\mathbf{S}_k^\mathrm{h} d\omega_k =0$.
     \item {$d\omega_k \in \mathcal{A}_{\mathrm {bw}}(\riem_k)$.}
 \end{enumerate}
\end{proposition}
\begin{proof}  We assume throughout that $k=1$. It is then necessary and  sufficient to show the equivalence for both the cases that $\riem_1$ is connected and that $\riem_2$ is connected. The case $k=2$ is obtained by symmetry. 

  We first show that (1) implies (3); assume that (1) holds. {If $\riem_2$ is connected, then $\omega_1$ is constant, so $d\omega_1=0$ and (3) follows trivially.  Now assume that $\riem_1$ is connected.}   Let $H \in \mathcal{D}_{\mathrm{bw}}(\riem_1)$ be such that $\overline{\partial} H= \overline{\partial} \omega_1$.  Then $\mathbf{O}_{1,2} H$ is constant on connected components of $\riem_2$, so by Theorem \ref{th:J_same_both_sides}, $\dot{\mathbf{J}}_{1,2} H = - \dot{\mathbf{J}}_{2,2} \dot{\mathbf{O}}_{1,2} \dot{H}$ =0.  So by Theorem \ref{th:jump_derivatives}
  \[  \mathbf{T}_{1,2} \overline{\partial} H + \overline{\mathbf{R}}_1 \overline{\mathbf{S}}_1 \overline{\partial} H = d \dot{\mathbf{J}}_{1,2} \dot{H} = 0.   \]
  Since the holomorphic and anti-holomorphic parts must both be zero, we have $\mathbf{T}_{1,2} \overline{\partial} \omega_1 = 0$ and $\overline{\mathbf{R}}_1 \overline{\mathbf{S}}_1 \overline{\partial} \omega_1 =0$.  The latter implies that  $\overline{\mathbf{S}}_1 \overline{\partial} \omega_1 =0$ by analytic continuation.  The former together with Corollary \ref{co:bridgeworthy_TFAE} to $\overline{\omega}_1$, implies that
  \[  \mathbf{R}_1 \mathbf{S}_1 \partial \omega_1 =  0.  \]
  Hence $\mathbf{S}_1 \partial \omega_1 =0$ and therefore $\mathbf{S}_1^\mathrm{h} d\omega_1 =0$. This shows that (1) implies (3).  A similar argument shows that (2) implies (3). 
  
 {Now assume that (3) holds, and that $\riem_2$ is connected.} We will show that both (1) and (2) hold. Since holomorphic and anti-holomorphic parts of $\mathbf{S}_1^\mathrm{h} d\omega_1$ are zero, we have $\mathbf{S}_1 \partial \omega_1 = 0$ and 
  $\overline{\mathbf{S}}_1 \overline{\partial} \omega_1=0$.  By Corollary \ref{co:bridgeworthy_TFAE}, we also have that 
\[  \mathbf{T}_{1,2} \overline{\partial} \omega_1 = 0\ \ \ \text{and} \ \ \ \overline{\mathbf{T}}_{1,2} \partial \omega_1 =0.  \]
Here, to show the left equation, we have applied the equivalence of parts (1) and (3) of Corollary \ref{co:bridgeworthy_TFAE} directly; whereas to show the right equation, we applied the conjugates of the equivalence of parts (1) and (3) to $\overline{\omega_1}$ to see that 
\[ 0 = \overline{\mathbf{S}}_1 \overline{\partial} \omega_1   =  \overline{\mathbf{S}_1 \partial \overline{\omega_1}} \ \ \Rightarrow  \ \ 0 = \overline{\mathbf{T}_{1,2} \overline{\partial} \overline{\omega_1}}  =\overline{\mathbf{T}}_{1,2} \partial \omega_1.  \]

We thus have that
\[ d \mathbf{J}^q_{1,2} \omega_1 = \mathbf{T}_{1,2} \overline{\partial} \omega_1 + \overline{\mathbf{R}}_1 \overline{\mathbf{S}}_1 \overline{\partial} \omega_1 = 0,  \]
and so $\mathbf{J}^q_{1,2} \omega_1$ is constant on connected components of $\riem_2$. Hence $H = \mathbf{O}_{2,1} \mathbf{J}_{1,2}^q \omega_1$ is bridgeworthy.  Applying Theorem \ref{th:Overfare_with_correction_functions} yields 
\[  \dot{\mathbf{J}}_{1,1} \dot{\omega}_1 - \dot{H} = \dot{\omega}_1,  \]
from which we obtain
\[  d \dot{\mathbf{J}}_{1,1} \dot{\omega}_1 - d \dot{H}  = d \dot{\omega}_1.   \]
Now Theorems \ref{th:jump_derivatives} and \ref{th:S_on_harmonic_measures} yield that
\[ d \dot{\mathbf{J}}_{1,1} \dot{\omega}_1 =\partial \dot{\omega}_1 + \mathbf{T}_{1,1} \overline{\partial} \dot{\omega}_1 + \overline{\mathbf{R}}_1 \overline{\mathbf{S}}_1 \overline{\partial} \dot{\omega}_1 = 0,  \]
and inserting this in the above equation we obtain that $\overline{\partial} \omega_1 = - \overline{\partial} H$. This proves that (3) implies (1) {in the case that $\riem_2$ is connected.} 
A similar argument, using the fact that $\overline{\mathbf{J}}_{1,2}^q \omega_1$ is constant on connected components of $\riem_2$ shows that $\partial \omega_1 = - \partial G$ 
where $G = \mathbf{O}_{2,1} \mathbf{J}_{1,2}^q \omega_1$, so (3) implies (2) {in the case that $\riem_2$ is connected.}

{Now assume that $\riem_1$ is connected and that (3) holds.  We will show that (4) holds. We have that $\mathbf{S}_1^h d\omega_1 = 0$.  Then by Theorem \ref{th:S_on_harmonic_measures} $\mathbf{S}_2^h d\mathbf{O}_{2,1} \omega_1 =0$.  Thus
\[   \mathbf{S}_2 \partial \mathbf{O}_{1,2} \omega_1 = 0 \ \ \ \text{and} \ \ \ \overline{\mathbf{S}}_2 \overline{\partial} \mathbf{O}_{1,2} \omega_1 = 0.    \]
Applying Corollary \ref{co:bridgeworthy_TFAE} parts (1)--(3) and Theorem \ref{th:jump_derivatives} we see that
\[ d {\mathbf{J}}_{2,1} {\mathbf{O}}_{1,2} \omega_1 = \mathbf{T}_{2,1} \overline{\partial} \mathbf{O}_{2,1} \omega_1 + \overline{\mathbf{R}}_1
 \overline{\mathbf{S}}_2  \overline{\partial} \mathbf{O}_{1,2} \omega_1 =0. \]
Thus $\mathbf{J}_{2,1} \mathbf{O}_{1,2}$ is locally constant.  Similarly
\[   d \mathbf{J}_{2,2} \mathbf{O}_{1,2} \omega_1 = \overline{\partial} \mathbf{O}_{1,2} \omega_1 + \mathbf{T}_{2,2} \overline{\partial} \mathbf{O}_{1,2} \omega_1 + \overline{\mathbf{R}}_2 \overline{\mathbf{S}}_2 \overline{\partial} \mathbf{O}_{1.2} \omega_1 =0    \]
so $\mathbf{J}_{2,2} \mathbf{O}_{1,2} \omega$ is constant.  Thus by Theorem 
\ref{th:Overfare_with_correction_functions} part (2) we see that
\[ 0= - \dot{\mathbf{O}}_{1,2} \dot{\mathbf{J}}_{2,1} \dot{\mathbf{O}}_{1,2} \dot{\omega}_1 + \dot{\mathbf{J}}_{2,2} \dot{\mathbf{O}}_{1,2} \dot{\omega}_1 = \dot{\mathbf{O}}_{1,2} \dot{\omega}_1. \]
Thus $\mathbf{O}_{1,2} \omega$ is constant, that is, $\omega_1$ is bridgeworthy.  Thus (3) implies (4) in the case that $\riem_1$ is connected. 
 It is obvious that (4) implies (1) and (4) implies (2) independently of the connectivity assumption.  
 
 In summary, we have shown the equivalence of (1), (2), and (3), and furthermore that (4) implies (1) and (2).  It remains to show that (1) implies (4).   }
Assuming that (1) holds, we have that $\omega_1 = \tilde{\omega_1} + h$ where $\tilde{\omega_1}$ is bridgeworthy and $h$ is holomorphic on $\riem_1$.  So $h=\omega_1-\tilde{\omega_1}$ has constant boundary values on $\partial \riem_1$. Fix a connected component $\riem_1^0$ and treat it as a subset of its double, so that the boundary is an analytic curve. We have that by Schwarz reflection $h$ extends to a holomorphic function on a neighbourhood of $\partial \riem_1^0$. Since $h$ is constant on the boundary it is constant on $\riem_1^0$.
{We have shown that $\omega_1 = \tilde{\omega_1} + c$ where $c$ is constant on connected components and $\tilde{\omega_1}$ is bridgeworthy; thus $d\omega_1 = d \tilde{\omega_1} \in \mathcal{A}_{\mathrm bw}(\riem_1)$.
}      
\end{proof}


\end{subsection}
\end{section} 
\begin{section}{Dirichlet problem for $L^2$ harmonic one-forms}  \label{se:Dirichlet_problem}
\begin{subsection}{Assumptions throughout this section}
 In this section, we consider a Riemann surface $\riem$  of type $(g,n)$.
\end{subsection}
\begin{subsection}{About this section}  

  In this section, we give a complete theory and solution of the Dirichlet problem for $L^2$ one-forms. This includes developing a theory of their boundary values, which we show can be identified with the Sobolev space $H^{-1/2}(\partial \riem)$. Given an element of $H^{-1/2}(\partial \riem)$ together with sufficient cohomological data, there is a unique $L^2$ harmonic one-form on $\riem$ with those boundary values. Furthermore, the solution depends continuously on the data.  
 
 We also characterize the boundary values in terms of equivalence classes of $L^2$ harmonic one-forms defined in collar neighbourhoods. We show that there is a one-to-one correspondence between elements of $H^{-1/2}$ and such equivalence classes, and this allows us to use the theory of CNT boundary values developed in Section \ref{se:CNT_all} to solve the problem. 
 Anti-derivatives of such forms have well-defined CNT boundary values, which can be identified with elements of $H^{1/2}$ (after removing a period). This reflects the fact that $H^{-1/2}$ is in some sense the set of distributional derivatives of elements of $H^{1/2}$.  
 
 We outline the approach. In Section \ref{se:regular_Dirichlet_forms} we give the routine solution to the Dirichlet problem for smooth boundary values. This section does not contain any original material, but rather serves to outline how the cohomological data is dealt with without the distraction of analytic complications. In particular it establishes the cohomological preliminaries used in the proof of the general case. In Section \ref{se:boundary_values_and_Hnegativeonehalf}, we show the equivalence between the CNT and $H^{-1/2}$ boundary values of one-forms.
 The bulk of the main results, namely the proof of the well-posedness of the Dirichlet problem for CNT boundary values, is given in Section \ref{se:Dirichlet_full_solution}.  Finally, 
 in Section \ref{se:Dirichlet_forms_Hnegativeonehalf} we use the equivalence between $H^{-1/2}$ and CNT boundary values of one-forms, together with the solution to the CNT boundary value problem given in Section \ref{se:Dirichlet_full_solution}, to solve the $H^{-1/2}$ Dirichlet problem for $L^2$ one-forms. 
 

\end{subsection}  
\begin{subsection}{Formulation of the regular Dirichlet problem}
\label{se:regular_Dirichlet_forms}
 
 Let $\riem$ be a Riemann surface of type $(g,n)$.  We describe a network of smooth curves on $\riem$.  By Corollary \ref{co:caps_can_be_sewn_on} we can treat $\riem$ as a subset of a compact Riemann surface $\mathscr{R}$ obtained by either sewing on caps, or as a subset of the double.  In the latter case, the boundary curves are analytic, and in the former, the boundary curves can be taken to be analytic, if one sews on caps via analytic parametrizations.

 For the moment, let $\gamma_1,\ldots,\gamma_{2g}$ be specific simple smooth closed curves which are generators of the homology of the surface $\mathscr{R}$ obtained by sewing on caps.  We choose these curves such that they lie in $\riem$, and furthermore, such that when $\mathscr{R}$ is cut along these curves we obtain a polygonal decomposition of $\mathscr{R}$ in the standard way.  Denote by $c_k$ curves which are isotopic to the boundaries $\partial_k \riem$ for $k=1,\ldots,n$; we assume that these are non-intersecting. See Figure \ref{fig:polygon_with_holes} for a picture of the polygonal decomposition. 
 \begin{figure}
     \includegraphics[width=6cm]{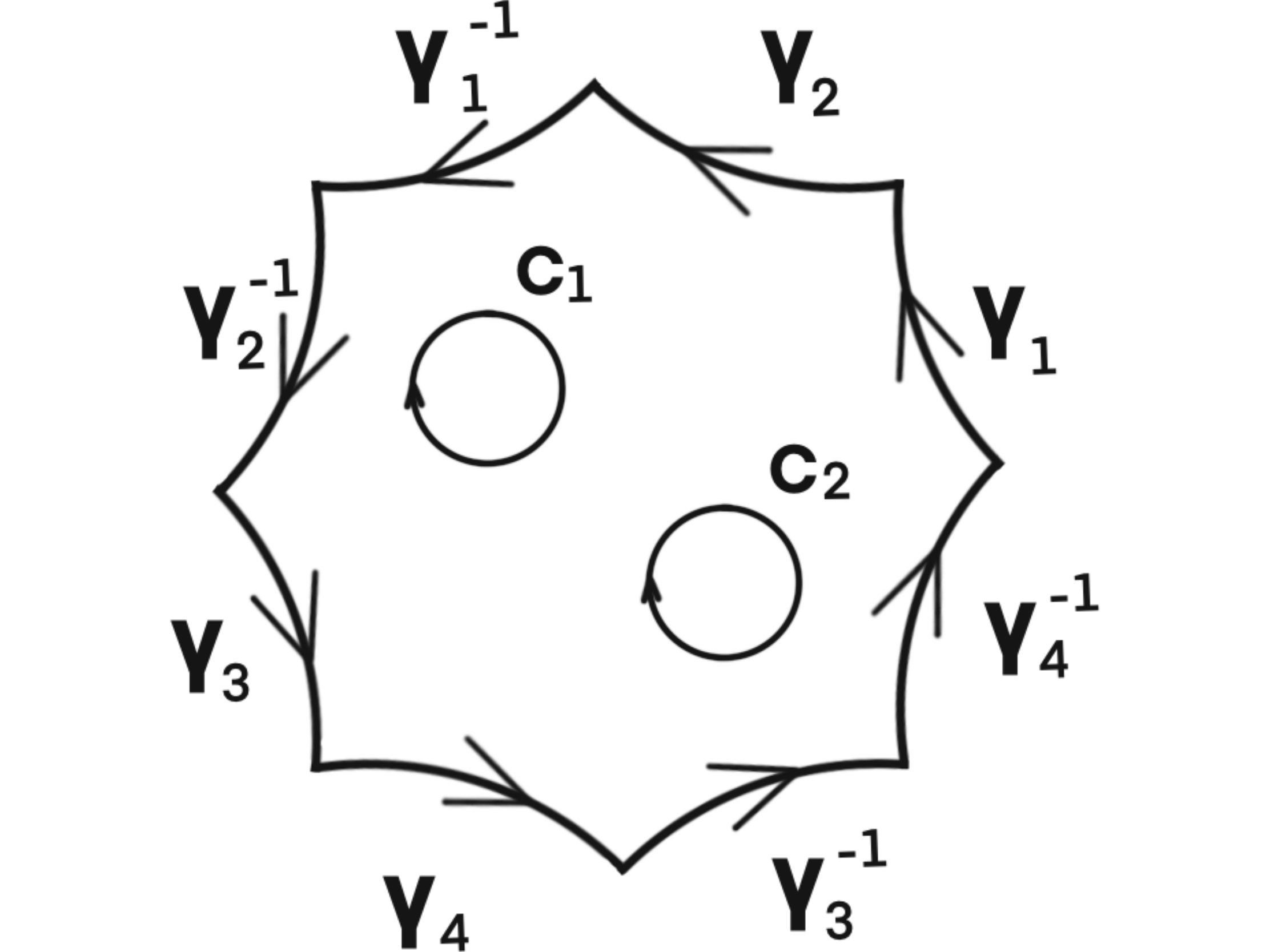}
     \caption{Polygononal decomposition of the bordered surface}
     \label{fig:polygon_with_holes}
 \end{figure}
  For any $\alpha \in \mathcal{A}_{\mathrm{harm}}(\riem)$, we have that $\int_{\gamma} \alpha$ depends only on the homotopy class of $\gamma$, so we can define 
  \[  \int_{\partial_k\riem} \alpha = \int_{c_k} \alpha, \]
  and with this definition, it is clear that if we let $\gamma$ denote the boundary of the polygon then 
  \begin{equation} \label{eq:sum_of_boundary_periods_zero}
   \sum_{k=1}^n  \int_{\partial_k \riem} \alpha = - \int_{\gamma} \alpha = 0     
  \end{equation}
  for any $\alpha \in \mathcal{A}_{\mathrm{harm}}(\riem)$.\\
 
 

 We also need the following facts regarding the double $\riem^d$ of $\riem$.  There is a basis of generators $\{ \Gamma_1,\ldots,\Gamma_{4g + 2n-2} \}$ for the homology of $\riem^d$ so that $\Gamma_k = \gamma_k$ for $k=1,\ldots,2g$ and $\Gamma_k = \partial_k\riem$ for $k=2g+2,\ldots,2g+n-1$.   
 
 Now let $\{\varepsilon_1,\ldots,\varepsilon_{4g+2n-2} \}$ be a dual basis of closed one-forms on $\riem^d$.  By the Hodge decomposition theorem these can be chosen to be harmonic.  
 We thus have 
 \begin{equation} \label{eq:dual_basis_double}
   \int_{\Gamma_k} \varepsilon_j = \delta_j^k,  \ \ \ \  j,k =1,\ldots,4g+2n-2 
 \end{equation}
 where $\delta_j^k$ is the Kronecker delta.

 Our data in the Dirichlet problem will consist of continuous one-forms on the boundary curves together with specified period information.  Since $\riem$ is a bordered surface, the notion of continuous or smooth one-forms is well-defined; explicitly, $\alpha$ is a continuous or smooth one-form if for some collar chart $\phi$ of $\partial_k \riem$, setting $\psi = \left. \phi \right|_{\partial_k \riem}$ its expression in coordinates is $\psi^* \alpha = h(e^{i\theta})\, d\theta$  for some continuous or smooth function $h$ (see Remark \ref{re:border_form_regularity_meaning}).   
 
 The $\mathcal{C}^\infty$ Dirichlet problem for one-forms is as follows. 
 Let $\riem$ be a Riemann surface of type $(g,n)$.  We refer to the following data as smooth Dirichlet data for forms on a Riemann surface:
 \begin{enumerate}[label=\roman*.]
     \item $\mathcal{C}^\infty$ one-forms $\beta_k$ on $\partial_k \riem$ for each $k=1,\ldots,n$, satisfying
      \[  \int_{\partial_1 \riem} \beta_1 + \cdots + \int_{\partial_n \riem} \beta_n =0;  \] 
     \item constants $\rho_1,\ldots,\rho_{n} \in \mathbb{C}$ satisfying
     \[  \rho_1 + \cdots + \rho_n =0;\]
     and 
     \item constants $\sigma_1,\ldots,\sigma_{2g} \in \mathbb{C}$.  
 \end{enumerate}
 \vspace{0.5cm}
\begin{definition}\label{defn:cont dirichlet data}
We say that a harmonic one-form $\alpha$ on $\riem$ solves the Dirichlet problem with data  $(\beta,\rho,\sigma)$ if 
 $\alpha$ extends smoothly to $\partial \riem$ and 
 \begin{enumerate}
     \item for any tangent vector $v_p$ to $\partial_k \riem$ at any point $p \in \partial_k \riem$, $\alpha(v_p)=\beta_k(v_p)$;
     \item for all $k =1,\ldots,n$
     \[  \int_{\partial_k \riem} \ast \alpha = \rho_k;   \]
     and 
     \item for all $k=1,\ldots,2g$
     \[  \int_{\gamma_k} \alpha = \sigma_k.  \]
 \end{enumerate} 
\end{definition}
 Note that the one-forms $\beta_k$ specify the periods around the boundary curves $\partial_k \riem$. Condition (2) is motivated as follows. For any harmonic measure $\sum_k d \omega_k$ and any solution $\alpha$, the form $\alpha + \sum_k d\omega_k$ still satisfies (1) and (3), because $\sum_k d\omega_k$ is exact and $\sum_k d\omega_k=0$ along $\partial\riem$.  In fact, this is the only indeterminacy and the condition (2) uniquely determines the solution.\\

 It is elementary that a solution exists; in fact, the smooth Dirichlet problem for one-forms is essentially a smooth Dirichlet problem for functions. One simply subtracts off forms whose periods match the data, so that one obtains boundary values of exact forms. One then solves the Dirichlet problem for functions with respect to the primitive on the boundary.   The solution to the problem for functions is well-known:
\begin{theorem}  \label{th:smooth_Dirichlet_solutionfunction} Let $X$ be a compact Riemannian manifold with boundary $\partial X,$ and $\Delta$ is the Laplacian on $X$. Then the Dirichlet problem
   \begin{equation} \label{smooth dirichlets problem}
    \begin{cases}
    \Delta u=0 \\
    u|_{\partial X}= f\in \mathcal{C}^{\infty}(\partial X)
    \end{cases}
  \end{equation}
 has a unique solution $u\in \mathcal{C}^{\infty}(X)$.
 \end{theorem}
 For the proof see e.g. \cite{Hormander} page 264 Example 1.

  \begin{theorem} \label{th:smooth_Dirichlet_solution} For smooth Dirichlet data $(\beta,\rho,\sigma)$ there exists an $\alpha \in  \mathcal{C}^\infty (\mathrm{cl} (\, \riem))$ which solves the smooth Dirichlet problem.
 \end{theorem}
 \begin{proof}  We assume that $\riem$ is included in its double, so that the boundary curves are analytic.  Setting 
  \begin{equation*} \label{eq:Dirichlet_theorem_cont_temp1}
     \lambda_k = \int_{\partial_k \riem} \beta_k 
  \end{equation*}
  for $k = 1,\ldots,n$,  by Corollary \ref{co:boundary_periods_specified_starmeasure} there is a
  $\mu \in \ast \mathcal{A}_{\mathrm{hm}}$ such that 
  \begin{equation} \label{eq:Dirichlet_theorem_cont_temp2}
    \int_{\partial_k \riem} \mu = \lambda_k   
  \end{equation}
  for every $k$ and  a harmonic one-form $\eta$ in the span of $\{ \varepsilon_1,\ldots,\varepsilon_{2g} \}$, which were defined in connection to \eqref{eq:dual_basis_double}, such that 
  \begin{equation*} \label{eq:Dirichlet_theorem_cont_temp3}
    \int_{\gamma_j} \eta = \sigma_j - \int_{\gamma_j} \mu     
  \end{equation*}
  for $j=1,\ldots,2g$. Since $\partial_k \riem$ is null-homotopic in $\mathscr{R}$, we have 
  \begin{equation} \label{eq:Dirichlet_theorem_cont_temp4} 
    \int_{\partial_k \riem} \eta = 0 
  \end{equation}
  for $k=1,\ldots,n$.   Observe that the one-forms $\eta$ and $\mu$ are smooth on $\partial \riem$.  
  
  Define functions $h_k$ on the boundary curves $\partial_k \riem$ as follows.  Each $h_k$ is the anti-derivative of $\beta_k - \mu - \eta$ on $c_k$, that is, for any tangent vector $v$ to the boundary $c_k$
   \begin{equation*} 
     dh_k(v) = \beta_k(v) - \mu(v) - \eta(v).   
   \end{equation*}
   Note that each anti-derivative is single-valued by \eqref{eq:Dirichlet_theorem_cont_temp1} and the definition of $\varepsilon_k$.   By Theorem \ref{th:period_matrix_invertible} we can add a suitable harmonic measure $d\omega \in \mathcal{A}_{\mathrm{hm}}(\riem)$ (which is exact and does not change the periods) in order to ensure that condition (2) in Definition \ref{defn:cont dirichlet data} is satisfied.  
   Solving now the ordinary Dirichlet problem with smooth data $h_1,\ldots,h_{n}$ on the boundary curves using Theorem \ref{th:smooth_Dirichlet_solution}, we obtain a smooth $h \in \mathcal{D}_{\mathrm{harm}}(\riem)$.  Then 
   \[   \alpha = dh + \mu + \eta  \]
   is the desired solution to the problem.  
 It is not hard to show that the solution is unique by keeping track of the periods and using uniqueness in Theorem \ref{smooth dirichlets problem}.
 \end{proof}

 
 
\end{subsection} 
\begin{subsection}{Boundary values of $L^2$ forms and $H^{-1/2}$}
\label{se:boundary_values_and_Hnegativeonehalf}
 
 In this section, we will show that $H^{-1/2}(\partial_k \riem)$ of a boundary curve $\partial_k \riem$ can be identified with an equivalence class of harmonic one-forms defined in a collar neighbourhood. The idea is fairly simple, and we give a sketch in the case of the circle $\mathbb{S}^1$ before launching into the details.  We can think of smooth one-forms $h(\theta) d \theta$ on the circle as dual to functions on the circle via the pairing
 \[  L_{h d\theta} (f)  = \int_{\mathbb{S}^1} f \cdot h d\theta. \]
 Of course if $h d\theta$ is in $H^{-1/2}(\mathbb{S}^1)$ and $f \in H^{1/2}(\mathbb{S}^1)$, then this only makes sense distributionally.
 
 On the other hand, given an $\alpha \in \mathcal{A}_{\mathrm{harm}}(\mathbb{A}_{r,1})$ for an annulus $\mathbb{A}_{r,1}$, by the First Anchor Lemma \ref{le:anchor_lemma_one} one can define a pairing
 \begin{equation} \label{eq:the_pairing} 
  \lim_{r \nearrow 1} \int_{|z|=r} f \alpha.  
 \end{equation}
 If $\alpha$ were smooth, we could identify this integral with 
 \[  \int_{\mathbb{S}^1} f \alpha. \]
 In the general case that $f$ is in $H^{1/2}(\mathbb{S}^1)$, it turns out that the pairing makes sense, and in fact all elements of $H^{-1/2}(\mathbb{S}^1)$ can be represented this way. The same idea works for the border of a Riemann surface, provided that we treat it as an analytic curve (see Remark \ref{re:Sobolev_space_uses_analytic_boundary}). 
 
 The remainder of this section is dedicated to filling in the details of this sketch. The payoff of this approach is that it makes it possible to use the machinery of CNT boundary values to solve the Dirichlet problem for one-forms with $H^{-1/2}$ boundary data. In this way one obtains a complete theory of the boundary values of $L^2$ harmonic one-forms for bordered surfaces.
  
 We begin by defining an equivalence relation, such that the equivalence classes represent the boundary values of the one-form.  Later we will see that each equivalence class can be identified with a unique element of $H^{-1/2}$, and vice-versa. 
 \begin{definition}[Equivalence relation for CNT Dirichlet boundary values of one-forms]  \label{de:form_Dirichlet_equivalence_relation}
 For collar neighbourhoods $A$ and $B$ of $\partial_k \riem$, 
 let $\alpha \in \mathcal{A}_{\mathrm{harm}}(A)$ and $\beta \in \mathcal{A}_{\mathrm{harm}}(B)$.  We say that $\alpha \sim \beta$
 if 
 \begin{enumerate} 
  \item there is a collar neighbourhood $U_k \subseteq A \cap B$ of $\partial_k \riem$ and a $\delta \in \mathcal{A}_{\mathrm{harm}}(U_k)$ such that $\alpha - \delta,\beta - \delta \in \mathcal{A}^{\mathrm{e}}_{\mathrm{harm}}(U_k)$; 
  \item if $f,g \in \mathcal{D}_{\mathrm{harm}}(U_k)$ are such that $df = \alpha - \delta$ and $dg = \beta - \delta$, then the CNT boundary values of $f-g$ are constant on $\partial_k \riem$ up to a null set.   
 \end{enumerate}
 \end{definition}
  In brief, $\alpha$ and $\beta$ are equivalent if their multi-valued primitives agree on the boundary up to integration constant. When the boundary curve is not clear from context, we will say ``$\alpha \sim \beta$ on $\partial_k \riem$''.
 
It turns out that if $\alpha \sim \beta$ via some $\delta$, then any one-form $\delta' \in \mathcal{A}_{\mathrm{harm}}(U')$ satisfying (1) also satisfies (2).  To see this, choose a collar neighbourhood $V \subset U \cap U'$, which exists by Proposition \ref{prop:collar_charts_intersection}.  Observe that 
 $\delta' - \delta = (\alpha - \delta) - (\alpha - \delta')$ has a primitive $h$ on $V$.  So if $f$ and $g$ are the primitives of $\alpha - \delta$ and $\beta - \delta$ respectively, then $f-h$ and $g-h$ are the unique primitives of $\alpha - \delta'$ and $\beta - \delta'$ up to constants.  But $(f-h)-(g-h) =f-g$ has constant CNT boundary values on $\partial_k \riem$ up to a null set, which proves the claim.  
 With this fact in hand, it is not hard to verify that $\sim$ is an equivalence relation. 
  
\begin{definition}\label{defn: Hprimes}[CNT Dirichlet boundary values for one-forms]
   Let $\mathcal{U}_k$ denote the collection of collar neighbourhoods of $\partial_k \riem$. Define
 \[  \gls{Dbvaluescomp} = \{ \alpha \in \mathcal{A}_{\mathrm{harm}}(U_k);\, U_k \in \mathcal{U}_k \}/\sim.         \]
 We also denote
 \[   \gls{Dbvalues} = \{ ([\alpha_1],\ldots,[\alpha_n]);\, \alpha_k \in \mathcal{H}'(\partial_k \riem) \}.     \] 
 If $\alpha \in \mathcal{A}_{\mathrm{harm}}(U)$ where $U$ contains a collar neighbourhood $U_k$ of each boundary, then we set
 \[ [\alpha] := ([\alpha_1],\ldots,[\alpha_n]),    \]
 where $\alpha_k = \left. \alpha \right|_{U_k}$.
 \end{definition}

 For fixed $k$, any equivalence class $[\beta] \in \mathcal{H}'(\partial_k \riem)$ has a well-defined boundary period.  Given a representative $\beta \in \mathcal{A}_{\mathrm{harm}}(U_k)$ for some collar neighbourhood $U_k$, let $c_k$ be a smooth closed curve in $U_k$ which is homotopic to $\partial_k \riem$, and define  
 \[  \int_{\partial_k \riem} [\beta] = \int_{c_k} \beta.     \]
 To see that this is well-defined, let $\beta' \in \mathcal{A}_{\mathrm{harm}}(U_k')$ be another representative of $[\beta]$ and $c_k'$ be another such curve.  By 
 Proposition \ref{prop:collar_charts_intersection} there is a canonical collar chart $\phi_{k,r}:U_{k,r} \rightarrow \mathbb{A}_{r,1}$ such that the inner boundary $\Gamma$ of $U_{k,r}$ is contained in $U_k \cap U_k'$ and $\phi_{k,r}$ extends analytically to $\Gamma$.  Since $\Gamma$ is isotopic to $\partial_k \riem$, it is isotopic to $c_k$ in $U_k$ and isotopic to $c_k'$ in $U_k'$.  Thus
 \[    \int_{c_k'} \beta' = \int_{\Gamma} \beta = \int_{c_k} \beta,    \]
 proving the claim.
 
 {It also follows directly from the definition of the equivalence classes that $\mathcal{H}'(\partial \riem)$ is conformally invariant in the following sense.
 \begin{proposition} \label{pr:Hprime_conformal_invariance}
  Let $\riem_1$ and $\riem_2$ be bordered surfaces and fix borders $\partial_{k_1} \riem_1$ and $\partial_{k_2} \riem_2$ which are homeomorphic to $\mathbb{S}^1$. Let $U$ and $V$ be collar neighbourhoods of $\partial_{k_1} \riem_1$ and $\partial_{k_2} \riem_2$ respectively, and let $f:U \rightarrow V$ be a conformal map. Then for any two representatives $\alpha$ and $\beta$ of $[\alpha] \in \mathcal{H}'(\partial_{k_2} \riem_2)$ we have 
  \[  [f^* \alpha] = [f^* \beta].  \]
  In particular, we have a well-defined pull-back map 
  \begin{align*}
      f^*:\mathcal{H}'(\partial_{k_2} \riem_{2}) & \rightarrow \mathcal{H}'(\partial_{k_1} \riem_{1}) \\
      [\alpha] & \mapsto [f^*\alpha].
  \end{align*}
 \end{proposition}
 }
  
{We will require the following elementary lemma, in order to define a norm on $\mathcal{H}'(\partial_k \riem)$. }
\begin{lemma} \label{le:nice_Hprime_representative_disk}
 Let $[\alpha]  \in \mathcal{H}'(\mathbb{S}^1)$ treated as the border of a subset of the disk $\mathbb{D}$. Then $\alpha$ has a {unique} representative of the form 
 \[  \alpha = f(z) dz + \overline{g(z)} d \bar{z} + \delta \]
 where $f(z),g(z) \in \mathcal{D}(\disk)$ have the form
 \[  f(z) = \sum_{n=2}^\infty f_n \, z^n, \  \ \  \overline{g(z)} = \sum_{n=2}^\infty \overline{g_n} \, \overline{z}^n, \]
 and 
 \[  \delta = \frac{a}{4 \pi i} \left( \frac{dz}{z} - \frac{d\overline{z}}{\overline{z}} \right).  \]
 for some constant $a \in \mathbb{C}$
\end{lemma}

\begin{proof}
 This follows easily from the definition of $\mathcal{H}'(\Gamma)$ and the existence of solutions to the Dirichlet problem on the disk, after observing that 
 \[  b\left( \frac{d z}{z} + \frac{d\bar{z}}{\bar{z}} \right) \]
 is equivalent to $0$ in $\mathcal{H}'(\mathbb{S}^1)$ for any $b \in \mathbb{C}$.  Uniqueness is obvious.
\end{proof}
 
 {This allows us to define a norm on $\mathcal{H}'(\mathbb{S}^1)$. Given any $[\alpha]$ let 
 \[  \alpha = f(z) dz + \overline{g(z)} d \overline{z} + \frac{\lambda}{4 \pi i} \left( \frac{dz}{z} - \frac{d\bar{z}}{\bar{z}} \right)  \] 
 be the representative given by Lemma \ref{le:nice_Hprime_representative_disk}. We define 
 \[  \| [\alpha] \|^2_{\mathcal{H}'(\mathbb{S}^1)} = \| f(z) dz + \overline{g(z)} d\bar{z} \|^2_{\mathcal{A}_{\mathrm{harm}}(\disk)} + |\lambda|^2. \]
 
 For any boundary curve $\partial_k \riem$, we define a norm on $\mathcal{H}'(\partial_k \riem)$ as follows. Choose a collar chart $\phi:U \rightarrow \mathbb{A}_{r,1}$ of $\partial_k \riem$. Implicitly using Proposition \ref{pr:Hprime_conformal_invariance}, we define 
 \begin{equation} \label{eq:Hprime_norm_using_collar}
  \| [\alpha] \|_{\mathcal{H}'(\partial_k \riem)} = \| \phi^*[ \alpha] \|_{\mathcal{H}'(\mathbb{S}^1)}.   
 \end{equation}
 This norm of course depends on the collar chart. However, we will see ahead that different collar charts induce equivalent norms. 
 
 Given a collection $\phi=(\phi_1,\ldots,\phi_n)$ of collar charts of $\partial_1 \riem,\ldots,\partial_n \riem$, we define a norm on $\mathcal{H}'(\partial \riem)$ by
 \begin{equation} \label{eq:Hprime_norm_full_boundary_chart}
     \| ([\alpha_1],\ldots,[\alpha_n]) \|^2_{\mathcal{H}'(\partial \riem)} = \| [\alpha_1] \|^2_{\mathcal{H}'(\partial_1 \riem)} + \cdots 
     \| [\alpha_n] \|^2_{\mathcal{H}'(\partial_n \riem)}. 
 \end{equation}
 Again, this norm depends on the collection of collar charts $\phi$. 
 }

 {
Regarding the norm defined above, we state the following lemma which will be useful in connection to Theorem \ref{th:Honehalf_reformulation} and Lemma \ref{le:Bphi_bounded} ahead.

\begin{lemma} \label{le:homogen_equiv}
 Let $\phi:U \rightarrow \mathbb{A}_{r,1}$ be a collar chart defined near $\partial_k \riem$ for fixed $k$. Then 
 \[  h \mapsto h \circ \phi  \]
 is a bounded isomorphism from $H^{1/2}(\mathbb{S}^1)$ to $H^{1/2}(\partial_k \riem)$. 
\end{lemma} 
\begin{proof}  
  By Carath\'eodory's theorem and the Schwarz reflection principle, $\phi$ extends to a conformal map from a doubly connected neighbourhood $V$ of $\partial_k \riem$ to the annulus $\mathbb{A}_{r,1/r}$. The restriction of $\phi$ to $\partial_k \riem$ is thus an analytic diffeomorphism between the compact manifolds $\partial_k \riem$ and $\mathbb{S}^1$, from which and Lemma \ref{lem:invariance of sobolev under diffeo} the claim follows.  
\end{proof}

\begin{lemma}\label{homogen equivalence}
 Let $\varphi: C_1\to C_2$ be a quasisymmetric mapping between the closed smooth curves $C_j$, $j=1,2$. Then $\varphi$ induces an equivalence between  
 $\dot{H}^{\frac{1}{2}}(C_1)$ and $\dot{H}^{\frac{1}{2}}(C_2)$, i.e.
 
  $$\Vert f\Vert_{\dot{H}^{\frac{1}{2}}(C_2)}\approx \Vert f\circ \varphi\Vert_{\dot{H}^{\frac{1}{2}}(C_1)}.$$
  As a consequence, we  have that if $\phi_k$ is a quasisymmetric map from $ \mathbb{S}^1 \to \partial_k \Sigma $ then 
  $$\Vert f\Vert_{\dot{H}^{\frac{1}{2}}(\partial_k \Sigma)}\approx \Vert f\circ \varphi_k\Vert_{\dot{H}^{\frac{1}{2}}(\mathbb{S}^1)}.$$

 \end{lemma}
\begin{proof}
 This is just a special case of Theorem 5.1 in \cite{Koskela}.
\end{proof}

}
 
  Let $\riem$ be a bordered Riemann surface of type $(g,n)$.  Fixing $k$, we can define a pairing between elements of $H^{1/2}(\partial_k \riem)$ and $\mathcal{H}'(\partial_k \riem)$ as follows.  Given $[\alpha] \in \mathcal{H}'(\partial_k \riem)$ and $h \in H^{1/2}(\partial_k \riem)$, let $\alpha \in \mathcal{A}(U)$  be a representative of $[\alpha]$ for a collar neighbourhood $U$ of $\partial_k \riem$, and let $H \in \mathcal{D}_{\mathrm{harm}}(U')$ have CNT boundary values $h$. There exists at least one such $H$, by solving the Dirichlet problem on $\riem$ with $H=h$ on $\partial_k \riem$ and $0$ on the other boundary curves.
  By Proposition \ref{prop:collar_charts_intersection} we can choose a common collar neighbourhood $V \subset U \cap U'$.    Define
  \begin{equation} \label{eq:associated_H_minusonehalf_pairing}
   L_{[\alpha]} (h) = \int_{\partial_k \riem} [H\alpha] = \lim_{\epsilon \searrow 0} \int_{\Gamma_\epsilon} 
     H \alpha 
  \end{equation}
  for limiting curves $\Gamma_\epsilon$ approaching $\partial_k \riem$.  
  We have already shown that for fixed $H$ this is well-defined. By the second anchor Lemma \ref{le:anchor_lemma_two} for any two $H_m \in \mathcal{D}_{\mathrm{harm}}(U_m)$ on collar neighbourhoods $U_m$ for $m=1,2$ with the same boundary values, we have for fixed $\alpha$
  \[   \int_{\partial_k \riem} H_1 \alpha = \int_{\partial_k \riem} H_2 \alpha. \]
  Thus $L_{[\alpha]}$ is well-defined.\\ 
  
  {The pairing is invariant under pull-back. 
  \begin{proposition} \label{pr:pullback_pairing_confinv}
   Let $\riem_1$ and $\riem_2$ be bordered surfaces and fix borders $\partial_{k_1} \riem_1$ and $\partial_{k_2} \riem_2$ which are homeomorphic to $\mathbb{S}^1$. Let $U$ and $V$ be collar neighbourhoods of $\partial_{k_1} \riem_1$ and $\partial_{k_2} \riem_2$ respectively, and let $f:U \rightarrow V$ be a conformal map.  
   For any $H \in H^{1/2}(\partial_{k_2}\riem_{2})$, 
   \[  \int_{\partial_{k_2} \riem_2} [\alpha] H = \int_{\partial_{k_1} \riem_1} f^* [\alpha] \, H \circ f.  \]
  \end{proposition}
  \begin{proof}
    Let $\phi:U_2 \rightarrow \mathbb{A}_{r,1}$ be a collar chart of $\partial_{k_2} \riem_2$.  Then $\phi \circ f:U_1 \rightarrow \mathbb{A}_{r,1}$ is a collar chart of $\partial_{k_1} \riem$, shrinking $U_2$ if necessary.  Let $\Gamma^2_r$ be the limiting curves $\phi^{-1}(|z|=r)$ induced by $\phi$, and similarly $\Gamma^1_r$ by $\phi \circ f$ (so that $f(\Gamma^1_r) = \Gamma^2_r)$.  
    
    Now choose a representative $\alpha$ of $[\alpha]$ and let $h$ be an extension of $H$ to a collar neighbourhood of $\partial_{k_2} \riem_2$.
    Then 
    by the Anchor Lemmas \ref{le:anchor_lemma_one} and \ref{le:anchor_lemma_two} and a change of variables, we have
    \begin{align*}
        \int_{\partial_{k_2} \riem_2} [\alpha] H & = \lim_{r \nearrow 1} \int_{\Gamma^2_r} \alpha h = \lim_{r \nearrow 1} \int_{\Gamma^1_r}
        f^* \alpha \, h \circ f \\
        & = \int_{\partial_{k_1} \riem_1} f^* [\alpha] H \circ f
    \end{align*}
    where in the last equality we have also used Proposition \ref{pr:Hprime_conformal_invariance}. 
  \end{proof}
  }
 
 \begin{theorem}  \label{th:Honehalf_reformulation}
  Let $\riem$ be a bordered Riemann surface of type $(g,n)$.  For any fixed $k  \in \{1,\ldots,n \}$, the bijection 
  \begin{align*}
      \mathcal{H}'(\partial_k \riem) & \rightarrow H^{-1/2}(\partial_k \riem) \\
      [\alpha] & \mapsto L_{[\alpha]} 
  \end{align*}
  is a bounded isomorphism.
 \end{theorem} 
 The proof of the theorem will require the following result. 
 \begin{theorem} \label{th:limiting_dual_expression}
  Let $L$ be in $H^{-1/2}(\mathbb{S}^1)$.  Then there is an $\alpha \in \mathcal{A}_{\mathrm{harm}}(\mathbb{A}_{r,1})$ such that
  \begin{equation} \label{eq:limiting_dual}
     L(f) = \lim_{s \nearrow 1} \int_{|z|=s}  f \alpha.   
  \end{equation}
 \end{theorem}
 \begin{proof}
 Since $H^{1/2}(\mathbb{S}^1)$ is a Hilbert space, Riesz representation theorem yields that there exists a unique $F\in H^{1/2}(\mathbb{S}^1)$ such that, if   $f=\sum_{n=-\infty}^{\infty} \hat{f}(n) e^{in\theta}$ and $F=\sum_{n=-\infty}^{\infty}\hat{F}(n) e^{in\theta}$ then
 $$
  L(f)= \langle f, F \rangle_{H^{1/2}(\mathbb{S}^1)}=\sum_{n=-\infty}^{\infty}\left(1+|n|^{2}\right)^{1/2} \hat{f}(n)\overline{\hat{F}(n)}
.$$
Moreover $\Vert L\Vert_{H^{-1/2}(\mathbb{S}^1)}= \Vert F\Vert_{H^{1/2}(\mathbb{S}^1)}.$
Now by Parseval's formula we also have
\begin{equation}
   \sum_{n=-\infty}^{\infty}\left(1+|n|^{2}\right)^{1/2} \hat{f}(n)\overline{\hat{F}(n)}=\frac{1}{2\pi} \int_{0}^{2\pi} f( e^{i\theta}) ((1-\partial^2_{\theta})^{1/2}F)(e^{i\theta})\, d\theta.
\end{equation}
This and the requirement of harmonicity of $\alpha$ suggests that the desired $\alpha$ should be taken as the Poisson extension of $((1-\partial^2_{\theta})^{1/2}F)(e^{i\theta})$ (i.e. its convolution with the Poisson kernel of the unit disk), which for $ s\leq 1$ yields that \begin{equation}
  \alpha(se^{i\theta})=\sum_{n=-\infty}^{\infty}\left(1+|n|^{2}\right)^{1/2}\hat{F}(n)\, s^{|n|} \, e^{in\theta}.  
\end{equation}  
 
Moreover, a  calculation reveals that for $0<r<1$ one has

\begin{equation}
    \Vert \alpha\Vert^{2}_{L^2 (\mathbb{A}_{r,1})}= \pi \sum_{n=-\infty}^{\infty}(1-r^{2|n|+2}) \frac{1+|n|^{2}}{1+|n|} |\hat{F}(n)|^2 \lesssim \sum_{n=-\infty}^{\infty} (1+|n|^{2})^{1/2} |\hat{F}(n)|^2 <\infty
\end{equation}
since $F\in H^{1/2}(\mathbb{S}^1).$ Therefore $\alpha \in \mathcal{A}_{\mathrm{harm}}(\mathbb{A}_{r,1})$, as desired.
\end{proof}

We now return to the proof of Theorem \ref{th:Honehalf_reformulation}. 
\begin{proof}[{\emph{\bf{Proof of Theorem \ref{th:Honehalf_reformulation}}}}]
 {
 Let $\phi:U \rightarrow \mathbb{A}_{r,1}$ be a collar chart. For any $h \in H^{1/2}(\partial_k \riem)$, recall that we have
 \begin{equation} \label{eq:pairing_pullback_invariance}
   \int_{\partial_k \riem} [\alpha]  h = \int_{\mathbb{S}^1} \phi^*[\alpha] h \circ \phi   
 \end{equation}
 by Proposition \ref{pr:pullback_pairing_confinv}. Thus, by Lemma \ref{le:homogen_equiv} and recalling the definition (\ref{eq:Hprime_norm_using_collar}) of the chart-dependent norm, it is enough to prove the claim on $\mathcal{H}'(\mathbb{S}^1)$.  }

 We first need to show that for any given 
 $[\alpha] \in \mathcal{H}'(\mathbb{S}^1)$, the linear functional   $L_{[\alpha]}$ is bounded, and hence in $H^{-1/2}(\mathbb{S}^1)$. To see this, let $\alpha$ be a representative as in Lemma \ref{le:nice_Hprime_representative_disk}, so that $\alpha- \delta$ is exact where 
{ {\[ \delta = \frac{a}{4\pi i} \left( \frac{dz}{z} - \frac{d\bar{z}}{\bar{z}} \right) \]}}
 for some $a \in \mathbb{C}$.
 Observe that 
 \[  a = \int_{\mathbb{S}^1} [\alpha].  \]
 
 For any $h \in H^{1/2}(\mathbb{S}^1)$ let $H$ be its unique harmonic extension in $\mathcal{D}_{\mathrm{harm}}(\disk)$, and write $H(z)=H_1(z) + H(0)$ where $H_1(0)=0$.
 Recall that 
 \[  \| h \|^2_{\mathcal{H}^{1/2}(\mathbb{S}^1)} = |H(0)|^2 + \| dH \|^2_{\mathcal{D}_{\mathrm{harm}}(\disk)}. \]
 {By the mean-value theorem for the harmonic function $H_1$, one has}
 \begin{align*}
     \lim_{r \nearrow 1} \int_{|z|=r} \alpha(z) H(z) & = 
     \lim_{r \nearrow 1} \int_{|z|=r}  H(0) \alpha(z) +  \lim_{r \nearrow 1} \int_{|z|=r} \alpha H_1(z) \\
     & = \lim_{r \nearrow 1} \int_{|z|=r}  H(0) \delta(z)+  \lim_{r \nearrow 1}\int_{|z|=r} (\alpha - \delta)  H_1(z) \\
     & = {{a}}H(0) - \iint_{\mathbb{D}} (\alpha- \delta) \wedge_w  dH_1(z), 
 \end{align*}
 {so the  Cauchy-Schwarz inequality and Corollary \ref{good sobolev cor} yield}
 \begin{align*}
   \left| \lim_{r \nearrow 1} \int_{|z|=r} \alpha(z) H(z) \right| 
    & \leq  | a H(0) | + \| \alpha- \delta \|_{\mathcal{A}_{\mathrm{harm}}(\disk)} \| H \|_{\mathcal{D}_{\mathrm{harm}}(\disk)} \\ & {\leq C \| \alpha  \|_{\mathcal{H}'(\mathbb{S}^1)} \| h \|_{H^{1/2}(\mathbb{S}^1)} },
 \end{align*}
for some constant $C$. 
 Thus $L_{[\alpha]} \in H^{-1/2}(\mathbb{S}^1)$. {The same computation also shows that the map $[\alpha] \rightarrow L_{[\alpha]}$ is bounded. }

The map $[\alpha] \rightarrow L_{[\alpha]}$ is surjective by Theorem \ref{th:Honehalf_reformulation}, so it remains to show that it is injective.
Assume that $L_{[\alpha]} h =0$ for all $h \in H^{1/2}(\mathbb{S}^1)$. Let $\alpha,\delta,f,\overline{g}$ be as in Lemma \ref{le:nice_Hprime_representative_disk}. Since 
\[  0 = L_{[\alpha]} (1) = \lim_{r \nearrow 1} \int_{|z|=r}  \alpha = b \]
we must have $a=0$. Similarly using $0=L_{[\alpha]}(z^n)=L_{[\alpha]}(\bar{z}^n)$ for all $n=1,\ldots,\infty$ shows that $f=\overline{g}=0$, so $\alpha=0$. Thus $[\alpha]=0$.
\end{proof}

{This shows that different collar charts induce equivalent norms, as promised.
\begin{corollary}
 For any fixed $k$, and any pair of collar charts $\phi,\psi$ near $\partial_k \riem$, the norm induced on $\mathcal{H}'(\partial_k \riem)$ by $\phi$ and $\psi$ are equivalent.
 
 Similarly, for any two collections of collar charts $\phi=(\phi_1,\ldots,\phi_n)$ and $\psi=(\psi_1,\ldots,\psi_n)$ of the boundaries $\partial_1 \riem,\ldots,\partial_n \riem$, the norms induced on $\mathcal{H}'(\partial \riem)$ by $\phi$ and $\psi$ are equivalent.
\end{corollary}
\begin{proof}
  It suffices to establish the case of one boundary curve. Fixing a collar chart $\phi$ by Theorem \ref{th:Honehalf_reformulation} the map $[\alpha] \rightarrow L_{[\alpha]}$ is a bounded isomoprhism between $\mathcal{H}'(\partial_k \riem)$ and $H^{-1/2}(\partial_k \riem)$ with respect to the norm on $\mathcal{H}'(\partial_k \riem)$ induced by this chart. Since this is true for any collar chart, the norms induced by different collar charts must be equivalent. 
\end{proof}

}

Finally, we observe that harmonic measures generate the zero equivalence class of $\mathcal{H}'(\partial_k \riem)$ for any $k=1,\ldots,n$. 
\begin{proposition} \label{pr:harmonic_measures_zero}
  For any $d\omega \in \mathcal{A}_{\mathrm{hm}}(\riem)$ we have 
  \[  [d\omega]=0. \]
\end{proposition}
\begin{proof}
 By Theorem \ref{th:Honehalf_reformulation} it suffices to show that $L_{[d\omega]}=0$.  Since $L_{[d\omega]}$ is bounded, it suffices to show that it is zero on the dense subset $H^1_{\mathrm{conf}}(U)$ where $U$ is a doubly connected neighbourhood of $\partial_k \riem$ in the double of $\riem$.  Observing that $d\omega$ has an extension to the double, for any such $h \in H^1_{\mathrm{conf}}(U)$ we obtain
 \[ L_{[d\omega]} (h) = \int_{\partial_k \riem} h\, d\omega    \]
 where the integral on the right hand side can be evaluated directly on the curve $\partial_k \riem$.  Since $d\omega =0$ for vectors tangent to $\partial_k \riem$, this completes the proof.
\end{proof}

{ A model of the homogeneous space $\dot{H}^{-1/2}(\partial_k \riem)$ can also be given in terms of one-forms.  Consider the Sobolev space $\dot{H}^{1/2}(\partial_k \riem)$ to consist of functions modulo constants. Let $\dot{H}^{-1/2}(\partial_k \riem)$ denote its dual space.  

Observe that if $[\alpha]=[\beta]$ in $\mathcal{H}'(\partial_k \riem)$, and $\alpha$ is exact, then $\beta$ is also exact. Thus we may define
\[ \gls{BVexactcomp} = \{ [\alpha] \in \mathcal{H}'(\partial_k \riem) : [\alpha] \text{  has an exact representative} \}.    \]
We can similarly define $\gls{BVexact}$ as above. 

It is easy to see that for $[\alpha] \in \dot{H}'(\partial_k \riem)$, for any constant function $c \in H^{1/2}(\partial_k \riem)$ we have
\[  L_{[\alpha]} c = 0.  \]
Thus, $[\alpha]$ generates a well-defined functional on $H^{1/2}(\partial_k \riem)$.  We can define 

It is easy to see that we have 
 \begin{theorem}  \label{th:dot_Honehalf_reformulation}
  Let $\riem$ be a bordered Riemann surface of type $(g,n)$.  For any fixed $k  \in \{1,\ldots,n \}$, the bijection 
  \begin{align*}
      \dot{\mathcal{H}}'(\partial_k \riem) & \rightarrow \dot{H}^{-1/2}(\partial_k \riem) \\
      [\alpha] & \mapsto L_{[\alpha]} 
  \end{align*}
  is a bounded isomorphism.
 \end{theorem}  
}
\end{subsection}
\begin{subsection}{Formulation and solution of the CNT Dirichlet problem for $L^2$ one-forms}   \label{se:Dirichlet_full_solution}

 We can now state the general Dirichlet problem for $L^2$ one-forms. 
 \begin{definition}[CNT Dirichlet data for one-forms]
 By CNT Dirichlet data for one-forms, we mean $([\beta],\rho,\sigma)$ where 
 \begin{enumerate}
     \item $[\beta] = ([\beta_1],\ldots,[\beta_n]) \in \mathcal{H}'(\partial \riem)$ 
     such that 
     \[   \int_{\partial_1\riem} [\beta_1] + \cdots + \int_{\partial_n \riem} [\beta_n] = 0; \]
     \item $\rho=(\rho_1,\ldots,\rho_n)\in \mathbb{C}^n$ satisfying
     \[  \rho_1 + \cdots + \rho_n =0;\]
     and
     \item $\sigma= (\sigma_1,\ldots,\sigma_{2g}) \in \mathbb{C}^{2g}$.  
 \end{enumerate} 
 \end{definition}
 The Dirichlet problem for this data is as follows.
 \begin{definition}[CNT Dirichlet problem for one-forms]\label{defn:Cnt dirichlet data}
  We say that a harmonic one-form $\alpha$ on $\riem$  solves the CNT Dirichlet problem with data  $([\beta],\rho,\sigma)$,  if
  $([\beta],\rho,\sigma)$ is CNT Dirichlet data and 
 \begin{enumerate}
     \item $[\alpha] = ([\beta_1],\ldots,[\beta_n])$;
     \item for all $k =1,\ldots,n$  
     \[  \int_{\partial_k \riem} \ast \alpha = \rho_k;   \]
     and 
     \item for all $k=1,\ldots,2g$
     \[  \int_{\gamma_k} \alpha : = \sigma_k.  \]
 \end{enumerate} 
 \end{definition}
 The CNT Dirichlet problem has a solution which depends continuously on the data.\\

  \begin{theorem}[Well-posedness of Dirichlet's problem for CNT data] \label{th:CNT_Dirichlet_problem_oneforms} For \emph{CNT} Dirichlet data $([\beta],\rho,\sigma)$ there exists a unique $\alpha \in  \mathcal{A}_{\mathrm{harm}}(\riem)$ which solves the Dirichlet problem. Moreover, the operator 
  \[   \gls{Dir}_{\partial \riem,\riem}: \mathcal{H}'(\partial \riem) \oplus \mathbb{C}^{2g+n-1} \rightarrow \mathcal{A}_{\mathrm{harm}}(\riem)      \]
  taking $([\beta],\rho,\sigma)$ to the solution
  is bounded.  Here the entries of $\mathbb{C}^{2g+n-1}$ are \[  (\rho_1,\ldots,\rho_{n-1},\sigma_1,\ldots,\sigma_{2g}).   \]  
 \end{theorem}

   
 
Before we prove of this result,  we will need some preparations. To that end,
 fix $k \in \{1,\ldots,n\}$ and let 
 \[  \mathcal{H}_{\mathrm{e}}'(\partial_k \riem) = \left\{ [\alpha]
 \in \mathcal{H}'(\partial_k \riem) : \int_{\partial_k \riem} [\alpha] =0 \right\}.  \] 
 
 Let $\phi:U \rightarrow \mathbb{A}_{r,1}$ be a collar chart defined near $\partial_k \riem$. Define a linear map
 \[  \gls{Bphi}:\mathcal{H}_{\mathrm{e}}'(\partial_k \riem) \rightarrow H^1_{\mathrm{conf}}(\disk)  \]
 as follows.  Given $[\alpha] \in \mathcal{H}_{\mathrm{e}}'(\partial_k \riem)$, choose a representative $\alpha$ of $[\alpha]$ and 
 let $h \in H^1_{\mathrm{conf}}(U)$ be such that $h = d\alpha$  (shrinking $U$ if necessary using Proposition \ref{prop:collar_charts_intersection}).  
 Now let $H= \mathbf{G}_{\mathbb{A}_{r,1},\disk} h \circ \phi$ and
 observe that
 $H \in H^1_{\mathrm{conf}}(\disk)$ is the unique harmonic map whose CNT boundary values agree with those of $h \circ \phi$ up to a constant. We then impose the integral condition
 \begin{equation} \label{eq:zero_mode_condition}
    \int_{\mathbb{S}^1} H(e^{i\theta})\, d\theta = 0  
 \end{equation}
 on the boundary and set
 \[  \mathbf{B}(\phi) [\alpha] =H.     \]
 Equivalently, we may say that $H$ is the unique harmonic function on $\disk$ whose boundary values are $\left. h \circ \phi \right|_{\mathbb{S}^1}$ up to a constant, which satisfies \eqref{eq:zero_mode_condition}.
 \begin{lemma} \label{le:Bphi_bounded}
  For a collar chart $\phi:U \rightarrow \mathbb{A}_{r,1}$ near $\partial_k \riem$, $\mathbf{B}(\phi)$ is bounded.
 \end{lemma}
 \begin{proof}
  Treating $\partial_k \riem$ as an analytic curve in the double, observe that $\phi$ has a biholomorphic extension taking a doubly-connected neighbourhood of $\partial_k \riem$ to $\mathbb{A}_{r,1/r}$ and $h \mapsto h \circ \phi$ is a bijection which is bounded from $\dot{H}^{1/2}(\partial_k \riem)$ to  $\dot{H}^{1/2}(\mathbb{S}^1)$, by Lemma \ref{homogen equivalence}.  Furthermore, since the extension of $h \circ \phi$ from $\dot{H}^{1/2}(\mathbb{S}^1)$ to $\mathcal{D}_{\text{harm}}(\disk)$ with any choice of constant is bounded with respect to the Dirichlet norm, and since condition \eqref{eq:zero_mode_condition} yields that
  \[   \| H \|_{H^1_{\mathrm{conf}}(\disk)} \approx \| H \|_{\mathcal{D}_{\text{harm}}(\disk)}, \]
 one obtains the desired boundedness result.
 \end{proof}

 \bigskip
 \begin{proof}
 [\emph{\bf{Proof of Theorem \ref{th:CNT_Dirichlet_problem_oneforms}}}]   
 
First, we show that the exact solution to the Dirichlet problem depends continuously on the data.  Let 
  \[ \mathcal{H}'_\mathrm{e}(\partial\riem) = \oplus_{k=1}^n \mathcal{H}'_\mathrm{e}(\partial_k \riem).  \]
  The solution to the boundary value problem or exact forms with data in $\mathcal{H}'_{\mathrm{e}}(\partial \riem)$ is as follows: given $([\alpha],\rho_1,\ldots,\rho_{n-1}) \in \mathcal{H}'_{\mathrm{e}}(\partial \riem) \oplus \mathbb{C}^{n-1}$ we want a one-form $\beta \in \mathcal{A}_{\mathrm{harm}}(\riem)$ such that $[\beta] = [\alpha]$
  and
  \begin{equation} \label{eq:temp_goal1}
    \int_{\partial_k \riem} \ast \beta = \rho_k,  \ \ \ k=1,\ldots,n. 
  \end{equation}
  We will define a map 
  \[  \gls{E}:\mathcal{H}'_\mathrm{e}(\partial \riem) \oplus \mathbb{C}^{n-1} \rightarrow \mathcal{A}_{\mathrm{harm}}(\riem)   \]
  taking data to the solution as follows. 
  First, for $k=1,\ldots,n$ let $\psi_k:U_k \rightarrow \mathbb{A}_{r_k,1}$ be the collar chart constructed in Lemma \ref{le:harmonic_measure_collar_chart}. By the same lemma, for any $[\alpha_k] \in \mathcal{H}'_e(\partial_k \riem)$
  \begin{equation} \label{eq:modes_temp}
    \int_{\partial_k \riem} \mathbf{C}_{\phi_k^{-1}} \mathbf{B}(\psi_k) [\alpha_k] \ast d \omega_k = \int_{\mathbb{S}^1} \mathbf{B}(\psi_k)[\alpha_k] \,d \theta=0.
  \end{equation}
  If we set $\psi=(\psi_1,\ldots,\psi_n)$ and  
  define 
  \[  \mathbf{B}(\psi) = \oplus_{k=1}^n \mathbf{B}(\psi_k) :\mathcal{H}_\mathrm{e}'(\partial \riem)  \rightarrow \oplus_{k=1}^n H^1_{\mathrm{conf}}(\mathbb{D}), \]
  by Lemma \ref{le:Bphi_bounded} this is bounded. Let $\mathbf{C}_{\phi^{-1}}$ be as in Lemma \ref{le:for_density_factorization}. 
  Let $\mathbf{R}^{\mathrm{h}}_{\disk^n,\mathbb{A}^n}= \oplus_{k=1}^n \mathbf{R}^{\mathrm{h}}_{\disk,\mathbb{A}_{r_{k},1}}$.
  Finally define 
  \begin{equation} \label{eq:E_definition}
   \mathbf{E}([\alpha],\rho_1,\ldots,\rho_{n-1}) = d \mathbf{G}_{U,\riem} \mathbf{C}_{\phi^{-1}} \mathbf{R}^{\mathrm{h}}_{\disk^n,\mathbb{A}^n} \mathbf{B}(\phi)[\alpha]+ \sum_{m=1}^{n-1}   b_m d\omega_m 
  \end{equation} 
  where the $c_k$ are defined by 
  \[   c_k = \int_{\partial_k \riem} \ast dH \ \ \  \mathrm{with} \ \ H =  \mathbf{G}_{U,\riem} \mathbf{C}_{\phi^{-1}} \mathbf{R}^{\mathrm{h}}_{\disk^n,\mathbb{A}^n} \mathbf{B}(\phi)[\alpha]   \]
  and the $b_k$ are defined by 
  \[ \rho_k -c_k = \sum_{m=1}^{n-1} \Pi_{km} b_m , \]
  with the help of Theorem \ref{th:period_matrix_invertible}.  
  
  We show that $\mathbf{E}([\alpha],(\rho_1,\ldots,\rho_{n-1})$ solves the boundary value problem.  
  By construction $\beta = \mathbf{E} ([\alpha],(\rho_1,\ldots,\rho_{n-1}))$ satisfies 
  \[ [\beta]=[\alpha]  \]
  since $[d\omega_k]=0$ for all $k=1,\ldots,n$ by Proposition \ref{pr:harmonic_measures_zero}.  To see that \eqref{eq:temp_goal1} is satisfied, we set $\beta = \mathbf{E}([\alpha],(\rho_1,\ldots,\rho_{n-1}))$ and compute
  \begin{align*}
      \int_{\partial_k \riem} \ast \beta & = c_k +  \sum_{m=1}^{n-1}   b_m  \int_{\partial_k \riem}  \ast d \omega_m  = c_k +\sum_{m=1}^{n-1} b_m \Pi_{km}    \\ 
      & =  \rho_k.
  \end{align*}
  
  Finally we show that $\mathbf{E}$ is bounded. The boundedness of the first term follows from 
  Theorem \ref{th:bounce_bounded}, Lemma \ref{le:Bphi_bounded}, and Lemma \ref{le:for_density_factorization}.  To bound the second term, observe that
  \begin{align*}
    c_k & = \int_{\partial_k \riem} \ast dH = \int_{\partial \riem} \omega_k \ast dH \\
    & = \iint_{\riem}  d \omega_k \wedge \ast dH,
  \end{align*}
  so 
  \[ \| (c_1,\cdots,c_{n-1}) \|_{\mathbb{C}^{n-1}} \leq \| H \|_{H^1_{\mathrm{conf}}(\riem)} \sup_{k=1,\ldots,n} \| d\omega_k \|_{\mathcal{A}_{\mathrm{harm}}(\riem)}.  \]
  This together with the facts that $H$ is bounded by the data, and that $\Pi$ is a finite matrix and therefore bounded,
  proves the claim.

  The remainder of the proof takes into account the cohomological data. 
  We are given an arbitrary $([\beta],\rho,\sigma) \in \mathcal{H}'(\partial \riem) \oplus \mathbb{C}^{2g+n-1}$.
  Setting 
  \begin{equation} \label{eq:Dirichlet_theorem_temp1}
     \lambda_k = \int_{\partial_k \riem} [\beta_k] 
  \end{equation}
  for $k = 1,\ldots,n$,  by Corollary \ref{co:boundary_periods_specified_starmeasure} there is a
  $\delta \in \ast \mathcal{A}_{\mathrm{hm}}(\riem)$ such that 
  \begin{equation} \label{eq:Dirichlet_theorem_temp2}
    \int_{\partial_k \riem} \delta = \lambda_k   
  \end{equation}
  for every $k$.
  Furthermore, there is a unique harmonic one-form $\eta$ in the span of $\{ \varepsilon_1,\ldots,\varepsilon_{2g} \}$ such that 
  \begin{equation} \label{eq:Dirichlet_theorem_temp3}
    \int_{\gamma_j} \eta = \sigma_j - \int_{\gamma_j} \delta    
  \end{equation}
  for $j=1,\ldots,2g$.  We also have by definition of $\varepsilon_k$ that
  \begin{equation} \label{eq:Dirichlet_theorem_temp4} 
    \int_{\partial_k \riem} \eta = 0,
  \end{equation}
  for $k=1,\ldots,n$.  
  Thus $[\beta - \delta -\eta] \in \mathcal{H}'_\mathrm{e}(\partial \riem)$. 
  
  We will require several estimates; the reader should keep in mind that they are elementary due to the fact that only finite-dimensional spaces are involved.
  Since $\delta$ is in the span of the finite-dimensional space $\ast \mathcal{A}_{\mathrm{hm}}$, and uniquely determined by $\lambda=(\lambda_1,\ldots,\lambda_{n-1}),$ we have that 
  \begin{equation} \label{eq:temp_continuous_depend_1} 
   \| [\delta] \|_{\mathcal{H}'(\partial \riem)} \leq C \| \lambda \|_{\mathbb{C}^{n-1}} \leq C \| [\beta] \|_{\mathcal{H}'(\partial \riem) \oplus  \mathbb{C}^{2g+n-1}}.  
  \end{equation} 
  Similarly 
 \begin{equation} \label{eq:temp_continuous_depend_2} 
   \| \delta \|_{\mathcal{A}_{\mathrm{harm}}(\riem)} \leq C \| \lambda \|_{\mathbb{C}^{n-1}} \leq C \| [\beta] \|_{\mathcal{H}'(\partial \riem) \oplus  \mathbb{C}^{2g+n-1}}.  
  \end{equation} 
  If desired, an explicit estimate could be obtained from the supremum over $k=1,\ldots,n-1$ of the norms of $\ast d \omega_k$, but this won't be needed. 
  
  Similarly, since the span of $\{ \varepsilon_1,\ldots, \varepsilon_{2g} \}$ is finite-dimensional, the dependence of $\eta$ on the data is continuous. Observe that 
  \[  e_j = \int_{\gamma_j} \delta, \ \ \ j =1,\ldots,2g \]
  depend linearly on $\delta$ and hence continuously on $\| [\beta] \|_{\mathcal{H}'(\partial \riem) \oplus  \mathbb{C}^{2g+n-1}}$. 
  Now by the definition \eqref{eq:Dirichlet_theorem_temp3} of $\eta$, 
  using the fact that $\eta$ is restricted to a finite-dimensional space that (denoting $e=(e_1,\ldots,e_{2g})$) we obtain
  \begin{equation} \label{eq:temp_continuous_depend_3} 
     \| [\eta] \|_{\mathcal{H}'(\partial \riem)} \leq C \| \sigma-e \|_{\mathbb{C}^{2g}} \leq C \| ([\beta],\rho,\sigma) \|_{\mathcal{H}'(\partial \riem) \oplus  \mathbb{C}^{2g+n-1}}.
  \end{equation} 
  Similarly
  \begin{equation} \label{eq:temp_continuous_depend_4} 
     \| \eta \|_{\mathcal{A}_{\mathrm{harm}}(\riem)} \leq  C \| ([\beta],\rho,\sigma) \|_{\mathcal{H}'(\partial \riem) \oplus  \mathbb{C}^{2g+n-1}}.
  \end{equation} 
  
  We need one further bound. Set $d = (d_1,\ldots,d_{n-1})$ where
  \[  d_k = \int_{\partial_k \riem} \ast (\eta + \delta) , \ \ k=1,\ldots,n \]
  (note that $d_n$ is $1-d_1- \cdots d_{n-1}$).
  The $d_k$'s depend boundedly on $\delta$ and $\eta$, so 
  \begin{equation} \label{eq:temp_continuous_depend_5}
   \| d \|_{\mathbb{C}_{n-1}} \leq \| ([\beta],\rho,\sigma) \|_{\mathcal{H}'(\partial \riem) \oplus  \mathbb{C}^{2g+n-1}}.
  \end{equation}
  
  Given the definitions of $\delta$, $\eta$, and $d$, it is easily verified that the solution to the Dirichlet problem is 
  \begin{equation} \label{eq:temp_solution_dirichlet}
   \mathbf{Dir}_{\partial \riem,\riem} ([\beta],\rho,\sigma) = 
     \mathbf{E}([\beta-\eta-\delta],\rho-d) + \eta + \delta.   
  \end{equation}
  The continuous dependence is now easily obtained: by boundedness of $\mathbf{E}$, \eqref{eq:temp_continuous_depend_1}, \eqref{eq:temp_continuous_depend_3}, and \eqref{eq:temp_continuous_depend_5} we have
  \begin{align*}
   \|    \mathbf{E}([\beta-\eta-\delta],\rho-d) \|_{\mathcal{A}_{\mathrm{harm}}(\riem)}&  \leq \| [\beta - \eta - \delta ]  \|_{\mathcal{H}'(\partial \riem)} + \| \rho-d \|_{\mathbb{C}^{n-1}} \\
   & \leq \| \beta \|_{\mathcal{H}'(\partial \riem)} +  \| \delta \|_{\mathcal{H}'(\partial \riem)} +  \| \eta \|_{\mathcal{H}'(\partial \riem)}  + \| \rho-d \|_{\mathbb{C}^{n-1}} \\
   & \leq C \| ([\beta],\rho,\sigma) \|_{\mathcal{H}'(\partial \riem) \oplus  \mathbb{C}^{2g+n-1}}.
  \end{align*}
Therefore \eqref{eq:temp_solution_dirichlet},  the above bound together with \eqref{eq:temp_continuous_depend_2} and \eqref{eq:temp_continuous_depend_4} yield that
  \[ \| \mathbf{Dir}_{\partial \riem,\riem} ([\beta],\rho,\sigma)\|_{\mathcal{A}_{\mathrm{harm}}(\riem)} \leq C \| ([\beta],\rho,\sigma) \|_{\mathcal{H}'(\partial \riem) \oplus  \mathbb{C}^{2g+n-1}}.  \]
  
  It remains to show that the solution is unique.  
  Let $\alpha'$ be another solution to the Dirichlet problem.  Conditions (1) and (3) of Definition \ref{eq:Dirichlet_theorem_temp3} imply that $\alpha' - \alpha$ is exact and has a global primitive $h$, which has constant CNT boundary values on $\partial \riem$.  So $h$ is in the linear span of the harmonic measures.  Condition (2) implies that $\alpha' = \alpha$.  Summing up, we have shown that the Dirichlet problem with the aforementioned CNT data is well-posed in the spaces that are given in the statement of the theorem.
 \end{proof} 
 \begin{remark}
  Because of condition (1) on CNT Dirichlet boundary data, one of the constants $\lambda_n$ in the $\mathcal{H}'(\partial \riem)$ is redundant and depends continuously on the other constants.  So one constant can be removed from the norm of $\mathcal{H}'(\partial \riem)$ in the estimate.  
 \end{remark}
 \begin{remark}[Special cases $n=1$ and $g=0$] \label{re:nequals_one_Dirichlet} If there is only one boundary curve $\partial_1 \riem$, then condition (2) requires that 
 \[  \int_{\partial_1 \riem} \ast \alpha = 0. \]
 This is true for any $\ast \alpha \in \mathcal{A}_{\mathrm{harm}}(\riem)$, so condition (2) may be omitted.   Similarly, in condition (1) it is required that 
 \[  \int_{\partial_1 \riem} [\beta_1] =0, \]
  which is true for any $[\beta_1] \in \mathcal{H}'(\partial \riem)$, and thus this part of condition (1) can be omitted.  
  
  If the genus $g$ of $\riem$ is zero, then the third condition is omitted.
  
  In either case, some steps in the proof of Theorem \ref{th:CNT_Dirichlet_problem_oneforms} can be omitted.
 \end{remark}

 The following proposition verifies that the CNT Dirichlet problem is natural.
 \begin{proposition}  If the Dirichlet data $([\beta],\rho,\sigma)$ is such that $[\beta]$ has a representative on a collar neighbourhood which is $\mathcal{C}^\infty$, then $\mathbf{Dir}_{\partial \riem,\riem}([\beta],\rho,\sigma)$ is the solution to the $\mathcal{C}^\infty$ Dirichlet problem. 
 \end{proposition}
 \begin{proof}  
  Choose a representative $(\beta_1,\ldots,\beta_n)$ of $[\beta]$ on a collection of collar neighbourhoods $U_k$ of $\partial_k \riem$ for $k=1,\ldots,n$, which are smooth on $\partial_k \riem$.     
  By Theorem \ref{th:smooth_Dirichlet_solution} there is a $\mathcal{C}^\infty$ solution $\alpha$ to the Dirichlet problem with data given by $(\beta,\rho,\sigma)$ with $\beta$ given by the restriction of $\beta_k$ to the boundaries $\partial_k \riem$ for $k=1,\ldots,n$.  
  
  We claim that $\alpha$ is the solution to the CNT Dirichlet problem.  Once this is shown, the proof is complete thanks to uniqueness statement of Theorem \ref{th:CNT_Dirichlet_problem_oneforms}.  First, observe that since $\alpha$ is $\mathcal{C}^\infty$ on $\text{cl} \, \riem$, it is in $\mathcal{A}_{\mathrm{harm}}(\riem)$.  So we need only show that the CNT boundary values of the $\mathcal{C}^\infty$ solution are equal to $[\beta]$.
  
  To see this, choose one-forms $\delta_k$ on a collar neighbourhood $U_k$ of $\partial_k \riem$, which extend smoothly to $\partial_k \riem$ and such that
  \[  \int_{\partial_k \riem} (\alpha - \delta_k) =0 \]
  for $k=1,\ldots,n$.  This can be arranged for example by considering $\riem$ to be a subset of its double.   
  The primitive $h_k$ of $\alpha - \delta_k$ on $U_k$ is $\mathcal{C}^\infty$, and in particular extends continuously to $\partial_k \riem$ for $k=1,\ldots,n$.  But the CNT boundary values must equal the continuous extension by definition.  By definition of the $\mathcal{C}^\infty$ solution to the Dirichlet problem, $d h_k = \beta_k -\delta_k$ on the boundary so $[\alpha] = [\beta]$.  This completes the proof.
 \end{proof}
\end{subsection}
\begin{subsection}{Dirichlet problem for one-forms with $H^{-1/2}$ data}
 \label{se:Dirichlet_forms_Hnegativeonehalf}
 The solution to the Dirichlet problem can be phrased in terms of $H^{-1/2}$ boundary data as follows.  
 \begin{definition}[$H^{-1/2}$ data for one-forms]
  By $H^{-1/2}$ data for one-forms we mean the following:
  \begin{enumerate}
     \item $L=(L_1,\ldots,L_n) \in \bigoplus_{k=1}^n H^{-1/2}(\partial_k \riem)$ 
     such that 
     \[   L_1(1) + \cdots L_n(1)=0; \]
     \item $\rho=(\rho_1,\ldots,\rho_n)\in \mathbb{C}^n$ satisfying
     \[  \rho_1 + \cdots + \rho_n =0;\]
     and
     \item $\sigma= (\sigma_1,\ldots,\sigma_{2g}) \in \mathbb{C}^{2g}$.  
 \end{enumerate} 
 \end{definition}
 
 In the following, recall the definition \eqref{eq:associated_H_minusonehalf_pairing} for the element $L_{[\alpha]}$ of $H^{-1/2}(\partial_k \riem)$ associated to a one-form $\alpha$.  
 \begin{definition}[$H^{-1/2}$ Dirichlet problem for one-forms]\label{defn:Hminus12_dirichlet_data}
  We say that a harmonic one-form $\alpha$ on $\riem$  solves the $H^{-1/2}$ Dirichlet problem with $H^{-1/2}$ Dirichlet data $(L,\rho,\sigma)$ if
 \begin{enumerate}
     \item for $k=1,\ldots,n$, for any $h_k \in \mathcal{H}^{1/2}(\partial_k \riem)$ we have
     \[  L_k(h_k) = L_{[\alpha]} h_k;   \] 
     \item for all $k =1,\ldots,n$  
     \[  \int_{\partial_k \riem} \ast \alpha = \rho_k;   \]
     and 
     \item for all $k=1,\ldots,2g$
     \[  \int_{\gamma_k} \alpha : = \sigma_k.  \]
 \end{enumerate} 
 \end{definition}
 The CNT Dirichlet problem has a solution which depends continuously on the data.\\

  \begin{theorem}[Well-posedness of Dirichlet's problem for $H^{-1/2}$ data] \label{th:Hminusonehalf_Dirichlet_problem_oneforms} For $H^{-1/2}$  Dirichlet data $(L,\rho,\sigma)$ there exists a unique $\alpha \in  \mathcal{A}_{\mathrm{harm}}(\riem)$ which solves the Dirichlet problem.  The operator 
  \[   \widetilde{\mathbf{Dir}}_{\partial \riem,\riem}: \bigoplus_{k=1}^n \mathcal{H}^{-1/2}(\partial_k \riem) \oplus \mathbb{C}^{2g+n-1} \rightarrow \mathcal{A}_{\mathrm{harm}}(\riem)      \]
  taking $(L,\rho,\sigma)$ to the solution
  is bounded.  Here the entries of $\mathbb{C}^{2g+n-1}$ are \[  (\rho_1,\ldots,\rho_{n-1},\sigma_1,\ldots,\sigma_{2g}).   \]  
 \end{theorem}
 \begin{proof}
  This follows immediately from Theorems \ref{th:Honehalf_reformulation} and \ref{th:CNT_Dirichlet_problem_oneforms}.  
 \end{proof}
 
\end{subsection}
\end{section}
\begin{section}{Overfare of harmonic one-forms} \label{se:Overfare}
\begin{subsection}{Assumptions throughout this section}  \label{se:Overfare_section_assumptions}
 The following assumptions 
 which will be in force throughout Section \ref{se:Overfare}. Additional hypotheses are added to the statement of each theorem where necessary.
 \begin{enumerate}
     \item $\mathscr{R}$ is a compact Riemann surface;
     \item $\Gamma = \Gamma_1 \cup \cdots \cup \Gamma_n$ is a collection of quasicircles;
     \item $\Gamma$ separates $\mathscr{R}$ into $\riem_1$ and $\riem_2$ in the sense of Definition \ref{de:separating_complex}.
 \end{enumerate} 
  
  We will furthermore assume that the ordering of the boundaries of $\partial \riem_1$ and $\partial \riem_2$ is such that $\partial_k \riem_1 = \partial_k \riem_2 = \Gamma_k$ as sets for $k=1,\ldots,n$.  
\end{subsection}
\begin{subsection}{About this Section}
 In this Section, we address the problem of overfare of one-forms. That is, given an $L^2$ harmonic one-form on $\riem_1$, we show that there is an $L^2$ harmonic one-form on $\riem_2$ with the same boundary values. To do this, we first show that the local boundary values in $H^{-1/2}(\partial_k \riem_1)$ (equivalently, in $\mathcal{H}'(\partial_k \riem_1)$) uniquely determine boundary values
 in $H^{-1/2}(\partial_k \riem_2)$ (equivalently, in $\mathcal{H}'(\partial_k \riem_2)$).   
 
 Of course, to uniquely determine the one-form on $\riem_2$ one also needs to specify cohomological data. One way to do this is simply to specify the CNT Dirichlet data for forms on $\riem_2$ as in Section \ref{se:Dirichlet_full_solution}. We also give an alternate approach, using forms in $\mathcal{A}_{\mathrm{harm}}(\mathscr{R})$ to specify the extra data. We call these forms catalyzing forms. This point of view  illuminates the scattering process and the cohomological properties of the Schiffer operators, as we will see in Sections \ref{se:index_cohomology} and \ref{se:scattering} ahead. It also plays a central role in our approach to the generalized period matrices in Section \ref{se:period_mapping}.  
\end{subsection}
\begin{subsection}{Partial overfare of one-forms}
 In this section we define overfare of one-forms and functions, and show that it exists and is bounded. 
 
 This subsection is devoted to a kind of ``partial'' overfare, where only the boundary data is mapped into the new surface. We first define this for $H^{1/2}$.  Recall that the Sobolev spaces are defined by treating the boundary curves of $\riem_k$ as analytic curves in the double. Thus, we distinguish $H^{1/2}(\partial_k \riem_1)$ and $H^{1/2}(\partial_k \riem_2)$.
 
  We define the partial overfare as follows. 
  Let $h_1 \in \mathcal{H}^{1/2}(\partial_1 \riem)$.  Let $\phi:U \rightarrow \mathbb{C}$ be a doubly-connected chart defined in a neighbourhood of $\partial_k \riem$, whose inner curves are analytic.  For any extension $H_1 \in \mathcal{D}_{\mathrm{harm}}(U_1)$ whose CNT boundary values equal $h$, let $H_2 \in \mathcal{D}_{\mathrm{harm}}(U_2)$ be as in Lemma \ref{le:local_Overfare}, and let $h_2$ be its CNT boundary values. We set
  \begin{align*}
   \mathbf{O}(\partial_k \riem_1,\partial_k \riem_2):H^{1/2}(\partial_k \riem_1) & \rightarrow H^{1/2}(\partial_k \riem_2) \\
   h_1 & \mapsto h_2.
  \end{align*}  
  We define $\mathbf{O}(\partial_k \riem_2,\partial_k \riem_1)$ similarly.
  
  \begin{proposition} \label{pr:local_overfare_functions_well_defined}
   Given $h \in H^{1/2}(\partial_k \riem_1)$, 
   let $H$ be any element of $\mathcal{D}_{\mathrm{harm}}(\riem_1)$ whose CNT boundary values equal $h$ on $\partial_k \riem_1$. Then the boundary values of $\mathbf{O}_{1,2} H$ equal $\mathbf{O}(\partial_k \riem_1,\partial_k \riem_2)h$.  
  \end{proposition}
  \begin{proof}
   This follows immediately from the observation that the CNT boundary values of $\mathbf{O}(\partial_k \riem_1,\partial_k \riem_2)h$ agree with those of $h$, and therefore with those of $H$. By definition of overfare, the CNT boundary values of $\mathbf{O}_{1,2} H$ agree with those of $H$.  
  \end{proof}
  
  In particular, $\mathbf{O}(\partial_k \riem_1,\partial_k \riem_2)$ is independent of the choice of extension $H_1$ and doubly-connected chart.
  
  We also have that the partial overfare is bounded.
  \begin{proposition} \label{pr:local_overfare_functions_bounded} 
  The following statements hold:\\
   
   \emph{(1)} 
    $\mathbf{O}(\partial_k \riem_1,\partial_k \riem_2)$ is bounded 
    as a map from $\dot{H}^{1/2}(\partial_k \riem_1)$ to $\dot{H}^{1/2}(\partial_k \riem_2)$.\\ 
    
  \emph{(2)}  If $\partial_k \riem_1$ is a \emph{BZM} quasicircle, then $\mathbf{O}(\partial_k \riem_1,\partial_k \riem_2)$ is bounded 
    as a map from ${H}^{1/2}(\partial_k \riem_1)$ to ${H}^{1/2}(\partial_k \riem_2)$.
  \end{proposition}
  \begin{proof}
   By Proposition \ref{pr:local_overfare_functions_well_defined} we may choose any doubly-connected chart to define the partial overfare. Choose a such a chart $\phi$ on a doubly-connected domain $U$ and let $U_1,U_2$ be collar neighbourhoods of $\partial_k \riem_1$ and $\partial_k \riem_2$ as in Lemma \ref{le:local_Overfare}.  As in the proof of that lemma, we obtain a pair of domains in the plane $\Omega_k$ bounded by $\gamma=\phi(\partial_k \riem_1) = \phi(\partial_k \riem_2)$. 
   Both claims now follows from Lemma \ref{le:local_Overfare}, and boundedness of Sobolev trace and extension from $H^1(\Omega_k)$ to $H^{1/2}(\gamma)$ and 
   $\dot{H}^1(\Omega)$ to $\dot{H}^{1/2}(\gamma)$. (Note that the definition of $H^{1/2}(\gamma)$ depends on the choice of side $\Omega_1$ or $\Omega_2$, treating $\gamma$ as an analytic curve in the double of $\Omega_1/\Omega_2$ respectively).   
  \end{proof}
  \begin{remark}[Unique extension from $H^{1/2}$ to $\mathcal{H}$] 
   \label{re:unique_extension_sobolev_to_CNT}
   Let $\Gamma$ be a border of a Riemann surface $\riem$. We treat $\Gamma$ as an analytic curve in the double. We assume for simplicity that there are no other boundary points, although the discussion holds in the general case.
   
   Elements of $H^{1/2}(\Gamma)$ which agree with each other almost everywhere are the same in that Sobolev space. On the other hand, functions in $\mathcal{H}(\Gamma)$ are the same only if they agree up to a (potential-theoretic) null set. Sets of measure zero need not be null; for example, in the circle, not every set of measure zero has logarithmic capacity zero. Thus, an element of  $H^{1/2}(\Gamma)$ does not a-priori lead to a well-defined element of $\mathcal{H}(\Gamma)$. 
   
   However, given $h \in H^{1/2}(\Gamma)$, a well-defined element of $\mathcal{H}(\Gamma)$ can be obtained as follows. Let $H \in H^1(\riem)$ be the unique harmonic Sobolev extension of $h$. In particular, $H \in \mathcal{D}_{\mathrm{harm}}(\riem)$ and thus has CNT boundary values $\tilde{h}$ defined except possibly on a null set.  Therefore $h$ determines a unique element of $\mathcal{H}(\Gamma)$.  
  \end{remark}
  \begin{remark}[Subtlety in defining overfare on $H^{1/2}$]
   There is an important technical subtlety in the definition of the partial overfare.  For simplicity, we assume that $\riem_1$ and $\riem_2$ have only one border $\partial \riem_1 = \partial \riem_2$ which is shared between them. As above, the discussion here applies to the general case. 
   
   Given $h_1 \in H^{1/2}(\partial_1 \riem)$, one might seek an element $h_2 \in H^{1/2}(\partial_2 \riem)$ which agrees with $h_1$ almost everywhere. This is not even well-defined, because sets of measure zero in $\partial_1 \riem$ are not necessarily of measure zero in $\partial_2 \riem$. For example, if $\Gamma$ is a quasicircle in the plane bounding $\Omega_1$ and $\Omega_2$, sets of measure zero in $\Gamma$ treated as an analytic curve in the double of $\Omega_1$ are precisely sets of harmonic measure zero. Sets of harmonic measure zero in $\Gamma$ with respect to $\Omega_1$ need not be harmonic measure zero with respect to $\Omega_2$. Thus the partial overfare cannot be formulated this way, necessitating the definition above and \ref{pr:local_overfare_functions_well_defined} and \ref{pr:local_overfare_functions_bounded}.
   
   On the other hand, using Remark \ref{re:unique_extension_sobolev_to_CNT} the definition of partial overfare can be stated succinctly as follows.  Given $h_1 \in H^{1/2}(\partial_k \riem_1)$, let $\tilde{h} \in \mathcal{H}(\partial_k \riem_1) = \mathcal{H}(\partial_k \riem_2)$ be the unique element corresponding to $h$. Then $\tilde{h}$ agrees with a unique element $h_2 \in H^{1/2}(\partial_k \riem_2)$, and we can set
   \[  h_2 = \mathbf{O}(\partial_k \riem_1,\partial_k \riem_2)h_1.  \]
  \end{remark} 
   
   Next, we will define a partial overfare of elements of $H^{-1/2}$. Again, recall that $H^{-1/2}(\partial_k \riem_m)$ is defined by treating $\partial_k \riem_m$ as an analytic curve in the double of $\riem_m$, and therefore we must distinguish $H^{-1/2}(\partial_k \riem_1)$ from $H^{-1/2}(\partial_k \riem_2)$.
   
   Let $L \in H^{-1/2}(\partial_k \riem_1)$. We define 
   \[  \mathbf{O}'(\partial_k \riem_1,\partial_k \riem_2):H^{-1/2}(\partial_k  \riem_1) \rightarrow H^{-1/2}(\partial_k \riem_2) \]
   by 
   \[   [\gls{oprime}(\partial_k \riem_1,\partial_k \riem_2) L](h) = -L(\mathbf{O}(\partial_k \riem_2,\partial_k \riem_1)h)  \ \ \ \text{for all} \ h \in H^{1/2}(\partial_k \riem_1).  \]
   $\mathbf{O}'(\partial_k \riem_2,\partial_k \riem_1)$ is defined similarly.
   \begin{remark}
    The negative sign is introduced in order to take into account the change of orientation of the boundary, as we will see below.
   \end{remark}
   
   We also define
   \[ \gls{oprimedot}_{1,2}:\dot{H}^{-1/2}(\partial_k \riem_1) \rightarrow \dot{H}^{-1/2}(\partial_k \riem_2)  \]
   and
   \[  \dot{\mathbf{O}}'_{2,1}:\dot{H}^{-1/2}(\partial_k \riem_2) \rightarrow \dot{H}^{-1/2}(\partial_k \riem_1)  \]
   in the obvious way. It is easily verified that these are well-defined. 
   
   \begin{proposition} \label{pr:overfare_prime_constant} For any $L \in H^{-1/2}(\partial_k \riem)$, 
     \[ [\mathbf{O}'(\partial_k \riem_1,\partial_k \riem_2)L] (1) = - L(1). \]
   \end{proposition}
   \begin{proof}
    This follows from the easily-verified fact that $\mathbf{O}(\partial_k \riem_1,\partial_k \riem_2) 1 = 1$.  
   \end{proof}
   \begin{proposition} \label{pr:Hminus_onehalf_local_over_bounded}
   The following statements are valid:\\
   
    \emph{(1)} The partial overfare $\gls{oprimedot}(\partial_k \riem_1,\partial_k \riem_2)$ is bounded  as a map from $\dot{H}^{-1/2}(\partial_k \riem_1)$ to $\dot{H}^{-1/2}(\partial_k \riem_2)$.\\ 
    
    \emph{(2)} If $\partial_k \riem$ is a \emph{BZM} quasicircle, then $\mathbf{O}'(\partial_k \riem_1,\partial_k \riem_2)$ is bounded  as a map from ${H}^{-1/2}(\partial_k \riem_1)$ to ${H}^{-1/2}(\partial_k \riem_2)$.
   \end{proposition}
   \begin{proof} 
    This follows immediately from Proposition \ref{pr:local_overfare_functions_bounded}. 
   \end{proof}

 The association between $H^{-1/2}(\partial_k \riem_m)$ and $\mathcal{H}'(\partial_k \riem_m)$ given by Theorem \ref{th:Honehalf_reformulation} immediately defines a bounded overfare
 \[  \mathbf{O}'(\partial_k \riem_1,\partial_k \riem_2) :\mathcal{H}'(\partial_k \riem_1) \rightarrow \mathcal{H}'(\partial_k \riem_2) \]
 and similarly for the homogeneous spaces
 \[  \mathbf{O}'(\partial_k \riem_1,\partial_k \riem_2) :\dot{\mathcal{H}}'(\partial_k \riem_1) \rightarrow \dot{\mathcal{H}}'(\partial_k \riem_2) \]
 
 We will use the same notation for the overfares on $H^{-1/2}(\partial_k \riem_m)$ and $\mathcal{H}'(\partial_k \riem_m)$.
 
 The partial overfare preserves periods:
 \begin{proposition} \label{pr:overfare_preserves_periods}
  For any $k =1,\ldots,n$ and $[\alpha] \in \mathcal{H}'(\partial_k \riem_1)$ we have that 
  \[  \int_{\partial_k \riem_2}  \mathbf{O}'(\partial_k \riem_1,\partial_k \riem_2) [\alpha] = - \int_{\partial_k \riem_1} [\alpha].  \]
  The same claim holds with the roles of $1$ and $2$ switched.
 \end{proposition}
 \begin{proof}
  This follows from Proposition \ref{pr:overfare_prime_constant} after observing that 
  \[  L_{[\alpha]}(1) = \int_{\partial_k \riem_1} [\alpha].  \]
 \end{proof}

   We immediately have the following:\\

   \begin{proposition}  \label{pr:extendible_forms_overfare_themselves}
    Let $U$ be a doubly-connected neighbourhood of $\partial_k \riem_1=\partial_k \riem_2$. 
    
     \emph{(1)} For any  
    $\alpha \in \mathcal{A}^{\mathrm{e}}_{\mathrm{harm}}(U)$ we have
    \[   \mathbf{O}'(\partial_k \riem_1,\partial_k \riem_2) [\alpha]=[\alpha] \]
    { where the equality above is in $\dot{H}^{-1/2}(\partial_k \riem_2)$.}
    
    \emph{(2)} If $\partial_k \riem_1$ is a \emph{BZM} quasicircle, then for any  
    $\alpha \in \mathcal{A}_{\mathrm{harm}}(U)$ we have
    \[   \mathbf{O}'(\partial_k \riem_1,\partial_k \riem_2) [\alpha]=[\alpha]. \]

   \end{proposition}
   \begin{proof}
    {Denote by $L^m_{[\alpha]}$ the elements of $H^{-1/2}(\partial_k \riem_m)$ induced by $\alpha$ for $m=1,2$. We need to show that $L^1_{[\alpha]}=L^2_{[\alpha]}$.} By Proposition \ref{pr:local_overfare_functions_bounded} it is enough to prove this on the dense set $\mathcal{D}_{\mathrm{harm}}(U)$ in both cases (1) and (2).  Let $\Gamma_\varepsilon^m$ denote the limiting curves and $U_\varepsilon$ denote the region bounded by these curves. For $H$ in this dense set, we have
    \begin{align*}
     L^2_{[\alpha]} H - L^1_{[\alpha]} H & = \lim_{\epsilon \rightarrow 0}  \left( \int_{\Gamma^2_\varepsilon} \alpha H - \int_{\Gamma^1_\varepsilon} \alpha H \right) \\
     & = - \lim_{\epsilon \rightarrow 0} \iint_{U_\varepsilon} \alpha \wedge dH.
    \end{align*}
    Therefore by the Cauchy-Schwarz inequality, for all $\varepsilon>0$
    \[ 
    \left| L^2_{[\alpha]} H - L^1_{[\alpha]} H\right| \leq \| dH \|_{\mathcal{A}_{\mathrm{{{harm}}}}(U_\varepsilon)} \cdot \| \alpha \|_{\mathcal{A}_{\mathrm{harm}}(U_\varepsilon)}.   \]
    {Letting $\varepsilon$ go to  zero, the claim now follows from the facts that $U_{\varepsilon}\subset U$, $dH\in \mathcal{A}_{\mathrm{harm}}(U)$ and $\cap_{ \varepsilon} U_\varepsilon$ has measure zero because quasicircles have measure zero.} 
   \end{proof}
   In other words, one-forms which extend harmonically across a border are their own overfare. We will use this repeatedly in the next few sections.
 
  
 \end{subsection}
 \begin{subsection}{Overfare of one-forms}  \label{se:subsection_overfare_oneforms}
   We first recall some notation and establish conventions. Assume that $\riem_k$ are connected and have genus $g_k$ for $k=1,2$. Let
 \[   \{ \gamma_1^k,\ldots,\gamma_{2g_k}^k, \partial_1 \riem_k,\ldots,\partial_{n-1} \riem_k \}  \]
 be a set of generators for the fundamental group of $\riem_k$.   The generators $\partial_j \riem_k$ are common to both $\riem_1$ and $\riem_2$, when viewed as subsets of $\mathscr{R}$.  Note that these are not the same generators as those appearing in Section \ref{se:Dirichlet_problem}, since $\mathscr{R}$ need not be the double of either $\riem_1$ or $\riem_2$.   
 
  In this section we define a notion of overfare of one-forms. That is, given $\alpha_2 \in \mathcal{A}_{\mathrm{harm}}(\riem_2)$, we see a form $\alpha_1 \in \mathcal{A}_{\mathrm{harm}}(\riem_1)$ with the same boundary values. Needless to say, one must specify more data about $\alpha_1$ to make this well-posed, as we saw in Section \ref{se:Dirichlet_problem}.
  \begin{theorem}   \label{th:old_overfare_forms}
   Given $\alpha_2 \in \mathcal{A}_{\mathrm{harm}}(\riem_2)$, $\sigma_1,\ldots,\sigma_{2g} \in \mathbb{C}$ and $\rho_1,\ldots,\rho_{n-1}\in \mathbb{C}$, there is a unique $\alpha_1 \in \mathcal{A}_{\mathrm{harm}}(\riem_1)$ such that
   \begin{enumerate}
       \item \[  \mathbf{O}(\partial_k \riem_2,\partial_k \riem_1) [\alpha_2]=[\alpha_1],  \  \ \ k=1,\ldots,n;
      \]
      \item \[ \int_{\gamma_m} \alpha_1 = \sigma_m, \ \ \ m=1,\ldots,2g;   \]
      and
      \item \[  \int_{\partial_k \riem_1} \ast \alpha_1 = \rho_k, \ \ \ k=1,\ldots,n-1. \]
   \end{enumerate}
  \end{theorem}
  \begin{proof}
   This follows immediately from Theorems \ref{th:CNT_Dirichlet_problem_oneforms} and \ref{le:local_Overfare}.  
  \end{proof}
  {
  \begin{remark}
   One can formulate continuous dependence of $\alpha_1$ on $\alpha_2,\sigma_1,\ldots,\sigma_{2g}$, and $\rho_1,\ldots,\rho_{n-1}$, for BZM quasicircles. We will take a different approach to overfare of forms ahead.
   \end{remark}
   }
  
  In order to view overfaring forms as a scattering process, we will reformulate the conditions as follows. The main idea is that we will use a one-form on the surface $\mathscr{R}$ to determine the extra data in the overfare.
  \begin{definition} \label{de:weakly_compatible}
   Let $\alpha_k \in \mathcal{A}_{\mathrm{harm}}(\riem_k)$ for $k=1,2$, and let $\zeta \in \mathcal{A}_{\mathrm{harm}}(\mathscr{R})$.  We say that $\alpha_1$ and $\alpha_2$ are weakly compatible with respect to $\zeta$ if
   \begin{enumerate}
       \item $\mathbf{O}(\partial_k \riem_2,\partial_k \riem_1) [\alpha_2]=[\alpha_1]$ for  $k=1,\ldots,n$, 
       \item $\alpha_k - \mathbf{R}^{\mathrm{h}}_k \zeta \in \mathcal{A}^{\mathrm{e}}_{\mathrm{harm}}(\riem_k)$ for $k=1,2$.\\
   \end{enumerate}
   
  We call $\zeta$ a \emph{weakly catalyzing one-form} for the pair $\alpha_1,\alpha_2$.
  \end{definition}
  
  It follows immediately from Theorem \ref{th:old_overfare_forms} that  weakly compatible forms exist.  
  \begin{corollary}  
   Given $\alpha_2 \in \mathcal{A}_{\mathrm{harm}}(\riem_2)$ and $\zeta \in \mathcal{A}_{\mathrm{harm}}(\mathscr{R})$ such that $\alpha_2 - \mathbf{R}_1^h \zeta \in \mathcal{A}^e_{\mathrm{harm}}(\riem_2)$, there is an $\alpha_1$ such that $\alpha_1$ and $\alpha_2$ are weakly compatible with respect to $\zeta$.  
  \end{corollary} 
  Of course $\alpha_1$ is not unique, and weak compatibility is obviously equivalent to conditions (1) and (2) of Theorem \ref{th:old_overfare_forms}.  
  
  We add a third condition to deal with the ambiguity.
  \begin{definition} \label{de:strongly_compatible}
   We say that $\alpha_k \in \mathcal{A}_{\mathrm{harm}}(\riem_k)$, $k=1,2$ are compatible with respect to $\zeta \in \mathcal{A}_{\mathrm{harm}}(\mathscr{R})$ if they are weakly compatible with respect to $\zeta$, and additionally
   \begin{enumerate}\setcounter{enumi}{2}
       \item
        \[ \mathbf{S}_1^{\mathrm{h}} \alpha_1 + \mathbf{S}_2^{\mathrm{h}} \alpha_2 = \zeta.  \]
   \end{enumerate}
   In this case we say that $\zeta$ is a catalyzing form.
  \end{definition}
  \begin{remark}
   We will say that $(\alpha_1,\alpha_2,\zeta)$ is a weakly compatible/compatible triple if $\alpha_1$ and $\alpha_2$ are weakly compatible/compatible with respect to $\zeta$.  Also,  we will say that $\alpha_1$ is weakly compatible/compatible with $\alpha_2$ and $\zeta$ if $(\alpha_1,\alpha_2,\zeta)$ is a compatible triple. 
  \end{remark}
  
  Some motivation for the third compatibility condition is in order.  Assume that $\riem_1$ and $\riem_2$ are connected, and refer to Theorem \ref{th:old_overfare_forms}.  An $\alpha_2+\overline{\beta}_2 \in \mathcal{A}_{\mathrm{harm}}(\riem_2)$ and $\zeta \in \mathcal{A}_{\mathrm{harm}}(\mathscr{R})$ specify the data (1) and (2) for the boundary value problem for $\alpha_1 + \overline{\beta}_1$ are specified in Theorem \ref{th:old_overfare_forms}. Thus $\alpha_1+\overline{\beta}_1$ is only determined up to a harmonic measure $d\omega \in \mathcal{A}_{\mathrm{harm}}(\riem_1)$, where the missing data (3) is required to determine a unique form.  
  
  Instead of giving the data in the form (3), we specify it using the third condition in Definition \ref{de:strongly_compatible}.  In this form, this data can be seen to also be specified by the catalying form $\zeta$. To see this, observe that weakly compatible forms $\alpha_k+\overline{\beta}_k$ with respect to the catalyzing form $\zeta$ are also weakly compatible with respect to $\zeta'$ if and only if $\zeta -\zeta' \in \mathcal{A}^{\mathrm{pe}}_{\mathrm{harm}}(\mathscr{R})$. By Theorem \ref{th:piecewise_exact_characterize} there is a $d\omega_1 \in \mathcal{A}_{\mathrm{harm}}(\riem_1)$ such that $\zeta-\zeta' =\mathbf{S}_1^{\mathrm{h}} d\omega_1$. Thus the remaining data is exactly specified by choosing a specific catalyzing form if one includes the third compatibility condition.  
  
  Incidentally, this also shows that compatible forms exist in the case that both $\riem_1$ and $\riem_2$ are connected. It is possible to extend this result by extending Theorem \ref{th:piecewise_exact_characterize}.  However we will prove existence in a different way. 
  
  The condition for compatibility is quite natural. Assuming that $\riem_2$ is connected, a reasonable definition for the overfare of $\alpha_2$ via the catalyzing form $\zeta$ is 
  \[  \alpha_1 = \mathbf{O}^{\mathrm{e}}_{2,1}\left[ \alpha_2 - \mathbf{R}_2^{\mathrm{h}} \zeta \right] + \mathbf{R}_1^{\mathrm{h}}\zeta,  \]
  in light of Proposition \ref{pr:extendible_forms_overfare_themselves}. 
  We will prove in Section \ref{se:scattering} that this $\alpha_1$ is indeed compatible, and in particular this proves existence of compatible forms in the case that only $\riem_2$ is connected. 
  
  We conclude with two observations on perturbations of compatible triples $(\alpha_1,\alpha_2,\zeta)$.   Given two forms $\alpha_k \in \mathcal{A}_{\mathrm{harm}}(\riem_k)$ with the same boundary values, there are many catalyzing one-forms $\zeta \in\mathcal{A}_{\mathrm{harm}}(\mathscr{R})$.
  \begin{proposition}  \label{pr:perturb_catalyst_by_pe}
   Let $\alpha_k \in \mathcal{A}_{\mathrm{harm}}(\riem_k)$, $k=1,2$ satisfy
   \[  \mathbf{O}(\partial_m {\riem_1},\partial_m \riem_2) [\alpha_1]=[\alpha_2]  \]
   for $k=1,2$.  There exists a {weakly }catalyzing one-form $\zeta \in \mathcal{A}_{\mathrm{harm}}(\mathscr{R})$ such that $(\alpha_1,\alpha_2,\zeta)$ is a weakly compatible triple. Furthemore, given any pair $\zeta,\zeta'$ of one-forms catalyzing the pair $\alpha_1,\alpha_2$, we have that $\zeta - \zeta'$ is piecewise exact.
  \end{proposition}
  \begin{proof}
   To prove existence, we need only choose any harmonic one-form on $\mathscr{R}$ whose periods agree with those of $\alpha_k$ on
   \[   \{ \gamma_1^k,\ldots,\gamma_{2g_k}^k, \partial_1 \riem_k,\ldots,\partial_{n-1} \riem_k \}  \]
   for $k=1,2$. This is possible because of the fact that
   \[  \int_{\partial_k \riem_1} [\alpha_1] = -\int_{\partial_k \riem_2} [\alpha_2] \]
   which follows from the condition $\mathbf{O}(\partial_m {\riem_1},\partial_m \riem_2) [\alpha_1]=[\alpha_2]$.
   
   Now let $\zeta,\zeta' \in \mathcal{A}_{\mathrm{harm}}(\mathscr{R})$ be catalyzing for the pair $\alpha_1,\alpha_2$.  Then 
   \[  \mathbf{R}^{\mathrm{h}}_k\zeta - \mathbf{R}^{\mathrm{h}}_k \zeta ' = (\mathbf{R}^{\mathrm{h}}_k\zeta- \alpha_k) - (\mathbf{R}^{\mathrm{h}}_k\zeta'- \alpha_k) \in \mathcal{A}_{\mathrm{harm}}^{\mathrm{e}}(\riem_k) \]
   for $k=1,2$, which completes the proof.
  \end{proof}
  
  Furthermore, 
  \begin{proposition}  \label{pr:perturb_both_by_harmonic_measure}
   {Assume that either $\riem_1$ or $\riem_2$ is connected. } 
   Let $\alpha_k \in \mathcal{A}_{\mathrm{harm}}(\riem_k)$ for $k=1,2$ be compatible with respect to $\zeta \in \mathcal{A}_{\mathrm{harm}}(\mathscr{R})$. Let $\omega_1$ and $\omega_2$ be harmonic functions which extend continuously to the boundary and are constant there.  Assume further that $\mathbf{O}_{1,2} \omega_1=\omega_2$. Then $\alpha_1 + d\omega_1$ and $\alpha_2 +d \omega_2$ are compatible with respect to $\zeta$. 
  \end{proposition}
  \begin{proof}
   By Proposition \ref{pr:harmonic_measures_zero} condition (1) of compatibility is satisfied by $\alpha_1 + d\omega_1$ and $\alpha_2 + d\omega_2$. The fact that (2) continues to be satisfied follows immediately from the fact that $d\omega_1$ and $d\omega_2$ are exact.
   Finally, observe that by Theorem \ref{th:S_on_harmonic_measures} 
   \[ \mathbf{S}^{\mathrm{h}}_1 (\alpha_1 + d\omega_1) + \mathbf{S}^{\mathrm{h}}_2 (\alpha_2 + d\omega_2) = \mathbf{S}^{\mathrm{h}}_1 \alpha_1 + \mathbf{S}^{\mathrm{h}}_2 \alpha_2 = \zeta,  \]
   completing the proof.
  \end{proof}
 \end{subsection}
\end{section} 
\begin{section}{Schiffer operators: cohomology and index theorems} \label{se:index_cohomology}
\begin{subsection}{Assumptions throughout this section} In this section we will once again use the assumptions that were in force in Subsection \ref{se:assumptions_scattering_section}.  Additional hypotheses are added to the statement of each theorem where necessary.


  

\end{subsection}
\begin{subsection}{About this Section}
 This section contains geometric and algebraic results about the Schiffer operators introduced in Subsection \ref{subsec:defSchiffer}. We give a characterization of the image and kernel of $\mathbf{T}_{1,2}$, and use this to prove an index theorem for this operator in the case that $\riem_1$ and $\riem_2$ are connected, and in the case that $\riem_2$ is of genus $g$ with $n$ boundary curves capped by $n$ simply connected domains.  This index theorem relates the conformally invariant index to purely topological quantities.
 
 We proceed as follows. First, we investigate the effect of the Schiffer operators $\mathbf{T}_{j,k}$ and $\mathbf{S}_k$ on cohomology in Section \ref{se:Schiffer_cohomology}. The main tool is the ``overfared'' jump formula, which is used to prove Theorem \ref{th:Schiffer_cohomology} which says that certain linear combinations of the Schiffer operators produce exact forms. Together with the fact that $\mathbf{S}_k \mathbf{R}_k$ is an isomorphism, this completely characterizes the effect of $\mathbf{T}_{j,k}$ on cohomology classes. In Section \ref{se:Schiffer_kernel_image} we determine the kernel and image of the operator $\mathbf{T}_{1,2}$. These results also play a central role in in the construction of the generalized period matrix in Section \ref{se:period_mapping}. Once this is accomplished, we prove the index theorem in Section \ref{se:Schiffer_kernel_image}. 
\end{subsection}
\begin{subsection}{Schiffer operators and cohomology}  \label{se:Schiffer_cohomology}
 Our goal here is to investigate the kernels, images and even Fredholm indices of Schiffer operators, and their interaction with the cohomology classes of $\riem_1$ and $\riem_2$. 

 \begin{theorem}  \label{th:S_kernel_range}   The Schiffer operators $\mathbf{R}_k$ and $\mathbf{S}_k$ satisfy
  \begin{align*} 
    \mathrm{Ker}(\mathbf{R}_k) & = \{0\} \\
    \mathrm{Im}(\mathbf{S}_k) & = \mathcal{A}(\mathscr{R}) \\
    \mathrm{Ker}(\mathbf{S}_k) & = [\mathbf{R}_k \mathcal{A}(\mathscr{R}) ]^\perp 
  \end{align*}
  for $k=1,2$.  The image of $\mathbf{R}_k$ is a $g$-dimensional subspace. 
  The corresponding statements hold for the complex conjugates.
 \end{theorem}
 \begin{proof} 
  The first statement is proven using analytic continuation.  The second statement follows from the first, and {Theorem \ref{th:adjoint_identities} which in turn yields that $\text{Im}(\mathbf{S}_k) = [\text{Ker} \, \mathbf{R}_k]^\perp$.}  The remaining statements are elementary.
 \end{proof}
 
 This yields the following:
 \begin{theorem}   \label{th:S_an_isomorphism} We have
  \begin{align*}
      \mathbf{S}_k \mathbf{R}_k : \mathcal{A}(\mathscr{R}) & \rightarrow \mathcal{A}(\mathscr{R}) \\
      \overline{\mathbf{S}}_k \overline{\mathbf{R}}_k : \overline{\mathcal{A}(\mathscr{R})} & \rightarrow \overline{\mathcal{A}(\mathscr{R})}
  \end{align*}
  are isomorphisms.
 \end{theorem}
 \begin{proof} It is enough to prove the first claim. 
  By Theorem \ref{th:S_kernel_range}, we have that $\mathbf{S}_k$ is surjective.  Thus for any $y \in \mathcal{A}(\mathscr{R})$ there is a $u \in \mathcal{A}(\riem_k)$ such that $\mathbf{S}_k u = y$.  Writing $u=v+w$ in terms of the orthogonal decomposition $\mathcal{A}(\riem_k) = [\mathbf{R}_k\mathcal{A}(\mathscr{R})]^\perp \oplus [\mathbf{R}_k \mathcal{A}(\mathscr{R})]$, since by Theorem \ref{th:S_kernel_range} we also have that $\mathbf{S}_k v=0$, we see that 
  \[ y = \mathbf{S}_k u = \mathbf{S}_k w   \]
  so $y \in \mathrm{Im}(\mathbf{S}_k \mathbf{R}_k),$ and thus $\mathbf{S}_k \mathbf{R}_k$ is surjective. Now since $\text{Ker}(\mathbf{S}_k \mathbf{R}_k )=(\mathrm{Im} (\mathbf{S}_k \mathbf{R}_k ))^{\perp}= \mathcal{A}(\mathscr{R}) ^{\perp}=0 $, we also have that $\mathbf{S}_k \mathbf{R}_k$ is injective.
  
 \end{proof}
  In what follows, we will apply the identities of Section \ref{se:Schiffer_Cauchy} to investigate how the Schiffer operators affect the cohomology classes of the one-forms to which they are applied.  The spaces
  \[  [\mathbf{R}_k \mathcal{A}(\mathscr{R})]^\perp = \{   {\alpha} \in \mathcal{A}(\riem_k) : (\alpha,{R}_k \beta) =0 \ \ \ \  \forall \beta \in \mathcal{A}(\mathscr{R}) \}= 
  [\mathbf{R}_k \mathcal{A}(\mathscr{R})]^\perp   \]
  and their complex conjugates
 \[   [\overline{\mathbf{R}}_k \overline{\mathcal{A}(\mathscr{R})}]^\perp  
    =  \{   \overline{\alpha} \in \overline{\mathcal{A}(\riem_k)} : (\overline{\alpha},{R}_k \overline{\beta}) =0 \ \ \ \  \forall \overline{\beta} \in \overline{\mathcal{A}(\mathscr{R})} \}  \]
  introduced in the proof above will play an important role.    
  \begin{remark}
    Throughout,  $[\mathbf{R}_k \mathcal{A}(\mathscr{R})]^\perp$ will always refer to the orthogonal complement in $\mathcal{A}(\riem_k)$ rather that in $\mathcal{A}_{\mathrm{harm}}(\riem_k)$, and similarly for $[\overline{\mathbf{R}}_k \overline{\mathcal{A}(\mathscr{R})}]^\perp$.  
  \end{remark}

     By a capped surface, we mean the special case that $\riem_1$ consists of $n$ simply-connected domains.   We say that $\riem_2$ is capped by $\riem_1$. 
    
    Recalling  $\mathbf{O}^\mathrm{e}_{2,1}$ from {Definition \ref{def:exact overfare}} and the projection operator \[  \overline{\mathbf{P}}_{1}: \mathcal{A}_{\mathrm{harm}}(\riem_1) \rightarrow \overline{\mathcal{A}(\riem_1)}, \] we have the following theorem of   M. Shirazi \cite{Shirazi_thesis,Schippers_Staubach_Shirazi}. 
  \begin{theorem} \label{th:Mohammad_isomorphism}
    Assume that $\riem_2$ is capped by $\riem_1$.  Then 
   \[\mathbf{T}_{1,2}([\overline{\mathbf{R}}_1 \overline{\mathcal{A}(\mathscr{R})}]^\perp  ) = 
   \mathcal{A}^\mathrm{e}(\riem_2) \]
   and 
   $\mathbf{T}_{1,2}:[\overline{\mathbf{R}}_1 \overline{\mathcal{A}(\mathscr{R})}]^\perp   \rightarrow \mathcal{A}^\mathrm{e}(\riem_2)$ is an isomorphism with inverse $-\overline{\mathbf{P}}_1 \mathbf{O}^\mathrm{e}_{2,1}$. In particular, 
   \[ \mathrm{Im}(\overline{\mathbf{P}}_1\mathbf{O}_{2,1}^{\mathrm e}) = [\overline{\mathbf{R}}_1 \overline{\mathcal{A}(\mathscr{R})}]^\perp.  \]
  \end{theorem}
  \begin{proof}
   We show that $\mathbf{T}_{1,2}([\overline{\mathbf{R}}_1 \overline{\mathcal{A}(\mathscr{R})}]^\perp) \subseteq  
   \mathcal{A}^\mathrm{e}(\riem_2)$.  Since each connected component of $\riem_1$ is simply connected, for any $\overline{\alpha} \in [\overline{\mathbf{R}}_1 \overline{\mathcal{A}(\mathscr{R})}]^\perp$, there is an $H \in \overline{\mathcal{D}(\riem_1)}$ such that $\overline{\alpha} =\overline{\partial} H$. Since by Theorem \ref{th:S_kernel_range} $\overline{\mathbf{S}}_1 \overline{\alpha} =0$, Theorem   \ref{th:jump_derivatives} yields that
   \[ \mathbf{T}_{1,2} \overline{\alpha} = \mathbf{T}_{1,2} \overline{\alpha}  + \overline{\mathbf{S}}_1 \overline{\alpha} = d \mathbf{J}^q_{1,2} H \in \mathcal{A}^\mathrm{e}(\riem_2).   \]
   Note that the computation is valid for any fixed value of $q$.
   
   To show that it is onto, let ${\beta} \in \mathcal{A}^\mathrm{e}(\riem_2)$, so that there is some $h \in \mathcal{D}(\riem_2)$ such that $\partial h = \beta$. 
   Setting $H= - \mathbf{O}_{2,1} h$, we have by Theorem \ref{th:J_same_both_sides} 
   \[  \dot{\mathbf{J}}_{1,2} \dot{\mathbf{O}}_{1,2} H = - \dot{\mathbf{J}}_{1,2} \dot{\mathbf{O}}_{2,1} h = + \dot{\mathbf{J}}_{2,2}  h = \dot{h}. \]
    where we have used Theorem \ref{th:jump_on_holomorphic}.  
   Since $\beta = \partial h = dh$, Theorem
   \ref{th:jump_derivatives} yields
   \[  \beta = d \dot{\mathbf{J}}_{1,2} \dot{H} = \mathbf{T}_{1,2} \overline{\partial} H + \overline{\mathbf{R}}_2 \overline{\mathbf{S}}_1 \overline{\partial} H.  \]
   Since the left hand side is holomorphic, $\overline{\mathbf{R}}_2 \overline{\mathbf{S}}_1 \overline{\partial} H=0$, so by Theorem \ref{th:S_kernel_range} $\overline{\partial} H \in
  [\overline{\mathbf{R}}_1 \overline{\mathcal{A}(\mathscr{R})}]^\perp$.
  
  Next we show that $\mathbf{T}_{1,2}$ is injective on $[\overline{\mathbf{R}}_1 \overline{\mathcal{A}(\mathscr{R})}]^\perp$. Let $\overline{\alpha}  \in [\overline{\mathbf{R}}_1 \overline{\mathcal{A}(\mathscr{R})}]^\perp$. Again since the components of $\riem_1$ are simply connected, we can assume that $\overline{\alpha} = \overline{\partial} H$ for some $H \in \overline{\mathcal{D}(\riem_1)}$. By Theorem
  \ref{th:Overfare_with_correction_functions} 
  \[   \dot{\mathbf{O}}_{2,1} \dot{\mathbf{J}}_{12}  \dot{H} = \dot{\mathbf{J}}_{1,1} \dot{H} - \dot{H}
         \]
  so differentiating and applying Theorems \ref{th:jump_derivatives}
  and \ref{th:S_kernel_range} we obtain
  \[  \mathbf{O}^e_{1,2} \mathbf{T}_{1,2} \overline{\alpha}  = d \mathbf{J}^q_{1,1} H = \partial H + \mathbf{T}_{1,1} \overline{\partial} H - dH = -\overline{\alpha} + \mathbf{T}_{1,1} \overline{\alpha}.   \]
  Thus $-\overline{\mathbf{P}}_1\mathbf{O}^\mathrm{e}_{2,1}$ is a left inverse for the restriction of $\mathbf{T}_{1,2}$ to $[\overline{\mathbf{R}}_1 \overline{\mathcal{A}(\mathscr{R})}]^\perp$. This proves injectivity, and since we already have surjectivity and boundedness, the restriction of $\mathbf{T}_{1,2}$ is invertible with inverse $-\overline{\mathbf{P}}_1\mathbf{O}^\mathrm{e}_{2,1}$ as claimed.
  \end{proof}
 
 We will improve and extend this theorem in different ways below. 
 The following corollary and lemma allows us to make use of the jump formula to examine cohomology classes.  
 
\begin{corollary} \label{co:specify_internal_periods}
 Let $\riem$ be a Riemann surface of type $(g,n)$ with internal homology basis $\{\gamma_1,\ldots,\gamma_{2g}\}$. For any constants $\lambda_1,\ldots,\lambda_{2g} \in \mathbb{C}$ there is an $\alpha \in \mathcal{A}(\riem)$ such that
 \[  \int_{\gamma_k} \alpha =  \lambda_k,  \ \ \ \ \ k=1,\ldots,2g.    \]
 The same claim obviously holds for $\overline{\mathcal{A}(\riem)}$.  
\end{corollary}
\begin{proof}  Sew on caps to $\riem$ to obtain a compact surface $\mathscr{R}$ of genus $g$, where $\riem_1$ are the caps and $\riem_2=\riem$.  So we prove the claim for $\riem= \riem_2$.  

 By the Hodge theorem applied to $\mathscr{R}$, there is a $\zeta = \xi + \overline{\eta} \in \mathcal{A}_{\text{harm}}(\mathscr{R})$ such that
 \[   \int_{\gamma_k} \zeta = \lambda_k      \]
 for $k=1,\ldots,n$.  Now since $\mathbf{S}_1 \mathbf{R}_1:\mathcal{A}(\mathscr{R}) \rightarrow \mathcal{A}(\mathscr{R})$ is an isomorphism by Theorem \ref{th:S_kernel_range},   there is a $\overline{\sigma} \in \overline{\mathcal{A}(\mathscr{R})}$ such that $\overline{\mathbf{S}}_1 \overline{\mathbf{R}}_1 \overline{\sigma} = \overline{\eta}$. 
 
 Since the components of $\riem_1$ are simply connected, $\overline{\mathbf{R}}_1 \overline{\sigma}$ is exact, so there is an $H \in \overline{\mathcal{D}_{\mathrm{harm}}(\riem_1)}$ such that $\overline{\partial} H = \overline{\mathbf{R}}_1 \overline{\sigma}$.  So 
 by Theorem \ref{th:jump_derivatives} we have 
 \[ \overline{\mathbf{R}}_2 \overline{\mathbf{S}}_1 \overline{\mathbf{R}}_1 \overline{\sigma}  + \mathbf{T}_{1,2} \overline{\mathbf{R}}_1 \overline{\sigma} = d \mathbf{J}^q_{1,2} H \in \mathcal{A}^\mathrm{e}_{\text{harm}}(\riem_2)  \]
 so 
 $\mathbf{R}_2 \xi - \mathbf{T}_{1,2} \overline{\mathbf{R}}_1 \overline{\sigma}$
 has the desired periods. 
\end{proof}

 \begin{lemma} \label{le:dbar_representative}
  Let $\riem$ be an arbitrary Riemann surface of type $(g,n)$.  Given any $\overline{\alpha} \in \overline{\mathcal{A}(\riem)}$, there is an $h \in \mathcal{D}_{\mathrm{harm}}(\riem)$ such that $\overline{\partial} h = \overline{\alpha}$.  Any other such $\tilde{h}$ is such that $h-\tilde{h} \in \mathcal{D}(\riem)$.  
  
  The corresponding statement holds for $\alpha \in \mathcal{A}(\riem)$, replacing $\overline{\partial}$ with $\partial$.  
 \end{lemma}
\begin{proof} Fix $\overline{\alpha} \in \overline{\mathcal{A}(\riem)}$.
 First, we show that there is a $\beta \in \mathcal{A}(\riem)$ such that $\overline{\alpha} - \beta$ is exact. 
 By Corollary \ref{co:specify_internal_periods}, for any $\overline{\alpha} \in \overline{\mathcal{A}(\riem)}$ we may find a ${\delta} \in {\mathcal{A}(\riem)}$ such that 
 \[  \int_{\gamma_k} ( \overline{\alpha} - {\delta} ) = 0    \]
 for $k=1,\ldots,2g$.  So it is enough to show that for any constants $\mu_1,\ldots,\mu_n$ such that $\mu_1 + \cdots + \mu_n = 0$  there is a ${\nu} \in {\mathcal{A}(\riem_2)}$ such that 
 \begin{equation} \label{eq:rep_lemma_temp}
    \int_{\partial_k \riem} \nu = \mu_k    
 \end{equation}
 for $k=1,\ldots,n$.  By Corollary \ref{co:boundary_periods_specified_starmeasure} there is a harmonic measure $d\omega \in \mathcal{A}_{\text{hm}}(\riem)$ such that $\ast d \omega$ satisfies \eqref{eq:rep_lemma_temp}.  Setting
 \[   \nu = \ast d \omega + i d\omega \in {\mathcal{A}(\riem)},      \]
 since $d\omega$ is exact, $\nu$ has the same periods as $\ast d \omega$. Setting $\beta = \delta + \nu$ proves the claim. 
 
 So let $\beta$ be such that $\overline{\alpha} - \beta$ is exact.  Letting $h$ be such that $d h= \overline{\alpha} - \beta$ we then have that $\overline{\partial} h=\overline{\alpha}$ as claimed. If $\overline{\partial} \tilde{h} = \overline{\alpha}$ then $\overline{\partial} (h-\tilde{h}) =0$ so $h-\tilde{h} \in \mathcal{D}(\riem)$.  
\end{proof}

 \begin{theorem} \label{th:Schiffer_cohomology}
  For any $\overline{\alpha} \in \overline{\mathcal{A}(\riem_1)}$,
  \begin{enumerate}
      \item  [$(1)$] $\mathbf{T}_{1,2} \overline{\alpha}+\overline{\mathbf{R}}_2 \overline{\mathbf{S}}_1 \overline{\alpha}$ is exact on $\riem_2$; and 
      \item [$(2)$] $-\overline{\alpha} + \mathbf{T}_{1,1} \overline{\alpha}+\overline{\mathbf{R}}_1 \overline{\mathbf{S}}_1 \overline{\alpha}$ is exact on $\riem_1$.
  \end{enumerate}
  If $\mathscr{R}$ has genus zero, then $\mathbf{T}_{1,2} \overline{\alpha}$ and $-\overline{\alpha} + \mathbf{T}_{1,2} \overline{\alpha}$ are exact. 
  
  The same statements apply to the complex conjugates, and all statements hold with $1$ and $2$ interchanged.  
 \end{theorem}
 \begin{proof}
  By Lemma \ref{le:dbar_representative} there is an $h \in \mathcal{D}_{\mathrm{harm}}(\riem_1)$ such that $\overline{\partial} h = \overline{\alpha}$. 
  Also, Theorem \ref{th:jump_derivatives} yields that
  \begin{equation} \label{eq:jump_cohomology_identity}
    d \mathbf{J}^q_1 h = \left\{ \begin{array}{ll}  \partial h + \mathbf{T}_{1,1} \overline{\partial} h + \overline{\mathbf{R}}_1 \overline{\mathbf{S}}_1   \overline{\partial} h  & \riem_1 \\
    \mathbf{T}_{1,2} \overline{\partial} h + \overline{\mathbf{R}}_2 
    \overline{\mathbf{S}}_1 \overline{\partial} h & \riem_2. 
    \end{array} \right.         
  \end{equation}
  Now using $dh= \partial h + \overline{\partial} h = \partial h + \overline{\alpha}$ we see that 
  \begin{equation} \label{eq:jump_cohomology_identity_11}
   d \mathbf{J}^q_{1,1} h =  dh - \overline{\alpha} + \mathbf{T}_{1,1} \overline{\alpha}  + \overline{\mathbf{R}}_1 \overline{\mathbf{S}}_1   \overline{\alpha}   
  \end{equation}
  and 
  \begin{equation} \label{eq:jump_cohomology_identity_12} 
       d \mathbf{J}^q_{1,2} h =  \mathbf{T}_{1,2} \overline{\alpha} + \overline{\mathbf{R}}_2 
    \overline{\mathbf{S}}_1 \overline{\alpha}. 
  \end{equation}
  This proves claims (1) and (2).  If $\mathscr{R}$ has genus zero, then the third claim follows from the fact that $K_\mathscr{R}=0$ (see Example \ref{ex:sphere_kernels}).  
  The remaining claims are obvious.
 \end{proof}
 
 This simple fact is surprisingly illuminating.  We list two immediate corollaries. 
 \begin{corollary}  \label{co:one_plus_T_periods}
  Let $\overline{\alpha} \in \overline{A(\riem_1)}$.  
  For any curve $c$ in $\riem_1$, 
  \begin{align*}
  \int_{c} (\overline{\alpha} - \mathbf{T}_{1,1} \overline{\alpha}) & =  \int_c \overline{\mathbf{S}}_1 \overline{\alpha} \\ & =  \left<  \overline{\alpha}, \ast \overline{\mathbf{R}}_1 H_c  \right>_{\riem_1} \\ & =  \int_c \left[ \mathbf{I} - \mathbf{T}_{1,1}^* \mathbf{T}_{1,1} - \mathbf{T}_{1,2}^* \mathbf{T}_{1,2} \right] \overline{\alpha} 
  \end{align*}
  where $H_c$ is associated to $c$ by \emph{(\ref{eq:H_forms_definition})}.  The same formulas hold with $1$ and $2$ interchanged, as do the complex conjugates.  
 \end{corollary} 
 \begin{proof}
  The first equality follows directly from Theorem \ref{th:Schiffer_cohomology}.  The second equality follows from the definition of $H_c$ and Theorem \ref{th:adjoint_identities}, observing that $\ast$ commutes with $\overline{\mathbf{R}}_1$.  The final equality follows from the identity $\mathbf{T}_{1,1}^* \mathbf{T}_{1,1} + \mathbf{T}_{1,2}^* \mathbf{T}_{1,2} = \mathbf{I} - \overline{\mathbf{R}}_1 \overline{\mathbf{S}}_1$ given in Theorem \ref{th:general_adjoint_double_identities}.  
 \end{proof}

 \begin{corollary}   \label{co:image_Tonetwo_periods}
  For any curve $c$ in $\riem_2$ and $\overline{\alpha} \in 
  \overline{\mathcal{A}(\riem_2)}$
  \begin{align*}
   -\int_{c} \mathbf{T}_{1,2}   \overline{\alpha} & =  \int_c \overline{\mathbf{S}}_2 \overline{\alpha} \\
   & = \left< \overline{\alpha}, \ast \overline{\mathbf{R}}_2 H_c \right>_{\riem_2}  \\
   & = \int_c \left[ \mathbf{I} - \mathbf{T}_{2,2}^* \mathbf{T}_{2,2} - \mathbf{T}_{2,1}^* \mathbf{T}_{2,1} \right] \overline{\alpha}.
  \end{align*}
  The same statement holds with $1$ and $2$ interchanged, and with complex conjugates.
 \end{corollary}
 \begin{proof}  The proof is identical to that of Corollary \ref{co:one_plus_T_periods}, except in the last step we use the identity
   $\mathbf{T}_{2,2}^* \mathbf{T}_{2,2} + \mathbf{T}_{2,1}^* \mathbf{T}_{2,1} + \overline{\mathbf{R}}_2 \overline{\mathbf{S}}_2= \mathbf{I}$
   of Theorem \ref{th:general_adjoint_double_identities}.  
 \end{proof}

   If $\riem_2$ is connected, recall Definition \ref{def:exact overfare}: for 
 $\alpha \in \mathcal{A}^{\mathrm e}_{\text{harm}}(\riem_2)$, we define
 \[   \mathbf{O}^{\mathrm{e}}_{2,1} \alpha= d \mathbf{O}_{2,1} h      \]
 for $dh = \alpha$.  
  \begin{proposition} \label{prop:exact_transmitted_jump} Assume that $\riem_2$ is connected.  
  For $\overline{\alpha} \in \overline{\mathcal{A}(\riem_1)}$ we have
  \[  \mathbf{O}^{\mathrm{e}}_{2,1} \left[ \mathbf{T}_{1,2} \overline{\alpha} + \overline{\mathbf{R}}_{2} \overline{\mathbf{S}}_1 \overline{\alpha} \right] = -\overline{\alpha} + \mathbf{T}_{1,1} \overline{\alpha} + \overline{\mathbf{R}}_{1} \overline{\mathbf{S}}_1 \overline{\alpha}. \]
  In particular, if $\overline{\alpha} \in (\overline{\mathbf{R}}_1 \overline{\mathcal{A}(\mathscr{R})})^\perp$ we have
   \[  \mathbf{O}^{\mathrm{e}}_{2,1}  \mathbf{T}_{1,2} \overline{\alpha} = - \overline{\alpha} + \mathbf{T}_{1,1} \overline{\alpha}.   \]
   The complex conjugate statements hold, as do the statements with the roles of $1$ and $2$ interchanged.
  \end{proposition}
  \begin{proof}
   By Lemma \ref{le:dbar_representative} there is an $h \in \mathcal{D}_{\mathrm{harm}}(\riem_1)$ such that $\overline{\partial} h = \overline{\alpha}$.   By Theorem \ref{th:S_kernel_range} $\mathbf{S}_1 \overline{\alpha} = 0$. Differentiating both sides of the first expression appearing in Theorem \ref{th:Overfare_with_correction_functions} part (1)
   proves the claim.
  \end{proof}
 \end{subsection}
 \begin{subsection}{Kernel and image of the Schiffer operator $\mathbf{T}_{1,2}$} \label{se:Schiffer_kernel_image}
 
 We require a generalization of Theorem \ref{th:Mohammad_isomorphism}. Namely, we would like to characterize the kernel and image of $\mathbf{T}_{1,2}$ in general. We begin with a partial characterization.
 
 \begin{theorem} \label{th:general_T12_kernel_on_perp}  Assume that $\riem_2$ is connected.  Then 
  \[ \mathrm{Ker} ( \mathbf{T}_{1,2}) \cap [\overline{\mathbf{R}}_1 \overline{\mathcal{A}(\mathscr{R})}]^\perp = \{0 \}.   \] 
 \end{theorem}
 \begin{proof}
  Assume that $\overline{\alpha} \in \mathrm{Ker} ( \mathbf{T}_{1,2}) \cap [\overline{\mathbf{R}}_1 \overline{\mathcal{A}(\mathscr{R})}]^\perp$. By Lemma \ref{le:dbar_representative} there is a $H \in \mathcal{D}_{\text{harm}}(\riem_1)$ such that $\overline{\partial} H = \overline{\alpha}$. We have $\overline{\mathbf{S}}_1 \overline{\partial} H=0$ by Theorem \ref{th:S_kernel_range} so by Theorem \ref{th:jump_derivatives}
  \[  d \mathbf{J}^q_{1,2} H = \mathbf{T}_{1,2} \overline{\alpha} + \overline{\mathbf{R}}_2 \overline{\mathbf{S}}_1 \overline{\alpha} = 0. \]
  Therefore $\mathbf{J}^q_{1,2} H$ is constant, from which it follows that $\mathbf{O}_{2,1} \mathbf{J}^q_{1,2} H$ is constant. By Theorem \ref{th:Overfare_with_correction_functions} we have
  \[  d(H - \mathbf{J}^q_{1,1} H) = - d \mathbf{O}_{2,1} \mathbf{J}^q_{1,2} H = 0   \]
  so again using Theorem \ref{th:jump_derivatives} and the fact that $\overline{\mathbf{S}}_1 \overline{\partial} H=0$ we obtain 
  \[ \partial H + \overline{\partial} H - \mathbf{T}_{1,1} \overline{\partial} H = 0,    \]
  so equating holomorphic and anti-holomorphic parts
  \[  \overline{\partial} H =0.  \]
  This completes the proof.
 \end{proof}
 
 We also have the following.
 \begin{theorem} \label{th:general_T12_on_perp_surjective_to_exact} 
 Assume that $\riem_2$ is connected.  
 The image of $[\overline{\mathbf{R}}_1 \overline{\mathcal{A}(\mathscr{R})}]^\perp$ under $\mathbf{T}_{1,2}$ is $\mathcal{A}^{\mathrm{e}}(\riem_2)$. 
 \end{theorem}
 \begin{proof}
  Given any $\beta \in \mathcal{A}^\mathrm{e}(\riem_2)$, let $h \in \mathcal{D}(\riem_2)$ be such that $\partial h = \beta$, which exists by conjugating Lemma \ref{le:dbar_representative}.  Note that $h$ is not necessarily uniquely defined. Set $H=-\mathbf{O}_{2,1}h$. Applying Theorems \ref{th:J_same_both_sides} and \ref{th:jump_on_holomorphic} we obtain that
  \[  \dot{\mathbf{J}}_{1,2} \dot{H} =  \dot{\mathbf{J}}_{2,2} h = \dot{h}.  \]
   Differentiating using Theorem \ref{th:jump_derivatives} we see that
  \[  \mathbf{T}_{1,2} \overline{\partial} H + \overline{\mathbf{R}}_2 \overline{\mathbf{S}}_1 \overline{\partial} H = \beta.  \]
  Since $\beta$ is holomorphic $\overline{\mathbf{R}}_2 \overline{\mathbf{S}}_1 \overline{\partial} H = 0$ so $\overline{\mathbf{S}}_1 \overline{\partial} H = 0$ and hence $\overline{\partial} H \in [\overline{\mathbf{R}}_1 \overline{\mathcal{A}(\mathscr{R})}]^\perp$ by Theorem \ref{th:S_kernel_range}. Furthermore $\mathbf{T}_{1,2} \overline{\partial} H = \beta$ completing the proof.
 \end{proof}
 
 Theorems \ref{th:general_T12_on_perp_surjective_to_exact} and \ref{th:general_T12_kernel_on_perp} taken together generalize Theorem \ref{th:Mohammad_isomorphism} and \cite[Theorem 4.22]{Schippers_Staubach_Plemelj}. We will extend it still further below. For now, we observe the following corollary.
 Recall the projections $\mathbf{P}_k= \mathbf{P}_{\riem_k}$ and 
 $\overline{\mathbf{P}}_k = \overline{\mathbf{P}}_{\riem_k}$ defined by 
 \eqref{eq:hol_antihol_projections_Bergman}.
 \begin{corollary} \label{co:connected_get_iso}
  Let $\riem_2$ be connected.  Then the restriction of $\mathbf{T}_{1,2}$ to $[\overline{\mathbf{R}}_1 \overline{\mathcal{A}(\mathscr{R})}]^\perp$ is an isomorphism onto $\mathcal{A}^\mathrm{e}(\riem_2)$, with inverse $-\overline{\mathbf{P}}_1 \mathbf{O}^\mathrm{e}_{2,1}$. 
 \end{corollary}
 \begin{proof} 
  The restriction of $\mathbf{T}_{1,2}$ is surjective by Theorem \ref{th:general_T12_on_perp_surjective_to_exact}. Since $\riem_2$ is connected, any function $\omega$ with the bridge property must have the same constant value on each boundary of $\riem_1$, and hence must be constant. So the kernel is trivial.  
  
  Observe that since $\riem_2$ is connected $\mathbf{O}^\mathrm{e}_{2,1}$ is well-defined by the requirement that $d \mathbf{O}_{2,1} = \mathbf{O}^\mathrm{e}_{2,1}$. The fact that this is the inverse follows as in previous proofs. Let $H \in \mathcal{D}_{\text{harm}}(\riem_1)$ be such that $\overline{\partial} H = \overline{\alpha}$. By Theorems \ref{th:Overfare_with_correction_functions},  \ref{th:jump_derivatives}, and the fact that $\overline{\mathbf{S}}_1 \overline{\alpha}=0$ we see that 
  \[ \mathbf{O}^\mathrm{e}_{2,1} \overline{\alpha} = -\overline{\alpha} + \mathbf{T}_{1,1} \overline{\alpha}. \]
  The claim follows immediately.
 \end{proof}

 In order to determine the image of $\mathbf{T}_{1,2}$, we define certain natural subspaces of $\mathcal{A}(\riem_2)$.  Assume that $\riem_2$ is connected (but not necessarily $\riem_1$).  
 Let $c^k_1,\ldots,c^k_{m_k}$ be a fixed homology basis of simple closed curves for $\riem_k$ for $k=1,2$.  
 
 We then have a linear map 
 \begin{align}\label{capital Xi}
  \Xi_k: \mathcal{A}(\mathscr{R}) & \rightarrow \mathbb{C}^{m_k} \\\nonumber
  u & \mapsto \left( \int_{c^k_1} u,\ldots, \int_{c^k_{m_k}} u \right)
 \end{align}
 and the linear map 
 \begin{align}\label{capital Xibar}
  \overline{\Xi_k}: \overline{\mathcal{A}(\mathscr{R})} & \rightarrow \mathbb{C}^{m_k} \\\nonumber
  \overline{u} & \mapsto \left( \int_{c^k_1} \overline{u},\ldots, \int_{c^k_{m_k}} \overline{u} \right). 
 \end{align}
 We then define the subspaces 
 \begin{align}\label{capital Xk}
    X_k & = \text{Im}(\Xi_k) \subseteq \mathbb{C}^{m_k} \\
    \overline{X}_k & = \text{Im}(\overline{\Xi}_k) \subseteq \mathbb{C}^{m_k}.
 \end{align}
 Although $\Xi_k$ and $\overline{\Xi}_k$ depend on the choice of basis, $X_k$ and $\overline{X}_k$ do not.  
 It will sometimes be convenient to choose specific homology bases and an ordering.  Note that some curves in the homology base of $\mathcal{A}(\mathscr{R})$ may appear in the homology base of both $\riem_1$ and $\riem_2$, and that some curves in the homology base of $\mathcal{A}(\mathscr{R})$ might not appear in the homology base of either $\riem_1$ or $\riem_2$.  

 
  
 We then define the following subspaces of $\mathcal{A}(\riem_2)$: 
 \begin{equation}\label{svensk a}
\gls{gota}({\riem}_k) = \left\{  \alpha \in \mathcal{A}({\riem}_k) : \left( \int_{c^k_1} \alpha ,\ldots,\int_{c^k_{m_k}} \alpha  \right) \in X_k  \right\}.    
 \end{equation}  
 and 
 \begin{equation}\label{svensk abar}
    \mathrm{\textgoth{A}}^-(\riem_k) = \left\{  \alpha \in \mathcal{A}(\riem_k) : \left( \int_{c^k_1} \alpha ,\ldots,\int_{c^{k}_{m_k}} \alpha  \right) \in \overline{X}_k         \right\}  
 \end{equation}     
 for $k=1,2$.  
 Note that it is most certainly {\bf not} true that $\mathrm{\textgoth{A}}^-(\riem_k) = \overline{\mathrm{\textgoth{A}}(\riem_k)}$.   
 
 The definition immediately implies that
 \begin{proposition}  For $k=1,2$, given $\alpha \in \mathcal{A}(\riem_k)$, it holds that
 $\alpha \in \emph{{\textgoth{A}}}(\riem_k)$ if and only if $\alpha$ is in the same cohomology class as an element of $\mathbf{R}_k\mathcal{A}(\mathscr{R})$.  Similarly, $\alpha \in \emph{{\textgoth{A}}}^-(\riem_k)$ if and only if $\alpha$ is in the same cohomology class as an element of $\overline{\mathbf{R}}_k \overline{\mathcal{A}(\mathscr{R})}$.  
 \end{proposition}
 
 Another useful fact is the following.
 \begin{proposition}  \label{pr:XplusXbar_is_all}  
  For $k=1,2$ 
  \[  X_k {+} \overline{X}_k = \mathbb{C}^{m_k}.  \]
 \end{proposition}
 \begin{proof}
   This is an immediate consequence of the Hodge theorem, which says that every cohomology class on $\mathscr{R}$ has a representative in $\mathcal{A}_{\text{harm}}(\mathscr{R})$.  Thus every possible configuration of periods of $c^k_1,\ldots,c^k_{m_k}$ in $\mathscr{R}$ (and so, in particular of of $c^k_1,\ldots,c^k_{m_k}$ in $\riem_k$)  can be attained by an element of  $\mathcal{A}_{\text{harm}}(\mathscr{R})$.  
 \end{proof}
 

 The following theorem establishes the behaviour of $\mathbf{T}_{1,2}$ on its entire domain.    

 \begin{theorem} 
  \label{th:Tonetwo_iso_improved}  Assume that $\riem_2$ is connected.  
  Let 
  \[ \overline{W}_1 = \left\{\overline{\mu} \in \overline{\mathbf{R}}_1 \overline{\mathcal{A}(\mathscr{R})};\,\, \overline{\mathbf{R}}_2 \overline{\mathbf{S}}_1 {\overline{\mu} } \in \overline{\mathcal{A}^{\mathrm{e}}(\riem_2)} \right\}  \]
    We have 
    \begin{equation*}
      \mathrm{Im}(\mathbf{T}_{1,2})  = \emph{{\textgoth{A}}}^-(\riem_2)
    \end{equation*}
    and
    \begin{equation*}
     \mathrm{Ker}(\mathbf{T}_{1,2}) \cong  \overline{W}_1.
    \end{equation*}
    The same claim follows for the complex conjugates and with $1$ and $2$ interchanged.
 \end{theorem}
 \begin{proof}
    It follows from Theorem \ref{th:Schiffer_cohomology} that $\text{Im}(\mathbf{T}_{1,2}) \subseteq \mathrm{\textgoth{A}}^-(\riem_2)$.  We show that $\mathrm{\textgoth{A}}^-(\riem_2)  \subseteq  \text{Im}(\mathbf{T}_{1,2})$. 
 Let $\beta \in \mathrm{\textgoth{A}}^-(\riem_2)$. Again applying Theorem  \ref{th:Schiffer_cohomology}, since $\overline{\mathbf{S}}_1 \overline{\mathbf{R}}_1$ is an isomorphism by Theorem \ref{th:S_an_isomorphism}, we can find $\overline{\gamma} \in \overline{\mathcal{A}(\mathscr{R})}$ such that 
   \[  \beta - \mathbf{T}_{1,2}   \overline{\mathbf{R}}_1 \overline{\gamma}  \in \mathcal{A}^e(\riem_2).  \]
 By Theorem \ref{th:general_T12_on_perp_surjective_to_exact} there is an $\overline{\alpha} \in (\overline{\mathbf{R}}_1 \overline{\mathcal{A}(R)})^\perp$ such that 
   \[  \mathbf{T}_{1,2} \overline{\alpha} = \beta - \mathbf{T}_{1,2}  \overline{\mathbf{R}}_1 \overline{\gamma} \]
   which completes the proof of the first claim.
   
   We now prove the second claim. 
   Let $\overline{\mu} \in \overline{W}_1$. Since $\mathbf{T}_{1,2} \overline{\mu} + \overline{\mathbf{R}}_2 \overline{\mathbf{S}}_1 \overline{\mu}$ is exact by Theorem \ref{th:Schiffer_cohomology}, so is $\mathbf{T}_{1,2} \mu$.  Thus by Theorem \ref{th:general_T12_on_perp_surjective_to_exact} there is an $\overline{\alpha} \in  (\overline{\mathbf{R}}_1 \overline{\mathcal{A}(R)})^\perp$ such that $\mathbf{T}_{1,2} \overline{\alpha} = - \mathbf{T}_{1,2} \overline{\mu}$ so $\overline{\alpha} + \overline{\mu} \in \mathrm{Ker} (\mathbf{T}_{1,2})$.  We define 
   \begin{align*}
      \Phi:\overline{W}_1 & \rightarrow \mathrm{Ker}(\mathbf{T}_{1,2}) \\
      \overline{\mu} & \mapsto \overline{\alpha} + \overline{\mu}.
   \end{align*}
   This is well-defined, since if $\overline{\alpha} + \overline{\mu}$ and $\overline{\beta} + \overline{\mu}$ are both in $\mathrm{Ker}( \mathbf{T}_{12})$ for $\overline{\beta},\overline{\alpha} \in [\overline{\mathbf{R}}_1 \overline{\mathcal{A}(\mathscr{R})}]^\perp$ then $\overline{\alpha} - \overline{\beta} \in [\overline{\mathbf{R}}_1 \overline{\mathcal{A}(\mathscr{R})}]^\perp \cap \mathrm{Ker}( \mathbf{T}_{1,2})$ so $\overline{\alpha} -\overline{\beta} =0$ by Theorem \ref{th:general_T12_kernel_on_perp}.  
   
   This map is surjective. Assume that $\overline{\gamma}  \in \mathrm{Ker}(\mathbf{T}_{1,2})$.  Write $\overline{\gamma} = \overline{\alpha} + \overline{\mu}$ for $\overline{\alpha} \in  (\overline{\mathbf{R}}_1 \overline{\mathcal{A}(R)})^\perp$ and $\overline{\mu} \in \overline{\mathbf{R}}_1 \overline{\mathcal{A}(\mathscr{R})}$. Since $\mathbf{T}_{1,2} \overline{\alpha} = -\mathbf{T}_{1,2} \overline{\mu}$ and the former is exact by Theorem \ref{th:general_T12_on_perp_surjective_to_exact} we see that $\mathbf{T}_{1,2} \overline{\mu}$ is exact. Thus by Theorem \ref{th:Schiffer_cohomology} $\overline{\mathbf{R}}_2 \overline{\mathbf{S}}_1 \overline{\mu}$ is exact.  So $\overline{\mu} \in \overline{W}_1$ and $\Phi(\overline{\gamma}) = \overline{\gamma}$. 
   
   This map is injective.  Assume that $\Phi(\overline{\mu})=0$. Then $\overline{\alpha} + \overline{\mu} =0$.  Using Theorem \ref{th:general_T12_kernel_on_perp} and the fact that $\overline{\alpha} \in [\overline{\mathbf{R}}_1 \overline{\mathcal{A}(\mathscr{R})}]^\perp$ we obtain that $\overline{\mu} \in [\overline{\mathbf{R}}_1 \overline{\mathcal{A}(\mathscr{R})}]^\perp$.  So $\overline{\mu}=0$. 
 \end{proof}
 \begin{remark} \label{re:explicit_Tpartialinverse_construction}
  The element $\overline{\alpha}$ corresponding to $\overline{\mu}$ in the definition of $\Phi$ can be constructed explicitly as follows.  Given $\overline{\mu} \in \overline{W}_1$, since $\mathbf{T}_{1,2} \overline{\mu}$ is exact, there is an $h \in \mathcal{D}(\riem_2)$ such that $\partial h = \mathbf{T}_{1,2} \overline{\mu}$. Set $H=  \mathbf{O}_{2,1} h$.   So 
  \[  \mathbf{J}_{1,2}^q H = -\mathbf{J}^q_{2,2} h = -h + c \]
  for some $c$ which is constant on connected components, by Theorems \ref{th:J_same_both_sides} and \ref{th:jump_on_holomorphic}.  So applying Theorem \ref{th:jump_derivatives} we obtain 
  \[  \mathbf{T}_{1,2} \overline{\partial} H + \overline{\mathbf{R}}_2  \overline{\mathbf{S}}_1 \overline{\partial} H =  d \mathbf{J}_{1,2} H = - \partial h = -\mathbf{T}_{1,2} \overline{\mu}. \]
  But since the right hand side is holomorphic, we must have that $\overline{\mathbf{R}}_2  \overline{\mathbf{S}}_1 \overline{\partial} H =0$ so by analytic continuation and Theorem \ref{th:S_kernel_range} we see that $\overline{\alpha} =\overline{\partial} H \in [\overline{\mathbf{R}}_1 \overline{\mathcal{A}(\mathscr{R})}]^\perp$ and
  \[  \mathbf{T}_{1,2} (\overline{\alpha} + \overline{\mu} ) =0. \]
  We may summarize this by saying that 
  \[  \Phi =  \mathbf{I} - \overline{\mathbf{P}}_1 d \mathbf{O}_{2,1} d^{-1} \mathbf{T}_{1,2}     \]
  observing that this is well-defined by the proof of the theorem.
 \end{remark}

   This has the following important consequence.
   \begin{corollary}  \label{co:T_plus_extra_surjective_capped}  Let $\riem_2$ be capped by $\riem_1$.  Then $\mathrm{ker}(\mathbf{T}_{1,2})$ is trivial.
   Furthermore, 
    any $\alpha \in \mathcal{A}(\riem_2)$ can be written 
    \[  \alpha = \mathbf{T}_{1,2} \overline{\gamma} + \mathbf{R}_2  \tau + \partial \omega  \]
    for unique $\overline{\gamma} \in \overline{\mathcal{A}(\riem_1)}$, $\tau \in \mathcal{A}(\mathscr{R})$, and $d\omega \in \mathcal{A}_{\mathrm{hm}}(\riem_2)$.     That is 
    \[    \mathcal{A}(\riem_2) = \emph{\textrm{\textgoth{A}}}^{-}(\riem_2) \oplus \mathbf{R}_2 \mathcal{A}(\mathscr{R}) \oplus \partial \mathcal{D}_{\mathrm{hm}}(\riem_2).       \]
    The same claim holds for complex conjugates.
   \end{corollary}
   \begin{proof}  
    Since any exact form on $\mathscr{R}$ is zero, we see that $\overline{W}_1 = \{0 \}$, so the kernel is zero by Theorem \ref{th:Tonetwo_iso_improved}.  
    Since the periods of elements of $\mathbf{R}_2 \mathscr{A}_{\mathrm{harm}}(\mathscr{R})$ are zero around the boundary curves $\partial_k \riem_2$ for all $k$, we see that $\ast \mathcal{A}_{\mathrm{hm}}(\riem_2)$ and $\mathbf{R}_2 \mathscr{A}_{\mathrm{harm}}(\mathscr{R})$ are linearly independent. Using the decomposition 
    \[  \partial \omega = \frac{1}{2}\left( d\omega + i \ast d\omega \right) \]
    shows that $\partial \mathcal{D}_{\mathrm{harm}}(\riem_2)$ and $\mathbf{R}_2 \mathcal{A}_{\mathrm{harm}}(\mathscr{R})$ are linearly independent. 
    
    Now $X_1$ and $\overline{X}_1$ are linearly independent, since each has dimension $g$, and the dimension of $X_1 + \overline{X}_1$ is $2g$.  Thus since $\text{Im}(\mathbf{T}_{1,2}) = {\mathrm{\textgoth{A}}}^{-}(\riem_2)$ by Theorem \ref{th:Tonetwo_iso_improved}, this proves that 
    \[  \mathcal{A}(\riem_2) = {\textrm{\textgoth{A}}}^{-}(\riem_2)  + \mathbf{R}_2 \mathcal{A}(\mathscr{R}) + \partial \mathcal{D}_{\mathrm{hm}}(\riem_2).  \]
    decomposition. Linear independence proves that the decomposition is a direct sum, and uniqueness of $\tau$. The uniqueness of $\overline{\partial} \omega$ follows from Theorem \ref{th:period_matrix_invertible}, and uniqueness of $\overline{\gamma}$ follows from triviality of the kernel of $\mathbf{T}_{1,2}$.  
   \end{proof}
   
\end{subsection}
\begin{subsection}{Index of the Schiffer operator} \label{th:index_and_examples}

 In the following, we first observe that 
  \begin{equation} \label{eq:kernel_Xi_and_W}
   \mathrm{dim} \, \mathrm{Ker}(\Xi_1) = \mathrm{dim}\, W_2. 
  \end{equation} 
  This follows directly from the definitions together with the fact that $\mathbf{R}_1 \mathbf{S}_1$ is an isomorphism by Theorem \ref{th:S_kernel_range}.  
 The same claim holds with $1$ and $2$ interchanged, as does the complex conjugate. 
 
 \begin{theorem}  \label{th:X_k_analysis}
  Assume that $\riem_1$ and $\riem_2$ are connected. Let $g$ be the genus of $\mathscr{R}$ and $g_1$,$g_2$ be the genuses of $\riem_1$ and $\riem_2$ respectively. 
  \begin{equation}  \label{eq:dimX1_and_W_2}
   \mathrm{dim}\, X_1   = g - \mathrm{dim} W_2 
  \end{equation}
  and
  \begin{equation}  \label{eq:dimX1_and_X1intersectX1bar}
   \mathrm{dim}\, X_1  = g_1 + \frac{n-1}{2} + \frac{1}{2} \, \mathrm{dim}(X_1 \cap \overline{X}_1). 
  \end{equation}
  The same claims hold with $1$ and $2$ interchanged. 
 \end{theorem}
 \begin{proof}
   The first claim follow from the fact that 
   \[  X_1 = \mathrm{Im} (\Xi_1), \ \ \  \]
   equation \eqref{eq:kernel_Xi_and_W}, and the fact that $\mathcal{A}(\mathscr{R})$ has dimension $g$.   
   
   To prove the second claim, first observe that every homology class in $\riem_1$ is represented by a homology class in $\mathscr{R}$ (note that this depends on the assumptions on the configuration $\riem_1$, $\riem_2$, $\mathscr{R}$).  So 
   \[ \mathrm{dim} (X_1+ \overline{X}_1) = 2g_1 + n-1.  \] 
   Using this together with the fact that 
   \[  \mathrm{dim}\, (X_1+\overline{X}_1) = \mathrm{dim}\, X_1 + \mathrm{dim}\,\overline{X}_1 - \mathrm{dim}\, (X_1 \cap \overline{X}_1) = 2 \, \mathrm{dim}\, X_1 - \mathrm{dim}\, (X_1 \cap \overline{X}_1)  \]
   proves the claim.
 \end{proof}
 Combining these two claims, together with properties of the Schiffer operator, results in the following. 
 \begin{theorem}  \label{th:Schiffer_index_theorem}
  Assume that $\riem_1$ and $\riem_2$ are connected. Let $g$ be the genus of $\mathscr{R}$ and $g_1$,$g_2$ be the genuses of $\riem_1$ and $\riem_2$ respectively. Then 
  \[ \mathrm{Index} (\mathbf{T}_{1,2}) = g_1 - g_2.  \]
  The same claim holds with $1$ and $2$ switched.  
 \end{theorem}
 \begin{proof}
  Combining the two equations in Theorem \ref{th:X_k_analysis}, we obtain \begin{align}  \label{eq:index_temp_zero}
      g- \mathrm{dim}\, W_2 & = g_1 + \frac{n-1}{2} + \frac{1}{2} \mathrm{dim}\,(X_1 \cap \overline{X}_1 ) \nonumber \\ 
      g- \mathrm{dim}\, W_1 & = g_2 + \frac{n-1}{2} + \frac{1}{2} \mathrm{dim}\,(X_2 \cap \overline{X}_2 ) 
  \end{align}
  so since $g_1 + g_2 + n-1 = g$ we obtain 
  \begin{align} \label{eq:index_proof_temp} 
     \mathrm{dim}\, W_2 + \frac{1}{2} \mathrm{dim}\, (X_1\cap \overline{X}_1) & = g_2 + \frac{n-1}{2} \nonumber \\ 
      \mathrm{dim}\, W_1 + \frac{1}{2} \mathrm{dim}\, (X_2\cap \overline{X}_2) & = g_1 + \frac{n-1}{2}
  \end{align}
  
  Next we compute the dimension of the cokernel of $\mathbf{T}_{1,2}$. By Theorem \ref{th:Tonetwo_iso_improved}, we have that all harmonic forms with periods in $\overline{X}_2$ are in the image of $\mathbf{T}_{1,2}$, from which we conclude that
  \begin{align*}
   \mathrm{dim}\, \mathrm{Coker}( \mathbf{T}_{1,2}) & = 2g_2 + n-1 -
   \mathrm{dim}\, X_2  \\
   & = g_2 + \frac{n-1}{2} - \frac{1}{2} \mathrm{dim}\,(X_2 \cap \overline{X}_2)
  \end{align*}
  where we have used equation \eqref{eq:dimX1_and_X1intersectX1bar} with $1$ replaced by $2$.  However, since
  $\mathbf{T}_{1,2}^* = \overline{\mathbf{T}}_{2,1}$, by Theorem \ref{th:adjoint_identities} we have 
  \[  \mathrm{dim}\, \mathrm{Coker}( \mathbf{T}_{1,2}) = \mathrm{dim}\, \mathrm{Ker}\, \mathbf{T}_{2,1} = \mathrm{dim}\, W_2 \]
  where we have used Theorem \ref{th:Tonetwo_iso_improved}. Thus 
  \begin{equation*}
   \mathrm{dim}\, W_2 = g_2 + \frac{n-1}{2} - \frac{1}{2} \mathrm{dim}\,(X_2 \cap \overline{X}_2)
  \end{equation*} 
  which upon comparison with \eqref{eq:index_proof_temp} yields that
  \begin{equation} \label{eq:intersections_X12_equal}
     \frac{1}{2} \, \mathrm{dim}\,(X_1 \cap \overline{X}_1) =  
  \frac{1}{2} \, \mathrm{dim}\,(X_2 \cap \overline{X}_2).   
  \end{equation} 
  Now using this fact together with equation \eqref{eq:index_proof_temp} we obtain  
  \begin{align*}
     \mathrm{dim}\, \mathrm{Ker}( \mathbf{T}_{1,2}) - \mathrm{dim}\, \mathrm{Coker}( \mathbf{T}_{1,2}) & = \mathrm{dim}\, \mathrm{Ker}( \mathbf{T}_{1,2}) - \mathrm{dim}\, \mathrm{Ker}( \overline{\mathbf{T}}_{2,1}) \\
     & =\mathrm{dim}\, W_1 - \mathrm{dim}\, W_2 \\
     & = g_1 - g_2 
  \end{align*}
  as claimed. 
  
  To prove the final claim, just switch the roles of $1$ and $2$ in the proof, which can be done by the symmetry of the conditions.
 \end{proof}
 
 \begin{remark}  Under the same assumptions, the proof also shows the following interesting facts. 
  By equations \eqref{eq:index_temp_zero} and \eqref{eq:index_proof_temp}, together with the fact that
  $\mathrm{dim}\, (X_1 \cap \overline{X}_1) = \mathrm{dim}\, (X_2 \cap \overline{X}_2)$ by \eqref{eq:intersections_X12_equal}, we obtain  
  \[ g_1 -g_2 = \mathrm{dim}\, W_1 - \mathrm{dim}\,W_2 = \mathrm{dim}\, X_1 - \mathrm{dim}\, X_2.  \] 
  Furthermore, \eqref{eq:index_proof_temp} implies that 
  \[  \mathrm{dim}\, (X_k \cap X_k) = n - 1 \ \  \mathrm{mod}\, 2  \]
  for $k=1,2$.  
 \end{remark}
 
 We also have the following.
 \begin{theorem} \label{th:index_capped_surface}
  Let $\riem_2$ be a surface of genus $g$ capped by $\riem_1$, where $\riem_1$ has $n$ connected components.  Then 
  \[  \mathrm{Index} ( \mathbf{T}_{1,2}) = - \mathrm{Index} ( \mathbf{T}_{2,1}) = 1-n-g.  \]
 \end{theorem}
 \begin{proof}
  The fact that the index of $\mathbf{T}_{1,2}$ is $n-1-g$ follows directly from Corollary \ref{co:T_plus_extra_surjective_capped} together with the facts that the dimension of $\partial \mathcal{D}_{\mathrm{harm}}(\riem_2)$ is $n-1$ and the dimension of $\mathcal{A}(\mathscr{R})$ is $g$. Using \ref{th:adjoint_identities} we have 
  \[ \mathrm{Index} ( \mathbf{T}_{2,1}) = - \mathrm{Index} ( \mathbf{T}_{2,1}^* )=  - \mathrm{Index} ( \overline{\mathbf{T}}_{1,2}) = - \mathrm{Index} ( {\mathbf{T}}_{1,2})\]
  which completes the proof.
 \end{proof}
 
 {The index theorems above connect conformally invariant quantities (the index of $\mathbf{T}_{1,2}$) to topologically invariant quantities. 
 The Schiffer operators are conformally invariant, as we saw in \eqref{eq:Schiffer_operators_conformally_invariant}. Thus their spectra, kernels, images, and indices are all conformally invariant. 
 Because the spaces $W_k$ are conformally but not obviously topologically invariant, it is interesting that they cancel in the proof of Theorem \ref{th:Schiffer_index_theorem}, and only topological data remains. The question then arises: are the dimensions of the cokernel and kernel themselves topological invariants? In other words, is it possible to choose topologically equivalent configurations with distinct dimensions for the cokernel and kernel of $\mathbf{T}_{1,2}$?  Either answer would be of great interest. 
 
 In fact, the kernels and cokernels are related to the image of the period map of $\mathscr{R}$, restricted to homology curves in $\riem_1$ or $\riem_2$.  The following example, the case of a genus two torus sliced by one curve, illustrates this. We also explicitly compute $W_1$ (defined implicitly in Theorem \ref{th:Tonetwo_iso_improved}) for this example, which turns out to be trivial. Although it therefore does not provide a counterexample to the topological invariance of the cokernel and kernel, the approach might however be a promising way to seek one. We leave this as an open problem.
 }

{Returning to the problem of computation of $W_1$, following \cite{Roydenperiods}, we start by recalling some basic facts regarding periods and related matrices. Let the compact Riemann surface \(\mathscr{R}\) have a canonical homology basis \(\left\{A_{j}, B_{j}\right\},\) where the \(A_{j}\) and \(B_{j}\)
are smooth simple closed curves with intersection numbers given
by
$$
\begin{array}{l}{\left[A_{j} \times B_{k}\right]=\delta_{j k}} \\ {\left[A_{j} \times A_{k}\right]=0} \\ {\left[B_{j} \times B_{k}\right]=0}\end{array}.
$$
When we are not interested in the intersection properties, we set $C_{j+g}=B_{j},\, 1 \leq j \leq g.$



If $\alpha$ and $\alpha^{\prime}$ are two closed forms with periods $a_{j}, b_{j}$ and $a_{j}^{\prime}, b_{j}^{\prime},$ respectively, around $\left\{A_{j}, B_{j}\right\}$, then  \emph{Riemann's bilinear relations} (or Riemann's period relations) state that
\begin{equation}\label{Riemann relations}
    \int_{\mathscr{R}} \alpha \wedge \alpha^{\prime}=\sum_{j}\left(a_{j} b_{j}^{\prime}-a_{j}^{\prime} b_{j}\right).
\end{equation}

Now let \(\omega_{j}\) be the harmonic one-form on \(\mathscr{R}\) whose period
around \(C_{k}\) is \(\delta_{j k} \) and \(w_{j}\) be the holomorphic one-form on \(\mathscr{R}\)  whose periods around \(A_{k}\) are \(\delta_{j k} .\) Then the entries of the so-called \emph{Riemann matrix} \(\Pi=\left[\pi_{j k}\right]\)
are given by
$$
\pi_{j k}=\int_{B_{k}} w_{j}.
$$
Since \(w_{j} \wedge w_{k}=0,\) 
\eqref{Riemann relations} yields that
\begin{equation}\label{symmetry of Riemann}
    0=\int w_{j} \wedge w_{k}=\pi_{k j}-\pi_{j k}.
\end{equation}

Thus \(\Pi\) is a symmetric matrix. 


Note also that, since period of \(w_{j}\) is \(\delta_{j k}\) around \(A_{k}\) and is \(\pi_{j k}\) around \(B_{k},\)
one has
\begin{equation}\label{Roydens trick}
  w_{j}=\omega_{j}+\sum_{k} \pi_{j k} \omega_{k+g}.  
\end{equation}
\begin{figure}
     \includegraphics[width=8cm]{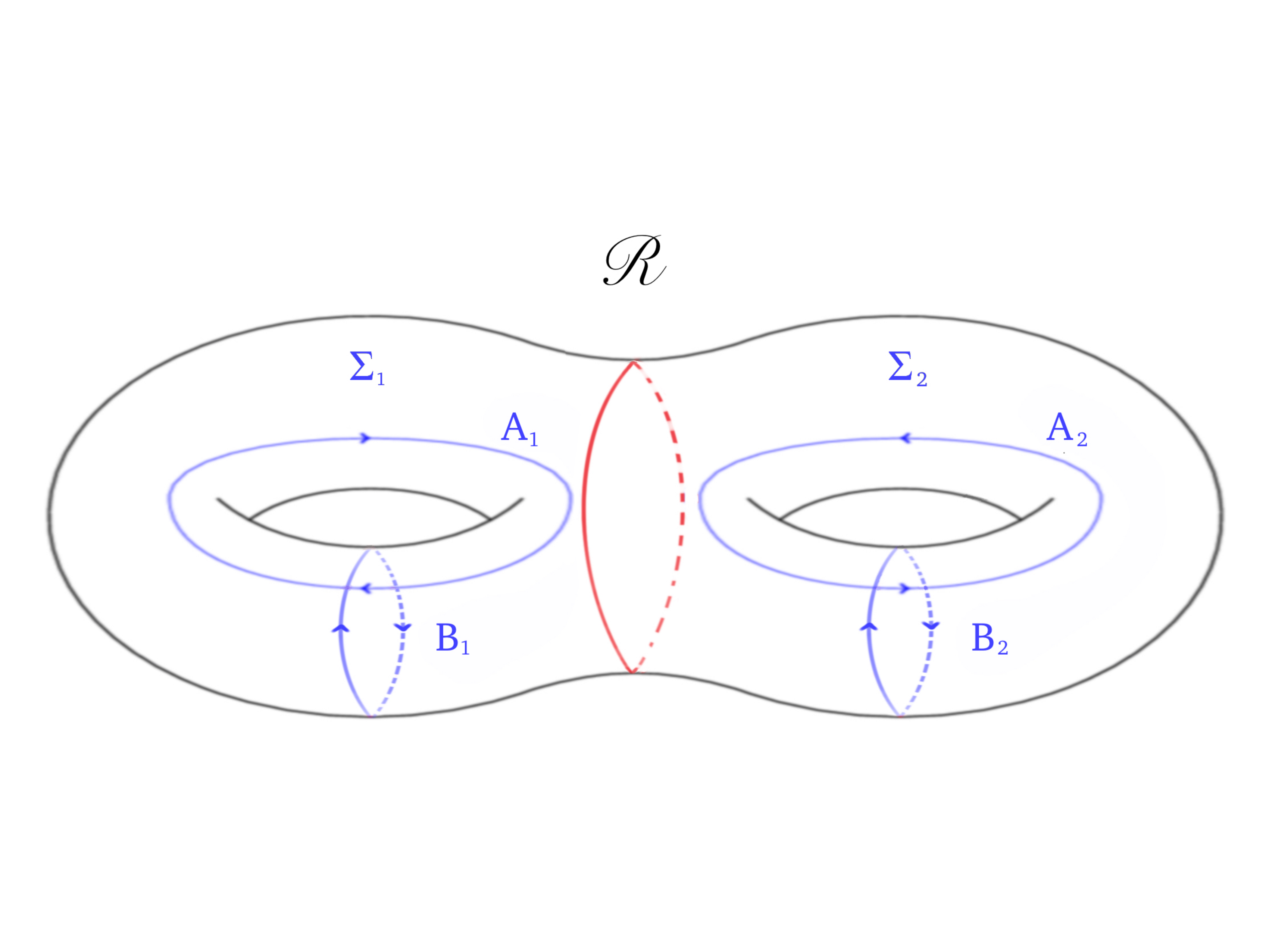}
     \caption{Genus two Riemann surface and its homology basis}
     \label{fig:period_curves}
 \end{figure} 
Now let $\mathscr {R}$ be the Riemann surface depicted in Figure \ref{fig:period_curves} and let us apply the information above to the problem of characterization of the set 
\begin{equation}\label{charac Vj}
    V_2:=\{w\in{\mathcal{A}(\mathscr{R})};\, \mathbf{R}_{2}{w} \in {\mathcal{A}^{\mathrm{e}}(\Sigma_2)}\}.
\end{equation}

We note that in the computation of $W_1$ mentioned above, we can confine ourselves to the computation of $V_2$. Here we have that \(w_{1}=\omega_{1}+\sum_{k=1}^{2}\pi_{1 k} \omega_{k+2}\) and
\(w_{2}=\omega_{2}+\sum_{k=1}^{2} \pi_{2 k} \omega_{k+2}\) and we seek a holomorphic one-form given by \(N_{1} w_{1}+N_{2} w_{2},\) with $N_j \in \mathbb{C},$ which is in \(V_{2}\). Therefore
$$\int_{A_{2}} (N_{1} w_{1}+N_{2} w_{2})=0 .$$ However, \eqref{Roydens trick} yields that $$0=\int_{A_{2}} N_{1} w_{1}+N_{2} w_{2}=\int_{A_{2}} \Big(N_{1} \omega_{1}+N_{1} \sum_{k=1}^{2} \pi_{1 k} \omega_{k+2}+N_{2} \omega_{2}+N_{2} \sum_{k} \pi_{2 k} \omega_{k+2}\Big)
=N_{2}.$$
Furthermore
\(\int_{B_{2}} N_{1} w_{1}=\int_{B_{2}} (N_{1} \omega_{1}+N_{1} \sum_{k=1} ^{2}\pi_{1 k} \omega_{k+2})=N_{1} \pi_{12}\).
Hence \(V_{2}\) is non-empty if and only if \(\pi_{12}=0\). Therefore by \eqref{symmetry of Riemann}, for $\mathscr{R}$ as in Figure \ref{fig:period_curves}, $V_2$ is non-empty if and only if the Riemann matrix has the form
\begin{equation}
  \begin{pmatrix}
\pi_{11} & 0\\
0 & \pi_{22}
\end{pmatrix}
 \end{equation}
 However, as was shown by M. Gerstenhaber in \cite{Gerstenhaber}, no  surface of genus 2 has  a diagonal matrix for a Riemann matrix, and therefore $V_2$ is indeed empty.

} 
\end{subsection}
\end{section}

\begin{section}{Scattering}   \label{se:scattering}

\begin{subsection}{Assumptions throughout this section} In this section we will once again use the assumptions that were in force in Subsection \ref{se:assumptions_scattering_section}.  

We will further assume that $\riem_2$ is connected. However we will explicitly state this assumption throughout.

\end{subsection}
\begin{subsection}{About this Section} 
 This section is devoted to the construction of the scattering matrix associated to the overfare of one-forms defined in Section \ref{se:subsection_overfare_oneforms}, and to the proof of its unitarity.  We give an explicit form in terms of the integral operators $\mathbf{T}_{1,k}$ and $\mathbf{S}_k$ of Schiffer.
 
 In a sense, we have already proven unitarity of the scattering matrix; indeed it follows immediately from the identities in Section \ref{se:Schiffer_adjoint_identities}.  The difficulty is to explicitly state and prove the form of the scattering matrix. 
 
 We do this in three main steps.  The first two steps are completed in Section \ref{se:decompositions_compatibility}, and the final step is completed in Section \ref{se:scattering_subsection}.

 The first step involves decomposing arbitrary harmonic one-forms on $\riem_2$ in terms of (a) restrictions of harmonic forms on $\mathscr{R}$, (b) forms in the image of $\mathbf{T}_{1,2}$ and $\overline{\mathbf{T}}_{1,2}$, and (c) harmonic measures on $\riem_2$.  Similar decompositions are given on $\riem_1$. They are motivated by the results of Section \ref{se:index_cohomology}, which showed the interrelation between the cohomology of forms  in the range of the operator $\mathbf{T}_{1,2}$ applied to the restrictions of forms on $\mathscr{R}$, and the cohomology of forms in the range of the adjoints of the restriction operator. The decompositions are by necessity somewhat intricate, but they also have a certain elegance and inevitability.

 Once this decomposition is given, in step two we apply the jump formula and cohomology identities to express the overfare $\mathbf{O}^{\mathrm{e}}_{2,1}$ in terms of the restriction operators, their adjoints, and the Schiffer operators $\mathbf{T}_{1,1}$.   These formulas show in particular that the overfare process produces compatible forms, as promised in Section \ref{se:Overfare}. This proves one of the unitarity relations.

 Section \ref{se:scattering_subsection} contains the third step in the proof of the unitarity of the scattering matrix, using the remaining adjoint identities of Section \ref{se:Schiffer_adjoint_identities} as well as the decompositions for the form and its overfare in \ref{se:decompositions_compatibility}.  Finally, in Section \ref{se:scattering_analogies} we give a heuristic discussion of the interpretation of our matrix as a scattering matrix. 
\end{subsection}

\begin{subsection}{Decompositions of harmonic forms and compatibility}  \label{se:decompositions_compatibility}   \ \ \

  We will need the following lemma.  In its statement and proof, we suppress restriction operators to reduce clutter, since they are clear from context.  Because of the asymmetry in the conditions for $\riem_1$ and $\riem_2$, in the statements and proofs there will be repeated division into the two cases.

  \begin{lemma} \label{le:scattering_preparation}   Assume that $\riem_2$ is connected.\\

 \emph{\bf{(1) (Case of $\riem_1$).}}
   Let $\xi \in \mathcal{A}(\mathscr{R})$, $\overline{\eta} \in \mathcal{A}(\mathscr{R})$, $\alpha_1 \in \mathcal{A}(\riem_1)$, and  $\overline{\beta_1} \in \mathcal{A}(\riem_1)$.  Assume that 
  $(\alpha_1 + \overline{\beta}_1 ) - \mathbf{R}^{\mathrm{h}}_1 (\xi + \overline{\eta}) \in \mathcal{A}^\mathrm{e}_{\mathrm{harm}}(\riem_1)$.
  There are $\overline{m},\overline{s} \in \overline{\mathcal{A}(\mathscr{R})}$,
  $n,t \in \mathcal{A}(\mathscr{R})$, such that 
  \begin{align*}
      \alpha_1 - \overline{m} - t & \in \mathcal{A}^\mathrm{e}_{\mathrm{harm}}(\riem_1) \\
      \overline{\beta}_1  - n - \overline{s} & \in \mathcal{A}^\mathrm{e}_{\mathrm{harm}}(\riem_1) 
  \end{align*}
  and 
  \begin{align*}
      \overline{\eta} & = \overline{m} + \overline{s} \\
      \xi & = n + t. 
  \end{align*} 
  
    \emph{\bf{(2) (Case of $\riem_2$).}}
  Let $\xi \in \mathcal{A}(\mathscr{R})$, $\overline{\eta} \in \mathcal{A}(\mathscr{R})$, $\alpha_2 \in \mathcal{A}(\riem_2)$, and  $\overline{\beta_2} \in \mathcal{A}(\riem_2)$.  Assume that 
  $(\alpha_2 + \overline{\beta}_2 ) - \mathbf{R}^{\mathrm{h}}_2 (\xi + \overline{\eta}) \in \mathcal{A}^\mathrm{e}_{\mathrm{harm}}(\riem_2)$.  
  There are $\overline{m},\overline{s} \in \overline{\mathcal{A}(\mathscr{R})}$,
  $n,t \in \mathcal{A}(\mathscr{R})$, and $d\omega \in \mathcal{A}_{\mathrm{bw}}(\riem_2)$ such that 
  \begin{align*}
      \alpha_2 - \partial \omega - \overline{m} - t & \in \mathcal{A}^\mathrm{e}_{\mathrm{harm}}(\riem_2) \\
      \overline{\beta}_2 - \overline{\partial} \omega - n - \overline{s} & \in \mathcal{A}^\mathrm{e}_{\mathrm{harm}}(\riem_2) 
  \end{align*}
  and 
  \begin{align*}
      \overline{\eta} & = \overline{m} + \overline{s} \\
      \xi & = n + t. 
  \end{align*}
  \end{lemma}
  Note that we are not claiming any relation between the $m,n,s,t$ in parts (1) and (2).  
  \begin{proof} 
   In the proof, we will require a basis for the cohomology of $\mathscr{R}$, which we now describe.  Let $g_k$, $k=1,\, 2,$ be the genus of $\Sigma_k$. Assume that there are $p$ curves in the complex $\Gamma$, and assume that there are $q$ connected components  $\riem_1^j$, $j=1,\ldots,q$ of $\riem_1$. Let $g_1^1,\ldots,g_1^q$ be the genuses of these components, so that $g_1 = g_1^1 + \cdots g_1^q$. 
 We then have that 
 \[ g_1+g_2+(p-q) = g,  \]
 so we can define $g_d= p-q$ to be the number of ``dissected handles''.
 Choose a homology basis for $\mathscr{R}$ consisting of 
 \begin{itemize}
     \item $2g_1$ curves $C^1_1,\ldots,C^1_{2g_1}$ corresponding to the handles in $\riem_1$;
     \item $2g_2$ curves $C^2_1,\ldots,C^2_{2g_2}$ corresponding to the handles in $\riem_2$; 
     \item a collection of boundary curves $\Gamma_1,\ldots,\Gamma_{p-q}$ containing $n_j-1$ boundary curves of $\riem_1^j$ for $j=1,\ldots,q$;
     \item a collection of curves $b_1,\ldots,b_{p-q}$ encircling each dissected handle. 
 \end{itemize} 
 There are $m$ boundary curves which are not in the span of this basis, one for each connected component of $\riem_1$.  
 
 We first claim that 
 \begin{enumerate}[label=(\alph*),font=\upshape]
     \item given any $\gamma_1 \in \mathcal{A}_{\mathrm{harm}}(\riem_1)$, there is a $\zeta_1 \in \mathcal{A}_{\mathrm{harm}}(\mathscr{R})$ such that $\gamma_1  - \zeta_1$ is exact in $\riem_1$; and 
     \item  given any $\gamma_2 \in \mathcal{A}_{\mathrm{harm}}(\riem_2)$, there is a $\zeta_2 \in \mathcal{A}_{\mathrm{harm}}(\riem_2)$ and a $d\omega\in \mathcal{A}_{\mathrm{bw}}(\riem_2)$ such that $\gamma_2 - \ast d\omega - \zeta_2$ is exact in $\riem_2$.
 \end{enumerate}
 To show this, consider a dual basis of harmonic one-forms on $\mathscr{R}$ which we denote by $H = \{ H_{C_j^k}, H_{\Gamma_l}, { H_{b_r}} \}$ with the usual meaning. Claim (a) follows from the fact that there is a unique element 
 \[  \zeta_1 \in \mathrm{span}  
 \{ H_{C_1^1},\ldots,H_{C_{2g_1}^1},H_{\Gamma_1},\ldots,H_{\Gamma_{p-q}}  \}  \]
 such that $\gamma_1 - \zeta_1$ is exact, since the set of curves spans the homology of each connected component of $\riem_1$.  
 
 To prove claim (b), observe that 
 one may remove all the $C^2_j$ periods of $\gamma_2$ using a 
 \[  \zeta_2 \in \mathrm{span} \{ H_{C_1^2},\ldots,H_{C_{2g_2}^2} \}.     \]
 That is, we can arrange that $\gamma_2 - \zeta_2$
 has zero periods over all $\Gamma_1,\ldots,\Gamma_{p-q}$ and $C^2_1,\ldots,C^2_{2g_2}$.
 
 As observed above, there are $q$ boundary curves,  $\gamma_1,\ldots,\gamma_{q}$ say, which are not contained in the collection $\Gamma_1,\ldots,\Gamma_{p-q}$. Any one of these, say $\gamma_q$, is a linear combination of the remaining curves $\gamma_1,\ldots,\gamma_{q-1}$.  {Let $\omega_j \in \mathcal{D}_{\mathrm {bw}}(\riem_2)$ be one on the boundary of the $j$th connected component of $\riem_1$ and $0$ on the others. Since one may specify the period of $\ast d\omega$ on $\gamma_1,\ldots,\gamma_{q-1}$ by Proposition \ref{pr:period_submatrix_also_positive}, it is enough to show that $\{ \ast d\omega_1,\ldots,\ast d\omega_{q-1} \} \cup \{ H_{C_1^2},\ldots,H_{C_{2g_2}^2} \}$
 is linearly independent.    }
 
 Assume that 
 \[  \sum_{j=1}^{q-1} \mu_j \ast d\omega_j + \sum_{l} \lambda_l H_{C^2_l} =0   \]
 where $H_l$ range over the elements of $H$.  We must have
 \[  0= \int_{\gamma_r}   \Big( \sum_{j=1}^{q-1} \mu_j \ast d\omega_j + \sum_{l} \lambda_l H_{C^2_l}\Big) = \mu_r \ast d \omega_r      \]
 for $r=1,\ldots,q-1$, so $\mu_r = 0$ by Theorem \ref{th:period_matrix_invertible}.  Thus 
 \[ \sum_{l} \lambda_l H_{C^2_l} =0 \]
 and the claim now follows from linear independence of elements of $\{ H_{C_1^2},\ldots,H_{C_{2g_2}^2} \}.$

 We shall first prove claim (2) of the lemma, which has the additional issue of bridgeworthy form. The proof of claim (1) is similar to that of (2), but without this complication. Apply claim (b) of the proof above to obtain $\zeta_2$, $\omega$, $\zeta_2'$, and $\omega'$ such that 
 \begin{align} \label{eq:remove_bridge_temp}
      \alpha_2 -  \zeta_2 - \partial \omega & \in \mathcal{A}^{\mathrm{e}}_{\mathrm{harm}}(\riem_2) \nonumber \\
     \overline{\beta}_2 - \zeta_2' - \overline{\partial} \omega' & \in \mathcal{A}^{\mathrm{e}}_{\mathrm{harm}}(\riem_2).
 \end{align}
 Here we have used the facts that 
 \begin{equation} \label{eq:partial_decomp_temp} 
   \partial \omega = 1/2 (d\omega +i \ast d \omega)  \ \ \ \text{and} \ \ \  \overline{\partial} \omega = 1/2 (d\omega - i \ast d \omega). 
 \end{equation}
 First, we will show that we may take $\overline{\partial} \omega'= \overline{\partial} \omega$ in \eqref{eq:remove_bridge_temp}.  To see this, observe that  
 \[  \alpha_2 + \overline{\beta}_2 - (\xi + \overline{\eta})   \]
 and 
 \[  \alpha_2 + \overline{\beta}_2 -  \zeta_2 - \zeta_2' - \partial \omega - \overline{\partial} \omega'  \]
 are both exact in $\riem_2$.  Subtracting we see that 
 \[  -\xi - \overline{\eta} + \zeta_2 + \zeta_2' + \partial \omega + \overline{\partial} \omega'   \]
 is exact. By the linear independence of the periods of  $\mathcal{A}_{\mathrm{harm}}(\mathscr{R})$ and $\ast d\omega$ for $\omega$ bridgeworthy established above, we must have that 
 $\partial \omega + \overline{\partial} \omega'$ is exact.  Again using \eqref{eq:partial_decomp_temp}, we see that the periods in $\riem_2$ of $\ast \omega$ and $\ast \omega'$ agree, so we may take $\omega'=\omega$ in \eqref{eq:remove_bridge_temp} as claimed. 
 
 Writing $\zeta_2=\overline{M}+T$ and $\zeta_2'=N + \overline{S}$ for 
   $\overline{M},\overline{S} \in \overline{\mathcal{A}(\mathscr{R})}$ and $N,T \in \mathcal{A}(\mathscr{R})$, we have thus shown 
    \begin{align*}
      \alpha_2 - \overline{M} - T - \partial \omega& \in \mathcal{A}^\mathrm{e}_{\mathrm{harm}}(\riem_2) \\
      \beta_2 - N - \overline{S} - \overline{\partial} \omega & \in \mathcal{A}^\mathrm{e}_{\mathrm{harm}}(\riem_2).
  \end{align*}  
  Note that $M$, $T$, $N$, and $S$ are not uniquely determined, and we must adjust them to complete the theorem. 
  Since $\alpha_2 + \overline{\beta}_2 - ( \xi + \overline{\eta})$ and $d\omega$ are in $\mathcal{A}^\mathrm{e}_{\mathrm{harm}}(\riem_2)$ we have 
  \[  \overline{M} + T + N + \overline{S} -(\xi + \overline{\eta}) \in \mathcal{A}^{\mathrm{e}}_{\mathrm{harm}}(\riem_2).       \]
  Define $u \in \mathcal{A}(\mathscr{R})$ and $\overline{v} \in \mathcal{A}(\mathscr{R})$ by
  \begin{align*}
      u & = N + T - \xi \\
      \overline{v} & = \overline{M} + \overline{S} - \overline{\eta};
  \end{align*}
  These satisfy $\Xi_2(u) = -\Xi_2(\overline{v})$, where $\Xi_2$ is the map defined in \eqref{capital Xi} where the integrals are evaluated over the curves $C^2_j$, $j=1, \dots, 2g_2$. Therefore if we set
  \begin{align*}
      \overline{m} & = \overline{M} - \overline{v}/2\\
      t & = T - u/2 \\
      \overline{s} & = \overline{S} -\overline{v}/2 \\
      n & = N -u/2
  \end{align*}
  it still holds that 
  \begin{align*}
    \alpha_2 - \overline{m} - t & \in \mathcal{A}^\mathrm{e}_{\text{harm}}(\riem_2)  \\
    \overline{\beta}_2 - \overline{s} -n & \in {\mathcal{A}^\mathrm{e}_{\text{harm}}(\riem_2)}.  
  \end{align*}
  Since $\overline{m} + \overline{s} = \overline{\eta}$ and $n + t = \xi$ this completes the proof of part (2) of the lemma.
  
  The proof of part (1) is identical, except that one may start directly with 
  \begin{align*}
      \alpha_1 - \overline{M}- T & \in \mathcal{A}^{\mathrm{e}}_{\mathrm{harm}}(\riem_1) \\
      \overline{\beta}_1 - N - \overline{S} & \in \mathcal{A}^{\mathrm{e}}_{\mathrm{harm}}(\riem_1).
  \end{align*}
  (Here of course the $M,T,N,S$ are not necessarily the same as those in the proof of (2).)
  \end{proof}

  In Section \ref{se:Overfare} we saw that given $\alpha_2+ \overline{\beta}_2 \in \mathcal{A}_{\mathrm{harm}}(\riem_2)$ and $\zeta = \xi + \overline{\eta} \in \mathcal{A}_{\mathrm{harm}}(\mathscr{R})$, there is a weakly compatible one-form $\alpha_1 + \overline{\beta}_1 \in \mathcal{A}_{\mathrm{harm}}(\riem_1)$ with respect to $\zeta$. As promised, we now show that there is a $\alpha_1 + \overline{\beta}_1$ which is in fact compatible, in the case that $\riem_2$ is connected. Furthermore, this compatible form is given by  
  \[  \alpha_1 + \overline{\beta}_1 = \mathbf{O}^\mathrm{e}\left[ \alpha_2 + \overline{\beta}_2 - \mathbf{R}_2 \xi - \overline{\mathbf{R}}_2 \overline{\eta} \right] +  \mathbf{R}_1 \xi + \overline{\mathbf{R}}_1 \overline{\eta}.  \]
  
  To do this, we require a decomposition for harmonic one-forms which is convenient from the point of view of the action of the Schiffer operators. This decomposition will also play a central role in the proof of the unitarity of the scattering matrix. \\
  
  \begin{lemma}[Decomposition lemma] \label{le:decomposition_lemma}
   Assume that $\riem_2$ is connected.\\ 
   
   \emph{\bf{(1) (Case of $\riem_1$).}}  
   Let $\alpha_1,\beta_1 \in \mathcal{A}(\riem_1)$ and $\xi,\eta \in \mathcal{A}(\mathscr{R})$ be such that $\alpha_1 + \overline{\beta}_1 - (\xi+\overline{\eta})$ is exact in $\riem_1$. Then there are $\tau_2,\sigma_2 \in \mathbf{R}_2 \mathcal{A}(\mathscr{R})$ and $\mu_2,\nu_2 \in [\mathbf{R}_2 \mathcal{A}(\mathscr{R}))]^\perp$  such that 
   \begin{align}  \label{eq:compatible_breakdown_A1}
      \alpha_1 - \overline{\mathbf{R}}_1 \overline{\mathbf{S}}_2 \overline{\mu}_2 - \mathbf{R}_1 \mathbf{S}_2 \tau_2 & \in \mathcal{A}^\mathrm{e}_{\mathrm{harm}}(\riem_1) \nonumber\\ 
      \overline{\beta}_1 - \mathbf{R}_1 \mathbf{S}_2 \nu_2 - \overline{\mathbf{R}}_1 \overline{\mathbf{S}}_2 \overline{\sigma}_2  & \in \mathcal{A}^\mathrm{e}_{\mathrm{harm}}(\riem_1)   
  \end{align}
  and 
 \begin{align} \label{eq:compatible_breakdown_B1}
    \overline{\mathbf{S}}_2 \overline{\mu}_2 +  \overline{\mathbf{S}}_2 \overline{\sigma}_2
   & = \overline{\eta} \nonumber  \\
    \mathbf{S}_2 \nu_2 + \mathbf{S}_2 \tau_2 & = \xi.  
 \end{align} 
 Furthermore, there are $\gamma_2, \rho_2 \in \mathcal{A}(\riem_1)$ such that
 \begin{align} \label{eq:compatible_breakdown_C1}
     \alpha_1 & = \mathbf{T}_{2,1} \overline{\gamma}_2 + \mathbf{R}_1 \mathbf{S}_2 \tau_2 \\ \nonumber
     \overline{\beta}_1 & = \overline{\mathbf{T}}_{2,1} \rho_2 + \overline{\mathbf{R}_1} \overline{\mathbf{S}}_2 \overline{\sigma}_2
 \end{align}
 where $\overline{\gamma}_k$ and $\rho_k$ are decomposed as follows:  
 \begin{align} \label{eq:compatible_breakdown_D1}
       \overline{\gamma}_2 & = - \overline{\mu}_2 + \overline{\delta}_2, \ \ \  \overline{\mu}_2 \in [\overline{\mathbf{R}}_2 \overline{\mathcal{A}(\mathscr{R})}], \ \ \  \overline{\delta}_2 \in [\overline{\mathbf{R}}_2 \overline{\mathcal{A}(\mathscr{R})}]^\perp,  \nonumber \\
       \rho_2 & = - \nu_2 + \varepsilon_2, \ \ \ \nu_2 \in \mathbf{R}_2 \mathcal{A}(\mathscr{R}), \ \ \ \varepsilon_2 \in [\mathbf{R}_2 \mathcal{A}(\mathscr{R}) ]^\perp. 
 \end{align} 
   
   \emph{\bf{(2) (Case of  $\riem_2$)}}.  Let $\alpha_2,\beta_2 \in \mathcal{A}(\riem_2)$ and $\xi,\eta \in \mathcal{A}(\mathscr{R})$ be such that $\alpha_2 + \overline{\beta}_2 - (\xi+\overline{\eta})$ is exact in $\riem_2$. Then there are $\tau_1,\sigma_1 \in \mathbf{R}_1 \mathcal{A}(\mathscr{R})$, $\mu_1,\tau_1 \in [\mathbf{R}_1 \mathcal{A}(\mathscr{R}))]^\perp$, 
   and a $d\omega \in \mathcal{A}_{\mathrm{bw}}(\riem_2)$, such that 
   \begin{align} \label{eq:compatible_breakdown_A2}
          \alpha_2 -\partial \omega - \overline{\mathbf{R}}_2 \overline{\mathbf{S}}_1 \overline{\mu}_1 - \mathbf{R}_2 \mathbf{S}_1 \tau_1 & \in \mathcal{A}^\mathrm{e}_{\mathrm{harm}}(\riem_2) \nonumber \\  
      \overline{\beta}_2 - \overline{\partial} \omega- \mathbf{R}_2 \mathbf{S}_1 \nu_1 - \overline{\mathbf{R}}_2 \overline{\mathbf{S}}_1 \overline{\sigma}_1  & \in \mathcal{A}^\mathrm{e}_{\mathrm{harm}}(\riem_2)  
   \end{align}
   and 
   \begin{align} \label{eq:compatible_breakdown_B2}
       \overline{\mathbf{S}}_1 \overline{\mu}_1 + \overline{\mathbf{S}}_1 \overline{\sigma}_1 & = \overline{\eta} \nonumber \\
       \mathbf{S}_1 \nu_1 + \mathbf{S}_1 \tau_1 & = \zeta. 
   \end{align}
   Furthermore, there are $\gamma_1,\rho_1 \in \mathcal{A}(\riem_1)$ such that 
   \begin{align} \label{eq:compatible_breakdown_C2}
         \alpha_2 & = \partial \omega + \mathbf{T}_{1,2} \overline{\gamma}_1 + \mathbf{R}_2 \mathbf{S}_1 \tau_1  \nonumber \\ 
     \overline{\beta}_2 & = \overline{\partial} \omega + \overline{\mathbf{T}}_{1,2} \rho_1 + \overline{\mathbf{R}}_2 \overline{\mathbf{S}}_1 \overline{\sigma}_1 
   \end{align}  
    where $\overline{\gamma}_1$ and $\rho_1$ are decomposed as follows:
 \begin{align} \label{eq:compatible_breakdown_D2}
       \overline{\gamma}_1 & = - \overline{\mu}_1 + \overline{\delta}_1, \ \ \  \overline{\mu}_1 \in [\overline{\mathbf{R}}_1 \overline{\mathcal{A}(\mathscr{R})}], \ \ \  \overline{\delta}_1 \in [\overline{\mathbf{R}}_1 \overline{\mathcal{A}(\mathscr{R})}]^\perp,  \nonumber \\
       \rho_1 & = - \nu_1 + \varepsilon_1, \ \ \ \nu_1 \in \mathbf{R}_1 \mathcal{A}(\mathscr{R}), \ \ \ \varepsilon_1 \in [\mathbf{R}_1 \mathcal{A}(\mathscr{R}) ]^\perp. 
 \end{align} 
  \end{lemma}
  \begin{proof}
   The claims \eqref{eq:compatible_breakdown_A1} and \eqref{eq:compatible_breakdown_B1} in part (1) follow directly from   Lemma \ref{le:scattering_preparation} part (1), using the fact that $\mathbf{S}_2 \mathbf{R}_2$ is an isomorphism by Theorem  \ref{th:S_an_isomorphism}.  Similarly, the claims \eqref{eq:compatible_breakdown_A2} and  \eqref{eq:compatible_breakdown_B2} in part (2) follow directly from   Lemma \ref{le:scattering_preparation} part (2), using the fact that $\mathbf{S}_1 \mathbf{R}_1$ is an isomorphism by Theorem  \ref{th:S_an_isomorphism}. 
   
   The claims \eqref{eq:compatible_breakdown_C1} and \eqref{eq:compatible_breakdown_C2} follow directly from Theorem  \ref{th:Tonetwo_iso_improved}. The decompositions of $\gamma_k$ and $\rho_k$ in \eqref{eq:compatible_breakdown_D1} and \eqref{eq:compatible_breakdown_D2} follow from Theorems \ref{th:S_kernel_range},  \ref{th:Schiffer_cohomology}, and \ref{th:general_T12_on_perp_surjective_to_exact}.
  \end{proof}
  
  \begin{theorem} \label{th:compatibility_existence}  
   Assume that $\riem_2$ is connected.\\
   
   \emph{\bf{(1) (Case of $\riem_1$).}}  Let $\alpha_1+ \beta_1 \in \mathcal{A}_{\mathrm{harm}}(\riem_1)$ and $\zeta = \xi + \overline{\eta} \in \mathcal{A}(\mathscr{R})$ be such that $\alpha_1 + \overline{\beta}_1 - \mathbf{R}_1^h \zeta \in \mathcal{A}^{\mathrm{e}}_{\mathrm{harm}}(\riem_1)$. 
   There is a $\alpha_2 + \overline{\beta}_2 \in \mathcal{A}_{\mathrm{harm}}(\riem_2)$ which is compatible with $\alpha_1 + \overline{\beta}_1$ {with respect to} $\zeta$.  Given any other compatible $\alpha_2'+\overline{\beta}_2'$, the difference $\alpha_2'+ \overline{\beta}_2' - \alpha_2 +\overline{\beta}_2 \in \mathcal{A}_{\mathrm{bw}}(\riem_2)$. 
   
   Furthermore, if $\rho_2$, $\gamma_2$, $\tau_2$, and $\sigma_2$ are given as in \emph{Lemma \ref{le:decomposition_lemma} part (1)}, then there is a $d\omega \in \mathcal{A}_{\mathrm{bw}}(\riem_2)$ such that  
   \begin{align} \label{eq:compatibility_theorem_overfared_1}
    \alpha_2 & = -\rho_2 + \mathbf{T}_{2,2} \overline{\gamma}_2 + \mathbf{R}_2 \mathbf{S}_2 \tau_2 + \partial \omega \nonumber \\
    \overline{\beta}_2 & = - \overline{\gamma}_2 + \overline{\mathbf{T}}_{2,2} \rho_2 + \overline{\mathbf{R}}_2 \overline{\mathbf{S}}_2 \overline{\sigma}_2 + \overline{\partial} \omega. 
   \end{align}

   \emph{\bf{(2) (Case of $\riem_2$).}} Let $\alpha_2+ \beta_2 \in \mathcal{A}_{\mathrm{harm}}(\riem_2)$ and $\zeta = \xi + \overline{\eta} \in \mathcal{A}(\mathscr{R})$ be such that $\alpha_2 + \overline{\beta}_2 - \mathbf{R}_2^{\mathrm{h}} \zeta \in \mathcal{A}^{\mathrm{e}}_{\mathrm{harm}}(\riem_2)$. 
   There is a unique $\alpha_1 + \overline{\beta}_1 \in \mathcal{A}_{\mathrm{harm}}(\riem_1)$ which is compatible with $\alpha_2 + \overline{\beta}_2$ {with respect to} $\zeta$, given by 
   \[  \alpha_1 + \overline{\beta}_1 = \mathbf{O}^{\mathrm{e}}_{2,1} \left( \alpha_2 + \overline{\beta}_2 - \mathbf{R}_2^{\mathrm{h}} \zeta\right) 
   + \mathbf{R}_1^{\mathrm{h}} \zeta.   \]
   
   Furthermore, $\rho_1$, $\gamma_1$, $\tau_1$, and $\sigma_1$ are given as in Lemma \ref{le:decomposition_lemma} part(2), then 
   \begin{align}  \label{eq:compatibility_theorem_overfared_2}
    \alpha_1 & = -\rho_1 + \mathbf{T}_{1,1} \overline{\gamma}_1 + \mathbf{R}_1 \mathbf{S}_1 \tau_1  \nonumber \\
    \overline{\beta}_1 & = - \overline{\gamma}_1 + \overline{\mathbf{T}}_{1,1} \rho_1 + \overline{\mathbf{R}}_1 \overline{\mathbf{S}}_1 \overline{\sigma}_1. 
   \end{align}
  \end{theorem}
  \begin{proof}
    We first prove (2). By the assumptions we have the decomposition \eqref{eq:compatible_breakdown_A2} - \eqref{eq:compatible_breakdown_D2} of Lemma \ref{le:decomposition_lemma}.  
    
    Using \eqref{eq:compatible_breakdown_C2}, \eqref{eq:compatible_breakdown_B2}, and Theorem \ref{th:S_kernel_range} in that order, we see that
    \begin{align}  \label{eq:tempy_mctempface}
     \alpha_2 + \overline{\beta}_2 - \mathbf{R}_2 \xi - \overline{\mathbf{R}}_2 \overline{\eta} & = d\omega + \mathbf{T}_{1,2} \overline{\gamma}_1 + \mathbf{R}_2 \mathbf{S}_1 \tau_1 + \overline{\mathbf{T}}_{1,2} \rho_1 + \overline{\mathbf{R}}_2 \overline{\mathbf{S}}_1 \overline{\sigma}_1 - \mathbf{R}_2 \xi - \overline{\mathbf{R}}_2 \overline{\eta} \nonumber \\
     & = d\omega + \mathbf{T}_{1,2} \overline{\gamma}_1 - \mathbf{R}_2 \mathbf{S}_1 \nu_1 + \overline{\mathbf{T}}_{1,1} \rho_1 - \overline{\mathbf{R}}_2 \overline{\mathbf{S}}_1 \overline{\mu}_1 \nonumber \\ 
      & = d\omega + \mathbf{T}_{1,2} \overline{\gamma}_1 + \mathbf{R}_2 \mathbf{S}_1 \rho_1 + \overline{\mathbf{T}}_{1,1} \rho_1 + \overline{\mathbf{R}}_2 \overline{\mathbf{S}}_1 \overline{\gamma}_1.
    \end{align}
    Now since $d\omega$ is bridgeworthy, $\mathbf{O}^{\mathrm{e}}_{2,1} d\omega = 0$. Together with Proposition \ref{prop:exact_transmitted_jump} we obtain 
    \[  \mathbf{O}_{2,1}^{\mathrm{e}} \left( \alpha_2 + \overline{\beta}_2 - \mathbf{R}_2^{\mathrm{h}} \zeta \right) = -\overline{\gamma}_1 + \mathbf{T}_{1,1} \overline{\gamma}_1 + \overline{\mathbf{R}}_1 \mathbf{S}_1 \overline{\gamma}_1 
    - \rho_1 + \overline{\mathbf{T}}_{1,1} \rho_1 + \mathbf{R}_1 \mathbf{S}_1 \rho_1. \] 
    If we define  
    \[ \alpha_1 + \overline{\beta}_1 =  \mathbf{O}_{2,1}^{\mathrm{e}} \left( \alpha_2 + \overline{\beta}_2 - \mathbf{R}_2^{\mathrm{h}} \zeta \right) + \mathbf{R}_1^{\mathrm{h}} \zeta  \]
    \eqref{eq:compatible_breakdown_C2} again, we obtain that 
    $\alpha_1 + \overline{\beta}_1$ satisfies \eqref{eq:compatibility_theorem_overfared_2}. Furthermore, by construction $\alpha_1 + \overline{\beta}_1$ is weakly compatible with $\alpha_2 + \overline{\beta}_2$ {with respect to} $\zeta$. 
    
    Next we show that $\alpha_1 + \overline{\beta}_1$ is compatible.  Applying $\mathbf{S}_1$ to the expression \eqref{eq:compatibility_theorem_overfared_2} for $\alpha_1$,
    we obtain
   \begin{align} \label{eq:S_compatibility_temp_with_correction}
     \mathbf{S}_1 \alpha_1 & = - \mathbf{S}_1 \rho_1 + \mathbf{S}_1 \mathbf{T}_{1,1} \overline{\gamma}_1 + \mathbf{S}_1 \mathbf{R}_1 \mathbf{S}_1 \tau_1  & \nonumber \\
      & = - \mathbf{S}_1 \rho_1 - \mathbf{S}_2 \mathbf{T}_{1,2} \overline{\gamma}_1 + \mathbf{S}_1 \mathbf{R}_1 \mathbf{S}_1 \tau_1  &  \  \text{{\scriptsize{ Equation \eqref{eq:general_adjoint_identity_double_final}}}} \nonumber \\
      & = - \mathbf{S}_1 \rho_1 - \mathbf{S}_2 \left(  \alpha_2 - \mathbf{R}_2 \mathbf{S}_1 \tau_1 \right) + \mathbf{S}_1 \mathbf{R}_1 \mathbf{S}_1 \tau_1  &  \text{\scriptsize Equation \eqref{eq:compatible_breakdown_C2}}  \\
      & = \mathbf{S}_1 \nu_1 - \mathbf{S}_2   \alpha_2  + \left( \mathbf{S}_2 \mathbf{R}_2    + \mathbf{S}_1 \mathbf{R}_1  \right) \mathbf{S}_1 \tau_1   &  \text{\scriptsize Equation \eqref{eq:compatible_breakdown_D2}, Thm \ref{th:S_kernel_range}} \nonumber \\
      & = - \mathbf{S}_2 \alpha_2 + \xi.
      & \text{\scriptsize Equation \eqref{eq:compatible_breakdown_B2}, Thm \ref{th:quadratic_adjoint_one}} \nonumber
   \end{align}  
   This proves the claim.
   
   Finally, if $\alpha_1' + \overline{\beta}_1'$ is another compatible form, it is in particular weakly compatible. Thus $\alpha_1' + \overline{\beta}_1' - (\alpha_1 + \overline{\beta}_1)$ is an exact form $d\omega$ vanishing on the boundary, and hence is in $\mathcal{A}_{\mathrm{hm}}(\riem_1)$. However, since both are compatible, we also have that $\mathbf{S}_1^{\mathrm{h}} d\omega = 0$ so by Proposition \ref{pr:kernel_S1h} $d\omega$ is bridgeworthy. Since $\riem_2$ is connected, the only bridgeworthy form on $\riem_1$ is $0$. This completes the proof of (2). 
   
   We now prove (1). By the assumptions we have the decomposition \eqref{eq:compatible_breakdown_A1} - \eqref{eq:compatible_breakdown_D1} of Lemma \ref{le:decomposition_lemma}.  Let $G,H \in \mathcal{D}_{\mathrm{harm}}(\riem_2)$ be such that $\partial G = \rho_2$ and $\overline{\partial} H = \overline{\gamma}_2$.  
   Such a $G$ and $H$ are guaranteed to exist by Lemma \ref{le:dbar_representative}. Since by Theorem  \ref{th:Overfare_with_correction_functions} 
   \[ \dot{\mathbf{J}}_{2,1} H = \dot{\mathbf{O}}_{2,1} \dot{\mathbf{J}}_{2,2} \dot{H} - \dot{\mathbf{O}}_{2,1} \dot{H} \]
   it follows using Theorem \ref{th:jump_derivatives} that
   \begin{equation} \label{eq:exact_same_boundary_values_one}
      d \mathbf{J}^q_{2,1} H = \mathbf{T}_{2,1} \overline{\gamma}_2 + \overline{\mathbf{R}}_1 \overline{\mathbf{S}}_2 \overline{\gamma}_2 \ \ \  \text{and} \ \ \  d(\mathbf{J}^q_{2,2} H - H )= - \overline{\gamma}_2 + \mathbf{T}_{2,2} \overline{\gamma}_2 + \overline{\mathbf{R}}_2 \overline{\mathbf{S}}_2 \overline{\gamma}_2 
   \end{equation}
   are exact and have primitives with the same CNT boundary values. Similarly using $G$ we obtain that
   \begin{equation} \label{eq:exact_same_boundary_values_two}
     \overline{\mathbf{T}}_{2,1} {\rho}_2 + {\mathbf{R}}_1  {\mathbf{S}}_2 {\rho}_2 \ \ \  \text{and} \ \ \   - {\rho}_2 + \overline{\mathbf{T}}_{2,2} {\rho}_2 + {\mathbf{R}}_2 {\mathbf{S}}_2 {\rho}_2 
   \end{equation}
   are exact and have primitives with the same boundary values. 
   
   Set now
   \begin{align*}
       \alpha_2 & = -\rho_2 + \mathbf{T}_{2,2} \overline{\gamma}_2 + \mathbf{R}_2 \mathbf{S}_2 \tau_2 \\
       \overline{\beta}_2 & = -\overline{\gamma}_2 + \overline{\mathbf{T}}_{2,2} \rho_2 + \overline{\mathbf{R}}_2 \overline{\mathbf{S}}_2 \overline{\mathbf{\sigma}}_2. 
   \end{align*}
   We will show that this is compatible with $\alpha_1 + \overline{\beta}_1$ with respect to $\zeta$. 
   
   To see that it is weakly compatible, it is enough to show that 
   \[  \alpha_2 + \overline{\beta}_2 - \mathbf{R}_2 \xi - \overline{\mathbf{R}}_2 \overline{\eta} \ \ \ \text{and} \ \ \ \alpha_1 + \overline{\beta}_1 - \mathbf{R}_1 \xi - \overline{\mathbf{R}}_1 \overline{\eta} \]
   are exact and have the same boundary values. We are given that the left expression is exact; a computation identical to \eqref{eq:tempy_mctempface} in part (2) shows that 
   \[  \alpha_1 + \overline{\beta}_1 - \mathbf{R}_1 \xi - \overline{\mathbf{R}}_1 \overline{\eta} = \mathbf{T}_{2,1} \overline{\gamma}_2 + \mathbf{R}_1 \mathbf{S}_2 \rho_2 + \overline{\mathbf{T}}_{2,1} \rho_2 + \overline{\mathbf{R}}_1 \overline{\mathbf{S}}_2 \overline{\gamma}_2 \]
   and we also have that
   \[  \alpha_2 + \overline{\beta}_2 - \mathbf{R}_2 \xi - \overline{\mathbf{R}}_2 \overline{\eta} = -\rho_2 + \mathbf{T}_{2,2} \overline{\gamma}_2 + \mathbf{R}_2 \mathbf{S}_2 \rho_2 - \overline{\gamma}_2 + \overline{\mathbf{T}}_{2,2} \rho_2 + \overline{\mathbf{R}}_2 \overline{\mathbf{S}}_2 \overline{\gamma}_2. \]
   Weak compatibility now follows from the fact that the left and right sides of \eqref{eq:exact_same_boundary_values_one} and \eqref{eq:exact_same_boundary_values_two} are exact and have primitives with the same boundary values.
   
   To show compatibility, we repeat the computation of \eqref{eq:S_compatibility_temp_with_correction} with the indices switched, and without the $\omega$ term.   
   
   Now any other weakly compatible form is of the form $\alpha_2 + \overline{\beta}_2 + d\omega$ for some $d\omega \in \mathcal{A}_{\mathrm{harm}}(\riem_2)$. If this is compatible, we must have that $\mathbf{S}_1^{\mathrm{h}} d\omega=0$, and thus $d\omega$ is bridgeworthy by Proposition \ref{pr:kernel_S1h}. 
  \end{proof}
  
  We have thus proven the characterization of compatibility promised in Section \ref{se:subsection_overfare_oneforms}. Though it is contained in the statement of Theorem \ref{th:compatibility_existence}, it deserves to be singled out.
  \begin{corollary}
   Assume that $\riem_2$ is connected. Then $\alpha_1 \in \mathcal{A}_{\mathrm{harm}}(\riem_1)$ and $\alpha_2 \in \mathcal{A}_{\mathrm{harm}}(\riem_2)$ are compatible with respect to $\zeta \in \mathcal{A}_{\mathrm{harm}}(\mathscr{R})$ if and only if 
   \[ \alpha_1 = \mathbf{O}^{\mathrm{e}}_{2,1} \left(\alpha_2 - \mathbf{R}_2^{\mathrm{h}} \zeta \right) + \mathbf{R}_1^{\mathrm{h}} \zeta.  \]
  \end{corollary}
  
  We also have several special cases worthy of attention.
  \begin{corollary}
   Assume that both $\riem_1$ and $\riem_2$ are connected. Given $\alpha_k \in \mathcal{A}_{\mathrm{harm}}(\riem_k)$ and $\zeta \in \mathcal{A}_{\mathrm{harm}}(\mathscr{R})$. The following are equivalent.
   \begin{enumerate}[label=(\arabic*),font=\upshape]
       \item $\alpha_k$ are compatible with respect to $\zeta$;
       \item $\alpha_1 = \mathbf{O}^{\mathrm{e}}_{2,1} \left( \alpha_2 - \mathbf{R}_2^{\mathrm{h}} \zeta \right) + \mathbf{R}_1^{\mathrm{h}} \zeta$;
       \item $\alpha_2 = \mathbf{O}^{\mathrm{e}}_{1,2} \left( \alpha_1 - \mathbf{R}_1^{\mathrm{h}} \zeta \right) + \mathbf{R}_2^{\mathrm{h}} \zeta$.
   \end{enumerate}
   In particular, given $\alpha_1$, there is a unique $\alpha_2$ which is compatible with $\alpha_1$ with respect to $\zeta$.  The same claim holds with the indices $1$ and $2$ interchanged.
  \end{corollary}
  
  \begin{corollary}  Assume that the separating complex of curves consists of a single curve. Given $\alpha_k \in \mathcal{A}_{\mathrm{harm}}(\riem_k)$ for $k=1,2$, they are weakly compatible with respect to $\zeta$ if and only if they are compatible with respect to $\zeta$.  
  \end{corollary}
 \end{subsection}
 \begin{subsection}{Unitarity of the scattering matrix}
 
  \label{se:scattering_subsection}
 
 We now show that the scattering matrix of overfare is a unitary matrix whose blocks are Schiffer operators.  We divide this into cases $g \neq 0$ and $g=0$.  
 \begin{theorem}[Scattering matrix, $g \neq 0$]   \label{th:unitary_scattering_genus_nonzero}  Assume that the genus of $\mathscr{R}$ is non-zero, and that $\riem_2$ is connected.  
 
  Assume that $\alpha_1 + \overline{\beta}_1$ and $\alpha_2 + \overline{\beta}_2$ are compatible with respect to $\zeta = \xi + \overline{\eta} \in \mathcal{A}_{\mathrm{harm}}(\mathscr{R})$.  Then
  \begin{equation}  \label{eq:scattering_matrix}
     \left( \begin{array}{c} \overline{\beta}_1 \\ \overline{\beta}_2 \\ \xi \end{array} \right)
     = \left( \begin{array}{ccc} - \overline{\mathbf{T}}_{1,1} & - \overline{\mathbf{T}}_{2,1} & \overline{\mathbf{R}}_1 \\
      - \overline{\mathbf{T}}_{1,2} & - \overline{\mathbf{T}}_{2,2} & \overline{\mathbf{R}}_2 \\ \mathbf{S}_1 & \mathbf{S}_2 & 0 \end{array} \right) 
      \left( \begin{array}{c} {\alpha}_1 \\ {\alpha}_2 \\ \overline{\eta} \end{array} \right).
  \end{equation}
  This matrix is unitary.
 \end{theorem}
 \begin{proof} 
  Unitarity follows from Theorems \ref{th:quadratic_adjoint_one}, \ref{th:general_adjoint_double_identities}, and \ref{th:general_adjoint_identity_double_final}.  So it only remains to show that the matrix equation holds.  The bottom entry of the left and right hand side are equal by part (3) of the definition of compatibility, so we need only demonstrate that the other two entries are equal.
  
  We then have that both parts of Lemma \ref{le:decomposition_lemma} hold. For $k=1,2$ let $\gamma_k$, $\rho_k$, $\tau_k$, $\sigma_k$, $\mu_k$, and $\nu_k$ be as in Lemma \ref{le:decomposition_lemma}, so that   \eqref{eq:compatible_breakdown_A1}-\eqref{eq:compatible_breakdown_D1} and \eqref{eq:compatible_breakdown_A2}-\eqref{eq:compatible_breakdown_D2} hold. 
 
 Note that $\mathbf{S}_k \nu_k = \mathbf{S}_k \rho_k$ and $\mathbf{S}_k \mu_k = \mathbf{S}_k \gamma_k$ for $k=1,2$, 
 since $\overline{\delta}_k \in [\overline{\mathbf{R}}_k \overline{\mathcal{A}(\mathscr{R})}]^\perp$ and 
 the integral kernel of $\mathbf{S}_k$ is in $\overline{\mathcal{A}(\mathscr{R})}$. Similarly  $\overline{\mathbf{S}}_k \overline{\mu}_k = \overline{\mathbf{S}}_k \overline{\gamma}_k$.

 Next, applying $\mathbf{T}_{1,1}$ to the first equation of  \eqref{eq:compatibility_theorem_overfared_2}, and inserting the second, we obtain
 \begin{align*}
     \overline{\mathbf{T}}_{1,1} \alpha_1 & = - \overline{\mathbf{T}}_{1,1} \rho_1 + \overline{\mathbf{T}}_{1,1} \mathbf{T}_{1,1} \overline{\gamma_1}  + \overline{\mathbf{T}}_{1,1} \mathbf{R}_1 \mathbf{S}_1 \tau_1 \\ 
      & = - \overline{\beta}_1 - \overline{\gamma_1} + \overline{\mathbf{R}}_1 \overline{\mathbf{S}}_1 \overline{\sigma}_1 + \overline{\mathbf{T}}_{1,1} 
      \mathbf{T}_{1,1} \overline{\gamma_1} + \overline{\mathbf{T}}_{1,1} 
      \mathbf{R}_1 \mathbf{S}_1 \tau_1.
 \end{align*}
 Now applying the first identity of Theorem \ref{th:general_adjoint_double_identities} and the first line of Theorem \ref{th:general_adjoint_identity_double_final} in that order, we obtain 
 \begin{align*}
     \overline{\mathbf{T}}_{1,1} \alpha_1 & = - \overline{\beta_1} + \overline{\mathbf{R}}_{1} \overline{\mathbf{S}}_1 \overline{\sigma_1} - \overline{\mathbf{T}}_{2,1} \mathbf{T}_{1,2} \overline{\gamma_1} - \overline{\mathbf{R}}_1 \overline{\mathbf{S}}_1 \overline{\gamma_1} 
     + \overline{\mathbf{T}}_{1,1} \mathbf{R}_1 \mathbf{S}_1 \tau_1 & \\ 
       & = - \overline{\beta}_1 + \overline{\mathbf{R}}_1 \overline{\mathbf{S}}_1 \overline{\sigma_1} - \overline{\mathbf{T}}_{2,1} \mathbf{T}_{1,2} \overline{\gamma_1} - \overline{\mathbf{R}}_{1} \overline{\mathbf{S}}_1 \overline{\gamma_1} - \overline{\mathbf{T}}_{2,1} \mathbf{R}_2 \mathbf{S}_1 \tau_1 & \\
       & =  - \overline{\beta}_1 - \overline{\mathbf{T}}_{2,1} \alpha_2 + \overline{\mathbf{T}}_{2,1} \partial \omega +
       \overline{\mathbf{R}}_1 \overline{\mathbf{S}}_1 \overline{\sigma}_1 + \overline{\mathbf{R}}_1 \overline{\mathbf{S}}_1  \overline{\mu}_1 & {\text{\scriptsize Eqn \eqref{eq:compatible_breakdown_C2}, Thm \ref{th:S_kernel_range}  }} \\
       & =  - \overline{\beta}_1 - \overline{\mathbf{T}}_{2,1} \alpha_2 + \overline{\mathbf{R}}_1 \overline{\eta}  & {\text{\scriptsize Eqn \eqref{eq:compatible_breakdown_B2}}}.
 \end{align*}
 Rearranging and using the fact that $\overline{\mathbf{T}}_{1,2} \partial \omega=0$ by Corollary \ref{co:bridgeworthy_TFAE}, we get
 \begin{equation} \label{eq:full_unitary_good2}
     \overline{\beta}_1 = - \overline{\mathbf{T}}_{1,1} \alpha_1  - \overline{\mathbf{T}}_{2,1} \alpha_2 + \overline{\mathbf{R}}_1 \overline{\eta}  
 \end{equation}
 as desired.
 
 The proof of the remaining equation is similar, with small differences arising from the asymmetry of the assumptions. Applying $\mathbf{T}_{2,2}$ to the first equation of  \eqref{eq:compatibility_theorem_overfared_1}, and inserting the second, we obtain
  \begin{align*}
     \overline{\mathbf{T}}_{2,2} \alpha_1 & = - \overline{\mathbf{T}}_{2,2} \rho_2 + \overline{\mathbf{T}}_{2,2} \mathbf{T}_{2,2} \overline{\gamma_2}  + \overline{\mathbf{T}}_{2,2} \mathbf{R}_2 \mathbf{S}_2 \tau_2 + \overline{\mathbf{T}}_{2,2} \partial \omega\\ 
      & = - \overline{\beta}_2 - \overline{\gamma_2} + \overline{\mathbf{R}}_2 \overline{\mathbf{S}}_2 \overline{\sigma}_2 + \overline{\mathbf{T}}_{2,2} 
      \mathbf{T}_{2,2} \overline{\gamma_2} + \overline{\mathbf{T}}_{2,2} 
      \mathbf{R}_2 \mathbf{S}_2 \tau_2
 \end{align*}
 where in the last equality we have used the fact that $\overline{\mathbf{T}}_{2,2} \partial \omega = - \overline{\partial} \omega$ by Corollary  \ref{co:bridgeworthy_TFAE}. 
  Now as above, applying the second identity of Theorem \ref{th:general_adjoint_double_identities} and the second line of Theorem \ref{th:general_adjoint_identity_double_final}  we obtain 
 \begin{align*}
     \overline{\mathbf{T}}_{2,2} \alpha_1 
       & = - \overline{\beta}_2 + \overline{\mathbf{R}}_2 \overline{\mathbf{S}}_2 \overline{\sigma_2} - \overline{\mathbf{T}}_{1,2} \mathbf{T}_{2,1} \overline{\gamma_2} - \overline{\mathbf{R}}_{2} \overline{\mathbf{S}}_2 \overline{\gamma_2} - \overline{\mathbf{T}}_{1,2} \mathbf{R}_1 \mathbf{S}_2 \tau_2 & \\
       & =  - \overline{\beta}_2 - \overline{\mathbf{T}}_{1,2} \alpha_1 +
       \overline{\mathbf{R}}_2 \overline{\mathbf{S}}_2 \overline{\sigma}_2 + \overline{\mathbf{R}}_2 \overline{\mathbf{S}}_2  \overline{\mu}_2 & {\text{\scriptsize Eqn \eqref{eq:compatible_breakdown_C1}, Thm \ref{th:S_kernel_range}  }} \\
       & =  - \overline{\beta}_2 - \overline{\mathbf{T}}_{1,2} \alpha_1 + \overline{\mathbf{R}}_2 \overline{\eta}  & {\text{\scriptsize Eqn \eqref{eq:compatible_breakdown_B1}}}.
 \end{align*}
 This completes the proof.
 \end{proof}

 In the genus zero case we have the following.
 \begin{theorem}[Scattering matrix genus zero]  \label{th:unitary_scattering_genus_zero} Assume that $g=0$ and $\riem_2$ is connected. and let $\alpha_k + \overline{\beta}_k \in \mathcal{A}_{\mathrm{harm}}(\riem_k)$.  Assume that $\alpha_1 + \overline{\beta}_1 = \mathbf{O}^{\mathrm e}(\alpha_2 + \overline{\beta_2})$.
  Then $\alpha_k$ and $\beta_k$ satisfy
   \begin{equation}  \label{eq:scattering_matrix_genus_zero}
     \left( \begin{array}{c} \overline{\beta}_1 \\ \overline{\beta}_2   \end{array} \right)
     = \left( \begin{array}{cc} {-} \overline{\mathbf{T}}_{1,1} &  { -}\overline{\mathbf{T}}_{2,1}  \\
       {-}\overline{\mathbf{T}}_{1,2} &  { -}\overline{\mathbf{T}}_{2,2}    \end{array} \right) 
      \left( \begin{array}{c} {\alpha}_1 \\ {\alpha}_2  \end{array} \right).
  \end{equation}
  and the matrix is unitary.
 \end{theorem}
 \begin{proof} 
   One obtains the much simpler proof by setting all elements of $\mathcal{A}_{\text{harm}}(\mathscr{R})$ to zero in the proof of Theorem \ref{th:unitary_scattering_genus_nonzero}. 
 \end{proof}
 
 We conclude this section with some observations on the action of the scattering matrix on harmonic measures. {These can be viewed as symmetries of the scattering process.}
 
 Fix $\alpha_k+ \overline{\beta}_k \in \mathcal{A}_{\mathrm{harm}}(\riem_k)$, $k=1,2$, are compatible with respect to $\zeta = \xi + \overline{\eta} \in \mathcal{A}_{\mathrm{harm}}(\mathscr{R})$. By Theorem  \ref{th:compatibility_existence} part (1), if $\alpha_2'+\overline{\beta}_2' \in \mathcal{A}_{\mathrm{harm}}(\riem_2)$ is another compatible form we have that 
 \begin{align*}
   \alpha_2' & = \alpha_2 + \partial \omega \\
   \overline{\beta}_2' & = \overline{\beta}_2 + \overline{\partial} \omega 
 \end{align*} 
 for some bridgeworthy form $d\omega \in \mathcal{A}_{\mathrm{bw}}(\riem_2)$. This is reflected by the matrix equation 
  \begin{equation}  \label{eq:scattering_matrix_on_bridgeworthy}
     \left( \begin{array}{c} 0  \\ \overline{\partial}  \omega \\ 0 \end{array} \right)
     = \left( \begin{array}{ccc} - \overline{\mathbf{T}}_{1,1} & - \overline{\mathbf{T}}_{2,1} & \overline{\mathbf{R}}_1 \\
      - \overline{\mathbf{T}}_{1,2} & - \overline{\mathbf{T}}_{2,2} & \overline{\mathbf{R}}_2 \\ \mathbf{S}_1 & \mathbf{S}_2 & 0 \end{array} \right) 
      \left( \begin{array}{c} 0  \\ \partial \omega \\ 0 \end{array} \right)
  \end{equation}
  which follows from Theorem \ref{th:general_T12_kernel_on_perp} and Corollary \ref{co:bridgeworthy_TFAE}. 
 
  By Proposition \ref{pr:perturb_both_by_harmonic_measure} if we fix $\zeta$ and simultaneously perturb the other two forms by a harmonic measure, the resulting forms are still compatible.  That is, if $\omega_k \in \mathcal{A}_{\mathrm{hm}}(\riem_k)$ for $k=1,2$ satisfy $\mathbf{O}_{1,2} \omega_1=\omega_2$ then 
  \begin{align*}
      \alpha_1' + \overline{\beta}_1' & = \alpha_1 + \overline{\beta}_1 + d\omega_1 \\
      \alpha_2' + \overline{\beta}_2' & = \alpha_2 + \overline{\beta}_2 + d\omega_2 
  \end{align*}
  are also compatible with respect to $\zeta$. This is in turn reflected in the following matrix equation:
  \begin{equation}  \label{eq:scattering_matrix_on_harmonic_measures}
     \left( \begin{array}{c} \overline{\partial} \omega_1 \\ \overline{\partial} \omega_2 \\ 0 \end{array} \right)
     = \left( \begin{array}{ccc} - \overline{\mathbf{T}}_{1,1} & - \overline{\mathbf{T}}_{2,1} & \overline{\mathbf{R}}_1 \\
      - \overline{\mathbf{T}}_{1,2} & - \overline{\mathbf{T}}_{2,2} & \overline{\mathbf{R}}_2 \\ \mathbf{S}_1 & \mathbf{S}_2 & 0 \end{array} \right) 
      \left( \begin{array}{c} \partial \omega_1 \\ \partial \omega_2 \\ 0 \end{array} \right)
  \end{equation}
  which follows from Theorems \ref{th:T_and_S_on_harmonic_measures} and \ref{th:S_on_harmonic_measures}.
 \end{subsection}
 \begin{subsection}{Analogies with classical potential scattering} \label{se:scattering_analogies}
  
  In this section we describe the analogy between the scattering matrix in Theorems \ref{th:unitary_scattering_genus_nonzero} and \ref{th:unitary_scattering_genus_zero} and scattering by a potential. We will restrict to the genus zero case in the former, and compare it to scattering by a potential well in one dimension.
  
  Let \(a(x, D)=\)
\(D^{2}+V(x)\) denote the $1$-dimensional Schr\"odinger operator, with \(D=-i \partial_{x}\), where the
potential \(V(x)\) is smooth and goes to 0 sufficiently fast as \(x\) goes to infinity. Because of this decay assumption at infinity, for $\lambda \in \R$ the solutions of the stationary Schr\"odinger equation

$$
a(x,D) u=\lambda^2 u,
$$
should behave like
$$
a_{+}^{l, r}(\lambda) \mathrm{e}^{i \lambda x }+a_{-}^{l, r}(\lambda) \mathrm{e}^{-i \lambda x },
$$
as $x \rightarrow \pm \infty$, where $l$ and $r$ stand for left and right and correspond respectively to $x \rightarrow-\infty$
and $x \rightarrow+\infty$. The so-called {\it{Jost solutions}} $\mathscr{J}_{\pm}^{l, r}$ are the solutions which behave exactly
as $\mathrm{e}^{i \lambda x}$ or $\mathrm{e}^{-i \lambda x}$ as $x \rightarrow \pm \infty$, and are obviously solutions to the equation $D^2 u= \lambda^2 u$, i.e. the original equation without any potential.\\
Now the scattering problem amounts to finding the components of a solution $u$ of the Schr\"odinger equation in
the basis $\left(\mathscr{J}_{+}^{r}, \mathscr{J}_{-}^{l}\right)$ of the outgoing Jost solutions, if one knows the components of $u$ in the basis
$\left(\mathscr{J}_{+}^{l}, \mathscr{J}_{-}^{r}\right)$ of the incoming Jost solutions.\\
In scattering theory and quantum mechanics,
the $2\times 2$ matrix 
$$\mathbb{S}(\lambda )=\left(\begin{array}{cc}s_{11} & s_{12} \\ s_{21} & s_{22}\end{array}\right)$$
that relates these components
is called the \emph{scattering matrix}. It also turns out that if the potential $V$ is real on
the real axis, then $\overline{\mathscr{J}_{+}^{l, r}}= \mathscr{J}_{-}^{l, r}$ and the scattering matrix
$\mathbb{S}(\lambda)$ is unitary. 

Turning to scattering in quasicircles, let $\riem_1$ and $\riem_2$ be identified with domains in the Riemann sphere $\sphere$. By M\"obius invariance we can assume without loss of generality that $\riem_2$ contains the point at $\infty$ whilst $\riem_1$ contains $0$.  The punctured plane is conformally a cylinder, with the points at $0$ and $\infty$ infinitely far away, with the quasicircle separating $0$ from $\infty$. 
We then identify left with $0$ and right with $\infty$. The quasicircle can be thought of as a potential well with a possibly highly irregular support set, whose Hausdorff dimension is in $[1,2)$.

The problem then is, given left moving solutions (harmonic forms in $\riem_1$) find right moving solutions (harmonic forms in $\riem_2$) which overfare through the potential well. The function behaves harmonically as $z \rightarrow 0/\infty$, but not across the potential well. The holomorphic forms are identified with solutions to the harmonic scattering problem  with expansions of the form 
\[  \alpha_k = \left( \sum_{k=1}^\infty \alpha_k^n z^n  \right) dz  \]
and the anti-holomorphic forms are identified with solutions to the harmonic scattering problem with  expansions of the form 
\[  \overline{\beta}_k = \left( \sum_{k=1}^\infty \beta^k_n z^{-n}  \right) dz.  \]
Thus we identify $\pm$ with holomorphic/anti-holomorphic respectively, and 
\[ \overline{\mathscr{J}_{+}^{l}} \sim \alpha_1, \ \ \  \overline{\mathscr{J}_{-}^{l}} \sim \overline{\beta}_1, \ \ \  \mathscr{J}_{+}^{r} \sim \alpha_2, \ \ \ \mathscr{J}_{-}^{r} \sim \overline{\beta}_2.  \]

 




 \end{subsection}
\end{section}
\begin{section}{The period mapping}  \label{se:period_mapping}
\begin{subsection}{Assumptions throughout this section}   
  
  \label{se:period_mapping_assumptions}
  Once again, we state the assumptions in order to avoid repetitions.  These will be in force throughout Section \ref{se:period_mapping}.   

 \begin{enumerate}
     \item $\mathscr{R}$ is a compact Riemann surface of genus $g$, with $n$ punctures $p_1,\ldots,p_n$;
     \item $\Gamma = \Gamma_1 \cup \cdots \cup \Gamma_n$ is a collection of quasicircles;
     \item $\Gamma$ separates $\mathscr{R}$ into $\riem_1$ and $\riem_2$ in the sense of Definition \ref{de:separating_complex};
     \item $\riem_2$ is connected; 
     \item $\riem_1 = \Omega_1 \cup \cdots \cup \Omega_n$ where $\Omega_1,\ldots,\Omega_n$ are simply-connected sets with disjoint closures;
     \item $p_k \in \Omega_k$ for $k=1,\ldots,n$.   
 \end{enumerate} 
 Recall that we refer to the domains $\Omega_k$ as ``caps'' (see Definition \ref{de:sewing_on_caps}).
 
 For definiteness, we will assume that $\Gamma$ have the orientation of $\riem_1$.  
\end{subsection}
\begin{subsection}{About this section}
     In this section we generalize the classical period mapping for compact surfaces to surfaces with border. This new period mapping relates both to the cohomology of the set of holomorphic one-forms on the compact surface $\mathscr{R}$, and to the structure of the set of boundary values of holomorphic one-forms on $\riem_2$. Thus it unifies both the classical polarization induced by the holomorphic one-forms on the compact surface (relating cohomology to complex structure) with the period maps of genus zero surfaces with boundary studied by various authors, including the Kirillov-Yuri'ev-Nag-Sullivan period map of the Teichm\"uller space of the disk \cite{KY2}, \cite{NS}, \cite{Takhtajan_Teo_Memoirs}, \cite{RSS_genus_zero},    
 
 In Section \ref{se:period_map_upsilon} we define a canonical isomorphism parametrizing the set of holomorphic one-forms on a surface $\riem_2$ of genus $g$ with $n$ boundary curves. We then use this to define a natural polarization and a map whose graph is the set of holomorphic one-forms on $\riem_2$.  In Section \ref{se:KYNS_period} we show how this generalizes both the classical and KYNS period mappings, and relate it to the Grunsky inequalities. Finally, in Section \ref{se:holomorphic_BVP} we use the machinery of the previous sections to give a reduction of the boundary value problem for holomorphic one-forms with $H^{-1/2}$ data to a non-singular integral equation on the $n$-fold direct sum of the Bergman space of the disk.
\end{subsection}
\begin{subsection}{The generalized period map}  \label{se:period_map_upsilon}

    Consider the space
    \[   \mathcal{A}_{\mathrm{harm}}(\riem_1) \oplus \mathcal{A}_{\mathrm{harm}}(\mathscr{R}).   \]
   
    We have the two projections  
   \[ \gls{pcap} = \mathbf{P}_{\riem_1} \oplus  {\overline{\mathbf{P}}}_{\mathscr{R}} : \mathcal{A}_{\mathrm{harm}}(\riem_1) \oplus \mathcal{A}_{\mathrm{harm}}(\mathscr{R}) \rightarrow \mathcal{A}(\riem_1) \oplus \overline{\mathcal{A}(\mathscr{R})}    \]
   and 
   \[   \overline{\mathbf{P}}_{\mathrm{cap}} = \overline{\mathbf{P}}_{\riem_1} \oplus  {\mathbf{P}}_{\mathscr{R}} : \mathcal{A}_{\mathrm{harm}}(\riem_1) \oplus \mathcal{A}_{\mathrm{harm}}(\mathscr{R}) \rightarrow \overline{\mathcal{A}(\riem_1)} \oplus {\mathcal{A}(\mathscr{R})}   \]
   where $\mathbf{P}_{\riem_1}$ was defined in \eqref{eq:hol_antihol_projections_Bergman} and  
   \[   {\mathbf{P}}_{\mathscr{R}}: \mathcal{A}_{\mathrm{harm}}(\mathscr{R}) \rightarrow \mathcal{A}(\mathscr{R})   \]
   is the projection onto the holomorphic part, and similarly $ {\overline{\mathbf{P}}}_{\mathscr{R}}$ is the projection onto the anti-holomorphic part.  The projections are obviously bounded.

   We define the following operator, which we will shortly show is an isomorphism.    
   \begin{align}\label{defn: captial theta}
       \gls{theta}: \overline{\mathcal{A}(\riem_1)} \oplus \mathcal{A}(\mathscr{R}) & \rightarrow \mathcal{A}^{\mathrm{se}}(\riem_2)  \\
      \nonumber (\overline{\gamma},\tau) & \mapsto - \mathbf{T}_{1,2} \overline{\gamma} +  {\mathbf{R}}_2 \tau.  
   \end{align}
   This is obviously bounded.

  We define an augmented overfare operator which contains the extra data of the cohomology class. This will be a factor of the inverse of $\Theta$. First, observe that given $\beta \in \mathcal{A}^{\mathrm{se}}_{\text{harm}}(\riem_2)$, there is a unique one-form $\sigma \in \mathcal{A}(\mathscr{R})$ whose restriction $\mathbf{R}^\mathrm{h}_2 \sigma$ is in the same cohomology class as $\beta$.  Thus $\alpha \in \mathcal{A}(\riem_1)$ and $\beta \in \mathcal{A}(\riem_2)$ can be compatible only via the form $\sigma$.  Also, since $\riem_2$ is connected, the exact overfare is well-defined.  Thus, given $\beta$, we have the following uniquely determined compatible form $\hat{\mathbf{O}} \beta \in \mathcal{A}_{\text{harm}}(\riem_1)$:
    \begin{align*}
       \gls{ohat}: \mathcal{A}^{\mathrm{se}}_{\mathrm{harm}}(\riem_2) & \rightarrow \mathcal{A}_{\mathrm{harm}}(\riem_1)  \\
      \beta & \mapsto \mathbf{O}^{\mathrm{e}}(\beta -  {\mathbf{R}}_2^{\mathrm{h}} \sigma) +
        {\mathbf{R}}_1^{\mathrm{h}}\sigma  
    \end{align*}
    where $\sigma$ is the unique element of 
    $\mathcal{A}_{\mathrm{harm}}(\mathscr{R})$ such that $\beta-  {\mathbf{R}}_2^{\mathrm{h}} \sigma$ is exact.  
 Using this, we define the augmented overfare map
 \begin{align}   \label{eq:Oaugmented_definition}  
   \gls{augo}: \mathcal{A}^{\mathrm{se}}_{\mathrm{harm}}(\riem_2) & \rightarrow \mathcal{A}_{\mathrm{harm}}(\riem_1) \oplus \mathcal{A}_{\mathrm{harm}}(\mathscr{R})  \nonumber \\
   \beta   & \mapsto (  \hat{\mathbf{O}} \beta, \sigma )
 \end{align}
 where $\sigma$ is the unique element of $\mathcal{A}_{\mathrm{harm}}(\mathscr{R})$ such that $\beta -  {\mathbf{R}}^{\mathrm{h}}_2  \sigma$ is exact.

   \begin{theorem}   \label{th:Theta_isomorphism} $\Theta$ is an isomorphism with inverse $\overline{\mathbf{P}}_{\mathrm{cap}} \mathbf{O}^{\mathrm{aug}}$.  
   \end{theorem}
   \begin{proof} The fact that $\Theta$ is an isomorphism follows directly from Corollary \ref{co:T_plus_extra_surjective_capped}.  
    
    We show it is injective.  Let $(\overline{\gamma},\tau) \in \overline{\mathcal{A}(\riem_1)} \oplus \mathcal{A}(\mathscr{R})$.  
    We then have that 
    \[    - \mathbf{T}_{1,2} \overline{\gamma} +  {\mathbf{R}}_2 \tau - (\overline{\mathbf{R}}_2 \overline{\mathbf{S}}_1 \overline{\gamma} + \mathbf{R}_2  {\tau} )    
            \]
    is exact by Theorem \ref{th:Schiffer_cohomology}.  So applying \eqref{eq:Oaugmented_definition}, we obtain 
    \begin{align}   \label{eq:augmented_overfare_explicit}
     \mathbf{O}^{\mathrm{aug}}  (- \mathbf{T}_{1,2} \overline{\gamma} +  {\mathbf{R}}_2 \tau ) & = ( \hat{\mathbf{O}}(-\mathbf{T}_{1,2} \overline{\gamma} +  {\mathbf{R}}_2 \tau ),\overline{\mathbf{S}}_1 \overline{\gamma} + \tau   )   \nonumber \\
     & = ( \overline{\gamma} -  \mathbf{T}_{1,1} \overline{\gamma}
       + \mathbf{R}_1 \tau_1, \overline{\mathbf{S}}_1 \overline{\delta} +  \tau  )
    \end{align}
    where we have used Proposition \ref{prop:exact_transmitted_jump} to show that
    \begin{align*}
        \hat{\mathbf{O}}(- \mathbf{T}_{1,2} \overline{\gamma} +  {\mathbf{R}}_2 \tau )  & = \mathbf{O}^{\mathrm{e}} (- \mathbf{T}_{1,2} \overline{\gamma} - {\overline{\mathbf{R}}}_1 \overline{\mathbf{S}}_1 \overline{\gamma})  + {\overline{\mathbf{R}}}_1 \overline{\mathbf{S}}_1 \overline{\gamma}  +  {\mathbf{R}}_1 \tau \\
       & = \overline{\gamma} -  \mathbf{T}_{1,1} \overline{\gamma} - \overline{\mathbf{R}}_1 \overline{\mathbf{S}}_1 \overline{\gamma} 
       + \overline{\mathbf{R}}_1 \overline{\mathbf{S}}_1 \overline{\gamma} + \mathbf{R}_1 \tau \\
       & = \overline{\gamma} -  \mathbf{T}_{1,1} \overline{\gamma}
       + {\mathbf{R}}_1 \tau.
    \end{align*}
    Thus 
    \begin{align*}
     \overline{\mathbf{P}}_{\mathrm{cap}} \mathbf{O}^{\mathrm{aug}} \Theta (\overline{\gamma},\tau) & =  \overline{\mathbf{P}}_{\mathrm{cap}} ( \overline{\gamma} -  \mathbf{T}_{1,1} \overline{\gamma}
       + \mathbf{R}_1 \tau,  \overline{\mathbf{S}}_1\overline{\gamma} +  \tau  ) \\
       & = (\overline{\gamma}, \tau).   
    \end{align*}
    This shows that $\overline{\mathbf{P}}_{\mathrm{cap}} \mathbf{O}^{\mathrm{aug}}$ is  a left inverse
    of $\Theta$, and hence $\Theta$ is one-to-one.   
    
    Since $\Theta$ is bounded and bijective, it is an isomorphism, and the left inverse equals the right inverse.
   \end{proof}
   
   The decomposition
   \begin{equation}  \label{eq:riemtwo_decomposition}
      \mathcal{A}^{\mathrm{se}}_{\mathrm{harm}}(\riem_2) = \mathcal{A}^{\mathrm{se}}(\riem_2) \oplus \overline{\mathcal{A}^{\mathrm{se}}(\riem_2)}    
   \end{equation}
   induces a polarization on 
   \[  \mathcal{A}_{\mathrm{harm}}(\riem_1) \oplus \mathcal{A}_{\mathrm{harm}}(\mathscr{R}) \]
  by 
   \begin{align} \label{eq:polarization_general_non_trivialized}
       \mathcal{A}_{\mathrm{harm}}(\riem_1) \oplus \mathcal{A}_{\mathrm{harm}}(\mathscr{R}) 
       & = \mathbf{O}^{\mathrm{aug}} \mathcal{A}^{\mathrm{se}}(\riem_2) \oplus \mathbf{O}^{\mathrm{aug}} \overline{\mathcal{A}^{\mathrm{se}}(\riem_2)} \nonumber \\
       & = W \oplus \overline{W}.  
   \end{align}
    We also have the fixed polarization 
   \[    \mathcal{A}_{\mathrm{harm}}(\riem_1) \oplus \mathcal{A}_{\mathrm{harm}}(\mathscr{R}) = W_0 \oplus \overline{W_0}  \]
   where 
   \[   W_0 = \overline{\mathcal{A}(\riem_1)} \oplus \mathcal{A}(\mathscr{R}).   \]
   Observe that by Theorem \ref{th:Theta_isomorphism} the new positive polarization $W$ can be written  
   \begin{equation*} 
     W = \mathbf{O}^{\mathrm{aug}} \, \Theta \, W_0. 
   \end{equation*}
   
   Furthermore we have the following result.  Define 
   \begin{align}\label{defn:upsilon}
      \gls{period} :\overline{\mathcal{A}(\riem_1)} \oplus \mathcal{A}(\mathscr{R})  &  \rightarrow \mathcal{A}(\riem_1) \oplus \overline{\mathcal{A}(\mathscr{R})} \\
      \nonumber(\overline{\gamma},\tau) & \mapsto ( - \mathbf{T}_{1,1} \overline{\gamma} +  {\mathbf{R}}_1 \tau,  \overline{\mathbf{S}}_1 \overline{\gamma} ). 
   \end{align}
   Then 
   \begin{theorem}  \label{th:upsilon_expression} We have that $\mathbf{\Upsilon} =  {\mathbf{P}}_{\mathrm{cap}} \mathbf{O}^{\mathrm{aug}} \Theta$.  In particular, 
     the positive polarization $W$ is the graph of $\Upsilon$.  
   \end{theorem}
   \begin{proof}
    The first claim follows from \eqref{eq:augmented_overfare_explicit}.  The second claim follows from the first claim together with the fact that $\overline{\mathbf{P}}_{\mathrm{cap}} \mathbf{O}^{\mathrm{aug}} \Theta = \mathbf{I}$ by Theorem \ref{th:Theta_isomorphism}.    
   \end{proof}
   
   We also have
   \begin{theorem} \label{th:upsilon_bounded}
    The operator norm $\| \mathbf{\Upsilon} \| <1$.   
   \end{theorem}
   \begin{proof}
    Using the notation of the proof of Theorem \ref{th:Theta_isomorphism}, we have 
    \begin{align}  \label{eq:upsilon_bound_temp_one} 
        \| - \mathbf{T}_{1,1} \overline{\gamma} +  {\mathbf{R}}_1 \tau \|^2 
        & = \| \mathbf{T}_{1,1} \overline{\gamma} \|^2 - 2 \text{Re} \left< \mathbf{T}_{1,1} \overline{\gamma}, \mathbf{R}_1 \tau \right> 
        + \| \mathbf{R}_1 \tau \|^2.
    \end{align}
    Now by Theorem \ref{th:adjoint_identities} and equation (\ref{eq:general_adjoint_identity_double_final}) we have that
    \begin{align} \label{eq:upsilon_bound_temp_two}
     \left< \mathbf{T}_{1,1} \overline{\gamma}, \mathbf{R}_1 \tau \right> & = \left< \mathbf{S}_1 \mathbf{T}_{1,1} \overline{\gamma} ,\tau \right> 
     = - \left< \mathbf{S}_2 \mathbf{T}_{1,2} \overline{\gamma} ,\tau \right> \nonumber \\
     & = - \left< \mathbf{T}_{1,2} \overline{\gamma}, \mathbf{R}_2 \tau \right>.
    \end{align}
    Furthermore Theorems \ref{th:adjoint_identities} and \ref{th:quadratic_adjoint_one} also yield that
    \begin{equation}  \label{eq:upsilon_bound_temp_three}
        \| \mathbf{R}_1 \tau \|^2 = \| \tau \|^2 - \| \mathbf{R}_2 \tau \|^2.
    \end{equation}
    Inserting \eqref{eq:upsilon_bound_temp_two} and \eqref{eq:upsilon_bound_temp_three} in \eqref{eq:upsilon_bound_temp_one} we obtain 
    \begin{equation} \label{eq:upsilon_bound_temp_four}
        \| - \mathbf{T}_{1,1} \overline{\gamma} +  {\mathbf{R}}_1 \tau \|^2  = 
        \| \mathbf{T}_{1,1} \overline{\gamma} \|^2 + 2 \text{Re} \left<  \mathbf{T}_{1,2} \overline{\gamma} ,\mathbf{R}_2 \tau \right> - \| \mathbf{R}_2 \tau \|^2 + \| \tau \|^2.  
    \end{equation}
    By Theorem \ref{th:quadratic_adjoint_one}, 
    \[  \| \mathbf{T}_{1,1} \overline{\gamma} \|^2 + \| \overline{\mathbf{S}}_{1} \overline{\gamma} \|^2 = \| \overline{\gamma} \|^2 - \| \mathbf{T}_{1,2} \overline{\gamma} \|^2  \]
    which when inserted in \eqref{eq:upsilon_bound_temp_four} yields
    \begin{align*}
        \| (-\mathbf{T}_{1,1} \overline{\gamma} +  {\mathbf{R}}_1 \tau,\overline{\mathbf{S}}_1 \overline{\gamma}) \|^2 & = 
        \| \overline{\gamma} \|^2 -\| \mathbf{T}_{1,2} \overline{\gamma} \|^2 + 2 \text{Re} \left<  \mathbf{T}_{1,2} \overline{\gamma} ,\mathbf{R}_2 \tau \right> - \| \mathbf{R}_2 \tau \|^2 + \| \tau \|^2  \\
        & = \| (\overline{\gamma},\tau )\|^2 - \| - \mathbf{T}_{1,2} \overline{\gamma} + \mathbf{R}_2 \tau \|^2. 
    \end{align*}
    Now since $(\overline{\gamma},\tau) \mapsto -\mathbf{T}_{1,2} \overline{\gamma} + \mathbf{R}_2 \tau$ is an isomorphism 
   by Theorem \ref{th:Theta_isomorphism}, there is a $c<1$ (uniform for all $(\overline{\gamma},\tau)$) such that
    \[ \| - \mathbf{T}_{1,2} \overline{\gamma} + \mathbf{R}_2 \tau \| \geq c \| (\overline{\gamma},\tau) \|,   \]
    therefore
    \begin{equation*}
        \| (-\mathbf{T}_{1,1} \overline{\gamma} + {\mathbf{R}}_1 \tau,\overline{\mathbf{S}}_1 \overline{\gamma}) \|^2 \leq (1-c^2) 
        \| (\overline{\gamma},\tau) \|^2,
    \end{equation*}
    and the proof is completed.
   \end{proof}
\end{subsection}
\begin{subsection}{Generalized polarizations}
 \label{se:KYNS_period} 
 Theorems \ref{th:upsilon_expression} and \ref{th:upsilon_bounded} generalize the Kirillov-Yuriev-Nag-Sullivan (KYNS) period map to Riemann surfaces of arbitrary genus and number of boundary curves, and unify it with the classical period map of compact surfaces. This fact will be shown below, but prior to that, we shall review some of the literature.  For the sake of clarity, we will take some liberties by imposing our notation and choice of function spaces in discussion of the literature. For example, we freely take advantage of the isomorphism between the homogeneous Sobolev space, Dirichlet space and the Bergman space of one-forms on the disk.
 
 Nag and Sullivan \cite{NS}, following Kirillov and Yuriev \cite{KrillovYuriev} in the smooth case, showed how the group of quasisymmetries of the circle $\mathrm{QS}(\mathbb{S}^1)$ acts symplectically on  $\dot{H}^{1/2}(\mathbb{S}^1)$. Setting $W_0 = \mathcal{A}(\disk)$, $\overline{W}_0 = \overline{\mathcal{A}(\disk)}$,
 we have the standard polarization
 \[ \dot{H}^{1/2}(\mathbb{S}^1) = W_0 \oplus \overline{W}_0.  \]
 Each quasisymmetry induces a new positive polarization $W \oplus \overline{W} = \dot{H}^{1/2}(\mathbb{S}^1)$, or equivalently an operator $\mathbf{Gr}:\overline{W}_0 \rightarrow W_0$ of norm strictly less than one whose graph is $W$. Takhtajan and Teo \cite{Takhtajan_Teo_Memoirs} made the important discovery that this operator can be identified with the classical Grunsky operator. They also showed that the resulting period mapping taking an element of the universal Teichm\"uller space to its operator $\mathbf{Gr}$ is holomorphic for both the full universal  Teichm\"uller space and the Weil-Petersson universal Teichm\"uller space.\\

Now fix a Riemann surface $\riem$ of type $(g,n)$.  We choose a collection 
 \[ \phi_k:\mathbb{S}^1 \rightarrow \partial_k \riem, \ \ \ k=1,\ldots,n \] 
 of quasisymmetric mappings, whose purpose is to map the boundary values into a fixed space.  By Theorem \ref{th:sewing_complex_structure} the resulting compact topological space has a unique complex structure compatible with $\riem_2$ and the disks $\disk$.  We call this Riemann surface $\mathscr{R}$, and the common boundaries of each $\disk$ and $\partial_k \riem_2$ for each $k$ are a separating complex of quasicircles satisfying the assumptions of Section \ref{se:period_mapping_assumptions}. After uniformizing, the copies of the disk are identified with simply connected domains $\Omega_1,\ldots,\Omega_n$ which are biholomorphic to $\disk$ and bounded by non-intersecting quasicircles.  Set $\riem_1 = \Omega_1 \cup \cdots \cup \Omega_n$ and let $\riem_2$ denote the complement of their closure in $\riem_1$. For $k=1,\ldots,n$, $\Omega_k$ is biholomorphic to the disk under the conformal extensions 
 \[ f_k: \disk \rightarrow  \Omega_k \]
 of the quasisymmetric maps $\phi_k$. Furthermore $\riem$ is biholomorphic to $\riem_2$ under the uniformizing map; we henceforth identify $\riem_2$ with $\riem$. 
 
 Let $\Omega(\mathscr{R})$ denote the cohomology classes of $L^2$ one-forms on $\riem$ which are semi-exact. 
 The complex structure on $\riem$ determines a complex structure on $\mathscr{R}$, which thus determines the class of harmonic forms on $\mathscr{R}$.  We thus obtain a map 
 \[ \mathbf{P}_{\mathrm{harm},\mathscr{R}}:L^2(\mathscr{R}) \rightarrow \mathcal{A}_{\mathrm{harm}}(\mathscr{R})  \]
 which depends on the complex structure of $\mathscr{R}$ (and hence of $\riem$).  We then define the projections onto the anti-holomorphic and holomorphic parts
 \[ \mathbf{P}_{\mathrm{harm},\mathscr{R}} = \mathbf{P}_\mathscr{R} \oplus \overline{\mathbf{P}}_\mathscr{R}:L^2(\mathscr{R}) \rightarrow \mathcal{A}(\mathscr{R}) \oplus \overline{\mathcal{A}(\mathscr{R})}  \] 
 (note that we are expanding the domain of $\mathbf{P}_{\mathscr{R}}$ to $L^2(\mathscr{R})$ for the sake of the discussion in this section ).
 Thus we obtain a polarization 
 \[ \Omega(\mathscr{R}) = W_{\mathscr{R}} \oplus \overline{W}_{\mathscr{R}}  \]
 via 
 \[  W = \mathbf{P}_{\mathscr{R}} \Omega(\mathscr{R}),  \ \ \  \overline{W} = \overline{\mathbf{P}}_{\mathscr{R}}  \Omega(\mathscr{R}).    \]
 
 Denoting $f=f_1 \times \cdots \times f_n$ and $\mathcal{A}_{\mathrm{harm}}(\disk)^n = \mathcal{A}_{\mathrm{harm}} \oplus \cdots \oplus \mathcal{A}_{\mathrm{harm}}$ we then have the polarization
 \[  \mathcal{A}_{\mathrm{harm}}(\disk)^n \oplus \Omega(\mathscr{R}) = 
   \mathcal{W}  \oplus \overline{\mathcal{W}}     \]
   where
   \begin{equation}
   \label{eq:general_polarization}
     \mathcal{W} = (f^*,\mathbf{Id}) W 
   \end{equation}
   and $W$ is given by \eqref{eq:polarization_general_non_trivialized}. 
Thus this combines both the classical polarization associated to the complex structure of the compact surface (in the second entry of \eqref{eq:general_polarization}) and the boundary values of the set of one-forms (in the first element of \eqref{eq:general_polarization}). 

 \begin{remark}
  Observe that both the set of quasisymmetries $\phi_k:\mathbb{S}^1 \rightarrow \partial_k \riem$ and the space of $L^2$ one-forms is unchanged under quasiconformal deformations $f:\riem \rightarrow \riem$. Thus the space $\Omega^{\mathrm{se}}(\riem)$
  is invariant under a quasiconformal deformation of the complex structure of $\riem$. This fact is one of the motivations for our analytic choices ($L^2$-boundary values and separating curves being quasicircles) in this paper.

 \end{remark}
 \end{subsection}
 \begin{subsection}{Generalized Grunsky inequalities}
 In this section, we show how Theorem \ref{th:upsilon_bounded} generalizes various versions of the Grunsky inequalities appearing in the literature to the case of surfaces of type $(g,n)$, after pulling back to $n$ copies of the disk.   In general, it is the overfare results (either in special cases or in general) which makes it possible to interpret the Grunsky portion of the polarization in terms of boundary values. Here we consider two cases:\\
 
 {\bf Case I: $g=0$.} If we assume that the genus of $\riem_1$ is zero and $n=1$, then $\mathscr{R} = \sphere$, $\mathcal{A}(\mathscr{R}) = \{0 \}$ and $\mathcal{A}^{\mathrm{se}}(\riem_2)=\mathcal{A}^\mathrm{e}(\riem_2)$. Also note that $(\overline{\mathbf{R}}_1 \overline{\mathcal{A}(\mathscr{R})})^\perp = \overline{\mathcal{A}(\riem_1)}$. 
 Thus the map $\Theta$ (defined previously by \eqref{defn: captial theta}) takes the form
 \begin{align*}
    \Theta:\overline{\mathcal{A}(\riem_1)} & \rightarrow \mathcal{A}^{\mathrm{e}}(\riem_2) \\    
    \overline{\gamma} & \mapsto \mathbf{T}_{1,2} \overline{\gamma}.
 \end{align*}
 In the case that $n=1$ the fact that $\mathbf{T}_{1,2}$ is an isomorphism was first proved by V. V. Napalkov and R. S. Yulmukhametov \cite{Nap_Yulm}. In the genus zero case for general $n$ this is due to Radnell, Schippers, and Staubach \cite{RSS_Dirichletspace}.
 
 The Grunsky inequalities are obtained as follows. With the observations above, the map $\Upsilon$ (defined previously by \eqref{defn:upsilon}) is seen to take the form 
 \begin{align*}
     \Upsilon: \overline{\mathcal{A}(\riem_1)} & \rightarrow \mathcal{A}(\riem_2) \\
     \overline{\gamma} & \mapsto - \mathbf{T}_{1,1} \overline{\gamma}. 
 \end{align*}
 so that Theorem \ref{th:upsilon_bounded} implies that
 \begin{equation}\label{norm of T11}
      \| \mathbf{T}_{1,1} \| <1.
 \end{equation}  
 This is equivalent to the classical estimate on the classical Schiffer operator given in example \ref{ex:sphere_kernels}; the estimate is a version of the Grunsky inequalities appearing in Bergman and Schiffer \cite{BergmanSchiffer}, though they assume that the boundary curves are analytic.  
 
 Explicitly, in the case that $n=1$, pulling back this estimate to the disk via the map $f:\disk \rightarrow \riem_1$ we obtain the usual form of the Grunsky inequalities. Following \cite{Schippers_Staubach_CAOT}, we define the Grunsky operator as follows
 \begin{align*}
     \gls{grunsk} : \overline{\mathcal{A}(\disk)} & \mapsto \mathcal{A}(\disk) \\
     \overline{\alpha} & \mapsto \mathbf{P}_{\disk} f^* \mathbf{O}^\mathrm{e}_{2,1} \mathbf{T}_{1,2} (f^{-1})^* \overline{\alpha}.
 \end{align*}
 See \cite{Schippers_Staubach_CAOT} for the relation to the usual Grunsky operator written in terms of Faber polynomials and Grunsky coefficients. Using Proposition \ref{prop:exact_transmitted_jump} and the obvious fact that $\mathbf{P}_{\disk} f^* = f^* \mathbf{P}_{\riem_1}$ we obtain
 \begin{align} \label{eq:Grunsky_conjugation_expression}
     \mathbf{Gr}_f \overline{\alpha} & = f^* \mathbf{P}_{\riem_1} \mathbf{O}^\mathrm{e} \mathbf{T}_{1,2} (f^{-1})^* \overline{\alpha} \nonumber \\
     & = f^* \mathbf{P}_{\riem_1} \mathbf{O}^\mathrm{e} ( - (f^{-1})^*\overline{\alpha} + \mathbf{T}_{1,1} (f^{-1})^* \overline{\alpha} ) \nonumber \\
     & = - f^* \mathbf{T}_{1,1} (f^{-1})^* \overline{\alpha}.
 \end{align}
 Since $f^*$ and $(f^{-1})^*$ are isometries \eqref{norm of T11} yields that 
 \[ \| \mathbf{Gr}_f \| <1. \]
 
 This is equivalent to an integral form of the Grunsky inequalities due to Bergman-Schiffer \cite{BergmanSchiffer}. To see this, using conformal invariance of the Schiffer $L$-kernel and Example \ref{ex:disk_kernels}, we have that
 \[  L_{\riem_1}(z,w) = - \frac{1}{2 \pi i} \frac{(f^{-1})'(w)  (f^{-1})'(z) \,dw \,dz }{(f^{-1}(w)-f^{-1}(w))^2}.    \]
 Combining this with Example \ref{ex:sphere_kernels} and equation \eqref{eq:nonsingular_Schiffer} we obtain
 \[  \mathbf{T}_{1,1} \overline{\alpha} = \iint_{\riem_1} \frac{dz}{2 \pi i} \left[ \frac{dw}{(w-z)^2} - \frac{(f^{-1})'(w) (f^{-1})'(z) \,dw}{(f^{-1}(w)-f^{-1}(z))^2}  \right] \wedge_w \overline{\alpha(w)}.    \]
 Now let $\overline{\alpha(w)} = \overline{h'(w)} d\bar{w} \in \overline{\mathcal{A}(\disk)}$ (where $h(w) \in \mathcal{D}(\disk)$). Then using the above together with \eqref{eq:Grunsky_conjugation_expression} we see that (after a change of variables)
 \[  \mathbf{Gr}_f \overline{\alpha}  = 
\frac{1}{\pi} \iint_{\disk} dz \left[ \frac{1}{(w-z)^2} - \frac{f'(w) f'(z)}{(f(w)-f(z))^2}  \right] \overline{h'(w)}\frac{ d\bar{w} \wedge dw}{2i}.  \]
It is a well-known fact, originating with Bergman and Schiffer \cite{BergmanSchiffer}, that the bound of one on the norm of this operator implies the Grunsky inequalities for the function $f$ (see e.g. \cite{Schippers_Staubach_CAOT,Schippers_Staubach_Grunsky_expository}).

 Similarly, in the case that $n>1$, pulling back to $\mathbb{D}^n$ via the maps $f_1,\ldots,f_n$ results in the Grunsky operator for multiply-connected domains (see \cite{RSS_Dirichletspace}). 
 
 For a detailed discussion of the literature surrounding the case $n=1$ see {\cite{Schippers_Staubach_Grunsky_expository}}.\\
 
 
 {\bf Case II: $g>0$.}  The Grunsky operator in higher genus was defined, and bounds obtained, by M. Shirazi \cite{Shirazi_thesis,Shirazi_Grunsky}, for the case of Dirichlet bounded functions. Here we formulate this in terms of $\mathcal{A}^\mathrm{e}(\riem_2)$, which is of course equivalent up to constants. First, as in \cite{Shirazi_Grunsky} we restrict our attention to the space $(\overline{\mathbf{R}}_1 \overline{\mathcal{A}(\mathscr{R})})^\perp$ and ignore the second component of $\Theta$; that is, we consider
 \[  \Theta' = \left. \Theta \right|_{((\overline{\mathbf{R}}_1 \overline{\mathcal{A}(\mathscr{R})})^\perp \oplus \{0\})}. \]
 In that case, the operator $\Theta$ takes the form 
 \begin{align*}
     \Theta':(\overline{\mathbf{R}}_1 \overline{\mathcal{A}(\mathscr{R})})^\perp & \rightarrow \mathcal{A}^\mathrm{e}(\riem_2) \\
     \overline{\alpha} & \mapsto - \mathbf{T}_{1,2} \overline{\alpha}.
 \end{align*}
 The fact that $\Theta'$ is an isomorphism was obtained by M. Shirazi \cite{Shirazi_thesis,Schippers_Staubach_Shirazi}. We have that the restriction  
 \[  \Upsilon' = \left. \Upsilon \right|_{((\overline{\mathbf{R}}_1 \overline{\mathcal{A}(\mathscr{R})})^\perp \oplus \{0\})} \]
 takes the form 
 \begin{align*}
     \Upsilon': (\overline{\mathbf{R}}_1 \overline{\mathcal{A}(\mathscr{R})})^\perp& \rightarrow \mathcal{A}^\mathrm{e}(\riem_2) \\
     \overline{\gamma} & \mapsto - \mathbf{T}_{1,1} \overline{\gamma}. 
 \end{align*}
 so that once again Theorem \ref{th:upsilon_bounded} implies that
 \[  \| \mathbf{T}_{11} \| <1. \]
 As in the genus zero case, we can define the Grunsky operator 
 \begin{align*}
     \mathbf{Gr}_f : V & \mapsto \bigoplus^n \mathcal{A}(\disk) \\
     \overline{\alpha} & \mapsto \mathbf{P}_{\disk} f^* \mathbf{O}^\mathrm{e}_{2,1} \mathbf{T}_{1,2} (f^{-1})^* \overline{\alpha}.
 \end{align*}
 where 
 \[  V= f^* (\overline{\mathbf{R}}_1 \overline{\mathcal{A}(\mathscr{R})})^\perp \]
 and $f^* = f_1^* \times \cdots \times f_n^*$.
 The Grunsky inequality obtained by M. Shirazi mentioned above is that the norm of $\mathrm{Gr}_f$ is less than one, which follows from $\| \Upsilon \| <1$. By Section \ref{se:KYNS_period} (restricting to exact one-forms), the graph of this Grunsky operator can be interpreted as the set of boundary values of holomorphic functions. See the work of Shirazi \cite{Shirazi_thesis,Shirazi_Grunsky}, for the details.
 
 Here we have not dealt with the deformation theory of Riemann surfaces, since that would require lengthening the paper impractically. The results of this entire paper, and in particular the above discussion, should be placed in the context of Teichm\"uller theory. This would include for example demonstration of the holomorphicity of this period map as well as holomorphicity of its restriction to the Weil-Petersson Teichm\"uller space. We hope to deal with this, along with a treatment of the symplectic group actions by quasisymmetric reparameterizations, in future publications. 
 
\end{subsection}
\begin{subsection}{The holomorphic boundary value problem} 
\label{se:holomorphic_BVP}
 We motivate the problem, placing analytic issues aside for the moment.\\
 
 {\bf{Problem.}} Given a one-form $\alpha$ on the boundary of $\riem_2$ and a fixed cohomology class on $\riem_2$, is there a holomorphic one-form on $\riem_2$ with boundary values equal to $\alpha$?\\
 
The cohomology class can be fixed by specifying periods, or equivalently any one-form in $L^2(\riem_2)$ in that cohomology class. We express the boundary values of the one-form $\alpha$ by parametrizing the boundary by maps $\phi_k:\mathbb{S}^1 \rightarrow \partial_k {\riem}_2$ from the circle to the boundaries .  That is, we look at the boundary parametrization as a kind of coordinate, and pull back the one-form to the circle, and specify the data on $\mathbb{S}^1$.  This data can be viewed as a one-form.   
 
 Adding analytic issues to the picture, assume now that the one-form is in $\mathcal{H}'(\partial \riem)$ (given in Definition \ref{defn: Hprimes}) and the boundary parametrization is a quasisymmetry.  If it has zero period around its boundaries, then the anti-derivative is an element of $\mathcal{H}(\partial \riem)$ (the Osborn space of Definition \ref{defn: Osborn space}), and its pull-back to the disk is an element of $\mathcal{H}(\mathbb{S}^1)$.  In the general case, the original data can be shown to be an element of $\mathcal{H}'(\mathbb{S}^1)$.  
 
 An equivalent picture is as follows.  We sew copies of the disk $\disk^+$ to each boundary curve via quasisymmetries $\phi_1,\ldots,\phi_n$ as in Section \ref{se:KYNS_period} to obtain the surface $\riem_2$ capped by $\riem_1$, with conformal maps $f_k:\disk \rightarrow \Omega_k$ where $\Omega_k$ are the connected components of $\riem_1$.   The data can now be taken to be elements of $\mathcal{H}'(\partial \riem_1)$, and the cohomology class can be specified by an element of $\mathcal{A}_{\mathrm{harm}}(\mathscr{R})$.    
 
 With this motivation, consider the following boundary value problem for holomorphic one-forms. We treat the case that the periods around boundary curves $\partial_k \riem_2$ are zero.
 From this point forward, we make careful analytic definitions and statements.\\
 
 We first state the problem in terms of $H^{-1/2}$ boundary values. 
 \begin{definition}[Holomorphic boundary value problem for semi-exact one-forms with $H^{-1/2}$ data]
 \[  \lambda = (\lambda_1,\ldots,\lambda_{2g}) \in \mathbb{C}^{2g}, \] and let $L \in H^{-1/2}(\partial \riem_2)$.  We say that $\beta \in \mathcal{A}^{\mathrm{se}}(\riem_2)$ solves the holomorphic boundary value problem if it satisfies  
 \[  L_{[\beta]} = L  \]
 and 
 \[ \int_{c_j} \beta = \lambda_j   \]
 for $j=1,\ldots,2g$.
 \end{definition}

 The problem is not well-posed in general.  We will give precise conditions for the existence of a solution momentarily.  
 
 First, we reformulate the problem using the theory of Sections \ref{se:Dirichlet_forms_Hnegativeonehalf} and \ref{se:Overfare}.
 Assume that $\beta$ solves the boundary value problem with respect to the data $\lambda$ and $L$. Assume also that $\delta \in \mathcal{A}_{\mathrm{harm}}(\riem_1)$ is the solution to the $H^{-1/2}$ boundary value problem on $\riem_1$ with respect to $\mathbf{O}'(\partial \riem_2,\partial \riem_1)L_{[\delta]}$. Such a solution is guaranteed to exist by Theorem \ref{th:Hminusonehalf_Dirichlet_problem_oneforms} applied separately to each connected component of $\riem_1$. Let $\zeta$ be the unique element of  $\mathcal{A}_{\mathrm{harm}}(\mathscr{R})$ with periods 
 \[ \int_{c_j} \zeta = \lambda_j.  \] 
 Then $\delta$ and $\beta$ are weakly compatible with respect to $\zeta$.  
 
 Conversely, if $\delta \in \mathcal{A}_{\mathrm{harm}}(\riem_1)$ and $\beta \in \mathcal{A}(\riem_2)$ are weakly compatible with respect to $\zeta \in \mathcal{A}_{\mathrm{harm}}(\mathscr{R})$ then $\beta$ solves the boundary value problem with data $L=\mathbf{O}'(\partial \riem_1,\partial \riem_2) L_{[\delta]}$.  
 
 Thus we have the following reformulation of the boundary value problem.
 \begin{definition}[Holomorphic CNT Dirichlet BVP for one-forms, semi-exact case]  Let \[  (\delta,\zeta) \in \mathcal{A}_{\mathrm{harm}}(\riem_1) \oplus \mathcal{A}_{\mathrm{harm}}(\mathscr{R}).  \]  We say that $\beta \in \gls{seform} (\riem_2)$ solves the holomorphic boundary value problem with respect to this data if $\delta$ and $\beta$ are weakly compatible with respect to the one-form $\zeta$.  
 \end{definition}

 This allows us to solve the BVP in the following way.  
 \begin{theorem}[Well-posedness of the semi-exact CNT BVP for holomorphic one-forms] \label{th:HBVP_solution_semi-exact}   Let the data $(\delta,\zeta)$ for the semi-exact holomorphic \emph{BVP} be given as above, and
  assume that $\nu,\tau$ are the unique elements of  $\mathbf{R}_1 \mathcal{A}(\mathscr{R})$ such that   $\overline{\mathbf{S}}_1 \overline{\nu} + \mathbf{S}_1 \tau = \zeta$.  
 The semi-exact holomorphic \emph{CNT} Dirichlet \emph{BVP} for forms has a solution with data $(\delta,\zeta)$ if and only if 
 \begin{equation} \label{eq:one_form_BVP_solvability_condition}
   \left[ \delta - \mathbf{R}_1 \mathbf{S}_1  \tau \right]  \in  \mathrm{Im}  [\mathbf{I} - \mathbf{T}_{1,1}].
 \end{equation}
  If this solution exists, it is unique and equals 
  \[  \beta =  -\mathbf{T}_{1,2} \overline{\gamma} + \mathbf{R}_2  \mathbf{S}_1 \tau \]
  where $\overline{\gamma} \in \overline{\mathcal{A}(\riem_1)}$ is the unique one-form such that 
  \[  \overline{\gamma} - \mathbf{T}_{1,1} \overline{\gamma} = \delta - \mathbf{R}_1  \mathbf{S}_1  \tau.  \]
  The component of this unique $\overline{\gamma}$ in $\overline{\mathbf{R}}_1 \overline{\mathcal{A}(\mathscr{R})}$ is $\overline{\nu}$.   Furthermore the solution depends continuously on the initial data.
 \end{theorem}
 \begin{proof}  
   Assume that there exists a solution $\beta \in \mathcal{A}(\riem_2)$. Then 
   \[    \beta - \mathbf{R}_2 \mathbf{S}_1 \tau - \overline{\mathbf{R}_2} \overline{\mathbf{S}}_1 \overline{\nu} \in \mathcal{A}^\mathrm{e}_{\mathrm{harm}}(\riem_2)  \]
  so by Corollary \ref{co:connected_get_iso} 
  \[   \beta  - \mathbf{R}_2 \mathbf{S}_1  \tau + \mathbf{T}_{1,2} \overline{\nu}  \in \mathcal{A}^\mathrm{e}(\riem_2).  \]
  Thus by Theorem \ref{th:Tonetwo_iso_improved} there is a unique $\overline{\alpha} \in [\mathbf{R}_1 \overline{\mathcal{A}(\mathscr{R})} ]^\perp$ such that 
  \[  - \mathbf{T}_{1,2} \overline{\alpha} = \beta - \mathbf{R}_1 \mathbf{S}_1 \tau + \mathbf{T}_{1,2} \overline{\nu}.       \]
  This implies that 
  \[  \beta- \mathbf{R}_2 \mathbf{S}_1 \tau = - \mathbf{T}_{1,2} [ \overline{\alpha} + \overline{\nu} ].     \]
  Since 
  \[  -\mathbf{O}^\mathrm{e}  \mathbf{T}_{1,2} [ \overline{\alpha} + \overline{\nu} ] =  \overline{\alpha} + \overline{\nu} - \mathbf{T}_{1,1} \left( \overline{\alpha} + \overline{\nu} \right)   \]
  and so
  \[  \mathbf{O}'(\partial \riem_2,\partial \riem_1) \left[ \beta - \mathbf{R}_2 \mathbf{S}_1 \tau \right]  = 
    \left[ \delta - \mathbf{R}_1 \mathbf{S}_1 \tau \right], \]
  this proves that $\delta = \overline{\gamma} - \mathbf{T}_{1,1} \overline{\gamma}$ for $\overline{\gamma} = \overline{\alpha} + \overline{\nu}$ and furthermore establishes that the solution has the claimed form.  Uniqueness follows from Theorem \ref{th:CNT_Dirichlet_problem_oneforms}, observing that the solution is also the solution to the Dirichlet problem with the specified data. 
  
  Conversely, assume that 
  \begin{equation} \label{eq:HBVP_proof_temp} 
   \delta - \mathbf{R}_1 \mathbf{S}_1 \tau = \overline{\gamma} - \mathbf{T}_{1,1} \overline{\gamma}
  \end{equation} 
   for some $\overline{\gamma} \in \overline{\mathcal{A}(\mathscr{R})}$. Let $\overline{\gamma} = \overline{\alpha} + \overline{\nu}$ be the decomposition of $\overline{\gamma}$ with respect to $\overline{\mathcal{A}(\riem_1)} = \overline{\mathbf{R}_1} \overline{\mathcal{A}(\mathscr{R})} \oplus [\overline{\mathbf{R}_1} \overline{\mathcal{A}(\mathscr{R})}]^\perp$.  Then we claim that 
  $\beta  = \mathbf{R}_2 \mathbf{S}_1 \tau - \mathbf{T}_{1,2} \overline{\gamma}$ satisfies 
  $[ \beta] = [ \delta ]$ and has the correct periods. 
  
  To see that $\beta$ has the correct periods, 
  observe that since $\overline{\alpha} \in [\overline{\mathbf{R}_1} \overline{\mathcal{A}(\mathscr{R})} ]^\perp$, $\overline{\mathbf{S}}_1 \overline{\alpha} = 0$ so
  by Theorem \ref{th:Schiffer_cohomology}
  \[  - \mathbf{T}_{1,2} \overline{\gamma} - \overline{\mathbf{R}}_2 \overline{\mathbf{S}}_1 \overline{\nu} = 
     - \mathbf{T}_{1,2} \overline{\gamma} - \overline{\mathbf{R}}_2 \overline{\mathbf{S}}_1 \overline{\gamma}      \]
  is exact, and therefore $- \mathbf{T}_{1,2} \overline{\gamma} + \mathbf{R}_1 \mathbf{S}_1  \tau$ has the 
  specified periods.  To see that the boundary values of $\beta$ are the right ones, we observe that
  \[  \beta = \overline{\mathbf{R}}_2 \overline{\mathbf{S}}_1 \overline{\gamma}
     + \mathbf{R}_2 \mathbf{S}_1 \tau + [ - \mathbf{T}_{1,2} \overline{\gamma} - \overline{\mathbf{R}}_2 \overline{\mathbf{S}}_1 \overline{\gamma} ]  \]
  and then apply Proposition \ref{prop:exact_transmitted_jump} to overfare the quantity in brackets, to show that 
  \begin{equation}  \label{eq:decomposition_foundation}
   [ \beta ] = \left[ \overline{\mathbf{R}}_1 \overline{\mathbf{S}}_1 
   \overline{\gamma} + \mathbf{R}_1 \mathbf{S}_1 \tau \right]  + \left[
     \overline{\gamma} - \mathbf{T}_{1,1} \overline{\gamma} - \overline{\mathbf{R}}_1
      \overline{\mathbf{S}}_1 \overline{\gamma} \right] = [ \delta ]   
  \end{equation}
  where we have used (\ref{eq:HBVP_proof_temp}) in the second equality.\\
  
  Finally we show continuous dependence of the solution on the data.
  Let $ \delta - \mathbf{R}_1 \mathbf{S}_1  \tau  \in  \mathrm{Im} (\mathbf{I} - \mathbf{T}_{1,1}).$ Then $\delta - \mathbf{R}_1 \mathbf{S}_1  \tau = (\mathbf{I} - \mathbf{T}_{1,1})\overline{\gamma}$ and $$\overline{\gamma}= \overline{\mathbf{P}_{\riem_1}}(\delta - \mathbf{R}_1 \mathbf{S}_1  \tau )= \overline{\mathbf{P}_{\riem_1}}\delta.$$ Therefore
  
  \begin{equation}\label{foersta normen}
    \Vert\overline{\gamma} \Vert = \Vert \overline{\mathbf{P}_{\riem_1}}\delta \Vert \lesssim  \Vert \delta\Vert.
  \end{equation}
  Furthermore 
   \begin{equation}\label{andra normen}
    \Vert\tau \Vert \leq  \Vert \tau + \overline{\nu} \Vert \leq  \Vert ({\mathbf{R}}_1
      {\mathbf{S}}_1)^{-1} \zeta\Vert \lesssim  \Vert \zeta\Vert.
  \end{equation} 
  
Thus \eqref{foersta normen} and \eqref{andra normen} and the boundedness of $\mathbf{T}_{1,2} $ and $\mathbf{R}_2  \mathbf{S}_1$ yield that
  \begin{equation*}
    \Vert\beta \Vert = \Vert -\mathbf{T}_{1,2} \overline{\gamma} + \mathbf{R}_2  \mathbf{S}_1 \tau\Vert \leq  \Vert \mathbf{T}_{1,2} \overline{\gamma}\Vert + \Vert  \mathbf{R}_2  \mathbf{S}_1 \tau\Vert \lesssim \Vert \overline{\gamma}\Vert + \Vert \tau\Vert\lesssim \Vert \delta\Vert + \Vert \zeta\Vert,
  \end{equation*}
which shows the continuous dependence of the solution $\beta$ on the initial data $(\delta, \zeta).$ Thus the semi-exact CNT BVP is well-posed in the Bergman space of forms {satisfying condition \ref{eq:one_form_BVP_solvability_condition}. }
 \end{proof}   
\end{subsection}
\end{section}

\clearpage

\printnoidxglossary[sort=def]


\begin{thebibliography}{99}
\bibitem{Ahlfors} Ahlfors, L. V. The complex analytic structure of the space of closed Riemann surfaces,
in Analytic Functions. Princeton University Press, \(1960 .\)
\bibitem{Ahlfors_Sario} Ahlfors, L. V.; Sario, L.
 { Riemann surfaces}.
   { Princeton Mathematical Series}, No. 26 Princeton University Press, Princeton, N.J. 1960.
   
  \bibitem{Aronszajn} Aronszajn, N. Boundary values of functions with finite Dirichlet integral. (Conference on partial differential equations, U. of Kansas (1954). Studies in eigenvalue problems. Technical report 14.
 
  
   \bibitem{AskBar}  Askaripour, N.; and Foth, T. On holomorphic $k$-differentials on open Riemann surfaces. Complex Var. Elliptic Equ. {\bf 57} (2012), no. 10, 1109–-1119.
   
  \bibitem{AskBar2} Askaripour, N.; and Barron, T.  On extension of holomorphic $k$-differentials on open Riemann surfaces. Houston J. Math. {\bf 40} (2014), no. 4, 117--1126.
 
\bibitem{van den ban} van den Ban, E. ; Crainic, M. Analysis on Manifolds Lecture notes for the 2009/2010 Master Class.

   \bibitem{BergmanSchiffer}  Bergman, S.; and Schiffer, M. 
  { Kernel functions and conformal mapping.} Compositio Math. {\bf 8}, (1951), 205--249.
   \bibitem{Booss} Booss-Bavnbek, B.; Wojciechowski, K. P. {Elliptic boundary problems for Dirac operators.} Mathematics: Theory and Applications. Birkhäuser Boston, Inc., Boston, MA, 1993.
   
\bibitem{Bishop} Bishop, C. J. Weil-Petersson curves, conformal energies, $\beta$-numbers, and minimal surfaces. Available at http://www.math.stonybrook.edu/~bishop/papers/wpbeta.pdf.
   \bibitem{brewster_mitrea}Brewster, K.; Mitrea, D.; Mitrea, I.; Mitrea, M. {
Extending Sobolev functions with partially vanishing traces from locally $(\varepsilon,\delta)$-domains and applications to mixed boundary problems.} 
J. Funct. Anal. {\bf 266} (2014), no. 7, 4314--4421.
\bibitem{Chipot}
Chipot, M. Elements of nonlinear analysis. Birkhäuser Advanced Texts: Basler Lehrbücher. [Birkhäuser Advanced Texts: Basel Textbooks] Birkhäuser Verlag, Basel, 2000. 
\bibitem{ConwayII}  Conway, J. B. { Functions of one complex variable }II. Graduate Texts in Mathematics, 159. Springer-Verlag, New York, 1995.

\bibitem{Courant_Schiffer}  Courant, R.  Dirichlet's principle, conformal mapping, and minimal surfaces. With an appendix by M. Schiffer. Reprint of the 1950 original. Springer-Verlag, New York-Heidelberg, 1977.

\bibitem{Douglas} Douglas, J. Solution of the problem of Plateau, Trans. Amer. Math. Soc. {\bf33} (1931), no. 1,
263–321.

\bibitem{Du1} Duff, G. F. D.  Differential forms in manifolds with boundary, Annals of Math. \(\mathbf{5 6} (1952)\),
115-127.
\bibitem{Du2} Duff, G. F. D. Boundary value problems associated with the tensor Laplace equation,
Canadian J. Math. $\mathbf{5}$ (1953), 196-210.
\bibitem{Du3} Duff, G. F. D. A tensor equation of elliptic type, Canadian J. Math. {\bf 5} (1953), 524-535.

\bibitem{Du4} Duff, G. F. D. A tensor boundary value problem of mixed type, Canadian J. Math. \(\mathbf{6}\)
(1954), 427-440.
\bibitem{DS1}  Duff, G. F. D.; Spencer, D. C. Harmonic tensors on manifolds with boundary, Proc.
Nat. Acad. Sci. USA \(\mathbf{3 7} (1951), 614-619 .\)
\bibitem{DS2} Duff, G. F. D.; and Spencer, D. C. Harmonic tensors on Riemannian manifolds with
boundary, Annals of Math. \(\mathbf{5 6} (1952), 128-156\).

\bibitem{EvansGariepy} Evans, L. C.; Gariepy, R. F. Measure theory and fine properties of functions. Studies in Advanced Mathematics. CRC Press, Boca Raton, FL, 1992. 

\bibitem{Eynard_notes} Eynard, B. Lecture notes on Riemann surfaces.  arXiv:1805.06405.


\bibitem{Farkas_Kra}  Farkas, H. M.; and Kra, I. Riemann surfaces. Second edition. Graduate Texts in Mathematics, 71. Springer-Verlag, New York, 1992.

\bibitem{Figalli} Figalli, A. On flows of $H^{3/2}$-vector fields on the circle. Math. Ann. {\bf 347} (2010), no. 1, 43–-57. 


\bibitem{Gay-Balmaz-Ratiu}  Gay-Balmaz, F.; Ratiu, T. S. The geometry of the universal Teichmüller space and the Euler-Weil-Petersson equation. Adv. Math. {\bf 279} (2015), 717–-778.

\bibitem{Gerstenhaber} Gerstenhaber, M.  On  a  theorem on  Haupt and Wirtinger concerning the  periods of  adifferential of the first  kind, anda related topological theorem, Proc. Amer. Math. Soc. {\bf 4} (1953), 476--481.
\bibitem{Gilkey} Gilkey, P. B. Invariance theory, the heat equation, and the Atiyah-Singer index theorem. Second edition. Studies in Advanced Mathematics. CRC Press, Boca Raton, FL, 1995.
\bibitem{Hormander} H\"ormander, L. Linear partial differential operators. Springer Verlag, Berlin-New York, 1976.

\bibitem{J} Jonsson, A. Besov spaces on closed subsets of \(\mathbf{R}^{n}\). - Trans. Amer. Math. Soc. {\bf 341} no. 1, (1994),
\(355--370 .\)

\bibitem{JW} Jonsson, A.; Wallin, H. Function spaces on subsets of \(\mathbf{R}^{n}\). - Math. Rep. {\bf 2} no. 1, 1984.

\bibitem{KY2}
Kirillov A.~A.; and Yuriev, D.~V. { Representations of the Virasoro
algebra by the orbit method}, J. Geom. Phys. {\bf 5} (1988), no.
3, 351--363.

\bibitem{Koskela} Koskela, Pekka; Yang, Dachun; Zhou, Yuan Pointwise characterizations of Besov and Triebel-Lizorkin spaces and quasiconformal mappings. Adv. Math. {\bf 226} (2011), no. 4, 3579--3621. 
\bibitem{KrillovYuriev} A. A. Kirillov,  and D. V. Yuri'ev, { K\"ahler geometry of the infinite-dimensional homogeneous
space $M = \mathrm{Diff}^+ (\mathbb{S}^1)/\mathrm{rot}(\mathbb{S}^1)$}, Funktsional. Anal. i Prilozhen. {\bf 21} (1987) no. 4, 35--46.
\bibitem{Lehto}
Lehto, O. {Univalent functions and {T}eichm\"uller spaces}. Graduate Texts
  in Mathematics, Vol. 109, Springer-Verlag, New York, 1987.
  
\bibitem{ME1} Morrey Jr., C. B.; Eells Jr., J. A variational method in the theory of harmonic
integrals, Proc. Nat. Acad. Sci. $\mathbf{41}$ (1955), 391--395 .

\bibitem{ME2}Morrey Jr., C. B.; Eells Jr., J. A variational method in the theory of harmonic
integrals I, Annals of Math. $\mathbf{63}$ (1956), 91--128 .

\bibitem{Nagbook} Nag, S.  {The Complex Analytic Theory of
 {T}eichm\"uller Spaces}, Canadian Mathematical Society Series of
  Monographs and Advanced Texts.  John Wiley \& Sons, Inc., New York, 1988.
  
\bibitem{NS}  Nag, S.; Sullivan, D. Teichm\"uller theory and the universal period mapping via quantum calculus and the $H^{1/2}$ space on the circle. Osaka J. Math. {\bf32} (1995), no. 1, 1--34. 
\bibitem{Nap_Yulm} Napalkov, V. V., Jr.; Yulmukhametov, R. S. On the Hilbert transform in the Bergman space. (Russian) Mat. Zametki {\bf 70} (2001), no. 1, 68--78; translation in Math. Notes {\bf 70} (2001), no. 1-2, 61--70.
\bibitem{Osborn}  H. Osborn, {  The Dirichlet functional} I. J. Math. Anal. Appl. {\bf 1} (1960), 61 -- 112.
 \bibitem{Pommerenke_boundary_behaviour}  Pommerenke, Ch.  { Boundary behaviour of conformal maps.}
  Grundelehren der mathematischen Wissenschaften v 299, Springer-Verlag, 1991.
   \bibitem{RadnellSchippers_monster} Radnell, D. and Schippers, E. {Quasisymmetric sewing in rigged Teichmüller space}, Commun. Contemp. Math. {\bf{8}} (2006), no. 4, 481--534.

 \bibitem{RS_fiber}  Radnell, D. and Schippers, E.  { Fiber structure and local coordinates for the Teichm\"uller space of a bordered Riemann surface}, Conform. Geom. Dyn. {\bf 14} (2010), 14--34.
 
 
\bibitem{RSS_genus_zero}  Radnell, D.; Schippers, E.; and Staubach, W. { A Model of the Teichmüller space of genus-zero bordered surfaces by period maps}. Conform. Geom. Dyn. {\bf 23} (2019), 32-–51. 
   
\bibitem{RSS_Dirichletspace} \bysame, { Dirichlet spaces of domains bounded by quasicircles}, Commun. Contemp. Math. {\bf 22} (2020), no. 3.


\bibitem{Riemann}Riemann, B. Theorie der Abelschen Functionen, J. Reine Angew. Math., {\bf 54} (1857), pp. 115–155; Werke, pp. 88-142.
   
   \bibitem{Royden}  Royden, H. L. {Function theory on compact Riemann surfaces}. J. Analyse Math. {\bf 18}, (1967), 295--327. 
   
   
\bibitem{Roydenperiods} Royden, H. L. The Variation of Harmonic Differentials and their Periods. Complex analysis, 211–223, Birkh\"auser, Basel, 1988. 

\bibitem{Rudin}  Rudin, W. Functional analysis. Second edition. International Series in Pure and Applied Mathematics. McGraw-Hill, Inc., New York, 1991. 
\bibitem{Schiffer_first} Schiffer, M. {The kernel function of an orthonormal system}. Duke Math. J. {\bf 13}, (1946). 529 -- 540.
 \bibitem{Schiffer_Spencer} Schiffer, M. and Spencer, D.  {Functionals on finite Riemann surfaces}. Princeton University Press, Princeton, N. J., 1954. 
\bibitem{Schippers_Shirazi_Staubach}  Schippers, E.; Shirazi, M.; and Staubach, W. Schiffer comparison operators and approximations on Riemann surfaces bordered by quasicircles. To appear in The Journal of Geometric Analyis.

\bibitem{Schippers_Staubach_PAMS} Schippers, E.; Staubach, W. A symplectic functional analytic proof of the conformal welding theorem. (English summary) 
Proc. Amer. Math. Soc. 143 (2015), no. 1, 265–278. 

\bibitem{Schippers_Staubach_JMAA} Schippers, E.; Staubach, W. Harmonic reflection in quasicircles and well-posedness of a Riemann-Hilbert problem on quasidisks. J. Math. Anal. Appl. {\bf448} (2017), no. 2, 864--884.


\bibitem{Schippers_Staubach_CAOT} Schippers, E.; Staubach, W. Riemann boundary value problem on quasidisks, Faber isomorphism and Grunsky operator. Complex Anal. Oper. Theory {\bf12} (2018), no. 2, 325--354.  

\bibitem{Schippers_Staubach_transmission}  Schippers, E.; Staubach, W. Transmission of harmonic functions through quasicircles on compact Riemann surfaces. Ann. Acad. Sci. Fenn. Math. {\bf 45} (2020), no. 2, 1111-–1134. .

\bibitem{Schippers_Staubach_Plemelj} Schippers, E.; Staubach, W. Plemelj-Sokhotski isomorphism for quasicircles in Riemann surfaces and the Schiffer operator. Math. Ann. {\bf 378} (2020), no. 3-4, 1613-–1653.

 \bibitem{Schippers_Staubach_transmission_sphere} Schippers, E. and Staubach, W. {Harmonic reflection in quasicircles and well-posedness of a Riemann-Hilbert problem on quasidisks}. J. Math. Anal. Appl. {\bf 448} (2017), no. 2, 864--884.

\bibitem{Schippers_Staubach_Grunsky_expository}  Schippers, E.; and Staubach, W.  Analysis on quasicircles-A unified approach through transmission and jump problems. To appear in EMS surveys in Mathematical Sciences. Available on Researchgate.


\bibitem{Schippers_Staubach_Shirazi} Schippers, E., Shirazi, M. and Staubach, W.  Schiffer comparison operators and approximations on Riemann surfaces bordered by quasicircles,  J. Geom. Anal. {\bf 31} (2021), no. 6, 5877–-5908.
 
 \bibitem{Sch} Schwarz, G. Hodge Decomposition- \(A\) method for solving boundary value problems, Lecture Notes in Mathematics, No. 1607, Springer-Verlag, \(1995 .\)

\bibitem{Shirazi_thesis}Shirazi, M. Faber and Grunsky Operators on Bordered Riemann Surfaces of Arbitrary Genus and the Schiffer Isomorphism". PhD Dissertation, University of Manitoba 2020. 
\bibitem{Shirazi_Grunsky} Shirazi, M. Faber and Grunsky Operators Corresponding to Bordered Riemann Surfaces, Conform. Geom. Dyn. {\bf 24} (2020), 177--201.
\bibitem{Sp} Spencer, D. C. Real and complex operators on manifolds, Contributions to the theory
of Riemann surfaces, Princeton University Press, 1953, pp. 204--227.

\bibitem{Takhtajan_Teo_Memoirs} Takhtajan, L.; Teo, L-P. Weil-Petersson Metric on the Universal Teichm\"uller Space,
 Memoirs of the American Mathematical Society. {\bf 183}  (2006) no. 861.
 
 \bibitem{Taylor}  Taylor, Michael E. Partial differential equations I.  Basic theory. Second edition. Applied Mathematical Sciences, 115. Springer, New York, 2011.
 
 
\bibitem{Triebel} Triebel, H. Theory of function spaces. Reprint of 1983 edition. Also published in 1983 by Birkhäuser Verlag. Modern Birkhäuser Classics. Birkhäuser/Springer Basel AG, Basel, 2010.

\bibitem{Vodopyanov} Vodop'yanov, K. Mappings of homogeneous groups and imbeddings of functional spaces, Siberian Math. Zh. {\bf 30} (1989), 25--41.
\end{thebibliography}
\end{document}